\documentclass[12pt,reqno]{amsbook}
\usepackage[utf8]{inputenc}

\usepackage{graphicx}
\usepackage{qtree}  
\graphicspath{{./Figures/}}

\usepackage{tikz}
\usepackage{tikz-qtree}

\usepackage[
backend=biber,
style=numeric
]{biblatex}  
\usepackage{amsmath,a4wide}
\usepackage{xfrac}
\usepackage{stmaryrd,mathrsfs,bm, bbm,amsthm,mathtools,yfonts,amssymb,color}
\usepackage{enumerate}
\usepackage[shortlabels]{enumitem}
\usepackage{xcolor}
\usepackage[citecolor=black,linkcolor=blue,colorlinks=true]{hyperref}  
\usepackage{hyperref}
\usepackage{caption}
\usepackage{subcaption}
\usepackage{graphicx}
\usepackage{qtree}  
\graphicspath{{./Figures/}}

\newcommand{\R}{\ensuremath{\mathbb R}}

\renewcommand{\div}{\operatorname{div}}
\newcommand{\al}{{\bar\alpha}}
\DeclareMathOperator{\curl}{curl}
\DeclarePairedDelimiter\abs\lvert\rvert
\DeclarePairedDelimiter\norm\lVert\rVert

\numberwithin{equation}{chapter}

\begingroup
\newtheorem{theorem}{Theorem}[section]
\newtheorem{lemma}[theorem]{Lemma}
\newtheorem{proposition}[theorem]{Proposition}
\newtheorem{corollary}[theorem]{Corollary}
\newtheorem{claim}[theorem]{Claim}
\endgroup

\theoremstyle{definition}
\newtheorem{definition}[theorem]{Definition}
\newtheorem{remark}[theorem]{Remark}

\title{Instability and nonuniqueness for the $2d$ Euler equations in vorticity form, after~M.~Vishik}

\author[Albritton]{Dallas Albritton}
\address[Dallas Albritton]{School of Mathematics, Institute for Advanced Study, 1 Einstein Dr., Princeton NJ 08540, USA}
\email{dallas.albritton@ias.edu}
\thanks{DA was supported by NSF Postdoctoral Fellowship Grant No. 2002023 and Simons Foundation Grant No.\ 816048.}

\author[Bru{\'e}]{Elia Bru\'e}
\address[Elia Bru\'e]{School of Mathematics, Institute for Advanced Study, 1 Einstein Dr., Princeton NJ 08540, USA}
\email{elia.brue@ias.edu}
\thanks{EB was supported by the Giorgio and Elena Petronio Fellowship}

\author[Colombo]{Maria Colombo}
\address[Maria Colombo]{Institute of Mathematics, EPFL SB, Station 8, CH-1015 Lausanne, Switzerland}
\email{maria.colombo@epfl.ch}
\thanks{MC was supported by the SNSF Grant 182565.}

\author[De Lellis]{Camillo De Lellis}
\address[Camillo De Lellis]{School of Mathematics, Institute for Advanced Study, 1 Einstein Dr., Princeton NJ 08540, USA}
\email{camillo.delellis@ias.edu}
\thanks{CDL was supported by the NSF under Grant No.
DMS-1946175.}

\author[Giri]{Vikram Giri}
\address[Vikram Giri]{Department of Mathematics, Princeton University, Washington Rd., Princeton, NJ 08544, USA}
\email{vgiri@math.princeton.edu}
\thanks{VG was supported by the NSF under Grant No. DMS-FRG-1854344.}

\author[Janisch]{Maximilian Janisch}
\address[Maximilian Janisch]{Institut für Mathematik, Universität Zürich, Winterthurerstrasse 190, 8057 Zürich, Switzerland}
\email{mail@maximilianjanisch.com}

\author[Kwon]{Hyunju Kwon}
\address[Hyunju Kwon]{School of Mathematics, Institute for Advanced Study, 1 Einstein Dr., Princeton NJ 08540, USA}
\email{hkwon@ias.edu}
\thanks{HK was supported by the NSF under Grant No. DMS-1926686.}

\makeindex

\begin{document}

\maketitle

\tableofcontents



\chapter{Introduction}

In these notes we will consider the Euler equations in the $2$-dimensional space in vorticity formulation\index{Euler equations@Euler equations}, which are given by
\begin{equation}\label{e:Euler}
\left\{
\begin{array}{ll}
&\partial_t \omega + (v\cdot \nabla) \omega = f\\ \\
& v (\cdot, t) = K_2 * \omega (\cdot, t)
\end{array}
\right.
\end{equation}
where $K_2$ is the usual $2$-dimensional Biot-Savart kernel and $f$ is a given external vorticity source, which we often call the force. \index{Biot-Savart kernel@Biot-Savart kernel}\index{aalf@$f$}\index{aalK_2@$K_2$}\index{external force@external force}\index{force, external@force, external}
$v$ is the velocity field, and it is a function defined on a space-time domain of type $\mathbb R^2 \times [0, T]$. By the Biot-Savart law we have $\omega = \curl v = \partial_{x_2} v_1 - \partial_{x_1} v_2 = \nabla\times v$\index{vorticity@vorticity}\index{velocity@velocity}\index{aagz@$\omega$}\index{aalv@$v$}.

\index{Cauchy problem@Cauchy problem}
We will study the Cauchy problem for \eqref{e:Euler} with initial data 
\begin{equation}\label{e:Cauchy}
\omega (\cdot, 0) = \omega_0\, 
\end{equation}
on the domain $\mathbb R^2 \times [0,\infty[$
under the assumptions that
\begin{itemize}
    \item[(i)] $\omega_0\in L^1\cap L^p$ for some $p>2$ and $v_0=K_2* \omega_0 \in L^2$;
    \item[(ii)] $f\in L^1 ([0,T], L^1\cap L^p)$ and $K_2*f\in L^1 ([0,T], L^2)$ for every $T<\infty$.
\end{itemize}
In particular we understand solutions $\omega$ in the usual sense of distributions, namely,
\begin{equation}\label{e:distrib}
\int_0^T \int_{\R^2} [\omega (\partial_t \phi + K_2* \omega \cdot \nabla \phi) + f \phi]\, dx\, dt
= - \int_{\R^2} \phi (x,0)\, \omega_0 (x)\, dx
\end{equation}
for every smooth test function $\phi\in C^\infty_c (\mathbb R^2 \times [0,T[)$.\index{Solution, weak} In view of (i)-(ii) and standard energy estimates we will restrict our attention to weak solutions which satisfy the following bounds:
\begin{itemize}
    \item[(a)] $\omega \in L^\infty ([0,T], L^1\cap L^p)$ and $v\in L^\infty ([0,T], L^2)$ for every $T<\infty$.
\end{itemize}
The purpose of these notes is to give a proof of the following:

\begin{theorem}\label{thm:main}
For every $p\in ]2, \infty[$ there is a triple $\omega_0, v_0$, and $f$ satisfying (i)-(ii)
with the property that there are uncountably many solutions $(\omega, v)$ of \eqref{e:Euler} and \eqref{e:Cauchy} on $\mathbb R^2\times [0, \infty [$ which satisfy the bound (a). Moreover, $\omega_0$ can be chosen to vanish identically.
\end{theorem}

In fact, the $f$ given by the proof is smooth and compactly supported on any closed interval of time $[\varepsilon, T]\subset ]0, \infty[$. Moreover, a closer inspection of the argument reveals that any of the solutions $(\omega, v)$ enjoy bounds on the $W^{1,4}_\textrm{loc}$ norm of $\omega (t, \cdot)$, and good decay properties at infinity, whenever $t$ is positive (and obviously such estimates degenerate as $t\downarrow 0$). In particular $v$ belongs to $C^1_\textrm{loc} (\mathbb R^2\times ]0, \infty[)$. It is not difficult to modify the arguments detailed in these notes to produce examples which have even more regularity and better decay for positive times, but we do not pursue the issue here.

\begin{remark}\label{r:bounded}\label{R:BOUNDED}
Recall that
\begin{equation}\label{e:bound-on-Biot-Savart}
\|K_2* \omega (\cdot, t)\|_{L^\infty}\leq C (p) (\|\omega (\cdot, t)\|_{L^1} + \|\omega (\cdot, t)\|_{L^p})
\end{equation}
whenever $p>2$ (cf. the Appendix for the proof). Therefore we conclude that each solution $v$ in Theorem \ref{thm:main} is bounded on $\mathbb R^2\times [0,T]$ for every positive $T$.
\end{remark}

The above groundbreaking result was proved by Vishik in the two papers \cite{Vishik1} and \cite{Vishik2} (upon which these notes are heavily based) and answers a long-standing open question in the PDE theory of the incompressible Euler equations, as it shows that it is impossible to extend to the $L^p$ scale the following classical uniqueness result of Yudovich. \index{Yudovich uniqueness Theorem@Yudovich uniqueness theorem}

\begin{theorem}\label{thm:Yudo}\label{THM:YUDO}
Consider a strictly positive $T$, an initial vorticity $\omega_0 \in L^1\cap L^\infty$ with $v_0=K_2*\omega_0 \in L^2$ and an external force $f\in L^1 ([0,T]; L^1\cap L^\infty)$ with $K_2*f\in L^1 ([0,T]; L^2)$. Then there is a unique solution $\omega$ of \eqref{e:Euler} and \eqref{e:Cauchy} on $\mathbb R^2\times [0, T]$ satisfying the estimates $\omega \in L^\infty ([0,T], L^1\cap L^\infty)$ and $v = K_2* \omega \in L^\infty ([0,T], L^2)$.
\end{theorem}

The above theorem in a bounded domain was originally proven by Yudovich in 1963 \cite{Yudovich1963}, who also proved a somewhat more technical statement on unbounded domains. We have not been able to find an exact reference for the statement above (cf. for instance \cite[Theorem 8.2]{MajdaBertozzi} and the paragraph right afterwards, where the authors point out the validity of the Theorem in the case of $f=0$). We therefore give a detailed proof in the appendix for the reader's convenience. When $f=0$, Theorem~\ref{thm:Yudo} has been extended to the Yudovich space of functions whose $L^p$ norms are allowed to grow moderately as $p \to +\infty$~\cite{yudovich1995uniqueness} and borderline Besov spaces containing $\textrm{BMO}$~\cite{VISHIK1999}.

\begin{remark}\label{r:A-priori-estimates}
We recall that the solution of Theorem~\ref{thm:Yudo} satisfies a set of important a priori estimates, which can be justified using the uniqueness part and a simple approximation procedure. Indeed if $(\omega, v)$ is a smooth solution of \eqref{e:Euler}, then the method of characteristics shows that, for every $t$, there exists a family of volume-preserving diffeomorphisms $T_s:\R^2\to\R^2, s\in[0, t]$, such that \begin{equation*}
        \omega(x, t)=\omega_0(T_0 x) + \int_0^t f(T_s x, s)\,\mathrm ds.
    \end{equation*}
    Therefore, since volume-preserving diffeomorphisms preserve all $L^q$ norms, we get, for all $q\in[1,\infty]$,
    \begin{equation*}
        \norm{\omega(\cdot, t)}_{L^q}\le\norm{\omega_0}_{L^q}+\int_0^t \norm{f(\cdot, s)}_{L^q}\,\mathrm ds.
    \end{equation*}
    Furthermore, a usual integration by parts argument, as seen in \cite[Lemma 1.1]{Yudovich1963}, shows that $v$ satisfies the estimate
    \begin{equation*}
        \norm{v(\cdot, t)}_{L^2}\le\norm{v_0}_{L^2}+\int_0^t \norm{K_2* f(\cdot,s)}_{L^2}\,\mathrm ds.
    \end{equation*}
\end{remark}

\begin{remark}\label{r:well-defined}
Recall that the Biot-Savart kernel is given by the formula
\begin{equation}\label{e:Biot-Savart}
K_2 (x_1, x_2) = \frac{x^\perp}{2\pi |x|^2} =  \frac{1}{2\pi |x|^2} (-x_2, x_1)\, .
\end{equation}
In particular, while $K_2\not \in L^p$ for any $p$, it can be easily broken into 
\begin{equation}\label{e:decomposition-Biot-Savart-Kernel}
K_2 = K_2 \mathbf{1}_{B_1} + K_2 \mathbf{1}_{B_1^c}\, ,
\end{equation}
where $B_1$ denotes the unit ball around $0$.
Observe that $K_2 \mathbf{1}_{B_1} \in L^q$ for every $q\in [1,2[$ and $K_2 \mathbf{1}_{B_1^c}\in L^r$ for every $r\in ]2, \infty]$. Under the assumption that $\omega \in L^{2-\delta}$ for some positive $\delta >0$, this decomposition allows us to define the convolution $K_2* \omega$ as $(K_2 \mathbf{1}_{B_1}) * \omega 
+ (K_2 \mathbf{1}_{B_1^c}) * \omega$, where each separate summand makes sense as Lebesgue integrals thanks to Young's convolution inequality.\footnote{Young's convolution inequality states that, if $g_1\in L^{p_1}$ and $g_2\in L^{p_2}$ with $1\leq \frac{1}{p_1} + \frac{1}{p_2} \leq 2$, then $g_1 (y-\cdot) g_2 (\cdot)$ belongs to $L^1$ for a.e. $y$ and $g_1* g_2\in L^r$ for $\frac{1}{r}=\frac{1}{p_1} + \frac{1}{p_2} -1$.}

On the other hand we caution the reader that, for general $\omega\in L^2$, $K_2*\omega$ may not be well-defined. More precisely, if we denote by $\mathscr{S}$\index{aalSscript@$\mathscr{S}$} the \index{Schwartz space@Schwartz space $\mathscr{S}$} Schwartz space of rapidly decaying smooth functions and by $\mathscr{S}'$\index{aalSscript'@$\mathscr{S}'$} the space of tempered distributions \index{tempered distribution@tempered distribution}\index{space of tempered distributions@space $\mathscr{S}'$ of tempered distributions} (endowed, respectively, with their classical Fr\'echet and weak topologies), it can be shown that there is no continuous extension of the operator $\mathscr{S} \ni \omega \mapsto K_2 *\omega\in \mathscr{S}'$ to a continuous operator from $L^2$  to $\mathscr{S}'$, cf. Remark~\ref{r:Camillo_dumb}. 

This fact creates some technical issues in many arguments where we will indeed need to consider a suitable continuous extension of the operator $\omega \mapsto K_2*\omega$ to {\em some} closed linear subspace of $L^2$, namely, $m$-fold rotationally symmetric functions in $L^2$ (for some integer $m\geq 2$). Such an extension will be shown to exist thanks to some special structural properties of the subspace.  
\end{remark}

\section{Idea of the proof}\label{sec:ideaproof}

We now describe, briefly, the rough idea of and motivation for the proof. An extensive description of the proof with precise statements can be found in Chapter~\ref{chapter:general}, which breaks down the whole argument into three separate (and independent) parts. The subsequent three chapters are then dedicated to the detailed proofs.

First, we recall two essential features of the two-dimensional Euler equations:
\begin{enumerate}
    \item \emph{Steady states}. The two-dimensional Euler equations possess a large class of explicit, radially symmetric steady states called \emph{vortices}:
\index{aagzbar@$\bar\omega$}
    \begin{equation}
        \label{eq:vorticesdef}
        \bar{\omega}(x) = g(|x|), \quad \bar{v}(x) = \zeta(|x|) x^\perp.
    \end{equation}
    \item \emph{Scaling symmetry}. The Euler equations possess a two-parameter scaling symmetry: If $(\omega,v)$ is a solution of~\eqref{e:Euler} with vorticity forcing $f$, and $\lambda, \mu > 0$, then
\begin{equation}
	\omega_{\lambda,\mu}(x,t) = \mu \omega(\lambda x, \mu t), \quad v_{\lambda,\mu}(x,t) = \frac{\mu}{\lambda} v(\lambda x, \mu t),
\end{equation}
define a solution with vorticity forcing\textbf{}
\begin{equation}
    f_{\lambda,\mu}(x,t) = \mu^2 f(\lambda x, \mu t).
\end{equation}
The scaling symmetry corresponds to the physical dimensions
\begin{equation}
[x] = L,\quad [t] = T, \quad [v] = \frac{L}{T}, \quad [\omega] = \frac{1}{T}, \quad \text{ and } \quad  [f] = \frac{1}{T^2}.
\end{equation}
\end{enumerate}

We now elaborate on the above two features:

\smallskip

\emph{1. Unstable vortices}. The stability analysis of shear flows $u = (b(y),0)$ and vortices~\eqref{eq:vorticesdef} is classical, with seminal contributions due to Rayleigh~\cite{Rayleigh1879}, Kelvin~\cite{Thomson1880},  Orr~\cite{Orr}, and many others. The linearized Euler equations around the background vortex $\bar{\omega}$ are
\begin{equation}
    \label{eq:linearizedeulerintro}
    \partial_t \omega - L_\textrm{st} \omega := 
    \partial_t \omega + \zeta(r) \partial_\theta \omega + (v \cdot e_r) g'(r) = 0, \quad v = K_2 \ast \omega.
\end{equation}
Consider the eigenvalue problem associated to the linearized operator $L_\textrm{st}$. It suffices to consider $\psi = e^{im\theta} \psi_m(|x|)$, $m \geq 0$, the stream function associated to a vorticity perturbation $\omega$ (that is, $\Delta \psi = \omega$). It is convenient to pass to an exponential variable $s = \log r$ and define $\phi(s) = \psi_m(e^s)$; $A(s) = e^s g'(e^s)$ ($r \; \times$ the radial derivative of the background vorticity); and $\Xi(s) = \zeta(e^s)$ (the differential rotation). The eigenvalue problem for $L_\textrm{st}$, with eigenvalue $\lambda = -imz$, can be rewritten as
\begin{equation}
    \label{eq:theeigenvaluequation}
    \left( \Xi(s) - z \right) \left( \frac{d^2}{ds^2}- m^2 \right) \phi - A(s) \phi = 0.
\end{equation}
This is \emph{Rayleigh's stability equation}. The eigenvalue $\lambda$ is unstable when $\textrm{Im}(z) > 0$, in which case we can divide by $\Xi - z$ and analyze a steady Schr{\"o}dinger equation. It is possible to understand~\eqref{eq:theeigenvaluequation} well enough to design vortices for which the corresponding linear operator has an unstable eigenfunction. The unstable modes are found to be bifurcating from neutral modes. For shear flows, this analysis goes back to Tollmien~\cite{Tollmien}, who calculated the necessary relationship between $z$ and $m$ for such a bifurcation curve to exist. The problem was treated rigorously by Z.~Lin~\cite{LinSIMA2003} for shear flows in a bounded channel and $\R^2$ and rotating flows in an annulus. We recognize also the important work of Faddeev~\cite{faddeev1973theory}. For those interested in hydrodynamic stability more generally, see the classic monograph~\cite{DrazinReid}. Chapter~4 therein concerns the stability of shear flows, including Rayleigh's criterion and a sketch of Tollmien's idea. See, additionally,~\cite{LinCC} and~\cite[Chapter 2]{Schmid2001}.

The case of vortices on $\R^2$, which is the crucial one for the purposes of these notes, was treated by Vishik in~\cite{Vishik2}, see Chapter~\ref{chapter:linearpartii} below. In the cases relevant to these notes, $L_\textrm{st}$ has at least one unstable eigenvalue $\lambda$. While the latter could well be real, for the sake of our argument let us assume that it is a complex number $\lambda= a_0 + b_0 i$ ($a_0, b_0 > 0$) and let $\bar\lambda = a_0- b_0 i$ be its complex conjugate. If we denote by $\eta$ and $\bar{\eta}$ two corresponding (nontrivial) eigenfunctions, it can be checked that they are {\em not} radially symmetric.

With the unstable modes in hand, one may seek a trajectory on the \emph{unstable manifold} associated to $\lambda$ and $\bar{\lambda}$. For example, one such trajectory may look like
\begin{equation}
    \omega = \bar{\omega} + \omega^\textrm{lin} + o(e^{a_0 t}),
\end{equation}
where $\omega^\textrm{lin} = \textrm{Re}(e^{\lambda t} \eta)$ is a solution of the linearized Euler equations~\eqref{eq:linearizedeulerintro}. These solutions converge to $\bar{\omega}$ exponentially in backward time.
Hence, we expect that certain unstable vortices exhibit a kind of \emph{non-uniqueness at time $t = -\infty$} and moreover break the radial symmetry. The existence of unstable manifolds associated to a general class of Euler flows in dimension $n \geq 2$ was demonstrated by Lin and Zeng~\cite{LinZengCPAM2013,LinZengCorrigendum2014}. There is a substantial mathematical literature on the nonlinear instability of Euler flows, see~\cite{friedlanderstraussvishikearly,FriedlanderHoward,friedlanderstraussvishik,bardosguostrauss,friedlandervishik,linnonlinear}.

\smallskip

\emph{2. Self-similar solutions}. It is natural to consider solutions invariant under the scaling symmetry and, in particular, it is natural to consider those self-similar solutions which live exactly at the desired integrability. If we fix a relationship $L^\alpha \sim T$ in the scaling symmetries, the similarity variables are
\begin{equation}\label{e:self-similar-scaling}
    \xi = \frac{x}{t^{\frac{1}{\alpha}}}, \quad \tau = \log t
\end{equation}
\begin{equation}
    v(x,t) = \frac{1}{t^{1-\frac{1}{\alpha}}} V(\xi, \tau), \quad \omega(x,t) = \frac{1}{t} \Omega(\xi, \tau).
\end{equation}
{We may regard the logarithmic time as $\tau = \log (t/t_0)$, so that $t$ is non-dimension- alized according to a fixed reference time $t_0 = 1$.} Notice that physical time $t=0$ corresponds to logarithmic time $\tau = -\infty$. The function $\Omega$ is known as the \emph{profile}.
The Euler equations, without force, in similarity variables are
\begin{equation}
    \label{eq:similarityvareulereqns}
    \left\{\begin{array}{ll}
    &\partial_\tau \Omega - \big( 1 + \frac{\xi}{\alpha} \cdot \nabla_\xi \big) \Omega + V \cdot \nabla_\xi \Omega = 0\\ \\
    & V= K_2* \Omega \, .
    \end{array}\right.
\end{equation}
Profiles $\Omega$ satisfying $\| \Omega(\cdot,\tau) \|_{L^p} = O(1)$ as $\tau \to -\infty$ satisfy $\| \omega(\cdot,t) \|_{L^p} = O(t^{-1+\frac{2}{\alpha p}})$ as $t \to 0^+$, and similarly in the weak $L^p$ norms. Hence, the Lebesgue and weak Lebesgue norms with $p = 2/\alpha$ would be $O(1)$ in either variables. To show sharpness of the Yudovich class, we consider $0 < \alpha \ll 1$.

\bigskip

The route to non-uniqueness through unstable vortices and self-similar solutions is as follows: Suppose that $\bar{\Omega}$ is an unstable steady state of the similarity variable Euler equations~\eqref{eq:similarityvareulereqns} (in particular, $\bar{\omega}(x,t) = t^{-1} \bar{\Omega}(\xi)$ is a self-similar solution of the usual Euler equations). Find a trajectory $\Omega$ on the unstable manifold associated to $\bar{\Omega}$. In similarity variables, the steady state $\bar{\Omega}$ will be ``non-unique at minus infinity", which corresponds to non-uniqueness at time $t=0$ in the physical variables.

One natural class of background profiles $\bar{\Omega}$ consists of  \emph{power-law vortices} $\bar{\omega} = \beta |x|^{-\alpha}$, $\beta \in \R$, which are simultaneously steady solutions and self-similar solutions without force. At present, we do not know whether the above strategy can be implemented with power-law vortices.

Instead, we choose a smooth vortex profile $g(|x|)$, with power-law decay as $|x| \to +\infty$, which is unstable for the Euler dynamics. Our background will be the self-similar solution with profile $\bar{\Omega} = g(|\xi|)$, which solves the Euler equations \emph{with a self-similar force}. This profile may be considered a well-designed smoothing of a power-law vortex. When the background is large, it is reasonable to expect that the additional term in the similarity variable Euler equations~\eqref{eq:similarityvareulereqns} can be treated perturbatively, so that $g(|\xi|)$ will also be unstable for the  similarity variable Euler dynamics. This heuristic is justified in Chapter~\ref{chapter:linearparti}.

In order to ensure that the solutions have finite energy, we also truncate the background velocity at distance $O(1)$ in physical space. This generates a different force. The truncation's contribution to the force is smooth and heuristically does not destroy the non-uniqueness, which can be thought of as ``emerging" from the singularity at the space-time origin. Our precise Ansatz is~\eqref{e:Ansatz-1}, which is the heart of the nonlinear part of these notes.

\section{Differences with Vishik's work}

While we follow the strategy of Vishik in~\cite{Vishik1,Vishik2}, we deviate from his proof in some ways. We start by listing two changes which, although rather minor, affect the presentation substantially.
\begin{enumerate}
    \item We decouple the parameter $\alpha$ in~\eqref{e:self-similar-scaling}  governing the self-similar scaling from the decay rate $\bar{\alpha}$ of the smooth profile $g$ at infinity. In \cite{Vishik1} these two parameters are equal; however, it is rather obvious that the argument goes through as long as $\alpha \leq \bar \alpha$. If we then choose $\alpha < \bar \alpha$ the resulting solution has zero initial data. This is a very minor remark, but it showcases the primary role played by the forcing $f$ in the equation.
    \item  Strictly speaking Vishik's Ansatz for the ``background solution'' is in fact different from our Ansatz (even taking into account the truncation at infinity). The interested reader might compare \eqref{e:tilde-v} and \eqref{e:curl-tilde-v} with \cite[(6.3)]{Vishik1}. Note in particular that the coordinates used in \cite{Vishik1} are not really \eqref{e:self-similar-scaling} but rather a more complicated variant. Moreover, Vishik's Ansatz contains a parameter $\varepsilon$, whose precise role is perhaps not initially transparent, and which is ultimately scaled away in~\cite[Chapter 9]{Vishik1}. This obscures that the whole approach hinges on finding a solution $\Omega$ of a truncated version of \eqref{eq:similarityvareulereqns}
    asymptotic to the unstable manifold of the steady state $\bar \Omega$ at $-\infty$.
    In our case, $\Omega$ is constructed by solving appropriate initial value problems for the truncated version of \eqref{eq:similarityvareulereqns} at negative times $-k$ and then taking their limit; this plays the role of Vishik's parameter~$\varepsilon$.
\end{enumerate}

We next list two more ways in which our notes deviate from \cite{Vishik1,Vishik2}. These differences are much more substantial. 
\begin{enumerate}
    \item[(3)] The crucial nonlinear estimate in the proof of Theorem \ref{thm:main} (cf. \eqref{e:asymptotic-in-t} and the more refined version \eqref{e:asymptotic-in-t-2}), which shows that the solution $\Omega$ is asymptotic, at minus infinity, to an unstable solution of the linearized equation, is proved in a rather different way. In particular our argument is completely Eulerian and based on energy estimates, while a portion of Vishik's proof relies in a crucial way on the Lagrangian formulation of the equation. The approach introduced here is exploited by the first and third author in~\cite{AC}, and we believe it might be useful in other contexts.
    \item[(4)] Another technical, but crucial, difference, concerns the simplicity of the unstable eigenvalue $\eta$. While Vishik claims such simplicity in \cite{Vishik2}, {the argument given in the latter reference is actually incomplete. After we pointed out the gap to him, he provided a clever way to fill it in \cite{Vishik3}}. These notes point out that such simplicity is not really needed in the nonlinear part of the analysis: in fact a much weaker linear analysis than the complete one carried in \cite{Vishik2} {is already enough to close the argument for Theorem \ref{thm:main}. However, for completeness and for the interested readers, we include in Appendix~\ref{a:better} the necessary additional arguments needed to conclude the more precise description of \cite{Vishik2}.}
\end{enumerate}

\section{Further remarks}\label{s:final-remarks}

Recently, Bressan, Murray, and Shen investigated in \cite{BressanAposteriori,BressanSelfSimilar} a different non-uniqueness scenario for~\eqref{e:Euler} which would demonstrate sharpness of the Yudovich class without a force. The scenario therein, partially inspired by the works of Elling~\cite{EllingAlgebraicSpiral,EllingSelfSimilar}, is also based on self-similarity and symmetry breaking but follows a different route.

Self-similarity and symmetry breaking moreover play a central role in the work of Jia, {\v S}ver{\'a}k, and Guillod~\cite{JiaSverakInventiones,JiaSverakIllposed,guillod2017numerical} on the conjectural non-uniqueness of weak Leray-Hopf solutions of the Navier-Stokes equations. One crucial difficulty in~\cite{JiaSverakIllposed}, compared to Vishik's approach, is that the self-similar solutions in~\cite{JiaSverakIllposed} are far from explicit. Therefore, the spectral condition therein seems difficult to verify analytically, although it has been  checked with non-rigorous numerics in~\cite{guillod2017numerical}. The work~\cite{JiaSverakIllposed} already contains a version of the unstable manifold approach, see p. 3759--3760, and a truncation to finite energy.

At present, the above two programs, while very intriguing and highly suggestive, require a significant numerical component not present in Vishik's approach. On the other hand, at present, Vishik's approach includes a forcing term absent from the above two programs, whose primary role is showcased by the fact that the initial data can be taken to be zero.

\bigskip


After the completion of this manuscript, the first three authors demonstrated  non-uniqueness of Leray solutions to the forced $3d$ Navier-Stokes equations. 
The strategy in~\cite{DallasEliaMariaNavierStokes} was inspired by the connection between instability and non-uniqueness recognized in~\cite{JiaSverakIllposed}, Vishik's papers~\cite{Vishik1,Vishik2}, and the present work. One of the main difficulties in~\cite{DallasEliaMariaNavierStokes} is the construction of a smooth, spatially decaying, and unstable (more specifically, due to an unstable eigenvalue) steady state of the forced $3d$ Euler equations. 
It is constructed as a vortex ring whose cross section is a small perturbation of the $2d$ vortex of Chapter~\ref{chapter:linearpartii} below. In this way, the work~\cite{DallasEliaMariaNavierStokes} relies crucially on the present work. Two spectral perturbation arguments are employed to ensure that (i) the vortex ring inherits the instability of the $2d$ vortex, at least when its aspect ratio is small, and (ii) the singularly perturbed operator with a Laplacian inherits the instability of the inviscid operator. Once the viscous instability is known, the non-linear argument is actually more standard in the presence of a Laplacian, due to the parabolic smoothing effect. In this direction, recent progress includes (a) non-uniqueness in bounded domains~\cite{DallasEliaMariaGluing} and (b) a quasilinear approach to incorporating the viscosity, exhibited in the hypodissipative $2d$ Navier-Stokes equations~\cite{AC}.

\bigskip

Much of the recent progress on non-uniqueness of the Euler equations has been driven by Onsager's conjecture, which was solved in~\cite{IsettOnsager}. With Theorem~\ref{thm:main} in hand, we can now summarize the situation for the Euler equations {in dimension three as follows:} 

\begin{itemize}[leftmargin=*]
    \item $\alpha \in (1,2)$: (\emph{Local well-posedness and energy conservation}) For each solenoidal $u_0 \in C^{\alpha}{(\mathbb{T}^3)}$ and force $f \in L^1(]0,T[;C^\alpha{(\mathbb{T}^3)})$, there exists $T' \in ]0,T[$ and a unique local-in-time solution $u \in L^\infty(]0,T'[;C^\alpha{(\mathbb{T}^3)})$. The solution $u$ depends continuously
    \footnote{The continuous dependence is more subtle for quasilinear equations than semilinear equations, and  uniform continuity is not guaranteed in the regularity class in which the solutions are found, see the discussion in~\cite{taonotes}. One can see this at the level of the equation for the difference of two solutions $u^{(1)}$ and $u^{(2)}$: One of the solutions becomes the ``background" and, hence, loses a derivative. One way to recover the continuous dependence stated above is to compare the above two solutions with initial data $u_0^{(1)}$, $u_0^{(2)}$ and forcing terms $f^{(1)}$, $f^{(2)}$ to approximate solutions $u^{(1),\varepsilon}$, $u^{(2),\varepsilon}$ with mollified initial data $u_0^{(1),\varepsilon}$, $u_0^{(2),\varepsilon}$ and mollified forcing terms $f^{(1),\varepsilon}$, $f^{(2),\varepsilon}$. One then estimates $\| u^{(1)} - u^{(2)} \| \leq \| u^{(1)} - u^{(1),\varepsilon} \| + \| u^{(1),\varepsilon} - u^{(2),\varepsilon} \| + \| u^{(2),\varepsilon} - u^{(2)} \|$. The approximate solutions, which are more regular, are allowed to lose derivatives in a controlled way.} 
    in the above class on its initial data and forcing term. Moreover, the solution $u$ conserves energy.
        \item $1/3 < \alpha < 1$: (\emph{Non-uniqueness and energy conservation}) There exist a time $T > 0$, a force $f \in L^1(]0,T[;{L^2 \cap C^\alpha(\R^2 \times \mathbb{T})})$, and two distinct weak solutions {$u_1,u_2 \in L^\infty(]0,T[;L^2 \cap C^\alpha(\R^2 \times \mathbb{T}))$} to the Euler equations with zero initial data and force $f$. For any $T>0$, any weak solution $u$ in the space $L^\infty(]0,T[;{L^2 \cap C^\alpha(\R^2 \times \mathbb{T})})$ with forcing in the above class conserve energy~\cite{constantinetiti}.
    \item $0 < \alpha < 1/3$: (\emph{Non-uniqueness and anomalous dissipation}) There exist a time $T>0$ and two distinct admissible weak solutions (see~\cite{OnsagerAdmissible}) $u_1,u_2 \in L^\infty(]0,T[;C^\alpha{(\mathbb{T}^3)})$ to the Euler equations with the same initial data and zero force and which moreover dissipate energy.
\end{itemize}

While we are not aware of the first two statements with force in the literature, the proofs are easy adaptations of those with zero force. In order to obtain the non-uniqueness statement in the region $1/3 < \alpha < 1$, one can {extend the non-unique solutions on $\R^2$ to be constant in the $x_3$ direction.}
The borderline cases may be sensitive to the function spaces in question. For example, the three-dimensional Euler equations are ill-posed in $C^k$, $k \geq 1$~\cite{bougainillposed}. Furthermore, of the above statements, only the negative direction of Onsager's conjecture is open in $n=2$. 


\cleardoublepage

\chapter{General strategy: background field and self-similar coordinates}
\label{chapter:general}

In this chapter, we give a detailed outline of the proof of Theorem~\ref{thm:main}. The last section contains a dependency tree which schematizes how the argument is subdivided into several intermediate statements, and which is supposed to help the reader navigate the different parts of this book. The chapter begins by describing the structure of the initial velocity field and the external body force. 
We then state Theorem~\ref{thm:main2}, a more precise version of Theorem~\ref{thm:main}, whose proof is broken into two main parts. The first part is devoted to constructing a linearly unstable vortex, see Theorem~\ref{thm:spectral} for the precise statement. 
The second part deals with the construction of nonlinear unstable trajectories, and it is outlined in Section~\ref{s:nonlinear}.

\section{The initial velocity and the force}

First of all, the initial velocity $v_0$ of Theorem \ref{thm:main} will have the following structure\index{aaga@$\alpha$}\index{aagb@$\beta$}
\begin{equation}\label{e:v_0}
v_0 (x) =
\begin{cases}
\beta (2-\alpha)^{-1} |x|^{-\alpha} \chi (\lvert x\rvert) x^\perp\;\; &\mbox{if }\al =\alpha\\
0   &\mbox{if }\al>\alpha
\end{cases}
\end{equation}
where $0<\alpha\leq \al<1$, $\chi$ is a smooth cut-off function, compactly supported in $\mathbb R$ and identically $1$ on the interval $[-1,1]$, and $\beta$ is a sufficiently large constant (whose choice will depend on $\alpha$). For simplicity we will assume that $\chi$ takes values in $[0,1]$ and it is monotone non-increasing on $[0, \infty[$, even though none of these conditions play a significant role.

A direct computation gives $\div v_0 = 0$. 
The corresponding $\omega_0$ is then given by\index{aagz_0@$\omega_0$}\index{aalv_0@$v_0$}\index{aagx@$\chi$}
\begin{equation}\label{e:omega_0}
\omega_0 (x) = 
\curl v_0 (x) =
\begin{cases}
\beta \left[ |x|^{-\alpha} \chi (|x|) + (2-\alpha)^{-1} \chi' (|x|) |x|^{1-\alpha}\right] \;\;&\mbox{if }\al=\alpha\\
0 &\mbox{if }\al>\alpha
\end{cases}
\end{equation}
and the relation $v_0 = K_2*\omega_0$ comes from standard Calder{\'o}n-Zygmund theory (since $\textrm{div}\, v_0 =0$, $\curl v_0=\omega_0$ and $v_0$ is compactly supported).
 $\al\in ]0,1[$ is chosen depending on $p$ in Theorem \ref{thm:main}, so that $\al p < 2$: in the rest of the notes we assume that $p$, $\al$, and $\alpha$ are fixed. In particular it follows from the definition that $\omega_0\in L^1\cap L^p$ and that $v_0 \in L^1 \cap L^\infty$. 
 
Next, the function $|x|^{-\al}$ will be appropriately smoothed to a (radial) function 
\begin{equation}\label{e:def-bar-Omega}
\bar \Omega (x) = g (|x|)
\end{equation}
\index{aalg@$g$}\index{aagZbar@$\bar \Omega$}
such that:
\begin{align}
&g \in C^\infty ([0,R]) \qquad \qquad &\forall R>0\, ,\\
&g (r) = r^{-\al} \qquad\qquad &\mbox{for $r\geq 2$,}\label{e:decay-at-infinity}\\ 
&g(r) = g(0) + \frac{g''(0)}{2} r^2\qquad \qquad &\mbox{for $r$ in a neighborhood of $0$.}\label{e:g-constant-around-0}
\end{align} 
This smoothing will be carefully chosen so as to achieve some particular properties, whose proof will take a good portion of the notes (we remark however that while a sufficient degree of smoothness and the decay \eqref{e:decay-at-infinity} play an important role, the condition \eqref{e:g-constant-around-0} is just technical and its role is to simplify some arguments). We next define the function $\bar V (x)$\index{aagf@$\zeta$} \index{aalVbar@$\bar V$} as
 \begin{equation}\label{e:def-barV}
 \bar V (x) = \zeta (|x|) x^\perp\, ,
 \end{equation}
where $\zeta$ is
\begin{equation}\label{e:def-zeta}
    \zeta(r) = \frac{1}{r^2}\int_0^r \rho g(\rho)\,\mathrm d\rho\, .
\end{equation}

\begin{remark}\label{r:well-defined-2} Observe that under our assumptions $\bar \Omega\in L^q(\R^2)$ for every $q>\frac{2}{\al}$, but it does not belong to any $L^q(\R^2)$ with $q\leq \frac{2}{\al}$. Since when $p\geq 2$ the condition $\al p <2$ implies $\al < 1$, we cannot appeal to Young's Theorem as in Remark \ref{r:well-defined} to define $K_2* \bar\Omega$. 

Nonetheless, $\bar V$ can be understood as a natural definition of $K_2* \bar\Omega$ for radial distributions of vorticity which are in $L^1_\textrm{loc}$. Indeed observe first that $\textrm{div}\, \bar V=0$ and $\textrm{curl}\, \bar V = \bar\Omega$, and notice also that $\bar V$ would decay at infinity like $|x|^{-1}$ if $\bar\Omega$ were compactly supported. This shows that $\bar V$ would indeed coincide with $K_2*\bar \Omega$ for compactly supported radial vorticities. Since we can approximate $\bar \Omega$ with $\bar\Omega_N := \bar\Omega \mathbf{1}_{B_N}$, passing into the limit in the corresponding formulas for $K_2* \bar\Omega_N$ we would achieve $\bar V$. 

Note also that in the remaining computations what really matters are the identities $\textrm{div}\,  \bar V = 0$ and $\textrm{curl}\, \bar V = \bar \Omega$ and so regarding $\bar V$ as $K_2* \bar\Omega$ only simplifies our terminology and notation.
\end{remark}

The force $f$ will then be defined in such a way that $\tilde \omega$, the curl of the velocity \index{aalf@$f$}\index{aalvtilde@$\tilde{v}$}\index{aagztilde@$\tilde{\omega}$}
\begin{equation}\label{e:tilde-v}
\tilde v (x, t) = \beta t^{1/\alpha-1} \bar V \Big(\frac{x}{t^{1/\alpha}}\Big) \chi (|x|)\, ,
\end{equation}
is a solution of \eqref{e:Euler}. In particular, since $(\tilde v\cdot\nabla)\tilde \omega=0$, the force $f$ is given by the explicit formula
\begin{equation}\label{e:def-f}
f (x,t) = \partial_t \tilde{\omega} (x,t)\, .    
\end{equation}
With this choice a simple computation, left to the reader, shows that $\tilde{\omega}$ solves \eqref{e:Euler} with initial data $\omega_0$. Note in passing that, although as pointed out in Remark \ref{r:well-defined-2} there is not enough summability to make sense of the identity $K_2* \bar \Omega = \bar V$ by using standard Lebesgue integration, the relation $K_2* \tilde\omega = \tilde{v}$ is made obvious by $\textrm{div}\, \tilde{v} =0$, $\textrm{curl}\, \tilde{v} = \tilde{\omega}$, and the boundedness of the supports of both $\tilde{\omega}$ and $\tilde{v}$.

The pair $(\tilde{\omega}, \tilde{v})$ is one of the solutions claimed to exist in Theorem \ref{thm:main}. The remaining ones will be described as a one-parameter family $(\omega_\varepsilon, v_\varepsilon)$ for a nonzero choice of the parameter $\varepsilon$, while $(\tilde{\omega}, \tilde{v})$ will correspond to the choice $\varepsilon =0$. We will however stick to the notation $(\tilde\omega, \tilde v)$ to avoid confusions with the initial data. 

It remains to check that $f$ belongs to the functional spaces claimed in Theorem \ref{thm:main}. 

\begin{lemma}\label{lem:Curl of tilde v}
    $\tilde\omega$ is a smooth function on $\{t>0\}$ which satisfies, for all $t>0$ and $x\in\R^2$,
    \begin{equation}\label{e:curl-tilde-v}
        \tilde \omega (x, t) = \beta t^{-1} \bar \Omega \bigg(\frac{x}{t^{1/\alpha}}\bigg) \chi (|x|) + \beta t^{-1} \zeta \bigg(\frac{|x|}{t^{1/\alpha}}\bigg) |x|\chi' (|x|)\, ,
    \end{equation}
while the external force $f$ and $\partial_t \tilde{v} = K_2*f$ belong, respectively, to the spaces $L^1([0,T]; L^1 \cap L^p)$ and $L^1 ([0,T], L^2)$ for every positive $T$. 

Likewise $\tilde\omega \in L^\infty ([0,T], L^1\cap L^p)$ and $\tilde{v} \in L^\infty ([0,T], L^2)$. 
\end{lemma}

We end the section with a proof of the lemma, while we resume our explanation of the overall approach to Theorem \ref{thm:main} in the next section.

\begin{proof} The formula \eqref{e:curl-tilde-v} is a simple computation. From it we also conclude that $\tilde\omega = \curl\tilde v$ is a smooth function on $\{t>0\}$ and hence differentiable in all variables.
Observe next that $|\bar V (x)|\leq C |x|^{1-\al}$  and we can thus estimate $|\tilde{v} (x,t)|\leq C {t^{\frac{\al}{\alpha}-1}}|x|^{1-\al}$. Since its spatial support is contained in $\textrm{spt}\, (\chi)$, we conclude that $\tilde v$ is bounded and belongs to $L^\infty ([0,{T}], L^2)$ {for any $T>0$.}

Using that $\bar \Omega (x) = |x|^{-\al}=g(\abs x)$ for $|x|\geq 2$, we write
	\begin{align*}
	\tilde{\omega} (x,t)	= & \beta t^{-1} g \left(\frac{|x|}{t^{1/\alpha}}\right) \chi (|x|) \mathbf{1}_{\{|x|\leq 2 t^{1/\alpha}\}}
	+ \beta { t^{\frac{\al}{\alpha}-1}}|x|^{-\al} \chi (|x|) \mathbf{1}_{\{|x|> 2t^{1/\alpha}\}} \\ 
	&+ \beta t^{-1}\zeta \left(\frac{|x|}{t^{1/\alpha}}\right) |x|\chi' (|x|)\, .
	\end{align*}
	In particular, recalling that $|\bar \Omega (x)|\leq C|x|^{-\al}$ and $\zeta (|x|) |x| \leq C |x|^{1-\al}$ we easily see that 
	\begin{align}
	\|\tilde\omega (\cdot, t)\|_{L^1} &\leq C \int_{\{|x|\in \textrm{spt}\, (\chi)\}} { t^{\frac{\al}{\alpha}-1}}|x|^{-\al} \, dx + C \int_{\{|x|\in \textrm{spt}\, (\chi')\}} { t^{\frac{\al}{\alpha}-1}}|x|^{1-\al}\, dx\, ,\\
	\|\tilde{\omega} (\cdot, t)\|_{L^p}^p &\leq C \int_{\{|x|\in \textrm{spt}\, (\chi)\}} { t^{\left(\frac{\al}{\alpha} -1\right)p}} |x|^{-p \al} \, dx\nonumber\\
 &\qquad\qquad + C \int_{\{|x|\in \textrm{spt}\, (\chi')\}} { t^{\left(\frac{\al}{\alpha} -1\right)p}}|x|^{p-p\al}\, dx\, .
	\end{align}
	This implies immediately that $\tilde\omega \in L^\infty ([0,{ T]}, L^1\cap L^p)$ {for any $T>0$ }, given that $\al p <2$ (and hence $|x|^{-\al p}$ is locally integrable). 
	
	We now differentiate in time in the open regions $\{|x|< 2t^{1/\alpha}\}$ and $\{|x| > 2t^{1/\alpha}\}$ separately to achieve
	\begin{align}
	f (x,t) = & - \beta \left(t^{-2} g \left(\frac{|x|}{t^{1/\alpha}}\right)
	+ \frac{1}{\alpha}{t^{-2-1/\alpha}} |x| g' \left(\frac{|x|}{t^{1/\alpha}}\right)\right) \chi (|x|) \mathbf{1}_{\{|x|\leq 2 t^{1/\alpha}\}}\nonumber\\
	& {+ \beta \left(\frac{\al}{\alpha}-1\right)t^{\frac{\al}{\alpha}-2} |x|^{-\al}\chi(|x|)\mathbf{1}_{\{|x|>2t^{1/\alpha}\}}
	}\nonumber\\
	& - \beta \left(t^{-2} \zeta \left(\frac{|x|}{t^{1/\alpha}}\right) +\frac{1}{\alpha} t^{-2-1/\alpha} \zeta' \left(\frac{|x|}{t^{1/\alpha}}\right) |x|\right) |x|\chi' (|x|)\nonumber\\
	=:& f_1 (x,t) + f_2 (x,t){ + f_3(x,t)}\, .\label{e:f1-f2}
	\end{align}
 Since we will only estimate integral norms of $f$, its values on $\{|x|= 2t^{1/\alpha}\}$ are of no importance. However, given that $f$ is in fact smooth over the whole domain $\{t>0\}$, we can infer the validity of the formula \eqref{e:f1-f2} for every point $x\in \{|x|= 2t^{1/\alpha}\}$ by approximating it with a sequence of points in $\{|x|< 2t^{1/\alpha}\}$ and passing to the limit in the corresponding expressions.

	We wish to prove that $f\in L^1 ([0,T], L^1\cap L^p)$. On the other hand, since for any $T_0>0$ both $f_1+f_2$ and $f_3$ are smooth and have compact support on $\mathbb R^2\times [T_0, T]$, it suffices to show that $f\in L^1 ([0,T_0], L^1\cap L^p)$ for a sufficiently small $T_0$. Recalling that $|g (|x|)| + |g' (|x|)||x| \leq C |x|^{-\al}$, we can then bound
	\begin{equation}\label{e:bound-f1}
	|f_1 (x,t)|\leq C {t^{-2+\frac{\al}{\alpha}}} |x|^{-\al} \mathbf{1}_{|x|\leq 2 t^{1/\alpha}} \qquad
	\mbox{for all $0<t<T_0$ and all $x$}. 
	\end{equation}
	Thus
	\begin{align}
	\|f_1\|_{L^1 (\mathbb R^2\times [0,T_0])} &\leq C \int_0^{T_0} t^{2/\alpha -2}\, dt\, , < \infty\\
	\|f_1\|_{L^1 ([0,T_0];L^p( \mathbb R^2))} &\leq C \int_0^{T_0} t^{2/(\alpha p) -2}\, dt < \infty\, ,
	\end{align}
	where the condition $2> \al p$ entered precisely in the finiteness of the latter integral. Coming to the second term,
	we observe that it vanishes when $\bar\alpha = \alpha$. When $\alpha < \bar\alpha$, since $\chi$ is compactly supported in $\mathbb R$, we get
	\begin{align*}
	 \|f_2\|_{L^1(\mathbb R^2\times [0,T_0])} &\leq C \int_0^{T_0} t^{\frac{\al}{\alpha}-2}(1+t^{\frac 1{\alpha}(2-\al)}) dt<+\infty\\
	 \|f_2\|_{L^1([0,T_0];L^p(\mathbb R^2))}
	 &\leq \int_0^{T_0} t^{\frac{\al}{\alpha}-2}(1+t^{\frac p{\alpha}(2-\al)})^{1/p} dt <+\infty
	 \, .
	\end{align*}
	The last term can be computed explicitly as
	\begin{align*}
	\zeta (r) &= \frac{C}{r^2} + \frac{1}{r^2} \int_2^r \rho^{1-\al}\,\mathrm d\rho = a r^{-2} +  b r^{-\al} & \text{for all } r \geq 2\, ,
	\shortintertext{where $a$ and $b$ are two fixed constants. Likewise}
		\zeta'(r)&= -2 a r^{-3} -\al b r^{-\al-1}\qquad &\text{for all } r \geq 2\, .
	\end{align*}
	Recall that $\chi' (|x|)=0$ for $|x|\leq 1$.
    Therefore, for $t\leq T_0$ sufficiently small, the functions $\zeta$ and $\zeta'$ are computed on $|x| t^{-1/\alpha} \geq 2$ in the formula for $f_3$ (cf. \eqref{e:f1-f2}). Thus, {
    \[
    f_3 (x,t) 
    =-\beta t^{-2}\Big(
    \Big(1-\frac 2{\al}\Big)a t^{\frac 2{\alpha}}|x|^{-1} + b\Big(1-\frac{\al}{\alpha}\Big)t^{\frac{\al}{\alpha}} |x|^{1-\al}
    \Big)\chi'(|x|)
    \]}%
    In particular $f_3$ has compact support. Since $\alpha <1$ the function
    \[
    -\beta t^{-2}
    \Big(1-\frac 2{\al}\Big)a t^{\frac 2{\alpha}}|x|^{-1} \chi'(|x|)\, ,
    \]
    is bounded, and thus belongs to
    $L^1 ([0,T_0], L^1\cap L^p)$. As for the second summand, it vanishes if $\alpha = \bar \alpha$, while its $L^p$ norm at time $t$ can be bounded by $C t^{-2+\frac{\bar \alpha}{\alpha}}$ if $\bar\alpha > \alpha$. The latter function however belongs to $L^1 ([0,T_0])$.
    
    Observe next that, since for every positive $t$ the function $f (\cdot, t)$ is smooth and compactly supported, $K_2* f (\cdot, t)$ is the unique divergence-free vector field which belongs to $L^1$ and such that its curl gives $f (\cdot, t)$. Hence, since $f (\cdot, t) = \curl \partial_t \tilde{v} (\cdot, t)$ and $\partial_t \tilde{v} (\cdot, t)$ is smooth and compactly supported, we necessarily have $K_2 * f (\cdot, t) = \partial_t \tilde{v} (\cdot, t)$. It remains to show that $\partial_t \tilde{v} \in L^1 ([0,T]; L^2)$ for every positive $T$. To that end we compute%
    \[
    \tilde{v} (x,t) = \beta t^{1/\alpha-1} \bar{V} \left(\frac{x}{t^{1/\alpha}}\right) \chi (|x|)
    = \beta t^{-1} \zeta \left(\frac{|x|}{t^{1/\alpha}}\right) x^\perp \chi (|x|) 
    \]
    \[
    \partial_t \tilde{v} (x,t ) = - \beta t^{-2} \chi (|x|) x^\perp \left(\zeta \left( \frac{|x|}{t^{1/\alpha}}\right) +\frac{1}{\alpha} \frac{|x|}{t^{1/\alpha}} \zeta' \left(\frac{|x|}{t^{1/\alpha}} \right)\right)\, .
    \]
    In order to compute the $L^2$ norm of $\partial_t \tilde{v} (\cdot, t)$ we break the space into two regions as in the computations above. In the region $\{|x|\leq 2 t^{1/\alpha}\}$ we use that $|\zeta| + |g|+ |\zeta'|$ are bounded to compute
    \[
    \int_{|x|\leq 2 t^{1/\alpha}} |\partial_t \tilde{v} (x,t)|^2\, dx \leq C t^{-4} \int_{|x|\leq t^{1/\alpha}} |x|^2\,dx \leq C t^{4/\alpha -4}\, ,
    \]
    which is a bounded function on $[0,1]$. On $\{|x|\geq t^{1/\alpha}\}$ we observe that the function can be explicitly computed as
    \[
    -\beta t^{-2} \chi (|x|) x^\perp \Big(a\Big(1-\frac{2}{\alpha}\Big) t^{2/\alpha} |x|^{-2} + b \Big(1-\frac{\bar \alpha}{\alpha}\Big) t^{\frac{\bar\alpha}{\alpha}}|x|^{-\bar \alpha}\Big)\, .
    \]
    If we let $\bar R>0$ be such that the support of $\chi$ is contained in $B_{\bar R}$, we use polar coordinates to estimate
    \[
    \int_{|x|\geq 2 t^{1/\alpha}} |\partial_t \tilde{v} (x,t)|^2\, dx \leq C t^{-4+4/\alpha} \int_{2t^{1/\alpha}}^{\bar R} \frac{d\rho}{\rho} + C |\alpha - \bar \alpha| t^{2\frac{\bar \alpha}{\alpha}-4}\, .
    \]
    We can therefore estimate the $L^2$ norm of $\partial_t \tilde{v}$ at time $t$ by
    \[
    \|\partial_t \tilde{v} (\cdot, t)\|_{L^2} \leq C + C |\alpha - \bar \alpha| t^{\frac{\bar \alpha}{\alpha} -2}\, .
    \]
    When $\alpha = \bar \alpha$ we conclude that the $L^2$ norm of $\partial_t \tilde{v}$ is bounded, while for $\bar\alpha > \alpha$ the function $t\mapsto t^{\frac{\bar \alpha}{\alpha} -2}$ belongs to $L^1 ([0,T])$.%
\end{proof}

\section{The infinitely many solutions}

We next give a more precise statement leading to Theorem \ref{thm:main} as a corollary.

\begin{theorem}\label{thm:main2}
Let $p\in ]2, \infty[$ be given and let $\alpha$ {and $\al$} be any positive number such that {$\alpha \leq \al$} and $\al p <2$. For an appropriate choice of the smooth function $\bar \Omega$ and of a positive constant $\beta$ as in the previous section, we can find, additionally:
\begin{itemize}
    \item[(a)] a suitable nonzero function $\eta\in (L^1\cap H^2) (\mathbb R^2; \mathbb C)$ with $K_2 * \eta\in L^2(\R^2; \mathbb C^2)$, \index{aagh@$\eta$}
    \item[(b)] a real number $b_0$ and a positive number $a_0>0$, \index{aalazero@$a_0$}\index{aalbzero@$b_0$}
\end{itemize}
with the following property.

Consider $\omega_0$, $v_0$, $\tilde{v}$, $\tilde\omega = \curl \tilde{v}$, and $f$ as defined in \eqref{e:v_0},\eqref{e:omega_0}, \eqref{e:tilde-v}, and \eqref{e:def-f}. Then for every $\varepsilon\in \R$ there is a solution $\omega_\varepsilon$ of \eqref{e:Euler} with initial data $\omega_0$ such that \index{aage@$\varepsilon$}\index{aagzepsilon@$\omega_\varepsilon$}
\begin{enumerate}[(i)]
    \item\label{item:1-omega in L infinity L1 Lp} $\omega_\varepsilon \in L^\infty ([0,T], L^1\cap L^p )$ for every $T>0$;
    \item\label{item:2-v in L infinity L2} $v_\varepsilon := K_2 * \omega_\varepsilon \in L^\infty ([0,T], L^2)$ for every $T>0$;
    \item\label{item:3-eigenvalue bound} as $t\to0$,
\begin{equation}\label{e:asymptotic-in-t}
\|\omega_\varepsilon (\cdot, t) - \tilde\omega (\cdot, t) - \varepsilon t^{a_0-1} \operatorname{Re} (t^{i b_0} \eta (t^{-1/\alpha} \cdot))\|_{L^2(\R^2)} = o (t^{a_0 +1/\alpha -1})\, ;
\end{equation}
\item\label{e:Camillo-is-silly} if $b_0=0$, then $\eta$ is real-valued. 
\end{enumerate}
\end{theorem}

Observe that, by a simple computation,
\begin{align*}
\| t^{a_0-1} \operatorname{Re} (t^{ib_0} \eta (t^{-1/\alpha} \cdot))\|_{L^2} = t^{a_0 +1/\alpha -1} \|\operatorname{Re} (t^{ib_0} \eta)\|_{L^2}\, ,
\end{align*}
and thus it follows from \eqref{e:asymptotic-in-t} that
\begin{equation}\label{eq:difference-of-the-omega}
\limsup_{t\downarrow 0} t^{1-1/\alpha - a_0} \|\omega_{\varepsilon} (\cdot, t) - \omega_{\bar\varepsilon} (\cdot, t)\|_{L^2} \geq |\varepsilon - \bar\varepsilon| \max_{\theta\in[0,2\pi]}\| \operatorname{Re} (e^{i\theta} \eta)\|_{L^2}\,
\end{equation}
(note that in the last conclusion we need (iv) if $b_0=0$). 
Since $\|\eta\|_{L^2} >0$, we conclude that the solutions $\omega_\varepsilon$ described in Theorem \ref{thm:main2} must be all distinct.

For each fixed $\varepsilon$, the solution $\omega_\varepsilon$ will be achieved as a limit of a suitable sequence of approximations $\omega_{\varepsilon, k}$\index{aagzepsilonk@$\omega_{\varepsilon, k}$} in the following way. After fixing a sequence of positive times $t_k$\index{aaltk@$t_k$} converging to $0$, which for convenience are chosen to be $t_k := e^{-k}$, we solve the following Cauchy problem for the Euler equations in vorticity formulation
\begin{equation}\label{e:Euler-later-times}
\left\{
\begin{array}{ll}
& \partial_t \omega_{\varepsilon,k} + ((K_2* \omega_{\varepsilon,k})\cdot \nabla) \omega_{\varepsilon,k} = f  \\ \\
& \omega_{\varepsilon, k} (\cdot, t_k) =  \tilde\omega (\cdot, t_k) + \varepsilon t_k^{a_0-1}  \operatorname{Re} (t_k^{ib_0} \eta (t_k^{-1/\alpha}\cdot))\, .
\end{array}\right.
\end{equation}
Observe that, since $t_k$ is positive, the initial data $\omega_{\varepsilon, k} (\cdot, t_k)$ belongs to $L^1\cap L^\infty$, while the initial velocity defining $v_{k, \varepsilon}:= K_2 * \omega_{\varepsilon, k} (\cdot, t_k)$ belongs to $L^2$. 

Since $K_2 * f \in L^1 ([0,T], L^2)$ for every $T$, we can apply the classical theorem of Yudovich (namely, Theorem \ref{thm:Yudo} and Remark \ref{r:A-priori-estimates}) to conclude that

\begin{corollary}\label{c:omega_k_epsilon}
For every $k$, $\varepsilon$, and every $T$ there exists a unique solution $\omega_{\varepsilon, k}$ of \eqref{e:Euler-later-times} with the property that $\omega_{\varepsilon , k} \in L^\infty ([t_k, T], L^1\cap L^\infty)$ and $v_{\varepsilon, k}\in L^\infty ([t_k, T], L^2)$ for every positive $T$. Moreover, we have the following bounds for every $t$
\begin{align}
\|\omega_{\varepsilon, k} (\cdot, t)\|_{L^1} \leq &\|\omega_{\varepsilon, k} (\cdot, t_k)\|_{L^1} + 
\int_{t_k}^t \|f (\cdot, s)\|_{L^1}\,\mathrm ds\\
\|\omega_{\varepsilon, k} (\cdot, t)\|_{L^p} \leq &\|\omega_{\varepsilon, k} (\cdot, t_k)\|_{L^p} + 
\int_{t_k}^t \|f (\cdot, s)\|_{L^p}\,\mathrm ds \label{e:omega_Lp_estimate}\\
\|v_{\varepsilon, k} (\cdot, t)\|_{L^2}\leq &\|v_{\varepsilon, k} (\cdot, t_k)\|_{L^2} +
 \int_{t_k}^t \|K_2* f (\cdot, s)\|_{L^2}\,\mathrm ds\, .
\end{align}
\end{corollary}

Next, observe that we can easily bound $\|\omega_{\varepsilon, k} (\cdot, t_k)\|_{L^1}$, $\|\omega_{\varepsilon, k} (\cdot, t_k)\|_{L^p}$, and $\|v_{\varepsilon, k} (\cdot, t_k)\|_{L^2}$ independently of $k$, for each fixed $\varepsilon$. We then conclude (for $\varepsilon$ fixed)
\begin{equation}\label{e:uniform_bound}
\sup_{k\in \mathbb N} \sup_{t\in [t_k, T]}
\left(\|\omega_{\varepsilon, k} (\cdot, t)\|_{L^1} + \|\omega_{\varepsilon, k} (\cdot, t)\|_{L^p} + \|v_{\varepsilon, k} (\cdot, t)\|_{L^2}
\right) < \infty\, .
\end{equation}
In turn we can use \eqref{e:uniform_bound} to conclude that, for each fixed $\varepsilon$, a subsequence of $\omega_{\varepsilon, k}$ converges to a solution $\omega_\varepsilon$ of \eqref{e:Euler} which satisfies the conclusions \ref{item:1-omega in L infinity L1 Lp} and \ref{item:2-v in L infinity L2} of Theorem \ref{thm:main2}. 

\begin{proposition}\label{p:convergence}\label{P:CONVERGENCE}
Assume $p, \alpha, \al, \omega_0, v_0, \tilde\omega, \tilde{v}, f, a_0, b_0$, and $\bar \eta$ are as in Theorem \ref{thm:main2} and let $\omega_{\varepsilon, k}$ be as in Corollary \ref{c:omega_k_epsilon}. Then, for every fixed $\varepsilon$, there is a subsequence, not relabeled, with the property that $\omega_{\varepsilon, k}$ converges (uniformly in $C ([0,T], L^q_w)$ for every positive $T$ and every $1< q\leq p$, where $L^q_w$ denotes the space $L^q$ endowed with the weak topology) to a solution $\omega_\varepsilon$ of \eqref{e:Euler} on $[0, \infty[$ with initial data $\omega_0$ and satisfying the bounds \ref{item:1-omega in L infinity L1 Lp} and \ref{item:2-v in L infinity L2} of Theorem \ref{thm:main2}.   
\end{proposition}

The proof uses classical convergence theorems and we give it in the appendix for the reader's convenience. The real difficulty in the proof of Theorem \ref{thm:main2} is to ensure that the bound (iii) holds. This is reduced to the derivation of suitable estimates on $\omega_{\varepsilon, k}$, which we detail in the following statement.

\begin{theorem}\label{thm:main3}
Assume $p, \alpha, \al$ are as in Theorem \ref{thm:main2} and fix $\varepsilon >0$. For an appropriate choice of $\bar \Omega$ and $\beta$ there is a triple $\eta$, $a_0$, and $b_0$ as in Theorem \ref{thm:main2} and three positive constants $T_0, \delta_0$, and $C$ with the property that
\begin{equation}\label{e:asymptotic-in-t-2}
\|\omega_{\varepsilon,k} (\cdot, t) - \tilde\omega (\cdot, t) - \varepsilon t^{a_0-1} \textrm{Re}\, (t^{ib_0} \eta (t^{-1/\alpha} \cdot))\|_{L^2} \leq C t^{a_0+1/\alpha-1+\delta_0} \;\; \forall t\in [t_k, T_0]\, .
\end{equation}
\end{theorem}

It is then obvious that the final conclusion \ref{item:3-eigenvalue bound} of Theorem \ref{thm:main2} is a consequence of the more precise estimate \eqref{e:asymptotic-in-t-2} on the approximations $\omega_{\varepsilon,k}$. The rest of these lecture notes are thus devoted to the proof of Theorem \ref{thm:main3} and we will start in the next section by breaking it into two main parts. 

\section{Logarithmic time scale and main Ansatz}

\index{similarity-variables@Similarity variables}\index{aagZ@$\Omega$}\index{aagt@$\tau$}\index{aago@$\xi$}
First of all, we will change variables and unknowns of the Euler equations (in vorticity formulation) in a way which will be convenient for many computations. Given a solution $\omega$ of \eqref{e:Euler} on $\R^2\times [T_0, T_1]$ with $0\leq T_0 \leq T_1$, we introduce a new function $\Omega$ on $\mathbb R^2 \times [\ln T_0, \ln T_1]$ with the following transformation. We set $\tau=\ln t$, $\xi=x t^{-1/\alpha}$ and
\begin{equation}\label{e:omega->Omega}
\Omega (\xi, \tau) \coloneqq e^{\tau} \omega (e^{\tau/\alpha} \xi, e^\tau)\, ,  
\end{equation}
which in turn results in
\begin{equation}\label{e:Omega->omega}
\omega (x, t) = t^{-1} \Omega (t^{-1/\alpha} x, \ln t)\, .
\end{equation}
Observe that, if $v (\cdot, t) = K_2 * \omega (\cdot, t)$ and $V( \cdot, \tau) = K_2 * \Omega (\cdot, \tau)$, we can derive similar transformation rules for the velocities as \index{aalV@$V$}
\begin{align}
V (\xi, \tau) &= e^{\tau (1-1/\alpha)} v(e^{\tau/\alpha} \xi, e^\tau)\label{e:v->V}\, ,\\
v (x,t) &= t^{-1+1/\alpha} V (t^{-1/\alpha} x, \ln t)\label{e:V-t>v}\, .
\end{align}
Likewise, we have an analogous transformation rule for the force $f$, which results in \index{aalF@$F$}
\begin{align}
F (\xi, \tau) &= e^{2\tau} f (e^{\tau/\alpha} \xi, e^\tau)\, ,\label{e:f->F}\\
f (x,t) &= t^{-2} F (t^{-1/\alpha} x, \ln t)\label{e:F->f}\, .
\end{align}
In order to improve the readability of our arguments, throughout the rest of the notes we will use the overall convention that, given some object related to the Euler equations in the ``original system of coordinates'', the corresponding object after applying the transformations above will be denoted with the same letter in capital case.

\begin{remark}
	Note that the naming of $\bar V$ and $\bar\Omega$ is somewhat of an exception to this convention, since $(\bar\Omega, \bar V)$ is a solution of \eqref{e:Euler} in Eulerian variables. However, if you ``force them to be functions of $\xi$,'' which is how they will be used in the non-linear part, then they solve the Euler equations in self-similar variables with forcing (see \eqref{e:Euler-transformed}).
\end{remark}

Straightforward computations allow then to pass from \eqref{e:Euler} to an equation for the new unknown $\Omega$ in the new coordinates. More precisely, we have the following

\begin{lemma}\label{l:coordinates-change}
Let $p>2$ and $\infty \geq T_1 > T_0\geq 0$. Then $\omega\in L^\infty_{\text{loc}} (]T_0, T_1[; L^1\cap L^p)$ and $v (\cdot, t) = K_2* \omega (\cdot, t)$ satisfy 
\begin{equation}\label{e:Euler-again}
\partial_t \omega + (v \cdot \nabla) \omega = f\, ,
\end{equation}
if and only if $\Omega$ and $V (\cdot, t) = K_2 * \Omega (\cdot, t)$ satisfy
\begin{equation}\label{e:Euler-transformed}
\partial_\tau \Omega - \Big(1 + \frac{\xi}{\alpha}\cdot \nabla\Big) \Omega + (V\cdot \nabla) \Omega = F\, . 
\end{equation}
\end{lemma}

We next observe that, due to the structural assumptions on $\tilde \omega$ and $\tilde v$, the corresponding fields $\tilde \Omega$ and $\tilde V$ can be expressed in the following way: \index{aalVtilde@$\tilde V$}\index{aagZtilde@$\tilde\Omega$}
\begin{align}
\tilde{V} (\xi, \tau) &= \beta \bar V (\xi) \chi (e^{\tau/\alpha} |\xi|)\, ,\label{e:tildeV}\\
\tilde{\Omega} (\xi, \tau) &= \beta \bar \Omega (\xi) \chi (e^{\tau/\alpha} |\xi|) + \beta \zeta (|\xi|) 
\chi' (e^{\tau/\alpha} |\xi|) e^{\tau/\alpha} |\xi|\, \label{e:tildeOmega}.
\end{align}
Observe that, for every fixed compact set $K$ there is a sufficiently negative $-T (K)$ with the property that
\begin{itemize}
    \item $\chi (e^{\tau/\alpha} |\cdot|)= 1$ and $\chi' (e^{\tau/\alpha} \cdot) = 0$ on $K$ whenever $\tau \leq - T (K)$.
\end{itemize}
Since in order to prove Theorem \ref{thm:main} we are in fact interested in very small times $t$, which in turn correspond to very negative $\tau$, it is natural to consider $\tilde\Omega$ and $\tilde{V}$ as perturbations of $\beta \bar \Omega$ and $\beta \bar V$. We will therefore introduce the notation 
\begin{align}
\tilde \Omega &= \beta \bar \Omega + \Omega_r\, ,\\
\tilde V & = \beta \bar V + V_r := \beta \bar V + K_2* \Omega_r\, .
\end{align}
We are thus lead to the following Ansatz for $\Omega_{\varepsilon,k} (\xi, \tau) = e^{\tau} \omega_{\varepsilon ,k} (e^{\tau/\alpha} \xi, e^\tau)$:
\begin{equation}\label{e:Ansatz-1}
\Omega_{\varepsilon, k} (\xi, \tau) = \beta \bar \Omega (\xi) + \Omega_r (\xi, \tau) + \varepsilon e^{\tau a_0} \textrm{Re}\, (e^{i\tau b_0} \eta (\xi)) + \Omega_{\text{per}, k} (\xi, \tau)\, .
\end{equation}
The careful reader will notice that indeed the function $\Omega_{\text{per},k}$ depends upon the parameter $\varepsilon$ as well, but since such dependence will not really play a significant role in our discussion, in order to keep our notation simple, we will always omit it. \index{aagZr@$\Omega_r$}\index{aagZepsilonk@$\Omega_{\varepsilon, k}$}\index{aagZperk@$\Omega_{\text{per}, k}$}\index{aalVr@$V_r$}

We are next ready to complete our Ansatz by prescribing one fundamental property of the function $\eta$. 
We first introduce the integro-differential operator \index{aalLss@$L_{\text{ss}}$}\index{Self-similar operator}
\begin{equation}\label{e:Lss}
L_{\text{ss}} (\Omega) := \Big(1+\frac{\xi}{\alpha} \cdot \nabla\Big) \Omega - \beta (\bar V \cdot \nabla) \Omega - \beta ((K_2* \Omega)\cdot \nabla) \bar \Omega\, .
\end{equation}
We will then prescribe that $\eta$ is an eigenfunction of $L_{\text{ss}}$ with eigenvalue $z_0 = a_0 + ib_0$, namely, \index{aalz0@$z_0$}
\begin{equation}\label{e:Ansatz-2}
L_{\text{ss}} (\eta) = z_0 \eta\, . 
\end{equation}
Observe in particular that, since $L_{\text{ss}}$ is a real operator (i.e. $L_{\text{ss}} (\eta)$ is real-valued when $\eta$ is real-valued, cf. Section \ref{s:abstract-operators}), the complex conjugate $\bar \eta$ is an eigenfunction of $L_{\text{ss}}$ with eigenvalue $\bar z_0$, so that, in particular, the function
\begin{equation}\label{e:Omega_lin}
\Omega_{\text{lin}} (\xi, \tau) := \varepsilon e^{a_0 \tau} \textrm{Re}\, (e^{i b_0 \tau} \eta (\xi)) 
= \frac{\varepsilon}{2} (e^{z_0 \tau} \eta (\xi) + e^{\bar z_0 \tau} \bar \eta (\xi))
\end{equation}
satisfies the linear evolution equation
\begin{equation}\label{e:evolution_of_Omega_lin}
\partial_\tau \Omega_{\text{lin}} - L_{\text{ss}} (\Omega_{\text{lin}})=0\, .
\end{equation}
The relevance of our choice will become clear from the discussion of Section \ref{s:nonlinear}. The point is that \eqref{e:evolution_of_Omega_lin} is close to the linearization of Euler (in the new system of coordinates) around $\tilde{\Omega}$. The ``true linearization'' would be given by \eqref{e:evolution_of_Omega_lin} if we were to substitute $\bar \Omega$ and $\bar V$ in \eqref{e:Lss} with $\tilde{\Omega}$ and $\tilde{V}$. Since however the pair $(\tilde \Omega, \tilde{V})$ is well approximated by $(\bar \Omega, \bar V)$ for very negative times, we will show that \eqref{e:evolution_of_Omega_lin} drives indeed the evolution of $\Omega_{\varepsilon,k}-\tilde{\Omega}$ up to an error term (i.e. $\Omega_{\text{per},k}$) which is smaller than $\Omega_{\text{lin}}$.

\section{Linear theory}

We will look for the eigenfunction $\eta$ in a particular subspace of $L^2$. More precisely for every 
$m\in \mathbb N\setminus \{0\}$ we denote by $L^2_m$ the set of those elements $\vartheta \in L^2 (\mathbb R^2, \mathbb C)$ which are $m$-fold symmetric, i.e., denoting by $R_\theta\colon \mathbb R^2\to \mathbb R^2$ the counterclockwise rotation of angle $\theta$ around the origin, \index{rotational-symmetry@Rotationally symmetric function space}\index{aalL2m@$L^2_m$}
they satisfy the condition
\begin{align*}
\vartheta &= \vartheta \circ R_{2\pi/m}\, .
\end{align*}
In particular, $L^2_m$ is a closed subspace of $L^2 (\mathbb R^2, \mathbb C)$. Note however that the term ``$m$-fold symmetric'' is somewhat misleading when $m=1$: in that case the transformation $R_{2\pi/m} = R_{2\pi}$ is the identity and in particular $L^2_1 = L^2 (\mathbb R^2, \mathbb C)$. Indeed we will look for $\eta$ in $L^2_m$ for a sufficiently large $m\geq 2$.

An important technical detail is that, while the operator $L^2 \cap \mathscr{S} \ni \omega \mapsto K_2* \omega \in \mathscr{S}'$ {\em cannot} be extended continuously to the whole $L^2$ (cf. Remark \ref{r:Camillo_dumb}), for $m\geq 2$ it {\em can} be extended to a continuous operator from $L^2_m$ into $\mathscr{S}'$: this is the content of the following lemma.

\begin{lemma}\label{l:extension}\label{L:EXTENSION}
For every $m\geq 2$ there is a unique continuous operator $T\colon L^2_m \to \mathscr{S}'$ with the following properties:
\begin{itemize}
    \item[(a)] If $\vartheta\in \mathscr{S}$, then $T (\vartheta) = K_2*\vartheta$ (in the sense of distributions);
    \item[(b)] There is $C>0$ such that for every $\vartheta \in L^2_m$, there is  $v=v(\vartheta)\in W^{1,2}_{\text{loc}}$ with
    \begin{itemize}
        \item[(b1)] $R^{-1} \|v\|_{L^2 (B_R)} + \|Dv\|_{L^2 (B_R)} \leq C\|\vartheta\|_{L^2 (\mathbb R^2)}$ for all $R>0$;
        \item[(b2)] $\textrm{div}\, v =0$ and $\langle T(\vartheta), \varphi\rangle = \int v\cdot \varphi$ for every test function $\varphi \in \mathscr{S}$.
    \end{itemize}
\end{itemize}
\end{lemma}

From now on the operator $T$ will still be denoted by $K_2*$ and the function $v$ will be denoted by $K_2*\omega$. Observe also that, if $\hat\Omega$ is an $L^2_{\text{loc}}$ function such that $\|\hat\Omega\|_{L^2 (B_R)}$ grows polynomially in $R$, the integration of a Schwartz function times $v \hat\Omega$ is a well defined tempered distribution. In the rest of the notes, any time that we write a product $\hat\Omega K_2 * \vartheta $ for an element $\vartheta\in L^2_m$ and an $L^2_{\text{loc}}$ function $\hat\Omega$ we will always implicitly assume that 
$\|\hat\Omega\|_{L^2 (B_R)}$ grows at most polynomially in $R$ and that the product is understoood as a well-defined element of $\mathscr{S}'$.
The relevance of this discussion is that, for $m\geq 2$, we can now consider the operator $L_{\text{ss}}$ as a closed, densely defined unbounded operator on $L^2_m$. We let
 \index{aalLss@$L_{\text{ss}}$}\index{Self-similar operator}
\begin{equation}\label{e:def-Lss-formal}
L_{\text{ss}} (\Omega) = \left(1- {\textstyle{\frac{2}{\alpha}}}\right) \Omega - \textrm{div}\, \big(\big(-{\textstyle{\frac{\xi}{\alpha}}} + \beta \bar V\big) \Omega\big) - \beta ( K_2*\Omega \cdot \nabla) \bar\Omega\, 
\end{equation}
and its domain is 
\begin{equation}\label{e:D(Lss)-formal}
D_m (L_{\text{ss}}) =\{\Omega\in L^2_m : L_{\text{ss}} (\Omega)\in L^2_m\}\, .
\end{equation}
When $\Omega\in \mathscr{S}$ it can be readily checked that $L_{\text{ss}}$ as defined in \eqref{e:def-Lss-formal} coincides with \eqref{e:Lss}. 

The definition makes obvious that $L_{\text{ss}}$ is a closed and densely defined unbounded operator over $L^2_m$. We will later show that $\Omega \mapsto (K_2*\Omega \cdot \nabla) \bar \Omega$ is in fact a compact operator from $L^2_m$ into $L^2_m$ and therefore we have
\begin{equation}\label{e:D(Lss)-formal-2}
D_m (L_{\text{ss}}) := \big\{\Omega\in L^2_m \colon \textrm{div} \big(\beta \bar V\Omega-     {\textstyle{\frac{\xi}{\alpha}}}\Omega\big)\, \in L^2_m\big\}\, .
\end{equation}
From now on, having fixed $m\geq 2$ and regarding $L_{\text{ss}}$ as an unbounded, closed, and densely defined operator in the sense given above, the spectrum \index{spectrum@spectrum} \index{aalspec@$\textrm{spec}_m$} $\textrm{spec}_m\, (L_{\text{ss}})$ on $L^2_m$ is defined as the (closed) set which is the complement of the {\em resolvent set} of $L_{\text{ss}}$, \index{resolvent@resolvent} the latter being the (open) set of $z_0 \in \mathbb C$ such that $L_{\text{ss}}-z_0$ has a bounded inverse $(L_{\text{ss}}-z_0)^{-1} \colon L^2_m \to L^2_m$. (The textbook definition would require the inverse to take values in $D_m (L_{\text{ss}})$; note however that this is a consequence of our very definition of $D_m (L_{\text{ss}})$.)

The choice of $\eta$ will then be defined by the following theorem which summarizes a quite delicate spectral analysis.

\begin{theorem}\label{thm:spectral}\label{THM:SPECTRAL}
For an appropriate choice of $\bar \Omega$ there is an integer $m\geq 2$ with the following property. For every positive $\bar a>0$, if $\beta$ is chosen appropriately large, then there is $\eta\in L^2_m\setminus \{0\}$ and $z_0=a_0+ib_0$ such that:
\begin{itemize}
    \item[(i)] $a_0 \geq \bar a$ and $L_{\text{ss}} (\eta) = z_0 \eta$;
    \item[(ii)] For any $z \in \textrm{spec}_m\, (L_{\text{ss}})$ we have $\textrm{Re}\, z\leq a_0$;
    \item[(iii)] If $b_0=0$, then $\eta$ is real valued;
    \item[(iv)] There is $k\geq 1$ integer and $e:\mathbb R^+\to \mathbb C$ such that $\eta (x) = e (r) e^{ikm \theta}$ if $b_0\neq 0$ and $\eta (x) = \textrm{Re}\, (e(r) e^{ikm\theta})$ if $b_0= 0$;
    \item[(v)] $\eta\in L^1 \cap H^2$ and $K_2*\eta\in L^2$
\end{itemize}
\end{theorem}

In fact we will prove some more properties of $\eta$, namely, suitable regularity and decay at infinity, but these are effects of the eigenvalue equation and will be addressed later. 

The proof of Theorem \ref{thm:spectral} will be split in two chapters. In the first one we regard $L_{\text{ss}}$ as perturbation of a simpler operator $L_{\text{st}}$, which is obtained from $L_{\text{ss}}$ by ignoring the $(1+\frac{\xi}{\alpha}\cdot \nabla)$ part: the intuition behind neglecting this term is that the remaining part of the operator $L_{\text{ss}}$ is multiplied by the constant $\beta$, which will be chosen appropriately large. The second chapter will be dedicated to proving a theorem analogous to Theorem \ref{thm:spectral} for the operator $L_{\text{st}}$. The analysis will heavily take advantage of an appropriate splitting of $L^2_m$ as a direct sum of invariant subspaces of $L_{\text{st}}$. The latter are obtained by expanding in Fourier series the trace of any element of $L^2_m$ on the unit circle. In each of these invariant subspaces the spectrum of $L_{\text{st}}$ can be related to the spectrum of a suitable second order differential operator in a single real variable.

\section{Nonlinear theory}\label{s:nonlinear}

The linear theory will then be used to show Theorem \ref{thm:main3}. In fact, given the decomposition introduced in \eqref{e:Ansatz-1}, we can now formulate a yet more precise statement from which we conclude Theorem \ref{thm:main3} as a corollary. 

\begin{theorem}\label{thm:main4}
Let $p$, $\alpha$, and $\al$ be as in Theorem \ref{thm:main2} and assume $\bar a$ is sufficiently large. Let $\bar \Omega$, $\eta$, $a_0$, and $b_0$ be as in Theorem \ref{thm:spectral} and for every $\varepsilon \in \mathbb R$, $k\in \mathbb N$ consider the solutions $\omega_{\varepsilon,k}$ of \eqref{e:Euler-later-times} and $\Omega_{\varepsilon,k} (\xi, \tau) = e^\tau \omega_{\varepsilon,k} (e^{\tau/\alpha} \xi, e^\tau)$. If we define $\Omega_{\text{per},k}$ through \eqref{e:Ansatz-1}, then there are $\tau_0 = \tau_0 (\varepsilon)$ and $\delta_0>0$, independent of $k$, such that
\begin{equation}\label{e:H2-estimate}
\|\Omega_{\text{per}, k} (\cdot, \tau)\|_{L^2} \leq e^{\tau (a_0+\delta_0)} \qquad\qquad\qquad \forall \tau\leq \tau_0\, . 
\end{equation}
\end{theorem}

\eqref{e:asymptotic-in-t-2} is a simple consequence of \eqref{e:H2-estimate} after translating it back to the original coordinates. In order to give a feeling for why \eqref{e:H2-estimate} holds we will detail the equation that $\Omega_{\text{per}, k}$ satisfies.
First of all subtracting the equation satisfied by $\tilde{\Omega}$ from the one satisfied by $\Omega_{\varepsilon, k}$ we achieve
\begin{align*}
&\partial_\tau \Omega_{\text{lin}} + \partial_\tau \Omega_{\text{per},k} - \left(1+{\textstyle{\frac{\xi}{\alpha}}}\cdot \nabla\right) \Omega_{\text{lin}}
-\left(1+{\textstyle{\frac{\xi}{\alpha}}}\cdot \nabla\right) \Omega_{\text{per}, k} \\
+ & (\tilde{V} \cdot \nabla) \Omega_{\text{lin}} + V_{\text{lin}}\cdot \nabla \tilde{\Omega} + (\tilde{V}\cdot \nabla ) \Omega_{\text{per},k} + (V_{\text{per},k}\cdot \nabla) \tilde{\Omega} + (V_{\text{lin}}\cdot \nabla) \Omega_{\text{per}, k}\\
+ & (V_{\text{per}, k} \cdot \nabla) \Omega_{\text{lin}}
+ (V_{\text{lin}}\cdot \nabla) \Omega_{\text{lin}} + (V_{\text{per},k}\cdot \nabla) \Omega_{\text{per}, k} = 0\, ,
\end{align*}
where we have used the convention $\tilde{V}=K_2*\tilde\Omega$, $V_{\text{per},k} = K_2* \Omega_{\text{per},k}$, and $V_{\text{lin}}= K_2* \Omega_{\text{lin}}$. Next recall that $\tilde{\Omega}=\beta\bar\Omega + \Omega_r$ and recall also the definition of $L_{\text{ss}}$ in \eqref{e:Lss} and the fact that $\partial_\tau \Omega_{\text{lin}} - L_{\text{ss}} (\Omega_{\text{lin}})= 0$. In particular formally we reach
\begin{align}
& (\partial_{\tau} - L_{\text{ss}}) \Omega_{\text{per}, k} + ((V_{\text{lin}}+V_r)\cdot \nabla) \Omega_{\text{per},k} + (V_{\text{per},k} \cdot \nabla) (\Omega_{\text{lin}} + \Omega_r)\nonumber\\
&\qquad+ (V_{\text{per},k}\cdot \nabla) \Omega_{\text{per},k}
= -(V_{\text{lin}}\cdot \nabla) \Omega_{\text{lin}} - (V_r\cdot \nabla) \Omega_{\text{lin}} - (V_{\text{lin}}\cdot \nabla) \Omega_r\, ,\label{e:master}
\end{align}
which must be supplemented with the initial condition 
\[
\Omega_{\text{per},k} (\cdot, -k)= 0\, .
\]
In fact, in order to justify \eqref{e:master} we need to show that $\Omega_{\text{per},k} (\cdot, \tau)\in L^2_m$ for every $\tau$, which is the content of the following elementary lemma. 

\begin{lemma}\label{l:evolution-in-L2m}
The function $\Omega_{\text{per},k} (\cdot, \tau)$ belongs to $L^2_m$ for every $\tau$.
\end{lemma}
\begin{proof}
It suffices to prove that $\omega_{\varepsilon, k} (\cdot, t)$ is $m$-fold symmetric, since the transformation rule then implies that $\Omega_{\varepsilon,k} (\cdot, \tau)$ is $m$-fold symmetric and $\Omega_{\text{per}, k} (\cdot, \tau)$ is obtained from the latter by subtracting $e^{a_0\tau} \textrm{Re} (e^{ib_0\tau} \eta) + \tilde{\Omega} (\cdot, \tau)$, which is also $m$-fold symmetric. In order to show that $\omega_{\varepsilon, k}$ is $m$-fold symmetric just consider that $\omega_{\varepsilon, k} (R_{2\pi/m} (\cdot), \tau)$ solves \eqref{e:Euler-later-times} because both the forcing term and the initial data are invariant under a rotation of $\frac{2\pi}{m}$ (and the Euler equations are rotationally invariant). Then the uniqueness part of Yudovich's statement implies $\omega_{\varepsilon, k} (\cdot, t) = \omega_{\varepsilon, k} (R_{2\pi/m} (\cdot), t)$.
\end{proof}
 
We proceed with our discussion and observe that $V_{\text{lin}} + V_r$ and $\Omega_{\text{lin}}+\Omega_r$ are both ``small'' in appropriate sense for sufficiently negative times, while, because of the initial condition being $0$ at $-k$, for some time after $-k$ we expect that the quadratic nonlinearity $(V_{\text{per},k}\cdot \nabla) \Omega_{\text{per},k}$ will not contribute much to the growth of $\Omega_{\text{per}, k} (\cdot, \tau)$. Schematically, we can break \eqref{e:master} as 
\begin{align}
& (\partial_{\tau} - L_{\text{ss}}) \Omega_{\text{per}, k} + \underbrace{((V_{\text{lin}}+V_r)\cdot \nabla) \Omega_{\text{per},k} + (V_{\text{per},k} \cdot \nabla) (\Omega_{\text{lin}} + \Omega_r)}_{\mbox{small linear terms}}\nonumber\\
&\qquad + \underbrace{(V_{\text{per},k}\cdot \nabla) \Omega_{\text{per},k}}_{\mbox{quadratic term}}
= \underbrace{-(V_{\text{lin}}\cdot \nabla) \Omega_{\text{lin}} - (V_r\cdot \nabla) \Omega_{\text{lin}} - (V_{\text{lin}}\cdot \nabla) \Omega_r}_{\mbox{forcing term } \mathscr{F}}\, ,\label{e:master-schematics}
\end{align}
In particular we can hope that the growth of $\Omega_{\text{per},k} (\cdot, \tau)$ is comparable to that of the solution of the following ``forced'' linear problem
\begin{equation}\label{e:master-linear}
(\partial_{\tau} - L_{\text{ss}}) \Omega = \mathscr{F}\, .
\end{equation}
Observe that we know that $\Omega_{\text{lin}} (\cdot, \tau)$ and $V_{\text{lin}} (\cdot, \tau)$ decay like $e^{a_0 \tau}$. We can then expect to gain a slightly faster exponential decay for $\mathscr{F} (\cdot, \tau)$ because of the smallness of $V_r$ and $\Omega_r$. On the other hand from Theorem \ref{thm:spectral} we expect that the semigroup generated by $L_{\text{ss}}$ enjoys growth estimates of type $e^{a_0\tau}$ on $L^2_m$ (this will be rigorously justified using classical results in the theory of strongly continuous semigroups). We then wish to show, using the Duhamel's formula for the semigroup $e^{\tau L_{\text{ss}}}$, that the growth of $\Omega_{\text{per},k}$ is bounded by $e^{a_0\tau} (e^{\delta_0 \tau} - e^{-\delta_0 k})$ for some positive $\delta_0$ for some time $\tau$ after the initial $-k$: the crucial point will be to show that the latter bound is valid for $\tau$ up until a ``universal'' time $\tau_0$, independent of $k$.

Even though intuitively sound, this approach will require several delicate arguments, explained in the final chapter of the notes. In particular:
\begin{itemize}
    \item we will need to show that the quadratic term $(V_{\text{per},k}\cdot \nabla) \Omega_{\text{per},k}$ is small up to some time $\tau_0$ independent of $k$, in spite of the fact that there is a ``loss of derivative'' in it (and thus we cannot directly close an implicit Gronwall argument using the Duhamel formula and the semigroup estimate for $L_{\text{ss}}$);
\item The terms $\Omega_r$ and $V_r$ are not really negligible in absolute terms, but rather, for very negative times, they are supported in a region in space which goes towards spatial $\infty$. 
\end{itemize}
The first issue will be solved by closing the estimates in a space of more regular functions, which contains $L^2$ and embeds in $L^\infty$ (in fact $L^2\cap W^{1,4}$): the bound on the growth of the $L^2$ norm will be achieved through the semigroup estimate for $L_{\text{ss}}$ via Duhamel, while the bound of the first derivative will be achieved through an energy estimate, which will profit from the $L^2$ one. The second point by restricting further the functional space in which we will close to estimates for $\Omega_{\text{per}, k}$. We will require an appropriate decay of the derivative of the solutions, more precisely we will require that the latter belong to $L^2 (|x|^2\,\mathrm dx)$. Of course in order to use this strategy we will need to show that the initial perturbation $\eta$ belongs to the correct space of functions.

\newpage
\section{Dependency tree} The following tree schematizes the various intermediate statetements that lead to the main result of this text, Theorem \ref{thm:main}. The dashed lines signify improvements of the intermediate statements which are not needed to obtain Theorem~\ref{thm:main}. We have included them because they deliver interesting additional information and because they are the original statements of Vishik, see \cite{Vishik1, Vishik2, Vishik3}. 

\vspace{1cm}
\begin{tikzpicture}
\tikzset{level distance=40pt}
\tikzset{level 1/.style={sibling distance=-50pt}}
\begin{scope}
\Tree%
[.{Main result, Theorem \ref{thm:main}}
  {Proposition \ref{p:convergence}} 
  [.{Theorem \ref{thm:main2}} 
    [.{Theorem \ref{thm:main3}}
      [.{Linear Theory}
        [.{Theorem \ref{thm:spectral}}
          [.\node(a2){Theorem \ref{thm:spectral2}};  
            [.\node(a1){Theorem \ref{thm:spectral3}};
              {Theorem \ref{thm:spectral5}}
            ]
          ]
        ]
      ]
      [.{Non-linear Theory}
        [.{Theorem \ref{thm:main4}} 
          [.{Theorem \ref{t:group}}  
            {Corollary \ref{c:structural}}
          ]
          {Proposition \ref{p:X-bounds}}
          [.{Lemma \ref{l:final-estimates}}
            {Lemma \ref{l:initial-bound}}
          ]
        ]
      ]
    ]
  ]
  {Formula \eqref{eq:difference-of-the-omega}} 
]
\end{scope}
\begin{scope}[xshift=1in,yshift=-4.45in]
\Tree%
[.\node(b1){Remark \ref{r:better2}};
  [.\node(c1){Theorem \ref{thm:spectral-stronger-2}};
    {Theorem \ref{thm:Vishikversion}}
  ]
]
\end{scope}
\begin{scope}[xshift=-1in,yshift=-4.3in]
\Tree%
[.\node(b2){Remark \ref{r:better}}; 
  \node(c2){Theorem \ref{thm:spectral-stronger}};
]
\end{scope}
\begin{scope}[dashed]
\draw  (a1) -- (b1); 
\draw  (a2) .. controls (-5.5,-7) and (-5.5,-11) .. (b2);
\end{scope}
\begin{scope}
\draw (c2) -- (c1);
\end{scope}
\end{tikzpicture}


\chapter{Linear theory: Part I}
\label{chapter:linearparti}

 In this chapter, we will reduce Theorem \ref{thm:spectral} to an analogous spectral result for another differential operator, and we will also show an important corollary of Theorem \ref{thm:spectral} concerning the semigroup that it generates. We start by giving the two relevant statements, but we will need first to introduce some notation and terminology. 
 
 First of all, in the rest of the chapter we will always assume that the positive integer $m$ is at least $2$. We then introduce a new (closed and densely defined) operator on $L^2_m$, which we will denote by $L_{\text{st}}$. The operator is defined by\index{aalLst@$L_{\text{st}}$}
 \begin{equation}
 L_{\text{st}} (\Omega) = - \text{div}\, (\bar V \Omega) - (K_2*\Omega \cdot \nabla) \bar \Omega\,  
 \end{equation}
 and (recalling that the operator $\Omega\mapsto (K_2* \Omega\cdot \nabla) \bar \Omega$ is bounded and compact, as  will be shown below) its domain in $L^2_m$ is given by
 \begin{equation}
D_m (L_{\text{st}}) = \{\Omega\in L^2_m : \textrm{div}\, (\bar V \Omega)\in L^2\}\, .
 \end{equation}
The key underlying idea behind the introduction of $L_{\text{st}}$ is that we can write $L_{\text{ss}}$ as
\[
L_{\text{ss}} = \big(1+{\textstyle{\frac{\xi}{\alpha}}}\cdot \nabla\big) + \beta L_{\text{st}} \, 
\]
and since $\beta$ will be chosen very large, we will basically study the spectrum of $L_{\text{ss}}$ as a perturbation of the spectrum of $\beta L_{\text{st}}$. In particular Theorem \ref{thm:spectral} will be derived from a more precise spectral analysis of $L_{\text{st}}$. Before coming to it, we split the space $L^2_m$ into an appropriate infinite sum of closed orthogonal subspaces. 

First of all, if we fix an element $\vartheta\in L^2 (\mathbb R^2)$ and we introduce the polar coordinates $(\theta, r)$ through $x= r (\cos \theta , \sin \theta)$, we can then use the Fourier expansion to write
\begin{equation}\label{e:Fourier}
\vartheta (x) =\sum_{k\in \mathbb Z} a_k (r) e^{ik\theta}\, 
\end{equation}
where
\[
a_k (r) := \frac{1}{2\pi} \int_0^{2\pi} \vartheta(r \cos(\theta),r\sin(\theta)) e^{-ik\theta}\,\mathrm d\theta .
\]
By Plancherel's formula,
\[
\|\vartheta\|_{L^2 (\R^2)}^2 = 2\pi \sum_{k\in \mathbb Z} \|a_k\|^2_{L^2 (\R^+, r\,\mathrm dr)}\, .
\]
In particular it will be convenient to introduce the subspaces\index{aalUk@$U_k$}
\begin{equation}
U_k :=\{f(r) e^{ik\theta} : f \in L^2 (\mathbb R^+, r\,\mathrm dr)\}\, .    
\end{equation}
Each $U_k$ is a closed subspace of $L^2$, distinct $U_k$'s are orthogonal to each other and moreover
\begin{equation}\label{e:Fourier-2}
L^2_m = \bigoplus_{k\in \mathbb Z} U_{km}\, .
\end{equation}

Each $U_{mk}$ is an invariant space of $L_{\text{st}}$, and it can be easily checked that $U_{mk}\subset D_m (L_{\text{st}})$ and that indeed the restriction of $L_{\text{st}}$ to $U_{mk}$ is a bounded operator. Following the same convention as for $L_{\text{ss}}$ we will denote by \index{spectrum@spectrum}\index{aalspec@$\textrm{spec}_m$} $\textrm{spec}_m\, (L_{\text{st}})$ the spectrum of $L_{\text{st}}$ on $L^2_m$.

\begin{theorem}\label{thm:spectral2}\label{THM:SPECTRAL2}
For every $m\geq 2$ and every $\bar\Omega$
we have
\begin{itemize}
    \item[(a)] each $z_i\in \textrm{spec}_m\, (L_{\text{st}})\cap \{z: \textrm{Re} \,z\, \neq 0\}$ belongs to the discrete spectrum and if $\textrm{Im}\, (z_i)=0$, then there is a nontrivial real eigenfunction relative to $z_i$.
\end{itemize}
Moreover, for an appropriate choice of $\bar\Omega$ there is an integer $m\geq 2$ such that:
\begin{itemize}
\item[(b)] $\textrm{spec}_m\, (L_{\text{st}})\cap \{z: \textrm{Re}\, z >0\}$ is nonempty.
\end{itemize}
\end{theorem}

\begin{remark}\label{r:better}\label{R:BETTER}
The theorem stated above contains the minimal amount of information that we need to complete the proof of Theorem \ref{thm:main2}. 
We can however infer some additional conclusions with more work. To this end, we {recall that in the case of an isolated point $z$ in the spectrum of a closed, densely defined operator $A$, the Riesz projector is defined as 
    \[
    \frac{1}{2\pi i} \int_\gamma (w -A)^{-1}\, dw
    \]
    for any simple closed rectifiable contour $\gamma$ bounding a closed disk $D$ with $D \cap \textrm{spec}\, (A) = \{z\}$. For an element of the discrete spectrum the Riesz projector has finite rank (the algebraic multiplicity of the eigenvalue \index{algebraic multiplicity@algebraic multiplicity} $z$).}

In the framework of Theorem \ref{thm:main2}, we can show that 
\begin{itemize}
   \item[(c)] $m$ can be chosen so that, in addition to (b), $\textrm{spec}_m\, (L_{\text{st}})\cap \{z: \textrm{Re}\, z >0\}$ is finite and the image of the Riesz projector
   \index{Riesz projector}\index{aalPz@$P_z$}
    $P_{z}$ of $L_{\text{st}}$ relative to each $z\in \textrm{spec}_m\, (L_{\text{st}})\cap \{z: \textrm{Re}\, z >0\}$ is contained in $U_m\cup U_{-m}$.
    \end{itemize}
Since this property is not needed to prove Theorem~\ref{thm:main2} we defer its proof to Appendix~\ref{a:better}. 
\end{remark}

In \cite{Vishik2} Vishik claims the following greatly improved statement.

\begin{theorem}\label{thm:spectral-stronger}\label{THM:SPECTRAL-STRONGER}
For a suitable $\bar \Omega$:
\begin{itemize}
    \item[(c')] $m$ can be chosen so that, in addition to  (b) and (c), $\textrm{spec}_m\, (L_{\text{st}})\cap \{z: \textrm{Re}\, z >0\}\cap U_m$ consists of a single element, with algebraic multiplicity $1$ in $U_m$. \index{algebraic multiplicity@algebraic multiplicity}
\end{itemize}
\end{theorem}
Since the spectrum of $L_{\text{st}}$ is invariant under complex conjugation (b), (c), and (c') imply that $\textrm{spec}_m\, (L_{\text{st}})\cap \{\textrm{Re}\, z>0\}$ consists either of a single real eigenvalue or of two complex conjugate eigenvalues. In the first case, the algebraic and geometric multiplicity \index{geometric multiplicity@geometric multiplicity} of the eigenvalue is $2$ and the space of eigenfunctions has a basis consisting of an element of $U_m$ and its complex conjugate in $U_{-m}$. In the second case the two eigenvalues $z$ and $\bar z$ have algebraic multiplicity $1$ and their eigenspaces are generated, respectively, by an element of $U_m$ and its complex conjugate in $U_{-m}$.
The argument given in \cite{Vishik2} for (c') is however not complete. Vishik provided later (\cite{Vishik3}) a way to close the gap. In Appendix~\ref{a:better} we will give a proof of Theorem \ref{thm:spectral-stronger} along his lines.

\medskip

In this chapter we also derive an important consequence of Theorem \ref{thm:spectral} for the semigroup generated by $L_{\text{ss}}$.\index{Semigroup}

\begin{theorem}\label{t:group}\label{T:GROUP}
For every $m\geq 2$, $L_{\text{ss}}$ is the generator of a strongly continuous semigroup on $L^2_m$ which will be denoted by $e^{\tau L_{\text{ss}}}$\index{aalEtauLss@$e^{\tau L_{\text{ss}}}$}, and the growth bound $\omega (L_{\text{ss}})$ of $e^{\tau L_{\text{ss}}}$ equals\index{Semigroup, growth bound}
\[
a_0 := \sup \{\textrm{Re}\, z_0 : z_0\in \textrm{spec}_m (L_{\text{ss}})\}\, <\infty\, 
\]
if $a_0 \geq 1-\frac{1}{\alpha}$.
In other words, for every $\delta>0$, there is a constant $M (\delta)$ with the property that
\begin{equation}\label{e:growth-bound}
\left\|e^{\tau L_{\text{ss}}} \Omega\right\|_{L^2} \leq M (\delta) e^{(a_0 +\delta) \tau} \|\Omega\|_{L^2}
\qquad \qquad \forall \tau\geq 0,\, \forall \Omega\in L^2_m\, .
\end{equation}
\end{theorem}

\section{Preliminaries}\label{s:abstract-operators}

In this section we start with some preliminaries which will take advantage of several structural properties of the operators $L_{\text{ss}}$ and $L_{\text{st}}$. First of all we decompose $L_{\text{st}}$ as
\begin{equation}\label{e:decompo_L_st}
L_{\text{st}} = S_1 + \mathscr{K}\, ,
\end{equation}\index{aalKscr@$\mathscr{K}$}\index{aalS1@$S_1$}where
\begin{align}
S_1 (\Omega)&:= - \textrm{div}\, (\bar V \Omega)\label{e:S1}\\
\mathscr{K} (\Omega) &:= - (K_2*\Omega \cdot \nabla) \bar \Omega \label{e:compatto}\, .
\end{align}
Hence we introduce the operator\index{aalS2@$S_2$}
\begin{equation}
S_2 (\Omega) := \textrm{div} \left(\left(\frac{\xi}{\alpha} - \beta \bar V\right) \Omega\right) - \frac{\Omega}{\alpha}\, ,
\end{equation}
so that we can decompose $L_{\text{ss}}$ as 
\begin{equation}
L_{\text{ss}} = \left(1-\frac{1}{\alpha}\right) + S_2 + \beta \mathscr{K}\, .
\end{equation}
The domains of the various operators $A$ involved are always understood as $D_m (A):= \{\Omega : A(\Omega)\in L^2\}$.

Finally, we introduce the real Hilbert spaces $L^2_m (\mathbb R)$ and $U_j (\mathbb R)$ by setting\index{aalL2mR@$L^2_m(\R)$}\index{aalUkR@$U_k(\R)$}
\begin{align}
L^2_m (\mathbb R) &:= \{\textrm{Re}\, \Omega : \Omega \in L^2_m\}\, 
\end{align}
and, for $j>0$ natural,
\begin{equation}
U_j (\mathbb R) :=\{\textrm{Re}\, \Omega: \Omega \in U_j\}\, .
\end{equation}
Observe that while clearly $L^2_m (\mathbb R)$ is a real subspace of $L^2_m$, $U_j (\mathbb R)$ is a real subspace of $U_j \oplus U_{-j}$. 

As it is customary, $L^2_m (\mathbb R)$ and its real vector subspaces are endowed with the inner product\index{Rotationally symmetric function space, inner product}
\begin{equation}
\langle \Omega, \Xi\rangle_{\mathbb R} = \int \Omega\, \Xi\, ,
\end{equation}
while $L^2_m$ and its complex vector subspaces are endowed with the Hermitation product
\begin{equation}
\langle \Omega, \Xi\rangle_{\mathbb C} = \int (\textrm{Re}\, \Omega\, \textrm{Re}\, \Xi +
\textrm{Im}\, \Omega\, \textrm{Im}\, \Xi) + i \int (\textrm{Im}\, \Omega\, \textrm{Re}\, \Xi - \textrm{Re} \, \Omega\, \textrm{Im}\, \Xi)\, .
\end{equation}
We will omit the subscripts from $\langle \cdot, \cdot \rangle$ when the underlying field is clear from the context. The following proposition details the important structural properties of the various operators. A closed unbounded operator $A$ on $L^2_m$ will be called \emph{real} if its restriction $A_{\mathbb R}$ to $L^2_m (\mathbb R)$ is a closed, densely defined operator with domain $D_m (A) \cap L^2_m (\mathbb R)$ such that $A(\Omega)\in L^2_m(\mathbb R)$ for all $\Omega\in D_m(A)\cap L^2_m(\mathbb R)$.

\begin{proposition}\label{p:abstract}
\begin{itemize}
    \item[(i)] The operators $\mathscr{K}$, $S_1$ and $S_2$ are all real operators.
    \item[(ii)] $\mathscr{K}$ is bounded and compact. More precisely there is a sequence of finite dimensional vector spaces $V_n \subset C^\infty_c (\mathbb R^2,\mathbb C)\cap L^2_m$ with the property that, if $P_n$ denotes the orthogonal projection onto $V_n$, then
    \begin{equation}\label{e:explicit-approx}
    \lim_{n\to\infty} \|\mathscr{K} - P_n\circ \mathscr{K}\|_O = 0\, ,
    \end{equation}
    where $\|\cdot\|_O$ denotes the operator norm. 
    \item[(iii)] $S_1$ and $S_2$ are skew-adjoint.
    \item[(iv)] $D_m (L_{\text{st}}) = D_m (S_1)$ and $D_m (L_{\text{ss}}) = D_m (S_2)$.
    \item[(v)] $U_{km}$ is an invariant subspace of $S_1, S_2, \mathscr{K}, L_{\text{st}}, L_{\text{ss}}$.
    \item[(vi)] The restrictions of $S_1$ and $L_{\text{st}}$ to each $U_{km}$ are bounded operators.
\end{itemize}
\end{proposition}
\begin{proof} The verification of (i), (iii), (iv), (v), and (vi) are all simple and left therefore to the reader. We thus come to (ii) and prove the compactness of the operator $\mathscr{K}$. Recalling Lemma \ref{l:extension}, for every $\Omega \in L^2_m$ we can write the tempered distribution $\mathscr{K} (\Omega)$ as 
\begin{equation}
\mathscr{K} (\Omega) = \nabla \bar \Omega \cdot V    
\end{equation}
where $V=V(\Omega)$ is a $W^{1,2}_{\text{loc}}$ function with the properties that
\begin{equation}\label{e:stima-W12}
R^{-1} \|V \|_{L^2 (B_R)} + \|DV\|_{L^2 (B_R)} \leq C \|\Omega\|_{L^2} \qquad \forall R \geq 0\, .   \end{equation}
Since $|\nabla \bar \Omega (\xi)| \leq C |\xi|^{-\al -1}$ for $|\xi|\geq 1$, whenever $R\geq 1$ we can estimate
\begin{align*}
\|\mathscr{K} (\Omega)\|^2_{L^2 (B_R^c)} &= \sum_{j=0}^\infty \|\mathscr{K} (\Omega)\|_{L^2 (B_{2^{j+1} R}\setminus B_{2^j R})}^2 
\\&
\leq C R^{-2-2\al} \sum_{j=0}^\infty 2^{-2 (1+\al) j} \|V\|^2_{L^2 (B_{2^{j+1} R})}\\
&\leq C R^{-2-2\al} \sum_{j=0}^\infty 2^{-2(1+\al) j} 2^{2j+2} R^2 \|\Omega\|_{L^2}^2
\leq C R^{-2\al} \|\Omega\|_{L^2}^2\, .
\end{align*}
This shows at the same time that
\begin{itemize}
    \item $\mathscr{K}$ is a bounded operator;
    \item If we introduce the operators
    \begin{equation}
    \Omega \mapsto \mathscr{K}_N (\Omega) := \mathscr{K} (\Omega) \mathbf{1}_{B_N}\, ,    
    \end{equation}
    then $\|\mathscr{K}_N- \mathscr{K}\|_{O}\to 0$.
\end{itemize}
Since the uniform limit of compact operators is a compact operator, it suffices to show that each $\mathscr{K}_N$ is a compact operator. This is however an obvious consequence of \eqref{e:stima-W12} and the compact embedding of $W^{1,2} (B_N)$ into $L^2 (B_N)$.

As for the remainder of the statement (ii), by the classical characterization of compact operators on a Hilbert space, for every $\varepsilon > 0$ there is a finite-rank linear map $L_N$ such that $\|\mathscr{K} - L_N\|_O \leq \frac{\varepsilon}{4}$. If we denote by $W_N$ the image of $L_N$ and by $Q_N$ the orthogonal projection onto it, given that $Q_N \circ L_N = L_N$ we can estimate
\[
\|Q_N \circ \mathscr{K} - \mathscr{K}\|_O \leq \|Q_N \circ \mathscr{K} - Q_N \circ L_N\|_O
+ \|L_N - \mathscr{K}\|_O \leq 2 \|L_N - \mathscr{K}\|_0 \leq \frac{\varepsilon}{2}\, .
\]
Fix next an orthonormal base $\{w_1, \ldots, w_N\}$ of $W_N$ and, using the density of $C^\infty_c (\mathbb R^2)$, approximate each element $w_i$ in the base with $v_i\in C^\infty_c (\mathbb R^2, \mathbb C)$. This can be done for instance convolving $w_i$ with a smooth radial kernel and multiplying by a suitable cut-off function. If the $v_i$'s are taken sufficiently close to $w_i$, the orthogonal projection $P_N$ onto $V_N = \textrm{span}\, (v_1, \ldots , v_N)$ satisfies $\|Q_N-P_N\|_O \leq \frac{\varepsilon}{2\|\mathscr{K}\|_O}$ and thus 
\[
\|\mathscr{K} - P_N \circ \mathscr{K}\|_O \leq \|\mathscr{K} - Q_N \circ \mathscr{K}\|_O + \|P_N-Q_N\|_O \|\mathscr{K}\|_O  \leq \varepsilon\, . \qedhere
\]
\end{proof}

\section{Proof of Theorem \ref{t:group} and proof of Theorem \ref{thm:spectral2}(a)}\label{subsect:Proof-of-spectral2-and-group}

The above structural facts allow us to gain some important consequences as simple corollaries of classical results in spectral theory, which we gather in the next statement. Observe in particular that the statement (a) of Theorem \ref{thm:spectral2} follows from it.

In what follows we take the definition of essential spectrum of an operator as given in \cite{EngelNagel}. We caution the reader that other authors use different definitions; at any rate the main conclusion about the essential spectra of the operators $L_{\text{ss}}$ and $L_{\text{st}}$ in Corollary \ref{c:structural} below depends only upon the property that the essential and discrete spectra are disjoint (which is common to all the different definitions used in the literature). 

\begin{corollary}\label{c:structural}
The essential spectrum of $L_{\text{st}}$ and the essential spectrum of $L_{\text{ss}} - \left(1-\frac{1}{\alpha}\right)$ are contained in the imaginary axis, while the remaining part of the spectrum is contained in the discrete spectrum. In particular, every $z\in \textrm{spec}_m (L_{\text{st}})$ (resp. $z\in \textrm{spec}_m (L_{\text{ss}})$) with nonzero real part (resp. real part different from $1-\frac{1}{\alpha}$) has the following properties.
\begin{itemize}
    \item[(i)] $z$ is isolated in $\textrm{spec}_m (L_{\text{st}})$ (resp. $\textrm{spec}_m (L_{\text{ss}})$);
    \item[(ii)] There is at least one nontrivial $\Omega$ such that $L_{\text{st}} (\Omega) = z \Omega$ (resp. $L_{\text{ss}} (\Omega) = z \Omega$) and if $\textrm{Im}\, (z)=0$, then $\Omega$ can be chosen to be real-valued; 
    \item[(iii)] The Riesz projection $P_z$ has finite rank; 
    \item[(iv)] $\textrm{Im}\, (P_z) = \bigoplus_{k\in \mathbb Z\setminus \{0\}} (\textrm{Im}\, (P_z)\cap U_{km})$ and in particular the set $\textrm{Im}\, (P_z) \cap U_{km}$ is empty for all but a finite number of $k$'s and it is nonempty for at least one $k$.  
\end{itemize}
Moreover, Theorem \ref{t:group} holds. 
\end{corollary}

\begin{proof} The points (i)-(iii) are consequence of classical theory, but we present briefly their proofs referring to \cite{Kato}. We also outline another approach based on the analytic Fredholm theorem. Observe that addition of a constant multiple $c$ of the identity only shifts the spectrum (and its properties) by the constant $c$. The statements for $L_{\text{ss}}$ are thus reduced to similar statements for $S_2+\beta \mathscr{K}$. 
Next since the arguments for $L_{\text{st}} = S_1 + \mathscr{K}$ only use the skew-adjointness of $S_1$ and the compactness of $\mathscr{K}$, they apply verbatim to $S_2+\beta\mathscr{K}$. We thus only argue for $L_{\text{st}}$. First of all observe that, since $S_1$ is skew-adjoint, its spectrum is contained in the imaginary axis. In particular, for every $z$ with $\textrm{Re}\, z \neq 0$ the operator $S_1-z$ is invertible and thus Fredholm with Fredholm index $0$. Hence by \cite[Theorem 5.26, Chapter IV]{Kato}, $L_{\text{st}}-z= S_1-z +\mathscr{K}$ is as well Fredholm and has index $0$. By \cite[Theorem 5.31, Chapter IV]{Kato} there is a discrete set $\Sigma \subset \{z: \textrm{Re}\, z\neq 0\}$ with the property that the dimension of the kernel (which equals that of the cokernel) of $L_{\text{st}} -z$  is constant on the open sets $\{\textrm{Re}\, z > 0\}\setminus \Sigma$ and $\{\textrm{Re}\, z<0\}\setminus \Sigma$. Since, for every $z$ such that $|\textrm{Re}\, z|> \|\mathscr{K}\|_O$, we know that $L_{\text{st}}-z$ has a bounded inverse from the Neumann series, the kernel (and cokernel) of $L_{\text{st}}-z$ equals $\{0\}$ on $\{\textrm{Re}\, z \neq 0\}\setminus \Sigma$. From \cite[Theorem 5.28, Chapter IV]{Kato} it then follows that $\Sigma$ is a subset of the discrete spectrum of $L_{\text{st}}$. Obviously the essential spectrum must be contained in the imaginary axis.  

An alernative approach is based on the analytic Fredholm theorem, see \cite[Theorem 7.92]{RenRogBook}. We apply such theorem to the family of compact operators on $L^2_m$ defined by 
$$z \to B(z)= (S_1 -z)^{-1} \circ \mathscr{K},$$ parametrized by 
$z\in \{\Re z>0\}$.
Observe that $(S_1 -z)$ is invertible if $\Re z>0$ and given by the formula 
$$(S_1 -z)^{-1} \Omega = \int_0^\infty e^{-sz} \Omega(\mathbf{X}_s)\, ds,$$
where $\mathbf{X}_s$ is the flow of the vector field $V$. Such formula also implies the existence of the complex derivative in $z$ of the operator. 
The theorem guarantees that either $(I-B(z))^{-1}$ does not exist for any $z\in \{\Re z>0\}$,  or  $(I-B(z))^{-1}$ exist for any $z\in \{\Re z>0\}$, apart from a discrete set. The first case can be discarded in our situation because for $z$ with large real part we know the operator is invertible.

In order to show (iv), denote by $P_k$ the orthogonal projection onto $U_{km}$ and observe that, since $L_{\text{st}} \circ P_k = P_k \circ L_{\text{st}}$, 
\begin{equation}\label{e:commute}
P_z \circ P_k = \frac{1}{2\pi i} \int_\gamma \frac{1}{w-L_{\text{st}}}\circ P_k \, dw 
= \frac{1}{2\pi i} \int P_k \circ \left(\frac{1}{w-L_{\text{st}}}\right)\, dw = P_k \circ P_z\, .
\end{equation}
Writing 
\begin{equation}\label{e:splitting}
P_z = \sum_k P_z\circ P_k
\end{equation}
and observing that the commutation \eqref{e:commute} gives the orthogonality of the images of the $P_z\circ P_k$, since $\textrm{Im}\, (P_z)$ is finite dimensional, we conclude that the sum is finite, i.e. that $P_z\circ P_k =0$ for all but finitely many $k$'s. Moreover, since $P_z^2 = P_z$ and $P_z$ equals the identity on $\textrm{Im}\, (P_z)$, we see immediately that $U_{km} \cap \textrm{Im}\, (P_z) = \textrm{Im}\, (P_z\circ P_k)$.

We now come to the proof of Theorem \ref{t:group}.
We have already shown that, if $\textrm{Re}\, \lambda$ is large enough, then $\lambda$ belongs to the resolvent of $L_{\text{ss}}$, which shows that $a_0 < \infty$. Next, observe that $L_{\text{ss}}$ generates a strongly continuous group if and only if $S_2+\beta \mathscr{K}$ does. On the other hand, using the skew-adjointness of $S_2$, we conclude that, if $\textrm{Re}\, z > \beta \|\mathscr{K}\|_O$, then $z$ is in the resolvent of $S_2+\beta \mathscr{K}$ and 
\[
\|(S_2+\beta \mathscr{K} - z)^{-1}\|_O \leq \frac{1}{\textrm{Re}\, z - \beta \|\mathscr{K}\|_O}\, .
\]
Therefore we can apply \cite[Corollary 3.6, Chapter II]{EngelNagel} to conclude that $S_2+\beta \mathscr{K}$ generates a strongly continuous semigroup. Since the same argument applies to $-S_2-\beta \mathscr{K}$, we actually conclude that indeed the operator generates a strongly continuous group. 

Next we invoke \cite[Corollary 2.11, Chapter IV]{EngelNagel} that characterizes the growth bound $\omega_0 (L_{\text{ss}})$ of the semigroup $e^{tL_{\text{ss}}}$ as 
\[
\omega_0 (L_{\text{ss}}) = \max \{\omega_{\text{ess}} (L_{\text{ss}}), a_0\}\, ,
\]
where $\omega_{\text{ess}}$ is the essential growth bound of \cite[Definition 2.9, Chapter IV]{EngelNagel}. By \cite[Proposition 2.12, Chapter IV]{EngelNagel}, $\omega_{\text{ess}} (L_{\text{ss}})$ equals $\omega_{\text{ess}} (1-\frac{1}{\alpha} + S_2)$ and, since $e^{\tau S_2}$ is a unitary operator, the growth bound of $e^{(1-1/\alpha +S_2)\tau}$ equals $1-\frac{1}{\alpha}$, from which we conclude that $\omega_{\text{ess}} (1-\frac{1}{\alpha} + S_2)\leq 1-\frac{1}{\alpha}$. In particular we infer that if $a_0\geq 1-\frac{1}{\alpha}$, then $\omega_0 (L_{\text{ss}})=a_0$.  
\end{proof}

\section{Proof of Theorem \ref{thm:spectral}: preliminary lemmas}

In this and the next section we will derive Theorem \ref{thm:spectral} from Theorem \ref{thm:spectral2}. It is convenient to introduce the following operator:
\begin{equation}\label{e:Lbeta}
L_\beta : = \frac{1}{\beta} \left(L_{\text{ss}} - \left(1-{\textstyle{\frac{1}{\alpha}}}\right)\right)
= \frac{1}{\beta} S_2 + \mathscr{K} \, .
\end{equation}
In particular
\begin{equation}\label{e:Lbeta-2}
L_\beta (\Omega) = \frac{1}{\beta}\left[\textrm{div}\, \left(\frac{\xi}{\alpha} \Omega\right)- \frac{\Omega}{\alpha}\right] + L_{st} (\Omega)\, .    
\end{equation}
Clearly the spectrum of $L_{\text{ss}}$ can be easily computed from the spectrum of $L_\beta$. The upshot of this section and the next section is that, as $\beta \to \infty$, the spectrum of $L_\beta$ converges to that of $L_{\text{st}}$ in a rather strong sense.

In this section we state two preliminary lemmas. We will use extensively the notation $P_V$ for the orthogonal projection onto some closed subspace $V$ of $L^2_m$.

\begin{lemma}\label{l:two}
Let $H= L^2_m, U_{km}, U_{-km}$, or any closed invariant subspace common to $L_{st}$ and all the $L_\beta$.
For every compact set $K\subset \mathbb C \setminus (i \mathbb R \cup \textrm{spec}_m (L_{\text{st}} \circ P_H))$, there is $\beta_0 (K)$ such that $K\subset \mathbb C \setminus (i\mathbb R \cup \textrm{spec}_m (L_\beta \circ P_H))$ for $\beta \geq \beta_0 (K)$. Moreover, 
\begin{equation}\label{e:op_norm_est}
\sup_{\beta \geq \beta_0 (K)} \sup_{z\in K} \|(L_\beta \circ P_H - z)^{-1}\|_O < \infty\, 
\end{equation}
and 
$(L_\beta \circ P_H -z)^{-1}$ converges strongly to $(L_{\text{st}}\circ P_H -z)^{-1}$ for every $z\in K$, namely,
\begin{equation}\label{e:strong_convergence}
\lim_{\beta\to \infty} \|(L_\beta \circ P_H -z)^{-1} (w) - (L_{\text{st}}\circ P_H -z)^{-1} (w)\| = 0\, \qquad \forall w\in L^2_m\, .
\end{equation}
\end{lemma}

\begin{lemma}\label{l:three}
For every $\varepsilon >0$ there is a $R=R (\varepsilon)$ such that
\begin{equation}\label{e:exclude_large_eigenvalue}
\textrm{spec}_m (L_\beta) \cap \{z : |z|\geq R, |\textrm{Re}\, z|\geq \varepsilon\} = \emptyset \qquad
\forall \beta \geq 1\, .
\end{equation}
\end{lemma}

\begin{proof}[Proof of Lemma \ref{l:two}] The proof is analogous for all $H$ and we will thus show it for $H=L^2_m$. Fix first $z$ such that $\textrm{Re}\, z \neq 0$ and recalling that $z- \beta^{-1} S_2$ is invertible, we write
\begin{equation}\label{e:invert_L_beta-z}
z- L_\beta = (z- \beta^{-1} S_2) (1 - (z-\beta^{-1} S_2)^{-1} \circ \mathscr{K})\, .
\end{equation}

\medskip

\textbf{Step 1} The operators $(\beta^{-1} S_2 -z)^{-1}$ enjoy the bound
\begin{equation}\label{e:uniform-bound-inverse-beta}
\|(z-\beta^{-1} S_2)^{-1}\|_O \leq |\textrm{Re}\, z|^{-1}    
\end{equation}
because $\beta^{-1} S_2$ are skew-adjoint. We claim that $(z-\beta^{-1} S_2)^{-1}$ converges strongly to $(z-S_1)^{-1}$ for $\beta \to \infty$. For a family of operators with a uniform bound on the operator norm, it suffices to show strong convergence of $(z-\beta^{-1} S_2)^{-1} w$ for a (strongly) dense subset. 

Without loss of generality we can assume $\textrm{Re}\, z >0$. Recalling that $\beta^{-1} S_2$ generates a strongly continuous unitary semigroup, we can use the formula
\begin{equation}\label{e:exponential_formula}
(z-\beta^{-1} S_2)^{-1} (w) = \int_0^\infty e^{-(z-\beta^{-1} S_2)\tau} (w)\, d\tau\, .
\end{equation}
Next observe that $\|e^{\beta^{-1} S_2\tau}\|_O=1$. Moreover if $w\in \mathscr{S}$, $e^{\beta^{-1}S_2 \tau} w$ is the solution of a transport equation with a locally bounded and smooth coefficient and initial data $w$. We can thus pass into the limit in $\beta\to \infty$ and conclude that $e^{\beta^{-1} S_2 \tau} w$ converges strongly in $L^2$ to $e^{S_1\tau} w$. We can thus use the dominated convergence theorem in \eqref{e:exponential_formula} to conclude that $(z-\beta^{-1} S_2)^{-1} (w)$ converges to $(z-S_1)^{-1} (w)$ strongly in $L^2$. Since $\mathscr{S}$ is strongly dense, this concludes our proof.

\medskip

\textbf{ Step 2} We next show that $(z-\beta^{-1} S_2)^{-1} \circ \mathscr{K}$ converges in the operator norm to 
$(z-S_1)^{-1} \circ \mathscr{K}$. Indeed using Proposition \ref{p:abstract} we can find a sequence of finite rank projections $P_N$ such that $P_N \circ \mathscr{K}$ converges to $\mathscr{K}$ in operator norm. From Step 1 it suffices to show that $(z-\beta^{-1} S_2)^{-1} \circ P_N \circ \mathscr{K}$ converges to $(z-S_1)^{-1} \circ P_N \circ \mathscr{K}$ in operator norm for each $N$. But clearly $(z-\beta^{-1} S_2)^{-1} \circ P_N$ is a finite rank operator and for finite rank operators the norm convergence is equivalent to strong convergence. The latter has been proved in Step 1.

\medskip

\textbf{ Step 3} Fix $z$ which is outside the spectrum of $L_{st}$. Because of Step 2 we conclude that
\[
(1- (z-\beta^{-1} S_2)^{-1} \circ \mathscr{K}) \to (1- (z-S_1)^{-1}\circ \mathscr{K})
\]
in the operator norm. Observe that $1-(z-S_1)^{-1}\circ \mathscr{K}$ is a compact perturbation of the identity. As such it is a Fredholm operator with index $0$ and thus invertible if and only if its kernel is trivial. Its kernel is given by $w$ which satisfy 
\[
 z w  - S_1 (w)- \mathscr{K} (w) = 0\, ,
\]
i.e. it is the kernel of $z-(S_1+\mathscr{K}) = z- L_{\text{st}}$, which we have assumed to be trivial since $z$ is not in the spectrum of $L_{\text{st}}$. Thus $(1-(z-S_1)^{-1} \circ \mathscr{K})$ is invertible and hence, because of the operator convergence, so is $(1-(z-\beta^{-1} S_2)^{-1}\circ \mathscr{K})$ for any sufficiently large $\beta$. Hence, by \eqref{e:invert_L_beta-z} so is $z-L_\beta$.

\medskip

\textbf{ Step 4} The inverse $(z- L_\beta)^{-1}$ is given explicitly by the formula
\begin{equation}\label{e:inversion_formula}
(z-L_\beta)^{-1} = (1-(z-\beta^{-1} S_2)^{-1} \circ \mathscr{K})^{-1} (z-\beta^{-1} S_2)^{-1} \, .
\end{equation}
Since $1-(z-S_2)^{-1} \circ \mathscr{K}$ converges to $1- (z-S_1)^{-1} \circ \mathscr{K}$ in the operator norm, their inverses converge as well in the operator norm. Since the composition of strongly convergent operators with norm convergent operators is strongly convergent, we conclude that $(z-L_\beta)^{-1}$ converges strongly to the operator 
\[
(1- (z-S_1)^{-1} \circ \mathscr{K})^{-1}  (z-S_1)^{-1} = (z-L_{\text{st}})^{-1} \, .
\]

\medskip

\textbf{ Step 5} Choose now a compact set $K\subset \mathbb C \setminus (i \mathbb R \cup \textrm{spec}_m (L_{\text{st}}))$. 
Recall first that 
\[
K \ni z \mapsto (z-S_1)^{-1} 
\]
is continuous in the operator norm. Thus $K\ni z \mapsto (1- (z-S_1)^{-1} \circ \mathscr{K})$ is continuous in the operator norm. We have already proved in Step 3 that such operators converge, pointwise for $z$ fixed, for $\beta \to +\infty$, and one can prove with the same argument the continuity in $\beta$ for $z$ fixed.

 We claim that $K\times [0,1]\ni (z, \beta) \mapsto (1-(z-\beta^{-1} S_2)^{-1} \circ \mathscr{K})$ is also continuous in the operator norm and in order to show this we will prove the uniform continuity in $z$ once we fix $\beta$, with an estimate which is independent of $\beta$. We first write 
\begin{align*}
& \|(1-(z - \beta^{-1} S_2)^{-1}\circ \mathcal{K}) - (1- (z'- \beta^{-1} S_2)^{-1} \circ \mathcal{K})\|_O\\
\leq & \|(z - \beta^{-1} S_2)^{-1} - (z'-\beta^{-1} S_2)^{-1}\|_O \|\mathcal{K}\|_O \, .
\end{align*}
Hence we compute
\begin{equation*}
    \begin{split}
        (z - &\beta^{-1} S_2)^{-1} - (z'-\beta^{-1} S_2)^{-1} 
\\&= (z - \beta^{-1} S_2)^{-1} \circ ((z' - \beta^{-1} S_2)-(z-\beta^{-1} S_2)) \circ (z'-\beta^{-1} S_2)^{-1}
    \end{split}
\end{equation*}
and use \eqref{e:uniform-bound-inverse-beta} to estimate
\begin{align*}
    \|(z - \beta^{-1} S_2)^{-1}& - (z'-\beta^{-1} S_2)^{-1}\|_O 
    \\&\leq |z-z'| \|(z - \beta^{-1} S_2)^{-1}\|_O \|(z'-\beta^{-1} S_2)^{-1}\|_O
\\&\leq \frac{|z-z'|}{| \textrm{Re}\, z| |\textrm{Re}\, z'|}\, .
\end{align*}
Since the space of invertible operators is open in the norm topology, this implies the existence of a $\beta_0>0$ such that $K\times [\beta_0, \infty] \ni (z, \beta) \mapsto (1-(z-\beta^{-1} S_2)^{-1}\circ \mathscr{K})^{-1}$ is well defined and continuous. Thus, for $\beta \geq \\beta_0$ we conclude that $1-(z-\beta^{-1} S_2)^{-1}\circ \mathscr{K}$ is invertible and the norm of its inverse is bounded by a constant $C$ independent of $\beta$ and $z\in K$. By \eqref{e:inversion_formula} and \eqref{e:uniform-bound-inverse-beta}, we infer that in the same range of $z$ and $\beta$ the norm of the operators $(z-L_\beta)^{-1}$ enjoy a uniform bound.
\end{proof}

\begin{proof}[Proof of Lemma \ref{l:three}]
We show \eqref{e:exclude_large_eigenvalue} for $\textrm{Re}\, z \geq \varepsilon$ replacing $|\textrm{Re}\, z|\geq \varepsilon$, as the argument for the complex lower half-plane is entirely analogous. 

Using \eqref{e:invert_L_beta-z}, we wish to show that there is $R = R (\varepsilon)$ such that the operator
\[
1 - (z-\beta^{-1} S_2)^{-1} \circ \mathscr{K}
\]
is invertible for all $\beta \geq 1$ and $z$ such that $|z|\geq R$ and $\textrm{Re}\, z \geq \varepsilon$. 
This will follow after showing that, for $\beta$ and $z$ in the same range
\begin{equation}\label{e:small}
\|(z-\beta^{-1} S_2)^{-1} \circ \mathscr{K}\|_O \leq \frac{1}{2}\, . 
\end{equation}
By \eqref{e:uniform-bound-inverse-beta}, we can use Proposition \ref{p:abstract} to reduce \eqref{e:small} to the claim
\begin{equation}\label{e:small-2}
\|(z-\beta^{-1} S_2)^{-1} \circ P_V \circ \mathscr{K}\|_O \leq \frac{1}{4}\, ,
\end{equation}
where $P_V$ is the projection onto an appropriately chosen finite-dimensional space $V\subset C^\infty_c$. If $N$ is the dimension of the space and $w_1, \ldots , w_N$ an orthonormal base, it suffices to show that
\begin{equation}\label{e:small-3}
\|(z-\beta^{-1} S_2)^{-1} (w_i)\|_{L^2} \leq \frac{1}{4N} \qquad \forall i\, .    
\end{equation}
We argue for one of them and set $w=w_i$. The goal is thus to show
\eqref{e:small-3} provided $|z|\geq R$ for some large enough $R$. We use again \eqref{e:exponential_formula} and write 
\[
(z-\beta^{-1} S_2)^{-1} (w) = \underbrace{\int_0^T e^{-(z-\beta^{-1} S_2) \tau} (w)\, d\tau}_{=:(A)} + \underbrace{\int_T^\infty e^{-(z-\beta^{-1} S_2) \tau} (w)\, d\tau}_{=:(B)}\, .
\]
We first observe that
\[
\|(B)\| \leq \int_T^\infty e^{-\varepsilon \tau}\, d\tau \leq \frac{e^{-\varepsilon T}}{\varepsilon}\, .
\]
Thus, choosing $T$ sufficiently large we achieve $\|(B)\| \leq \frac{1}{8N}$. 
Having fixed $T$ we integrate by parts in the integral defining (A) to get
\begin{align*}
(A)   
= \underbrace{\frac{w - e^{- (z-\beta^{-1} S_2) T} (w)}{z}}_{=: (A1)} 
+ \underbrace{\frac{1}{z} \int_0^T e^{-z} \beta^{-1} S_2 \circ e^{\beta^{-1} S_2 \tau} (w)\, d\tau}_{=: (A2)}\, .
\end{align*}
First of all we can bound
\[
\|(A1)\| \leq \frac{1+ e^{-T\varepsilon}}{|z|}\leq \frac{2}{R}\, .
\]
As for the second term, observe that $[0,T]\ni \tau \mapsto e^{\beta^{-1} S_2 \tau} (w)$ is the solution of a transport equation with smooth coefficients and smooth and compactly supported initial data, considered over a finite interval of time. Hence the support of the solution is compact and the solution is smooth. Moreover, the operators $\beta^{-1} S_2$ are first-order differential operators with coefficients which are smooth and whose derivatives are all bounded. In particular 
\[
\max_{\tau\in [0,T]} \|\beta^{-1} S_2 \circ e^{\beta^{-1} S_2 \tau} (w)\| \leq C
\]
for a constant $C$ depending on $w$ and $T$ but not on $\beta$, in particular we can estimate
\[
\|(A2)\|\leq \frac{C (T)}{R}
\]
Since the choice of $T$ has already been given, we can now choose $R$ large enough to conclude $\|(A)\|\leq \frac{1}{8N}$ as desired.
\end{proof}

\section{Proof of Theorem \ref{thm:spectral}: conclusion}

First of all observe that $z\in \textrm{spec}_m (L_\beta)$ if and only if $\beta z + 1-\frac{1}{\alpha}\in \textrm{spec}_m (L_{\text{ss}})$. Thus, in order to prove Theorem \ref{thm:spectral} it suffices to find $\beta_0$ and $c_0$ positive such that:
\begin{itemize}
\item[(P)] If $\beta \geq \beta_0$, then $\textrm{spec}_m (L_\beta)$ contains an element $z$ with $\textrm{Re}\, z \geq c_0$ such that $\textrm{Re}\, z = \max \{\textrm{Re}\, w : w\in \textrm{spec}_m\, (L_\beta)\}$.
\end{itemize}
Observe indeed that using the fact that the $U_{km}$ are invariant subspaces of $L_{\text{ss}}$, $\beta z + 1-\frac{1}{\alpha}$ have an eigenfunction $\vartheta$ which belongs to one of them, and we can assume that $k\geq 1$ by possibly passing to the complex conjugate $\bar z$. If $z$ is not real, we then set $\eta=\vartheta$ and the latter has the properties claimed in Theorem \ref{thm:spectral}. If $z$ is real it then follows that the real and imaginary part of $\vartheta$ are both eigenfunctions and upon multiplying by $i$ we can assume that the real part of $\vartheta$ is nonzero. We can thus set $\eta= \textrm{Re}\, \vartheta$ as the eigenfunction of Theorem \ref{thm:spectral}. $\eta$ then satisfies the requirements (i)-(iv) of the theorem. In order to complete the proof we however also need to show the estimates in (v). Those and several other properties of $\eta$ will be concluded from the eigenvalue equation that it satisfies and are addressed later in Lemma \ref{l:pointwise}.

We will split the proof in two parts, namely, we will show separately that
\begin{itemize}
    \item[(P1)] There are $\beta_1, c_0 >0$ such that  $\textrm{spec}_m (L_\beta)\cap \{\textrm{Re}\, z \geq c_0\}\neq \emptyset$ for all $\beta \geq \beta_1$.
    \item[(P2)] If $\beta \geq \beta_0:= \max \{\beta_1, 1\}$, then  $\sup \{\textrm{Re}\, w : w\in \textrm{spec}\, (L_\beta)\}$ is attained.
\end{itemize}

\medskip

\textbf{ Proof of (P1).} We fix $z\in \textrm{spec}\, (L_{\text{st}})$ with positive real part and we set $2c_0:= \textrm{Re}\, z$. We then fix a contour $\gamma \subset B_\varepsilon (z)$ which:
\begin{itemize}
\item it is a simple smooth curve;
\item it encloses $z$ and no other portion of the spectrum of $L_{\text{st}}$;
\item it does not intersect the spectrum of $L_{\text{st}}$;
\item it is contained in $\{w: \textrm{Re}\, w \geq c_0\}$. 
\end{itemize}
By the Riesz formula we know that
\[
P_z = \frac{1}{2\pi i} \int_\gamma (w-L_{\text{st}})^{-1}\, dw
\]
is a projection onto a subspace which contains all eigenfunctions of $L_{\text{st}}$ relative to the eigevanlue $z$. In particular this projection is nontrivial. By Lemma \ref{l:two} for all sufficiently large $\beta$ the curve $\gamma$ is not contained in the spectrum of $L_\beta$ and we can thus define
\[
P_{z,\beta} = \frac{1}{2\pi i} \int_\gamma (w-L_\beta)^{-1}\, dw\, .
\]
If $\gamma$ does not enclose any element of the spectrum of $L_\beta$, then $P_{z, \beta} = 0$. On the other hand, by Lemma \ref{l:two} and the dominated convergence theorem,
\[
P_{z,\beta} (u) \to P_z (u)
\]
strongly for every $u$. That is, the operators $P_{z,\beta}$ converge strongly to the operator $P_z$. If for a sequence $\beta_k\to \infty$ the operators $P_{z,\beta_k}$ where trivial, then $P_z$ would be trivial too. Since this is excluded, we conclude that the curve $\gamma$ encloses some elements of the spectrum of $L_\beta$ for all $\beta$ large enough. Each such element has real part not smaller than $c_0$.

\medskip

\textbf{ Proof of (P2).} Set $\varepsilon := c_0$ and apply Lemma \ref{l:three} to find $R>0$ such that
$\textrm{spec}_m (L_\beta)\setminus \overline{B}_R$ is contained in $\{w: \textrm{Re}\, w < c_0\}$. In particular, if $\beta \geq \max\{\beta_1, 1\}$, then the eigenvalue $z$ found in the previous step belongs to $\overline{B}_R$ and thus
\[
\sup\, \{\textrm{Re}\, w : w\in \textrm{spec}\, (L_\beta)\} =
\sup\, \{w: \textrm{Re}\, w \geq c_0, |w|\leq R\}\cap \textrm{spec}\, (L_\beta) \, .
\]
However, since $\textrm{spec}\, (L_\beta)\cap \{w: \textrm{Re}\, w \neq 0\}$ belongs to the discrete spectrum, the set $\{w: \textrm{Re}\, w \geq c_0, |w|\leq R\}\cap \textrm{spec}\, (L_\beta)$ is finite. 

\chapter{Linear theory: Part II}
\label{chapter:linearpartii}\label{sect:Proof-spectral-2}

This chapter is devoted to proving Theorem \ref{thm:spectral2}. Because of the discussions in the previous chapter, considering the decomposition 
\[
L^2_m = \bigoplus_{k\in \mathbb Z} U_{km}\, ,
\]
the statement of Theorem \ref{thm:spectral2} can be reduced to the study of the spectra of the restrictions $L_{\text{st}}|_{U_{km}}$ of the operator $L_{\text{st}}$ to the invariant subspaces $U_{km}$. \index{spectrum@spectrum}\index{aalspec()@$\textrm{spec}_m (\cdot, U_j)$}
For this reason we introduce the notation $\textrm{spec}\, (L_{\text{st}}, U_j)$ for the spectrum of the operator $L_{\text{st}}|_{U_j}$, understood as an operator from $U_j$ to $U_j$. The following is a very simple observation. 

\begin{lemma}
The restriction of the operator $L_{\text{st}}$ to the radial functions $U_0$ is identically $0$. Moreover, $z\in \textrm{spec}\, (L_{\text{st}}, U_j)$ if and only if $\bar z \in \textrm{spec}\, (L_{\text{st}}, U_{-j})$. 
\end{lemma}

We will then focus on proving the following statement, which is slightly stronger than what we need to infer Theorem \ref{thm:spectral2}.

\begin{theorem}\label{thm:spectral3}
For a suitable choice of $\bar \Omega$, there is an integer $m_0\geq 2$ such that
$\textrm{spec}\, (L_{\text{st}}, U_{m_0}) \cap \{\textrm{Re}\, z > 0\}$ is nonempty and $\textrm{spec}\, (L_{\text{st}}, U_{m_0})\cap \{\textrm{Re}\, z \geq \bar a\}$ is finite for every positive $\bar a$.
\end{theorem}

\begin{remark}\label{r:better2}\label{R:BETTER2} As it is the case for Theorem \ref{thm:spectral2} we can deepen our analysis and prove the following stronger statement:
\begin{itemize}
    \item[(i)] For a suitable choice of $m_0$, in addition to the conclusion of Theorem \ref{thm:spectral3} we have $\textrm{spec}\, (L_{\text{st}}, U_m) \subset i \mathbb R$ for every $m> m_0$.
\end{itemize}
This will be done in Appendix \ref{a:better}, where we will also show how conclusion (c) of Remark \ref{r:better} follows from it.
\end{remark}

Note that in \cite{Vishik2} Vishik claims the following stronger statement.

\begin{theorem}\label{thm:spectral-stronger-2}\label{THM:SPECTRAL-STRONGER-2}
For a suitable choice of $m_0$, in addition to the conclusion of Theorem \ref{thm:spectral3} and to Remark \ref{r:better2}(i), we have also
\begin{itemize}
    \item[(ii)] $\textrm{spec}\, (L_{\text{st}}, U_{m_0}) \cap \{\textrm{Re}\, z > 0\}$ consists of a single eigenvalue with algebraic multiplicity $1$.\index{algebraic multiplicity@algebraic multiplicity}
\end{itemize}
\end{theorem}
In Appendix \ref{a:better} we will show how to prove the latter conclusion and how Theorem \ref{thm:spectral-stronger} follows from it.

\section{Preliminaries}

If we write an arbitrary element $\Omega\in U_m$ as $\Omega (x) = e^{im\theta} \gamma (r)$ using polar coordinates, we find an isomorphism of the Hilbert space $U_m$ with the Hilbert space \index{aalHcal@$\mathcal{H}$}
\begin{equation}\label{e:def-H}
\mathcal{H}:= \left\{\gamma \colon \mathbb R^+ \to \mathbb C : \int_0^\infty |\gamma (r)|^2\, r\, dr < \infty\right\}    
\end{equation}
and thus the operator $L_{\text{st}}\colon U_m \to U_m$ can be identified with an operator 
\[
\mathcal{L}_{\text{st}} \colon \mathcal{H}\to \mathcal{H}\, .
\]
In fact, since $L_{\text{st}} = S_1+\mathscr{K}$, \index{aalKscr@$\mathscr{K}$} where $S_1$ \index{aalS1@$S_1$} is skew-adjoint and $\mathscr{K}$ compact, $\mathcal{L}_{\text{st}}$ is also a compact perturbation of a skew-adjoint operator. In order to simplify our notation and terminology, we will then revert our considerations to the operator $\mathcal{L}_m = i\mathcal{L}_{\text{st}}$, \index{aalLcalm@$\mathcal{L}_m$} which will be written as the sum of a self-adjoint operator, denoted by $\mathcal{S}_m$, \index{aalScalm@$\mathcal{S}_m$} and a compact operator, denoted by $\mathscr{K}_m$: \index{aalKscrm@$\mathscr{K}_m$}
\index{aagZbar@$\bar \Omega$}
\begin{equation}
    \mathcal{L}_m \gamma = m\zeta \gamma - \frac{m}{r} \psi g' = \mathcal{S}_m + \mathscr{K}_m \, ,
\end{equation}
where $\psi$ is defined in~\eqref{e:explicit3}.

\begin{lemma}\label{l:S-in-polar}
After the latter identification, if $\bar\Omega (x) = g (|x|)$ and $\zeta$ is given through the formula \eqref{e:def-zeta}, then $\mathcal{S}_m\colon \mathcal{H}\to \mathcal{H}$ is the following bounded self-adjoint operator:
\begin{equation}\label{e:explicit}
\gamma \mapsto \mathcal{S}_m (\gamma) = m \zeta \gamma\, .
\end{equation}
\end{lemma}
\begin{proof}
The formula is easy to check. The self-adjointness of \eqref{e:explicit} is obvious. Concerning the boundedness we need to show that $\zeta$ is bounded. Since $g$ is smooth (and hence locally bounded), $\zeta$ is smooth and locally bounded by \eqref{e:def-zeta}. To show that it is globally bounded recall that $g (r) = r^{-\al}$ for $r\geq 2$, so that
\[
\zeta (r) = \frac{\tilde c_0}{r^2} + \frac{1}{r^2} \int_2^r \rho^{1-\al}\, d\rho = \frac{c_0}{r^2} + \frac{c_1}{r^\al} \qquad \forall r\geq 2\, ,
\]
where $c_0$ and $c_1$ are two appropriate constants. 
\end{proof}

A suitable, more complicated, representation formula can be shown for the operator $\mathscr{K}_m$. 

\begin{lemma}\label{l:K-in-polar}
Under the assumptions of Lemma \ref{l:S-in-polar}, the compact operator $\mathscr{K}_m: \mathcal{H}\to \mathcal{H}$ is given by
\begin{equation}\label{e:explicit2}
\gamma \mapsto \mathscr{K}_m (\gamma)= - \frac{m}{r} \psi g'\, 
\end{equation}
where 
\begin{equation}\label{e:explicit3}
\psi (r) = - \frac{1}{2m} r^{m} \int_r^\infty \gamma (s) s^{1-m}\, ds - \frac{1}{2m} r^{-m} \int_0^r \gamma (s) s^{1+m}\, ds\, .
\end{equation}
\end{lemma}

\begin{remark}\label{r:potential-theory}
When $\gamma$ is compactly supported, $\phi (\theta,r):= \psi (r) e^{im\theta}$ with $\psi$ as in \eqref{e:explicit} gives the unique potential-theoretic solution of $\Delta \phi = \gamma e^{im\theta}$, namely, $\phi$ obtained as the convolution of $\gamma e^{im\theta}$ with the Newtonian potential $\frac{1}{2\pi} \ln r$. For general $\gamma\in \mathcal{H}$ we do not have enough summability to define such convolution using Lebesgue integration, but, as already done before, we keep calling $\phi$ the potential-theoretic solution of $\Delta \phi = \gamma e^{im\theta}$.  
\end{remark}

\begin{proof}[Proof of Lemma \ref{l:K-in-polar}]
First of all we want to show that the formula is correct when $\Omega = \gamma (r) e^{im \theta} \in C^\infty_c \cap L^2_m$. We are interested in computing $-i (K_2*\Omega\cdot \nabla) \bar \Omega$. We start with the formula $K_2* \Omega = \nabla^\perp \phi$, where $\phi$ is the potential-theoretic solution of $\Delta \phi = \Omega$. The explicit formula for $\phi$ reads
\[
\phi (x) = \frac{1}{2\pi} \int_{\R^2} \Omega (y) \ln |y-x|\,\mathrm dy\, .
\]
$\phi$ is clearly smooth and hence locally bounded. Observe that $\Omega$ averages to $0$ and thus
\[
\phi (x) = \frac{1}{2\pi} \int_{\R^2} \Omega (y) (\ln |y-x| - \ln |x|)\,\mathrm dy\, .
\]
Fix $R$ larger than $1$ so that $\textrm{spt}\, (\Omega) \subset B_R$ and choose $|x|\geq 2R$. We then have the following elementary inequality for every $y\in \textrm{spt}\, (\Omega)$:
\[
|\ln |x| - \ln |x-y||\leq \ln (|x-y| + |y|) - \ln (|x-y|)\leq \frac{|y|}{|y-x|} \leq \frac{2|y|}{|x|}\, ,
\]
from which we conclude that $|\phi (x)|\leq C (1+|x|)^{-1}$. Hence $\phi$ is the only solution to $\Delta \phi = \Omega$ with the property that it converges to $0$ at infinity. This allows us to show that $\phi$ satisfies the formula
\[
\phi (x) = \psi (r) e^{im\theta}
\]
where $\psi$ is given by formula \eqref{e:explicit3}. We indeed just need to check that the Laplacian of 
$\psi (r) e^{im\theta}$ equals $\gamma (r) e^{im\theta}$ and that $\lim_{r\to \infty} \psi (r) = 0$. 
Using the formula $\Delta = \frac{1}{r^2} \frac{\partial^2}{\partial \theta^2} + \frac{1}{r} \frac{\partial}{\partial r} + \frac{\partial^2}{\partial r^2}$ the first claim is a direct verification. Next, since $\gamma (r) =0$ for $r\geq R$, we conclude $\psi (r) = C r^{-m}$ for all $r\geq R$, which shows the second claim. 
Observe next that
\[
\nabla \phi = \frac{m i}{r^2} \psi (r) e^{im\theta} \frac{\partial}{\partial \theta} - \frac{\partial}{\partial r} \left(\psi (r) e^{im\theta}\right) \frac{\partial}{\partial r}\, ,
\]
which turns into
\[
\nabla \phi^\perp = - \frac{mi}{r} \psi (r) e^{im\theta} \frac{\partial}{\partial r} - \frac{1}{r} \frac{\partial}{\partial r} \left(\psi (r) e^{im\theta}\right) \frac{\partial}{\partial \theta}\, .
\]
Since $\bar\Omega (x) = g(r)$, we then conclude that 
\[
- (K_2*\Omega\cdot \nabla) \bar\Omega = \frac{mi}{r} \psi (r) e^{im\theta} g' (r)\, .
\]
Upon multiplication by $i$ we obtain formula \eqref{e:explicit2}. Since we know from the previous chapter that $\mathscr{K}$ is a bounded and compact operator and $\mathscr{K}_m$ is just the restriction of $i\mathscr{K}$ to a closed invariant subspace of it, the boundedness and compactness of $\mathscr{K}_m$ is obvious.
\end{proof}

Notice next that, while in all the discussion so far we have always assumed that $m$ is an integer larger than $1$, the operator $\mathcal{S}_m$ can in fact be easily defined for every {\em real} $m>1$, while, using the formulae \eqref{e:explicit2} and \eqref{e:explicit3} we can also make sense of $\mathscr{K}_m$ for every real $m>1$. In particular we can define as well the operator $\mathcal{L}_m$ for every $m>1$. The possibility of varying $m$ as a real parameter will play a crucial role in the rest of the chapter, and we start by showing that, for $m$ in the above range, the boundedness of $\mathcal{L}_m$ and $\mathcal{S}_m$ and the compactness of $\mathscr{K}_m$ continue to hold.
\index{aalLcalm@$\mathcal{L}_m$}\index{aalScalm@$\mathcal{S}_m$}\index{aalKscrm@$\mathscr{K}_m$}

\begin{proposition}\label{p:all-m}
The operators $\mathcal{L}_m$, $\mathcal{S}_m$, and $\mathscr{K}_m$ are bounded operators from $\mathcal{H}$ to $\mathcal{H}$ for every real $m>1$, with a uniform bound on their norms if $m$ ranges in a compact set. Moreover, under the same assumption $\mathscr{K}_m$ is compact. In particular:
\begin{itemize}
    \item[(i)] $\textrm{spec}\, (\mathcal{L}_m)$ is compact;
    \item[(ii)] for every $z$ with $\textrm{Im}\, z \neq 0$ the operator $\mathcal{L}_m-z$ is a bounded Fredholm operator with index $0$;
    \item[(iii)] every $z\in \textrm{spec}\, (\mathcal{L}_m)$ with $\textrm{Im}\, z \neq 0$ belongs to the discrete spectrum.
\end{itemize}
\end{proposition}
\begin{proof} The boundedness of $\mathcal{S}_m$ is obvious. Having shown the boundedness and compactness of $\mathscr{K}_m$, (i) follows immediately from the boundedness of $\mathcal{L}_m$, while (ii) follows immediately from the fact that $\mathcal{L}_m - z$ is a compact perturbation of the operator $\mathcal{S}_m -z$, which is invertible because $\mathcal{S}_m$ is selfadjoint, and (iii) is a standard consequence of (ii). 

First of all let us prove that $\mathscr{K}_m$ is bounded (the proof is necessary because, from what previously proved, we can just conclude the boundedness and compactness of the operator for {\em integer values} of $m$ larger than $1$). We observe first the pointwise estimate 
\begin{equation}
    \label{eq:ausefulpointwiseestimateonpsi}
    \| r^{-1}\psi\|_\infty \leq C  \|\gamma\|_{\mathcal{H}} \, ,
\end{equation} as it follows from Cauchy-Schwarz that
\begin{align*}
r^{m-1} \int_r^\infty |\gamma (s)| s^{1-m}\, ds&\leq
r^{m-1} \Big(\int_r^\infty |\gamma(s)|^2 s\, ds\Big)^{\frac{1}{2}} \Big(\int_r^\infty s^{1-2m}\, ds\Big)^{\frac{1}{2}}
\\
&\leq \frac{1}{\sqrt{2m-2}} \|\gamma\|_{\mathcal{H}}
\end{align*}
and
\begin{align*}
r^{-m-1} \int_0^r |\gamma (s)| s^{1+m}\, ds &\leq r^{-m-1} \Big(\int_0^r |\gamma (s)|^2 s\, ds\Big)^{\frac{1}{2}} \Big(\int_0^r s^{1+2m}\, ds\Big)^{\frac{1}{2}}
\\
&\leq \frac{1}{\sqrt{2m+2}} \|\gamma\|_{\mathcal{H}}\, .
\end{align*}
Since $g' (r) \leq C (1+r)^{-1-\al}$, it follows immediately that 
\begin{equation}\label{e:K_m-pointwise-bound}
|(\mathscr{K}_m (\gamma)) (r)| \leq \frac{C\|\gamma\|_{\mathcal{H}}}{(1+r)^{1+\al}}
\end{equation}
and in particular
\[
\|\mathscr{K}_m (\gamma)\|_{\mathcal{H}} \leq
C\|\gamma\|_{\mathcal{H}} \Big(\int_0^\infty \frac{s}{(1+s)^{2+2\al}}\, ds\Big)^{\frac{1}{2}}
\leq C \|\gamma\|_{\mathcal{H}} \, .
\]
This completes the proof of boundedness of the operator. In order to show compactness consider now a bounded sequence $\{\gamma_k\}\subset \mathcal{H}$. Observe that for every fixed $N$, \eqref{e:explicit3} gives the following obvious bound
\begin{equation}
\|\mathscr{K}_m(\gamma_k)\|_{W^{1,2} [N^{-1}, N]} \leq C (N) \|\gamma_k\|_{\mathcal{H}}\, .    
\end{equation}
In particular, through a standard diagonal procedure, we can first extract a subsequence of $\{\mathscr{K}_m(\gamma_k)\}$ (not relabeled) which converges strongly in the space $L^2 ([N^{-1}, N], rdr)$ for every $N$. It is now easy to show that $\{\mathscr{K}_m (\gamma_k)\}_k$ is a Cauchy sequence in $\mathcal{H}$. Fix indeed $\varepsilon>0$. Using \eqref{e:K_m-pointwise-bound} we conclude that there is a sufficiently large $N$ with the property that
\begin{equation}\label{e:Cauchy-1}
\sup_k \|\mathscr{K}_m (\gamma_k) \mathbf{1}_{[0, N^{-1}]\cup [N, \infty[}\|_{\mathcal{H}} < \frac{\varepsilon}{3}\, .
\end{equation}
Hence, given such an $N$, we can choose $k_0$ big enough so that
\begin{equation}\label{e:Cauchy-2}
\|(\mathscr{K}_m (\gamma_k) - \mathscr{K}_m (\gamma_j)) \mathbf{1}_{[N^{-1}, N]}\|_{\mathcal{H}} \leq 
\frac{\varepsilon}{3} \qquad \forall k,j \geq k_0\, .
\end{equation}
Combining \eqref{e:Cauchy-1} and \eqref{e:Cauchy-2} we immediately conclude
\[
\|\mathscr{K}_m (\gamma_k) - \mathscr{K}_m (\gamma_j)\|_{\mathcal{H}} < \varepsilon
\]
for every $j,k \geq k_0$. This completes the proof that $\{\mathscr{K}_m (\gamma_j)\}$ is a Cauchy sequence and hence the proof that $\mathscr{K}_m$ is compact.
\end{proof}

\section{The eigenvalue equation and the class \texorpdfstring{$\mathscr{C}$}{scrC}}
\label{s:eigenvalue-equation}

Using the operators introduced in the previous setting, we observe that Theorem~\ref{thm:spectral3} is equivalent to showing that $\textrm{spec}\, (\mathcal{L}_{m_0}) \cap \{\textrm{Im}\, z>0\}$ is nonempty.
We next notice that, thanks to Proposition \ref{p:all-m} and recalling that $\psi$ is defined through \eqref{e:explicit3}, the latter is equivalent to showing that the  equation
\begin{equation}\label{e:eigenvalue-equation}
m \zeta \gamma - \frac{m}{r} g' \psi = z \gamma
\end{equation}
has a nontrivial solution $\gamma \in \mathcal{H}$ for some integer $m=m_0\geq 2$ and some complex number $z$ with positive imaginary part. 

We thus turn \eqref{e:eigenvalue-equation} into an ODE problem by changing the unknown from $\gamma$ to the function $\psi$.
In particular, recall that the relation between the two is that $\Delta (\psi (r) e^{im\theta}) = \gamma (r) e^{im\theta}$, and $\psi e^{im\theta}$ is in fact the potential-theoretic solution. We infer that 
\[
\psi'' + \frac{1}{r} \psi' - \frac{m^2}{r^2} \psi = \gamma\, 
\]
and hence \eqref{e:eigenvalue-equation} becomes
\begin{equation}\label{e:eigenvalue-equation-2}
- \psi'' - \frac{1}{r}\psi' + \frac{m^2}{r^2} \psi + \frac{g'}{r (\zeta -m^{-1} z)} \psi = 0  \, .
\end{equation}
Notice that, by classical estimates for ODEs, $\psi \in W^{2,2}_{\text{loc}} (\mathbb R^+)$.
Observe, moreover, that if $\psi\in L^2 (\frac{dr}{r})\cap W^{2,2}_{\text{loc}}{(\mathbb{R}^+)}$ solves \eqref{e:eigenvalue-equation-2} and $z$ has nonzero imaginary part, it follows that
\begin{equation}\label{e:change-SK-1}
\gamma = \frac{mg'}{r (m \zeta -z)} \psi
\end{equation}
belongs to $L^2 (r dr)$ and solves \eqref{e:eigenvalue-equation}, because the function $\frac{mg'}{m \zeta -z}$ is bounded. Vice versa, assume that $\gamma \in L^2 (r dr)$ solves \eqref{e:eigenvalue-equation}. Then $\psi$ solves \eqref{e:eigenvalue-equation-2} and we claim that $\psi\in L^2 (\frac{dr}{r})\cap W^{2,2}_{\text{loc}} (\mathbb R^+)$. Again the $W^{2,2}_{\text{loc}}$ regularity follows from classical estimates on ODEs and therefore we just need to show $\psi\in L^2 (\frac{dr}{r})$.

When $m$ is an integer strictly larger than $1$, by classical Calder{\'o}n-Zygmund estimates, $\phi (x) := \psi (r) e^{im\theta}$ is a $W^{2,2}_{\text{loc}}$ function of $\mathbb R^2$. As such $\phi\in C^\omega (B_1)$ for every $\omega<1$ and therefore $\psi\in C^\omega ([0,1])$ and, by symmetry considerations, $\psi (0) =0$. Thus it turns out that $|\psi (r)|\leq C r^\omega$ for every $r\in [0,1]$, which easily shows that $\psi\in L^2 ([0,1], \frac{dr}{r})$. However, in the sequel we will consider the more general case of any real $m>1$. Under the latter assumption we just understand $\psi$ as given by the integral formula \eqref{e:explicit3} and obtain a pointwise bound $|\psi (r)|\leq C r$ from~\eqref{eq:ausefulpointwiseestimateonpsi}.

We next improve the linear growth bound on $\psi$ at $\infty$
to show that 
\begin{equation}\label{e:correzione-1}
\int_1^\infty \frac{|\psi (r)|^2}{r}\, dr < \infty\, .
\end{equation}
Here we recall that, for $r$ sufficiently large, $\zeta (r) = \frac{c_0}{r^2}+ \frac{c_1}{r^\al}$ for some constants $c_0$ and $c_1$, while $g' (r) = -\al r^{-(1+\al)}$. In particular, since $|m\zeta - z|\geq |\textrm{Im} \, z|>0$, from \eqref{e:change-SK-1} we infer
\begin{equation}\label{e:change-SK-3}
|\gamma (r)| \leq C r^{-2-\bar \alpha} |\psi (r)|\, \qquad \forall r>1\, ,
\end{equation}
while at the same time we use \eqref{e:explicit3} to bound 
\begin{equation}\label{e:change-SK-4}
|\psi (r)|\leq C r^{-m} + \frac{r^{-m}}{2m} \int_1^r |\gamma (s)| s^{1+m}\, ds 
+ \frac{r^m}{2m} \int_r^\infty |\gamma (s)| s^{1-m}\, ds \qquad \forall r>1\, .
\end{equation}
Now, insert the bound $|\psi (r)|\leq C r$ 
into \eqref{e:change-SK-3} to get $|\gamma (r)|\leq C r^{-1-\bar\alpha}$ and hence use the latter in \eqref{e:change-SK-4} to improve the bound on $\psi$ to 
\[
|\psi (r)|\leq C (r^{-m} + r^{1-\bar\alpha})\, .
\]
We then repeat the very same procedure to improve the last pointwise bound to $|\psi (r)|\leq C (r^{-m} + r^{1-2\bar\alpha})$. 
After a finite number of steps we arrive at the bound $|\psi (r)|\leq C r^{-m}$. The latter is obviously enough to infer~\eqref{e:correzione-1}.

Hence our problem is equivalent to understand for which $m$ and $z$ with positive imaginary part there is an $L^2 (\frac{dr}{r})\cap W^{2,2}_{\text{loc}}$ solution of \eqref{e:eigenvalue-equation-2}. The next step is to change variables to $t = \ln r$ and we thus set $\varphi (t) = \psi (e^t)$, namely, $\psi (r) = \varphi (\ln r)$. The condition that $\psi\in L^2 (\frac{dr}{r})$ translates then into $\varphi\in L^2 (\mathbb R)$ and $\psi\in W^{2,2}_{\text{loc}}$ translates into $\varphi\in W^{2,2}_{\text{loc}}$.
Moreover, if we substitute the complex number $z$ with $\frac{z}{m}$ we can rewrite
\begin{equation}\label{e:eigenvalue-equation-3}
- \varphi'' (t) + m^2 \varphi (t) + \frac{A(t)}{\Xi (t) - z} \varphi (t) = 0\, ,
\end{equation}
which is \emph{Rayleigh's stability equation}\index{Rayleigh's stability equation@Rayleigh's stability equation},
where the functions $A$ and $\Xi$ are given by changing variables in the corresponding functions $g'$ and \index{aalA@$A$}\index{aagO@$\Xi$}
$\zeta$:
\begin{align}
A (t) &= \frac{d}{dt} g(e^t)\\
\Xi (t) &= \int_{-\infty}^t e^{-2 (t-\tau)} g (e^\tau)\, d\tau\, .
\end{align}
Note in particular that we can express $A$ and $\Xi$ through the relation
\begin{equation}\label{e:A-Xi}
A = \Xi'' + 2 \Xi'\, .
\end{equation}
The function $g$ (and so our radial function $\bar\Omega$) \index{aagZbar@$\bar \Omega$}\index{aalg@$g$} can be expressed in terms of $\Xi$ through the formula
\begin{equation}\label{e:formula-g}
g (e^t) = e^{-2t} \frac{d}{dt} (e^{2t} \Xi (t))\, .     
\end{equation}
Rather than looking for $g$ we will then look for $\Xi$ in an appropriate class $\mathscr{C}$ \index{aalCscr@$\mathscr{C}$} which we next detail:

\begin{definition}\label{d:class-C}
The class $\mathscr{C}$ consists of those functions $\Xi: \mathbb R \to ]0, \infty[$ such that
\begin{itemize}
\item[(i)] $\Xi (-\infty) := \lim_{t\to - \infty} \Xi (t)$ is finite and there are constants $c_0>0$ and $M_0$ such that $\Xi (t) = \Xi (-\infty) - c_0 e^{2t}$ for all $t\leq M_0$;
 \item[(ii)] there is a constant $c_1$ such that $\Xi (t) = c_1 e^{-2t} + \frac{1}{2-\al} e^{-\al t}$ for $t\geq \ln 2$;
 \item[(iii)] $A$ has exactly two zeros, denoted by $a<b$, and $A' (a)>0$ and $A' (b)<0$ (in particular $A<0$ on $]-\infty,a[ \, \cup \, ]b, \infty[$ and $A>0$ on $]a,b[$);
 \item[(iv)] $\Xi ' (t) <0$ for every $t$.
\end{itemize}
\end{definition}
\begin{figure}[ht]
    \centering
    \includegraphics[width=\textwidth]{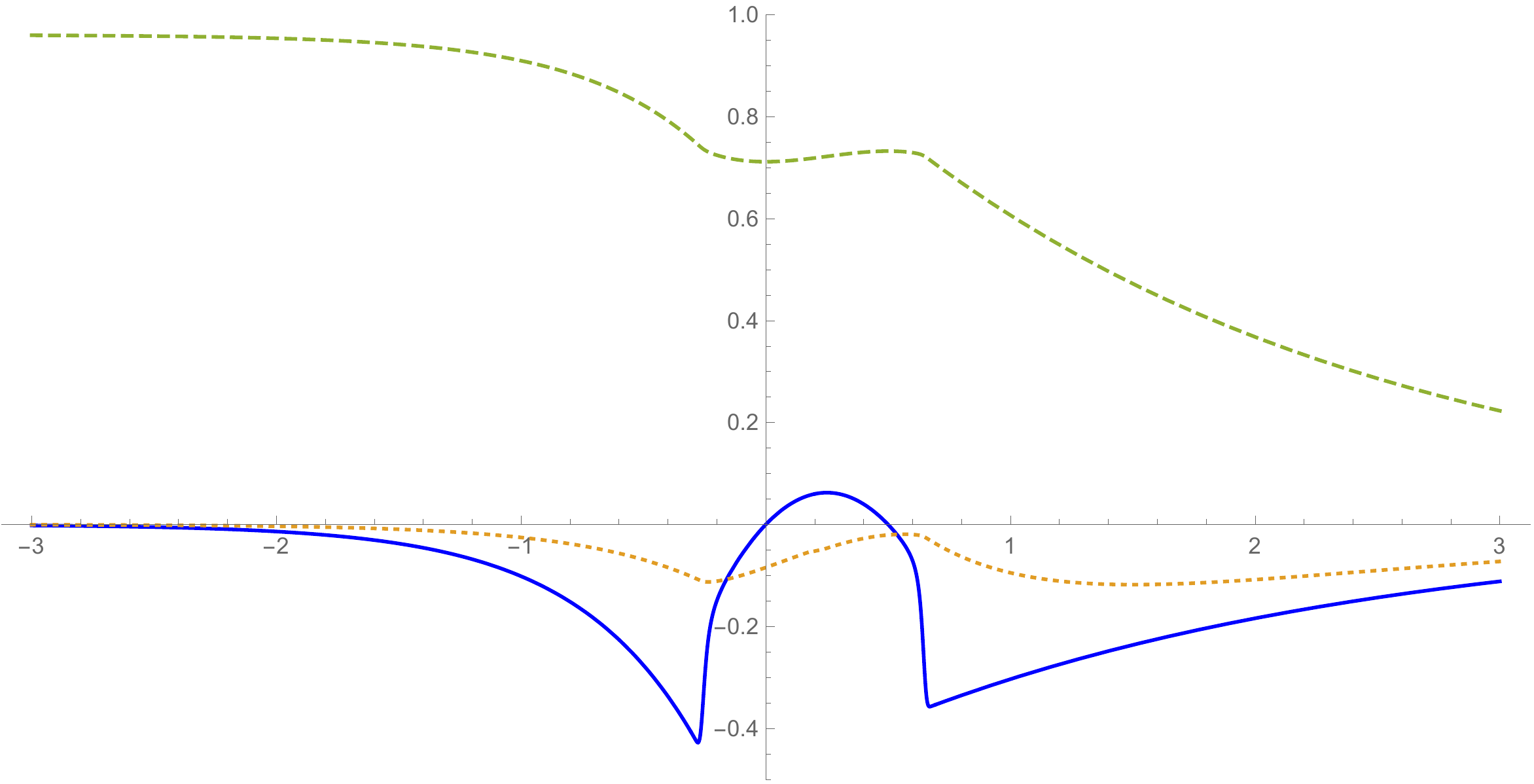}
    \caption{A sketch of the function in the class $\mathscr{C}$ which will be finally chosen in Section \ref{s:choice-of-A} to prove Theorem \ref{thm:spectral3}, in the $t= \log r$ axis. The graph of $A(t)$ is the solid curve, $G(t):=\Xi'(t)+2\Xi(t)$ the dashed one, and $\Xi'(t)$ the dotted one. Even though $A$ is smooth, its derivative undergoes a very sharp change around the points $t = \frac{1}{2}$ and the point $t= -\frac{1}{\sqrt{B}}$, where $B$ is an appropriately large constant, cf. Section \ref{s:choice-of-A}.}
    \label{fig:fig1}
\end{figure}
\begin{figure}[ht]
    \centering
    \includegraphics[width=0.75\textwidth]{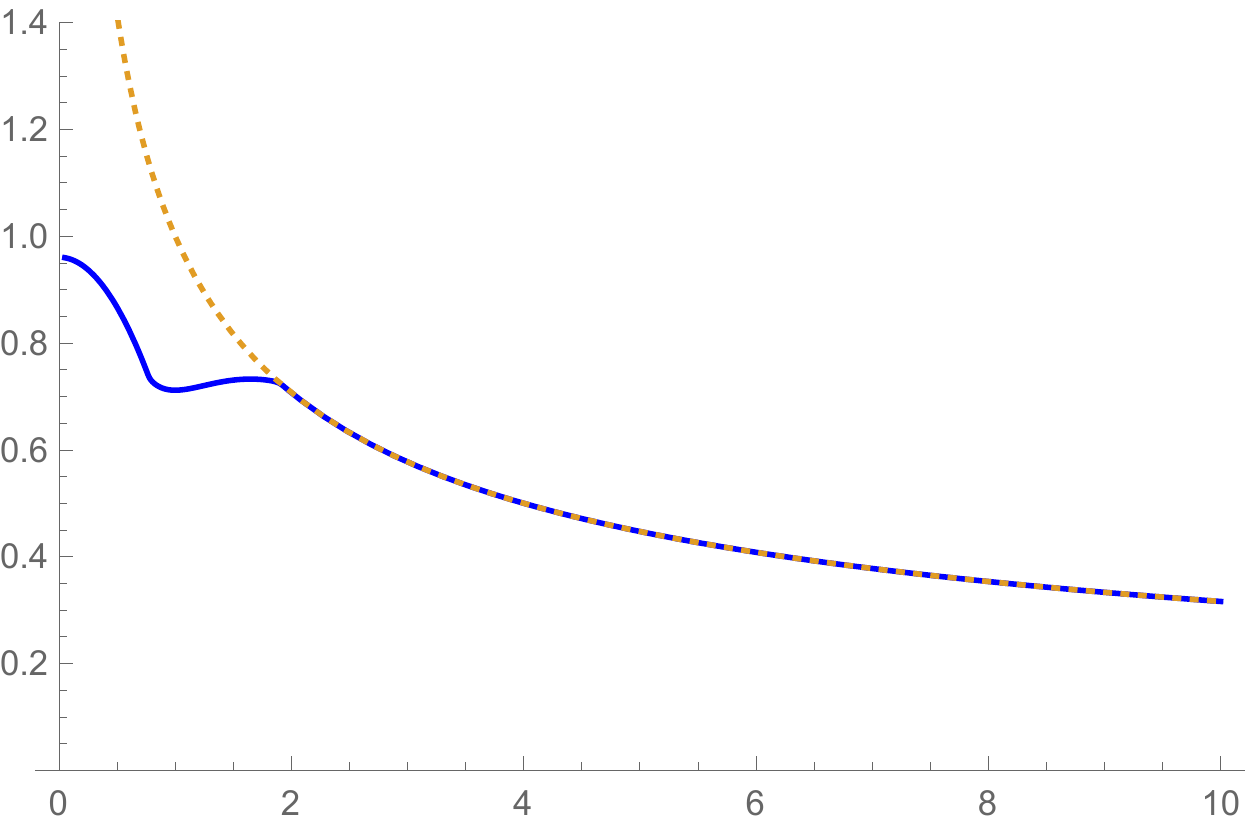}
    \caption{The profile of the background vorticity $\bar \Omega(x) = g(r)$ in the original coordinates (the solid curve). Compare with the exact singular profile $r^{-\al}$ (the dashed curve)}
    \label{fig:fig2}
\end{figure}

Fix $\Xi\in \mathscr{C}$. By \eqref{e:formula-g}, $g$ is then smooth, it equals $2\Xi (-\infty)- 4 c_0 r^2$ in a neighborhood of $0$, and it is equal to $r^{-\al}$ for $r\geq 2$, thanks to the conditions (i)-(ii). In particular the corresponding function $\bar \Omega (x) = g (|x|)$ satisfies the requirements of Theorem \ref{thm:spectral3}. We are now ready to turn Theorem \ref{thm:spectral3} into a (in fact stronger) statement for the eigenvalue equation \eqref{e:eigenvalue-equation-3}. In order to simplify its formulation and several other ones in the rest of these notes, we introduce the following sets

\begin{definition}\label{d:sets-U-m}
Having fixed $\Xi\in \mathscr{C}$ and a real number $m>1$, we denote by $\mathscr{U}_m$ the set of those complex $z$ with positive imaginary part with the property that there are nontrivial solutions $\varphi\in L^2\cap W^{2,2}_{\text{loc}} (\mathbb R, \mathbb C)$ of \eqref{e:eigenvalue-equation-3}.
\end{definition}

\index{aalUscrm@$\mathscr{U}_m$}

{\begin{remark}\label{r:m-factor-eigenvalue}
Observe that $z$ belongs to $\mathscr{U}_m$ if and only if it has positive imaginary part and $m z$ is an eigenvalue of $\mathcal{L}_m$.
\end{remark}}

\begin{theorem}\label{thm:spectral5}\label{THM:SPECTRAL5}
There is a function $\Xi\in \mathscr{C}$ and an integer $m_0\geq 2$ such that $\mathscr{U}_{m_0}$ is finite and nonempty.
\end{theorem}


\section{A formal expansion}
\label{sec:aformalexpansion}

\index{Rayleigh's stability equation@Rayleigh's stability equation}
In this section, we present a formal argument for the existence of solutions to Rayleigh's equation
\begin{equation}\label{e:eigenvalue-equation-recall}
(\Xi(t) - z) \Big(-\frac{d^2}{dt^2} + m^2\Big) \varphi (t) + A(t) \varphi (t) = 0 
\end{equation}
which correspond to unstable eigenmodes.

For convenience, we ignore the subtleties as $|t| \to +\infty$ by considering~\eqref{e:eigenvalue-equation-recall} in a bounded open interval $I$ with zero Dirichlet conditions. The background flow in question is a vortex-like flow in an annulus. In this setting, Rayleigh's equation~\eqref{e:eigenvalue-equation-recall} is nearly the same as the analogous equation 
for shear flows $(U(y),0)$ in a bounded channel: \index{shear flow@shear flow} the differential rotation $\Xi(t)$ corresponds to the shear profile $U(y)$, and $A(t)$, which is $r$ times the radial derivative of the vorticity, corresponds to $U''(y)$. In either setting, unstable eigenvalues $\lambda$ correspond to $z$ with $\textrm{Im} z > 0$. Moreover, Rayleigh's equation enjoys a conjugation symmetry: If $(\varphi,m,z)$ is a solution, then so is $(\bar{\varphi},m,\bar{z})$. Hence, to each unstable mode, there exists a corresponding stable mode, and vice versa. The equation is further symmetric with respect to $m \to -m$, and we make the convention that $m \geq 0$.

It will be convenient to divide by the factor $\Xi - z$ and analyze the resulting steady Schr{\"o}dinger equation. This is possible when:
\begin{itemize}
    \item[(i)] the eigenvalue is stable or unstable, i.e. $\textrm{Im} z \neq 0$, or 
    \item[(ii)]  $A(t)$ vanishes precisely where $\Xi(t) - z$ vanishes. 
    \end{itemize}
Note that we already know that $A(t)$ changing sign is necessary for instability. This is \index{Rayleigh's inflection point theorem@Rayleigh's inflection point theorem} Rayleigh's inflection point theorem, which can be demonstrated by multiplying~\eqref{e:eigenvalue-equation-3} with $\bar \varphi$ and integrating by parts (this procedure is in fact used a few times in the rest of the chapter, see for instance \eqref{e:tested}, \eqref{e:imaginary-trick}, and \eqref{e:imaginary-trick-2}). Suppose $A$ vanishes at $t_0$ and define $z_0 := \Xi(t_0)$, sometimes called a \emph{critical value} or \emph{critical layer}. By the assumption that $\Xi' < 0$, we have that\index{critical value@critical value}\index{critical layer@critical layer}
\begin{equation}
    \mathcal{K}(t) := \frac{A(t)}{\Xi(t) - z_0}
\end{equation}
is a well-defined potential, and
\begin{equation}
    \label{eq:aoverxi}
    \frac{A(t)}{\Xi(t) - z_0} \approx \frac{A'(t_0)}{\Xi'(t_0)} + O(t-t_0)
\end{equation}
as $t \to t_0$. After rearranging, the equation at $z = z_0$ becomes
\begin{equation}
    \underbrace{-\varphi'' + \mathcal{K} \varphi}_{ =: \mathcal{L} \varphi} = - m^2 \varphi \, .
\end{equation}
$\mathcal{L}$ is a self-adjoint operator on the appropriate function spaces. Using~\eqref{eq:aoverxi}, it is possible to design $\Xi$ such that $\mathcal{K}$ is negative enough to ensure that $\mathcal{L}$ has a negative eigenvalue; see Lemma~\ref{l:bottom}. Let $-m_0^2$ denote the bottom of the spectrum of $\mathcal{L}$, with the convention $m_0 > 0$. Let $\varphi_0$ denote an associated eigenfunction,
\begin{equation}
    \label{eq:varphievalequation}
    \mathcal{L} \varphi_0 = - m_0^2 \varphi_0 \, ,
\end{equation}
which we normalize to $\| \varphi_0 \|_{L^2}^2 = 1$. By Sturm-Liouville theory, the eigenvalue $-m_0^2$ is \index{algebraic multiplicity@algebraic multiplicity} algebraically simple, and its eigenfunction $\varphi_0$ does not vanish except at $\partial I$. In summary, $(\varphi_0,m_0,z_0)$ is a solution to Rayleigh's equation. The eigenfunction $\varphi_0$ is known as a \emph{neutral mode}, since it is neither stable nor unstable: $\textrm{Im} z = 0$. $m_0$ is the corresponding \emph{neutral wavenumber}, and in general it may not be an integer.
\index{neutral mode@neutral mode}\index{neutral wavenumber@neutral wavenumber}

We seek unstable modes $\varphi_\varepsilon$ nearby the neutral mode $\varphi_0$ in a perturbation expansion:
\begin{equation}
    \begin{aligned}
        \varphi_\varepsilon &= \varphi_0 + \varepsilon \varphi_1 + \varepsilon^2 \varphi_2 + \cdots \\
    m_\varepsilon &= m_0 + \varepsilon m_1 + \varepsilon^2 m_2 + \cdots \\
    z_\varepsilon &= z_0 + \varepsilon z_1  + \varepsilon^2 z_2 + \cdots \, .
    \end{aligned}
\end{equation}
We normalize $\| \varphi_\varepsilon \|_{L^2}^2 = 1$, which yields the condition
\begin{equation}
    \langle \varphi_0, \varphi_1 \rangle = 0 \, .
\end{equation}
Eventually, we will see that it is also possible to fix the freedom in the parameter $\varepsilon$ by choosing, e.g., $|m_1|=1$. Due to the conjugation symmetry of Rayleigh's equation, the perturbation expansion will simultaneously detect stable modes.

The equation to be satisfied by $\varphi_1$ is
\begin{align}
    -z_1 & \Big( -\frac{d^2}{dt^2} + m_0^2 \Big) \varphi_0  + 2(\Xi - z_0) m_0 m_1 \varphi_0\nonumber\\
    &+ \underbrace{(\Xi - z_0)  \Big( -\frac{d^2}{dt^2} + m_0^2 \Big) \varphi_1 + A \varphi_1}_{ = (\Xi - z_0) (\mathcal{L} + m_0^2) \varphi_1} = 0 \, .
\end{align}
We wish to divide by $\Xi - z_0$ and apply the equation~\eqref{eq:varphievalequation} for $\varphi_0$ to obtain
\begin{equation}
    z_1 (\Xi - z_0)^{-1} \mathcal{K} \varphi_0 + 2 m_0 m_1 \varphi_0 + (\mathcal{L} + m_0^2) \varphi_1 = 0 \, .
\end{equation}
However, we should not manipulate terms containing $(\Xi - z_0)^{-1}$ cavalierly; $(\Xi - z_0)^{-1}$ is not locally integrable. Rather, we divide by $\Xi - z_0 \pm i \delta$ and send $\delta \to 0^+$. We will find that one choice of $\pm$ will detect the stable eigenvalue, and the other choice will detect the unstable eigenvalue. We write
\begin{equation}
    z_1 \lim_{\delta \to 0^+} [(\Xi - z_0 \pm i\delta )^{-1} \mathcal{K} \varphi_0] + 2 m_0 m_1 \varphi_0 + (\mathcal{L} + m_0^2) \varphi_1 = 0 \, .
\end{equation}
By the Fredholm theory, the above equation is solvable provided that
\begin{equation}
    \langle z_1 (\Xi - z_0 \pm i0)^{-1} \mathcal{K} \varphi_0 + 2 m_0 m_1 \varphi_0, \varphi_0 \rangle = 0 \, ,
\end{equation}
borrowing the physics shorthand $\pm i0$ to track the regularization procedure. By the Plemelj formula,\index{Plemelj's formula@Plemelj's formula}
\begin{align}
    \int (\Xi - z_0 \pm i0)^{-1} \mathcal{K} |\varphi_0|^2 \, dt = & \mp i\pi \Xi'(t_0)^{-1} \mathcal{K}(t_0) |\varphi_0(t_0)|^2\nonumber\\
    &+ \textrm{pv} \, \int (\Xi - z_0)^{-1} \mathcal{K} |\varphi_0|^2 \, dt \, .
\end{align}
Hence,
\begin{equation}
    z_1 \left[ \mp i\pi \Xi'(t_0)^{-1}\mathcal{K}(t_0) |\varphi_0(t_0)|^2  + \textrm{pv} \, \int (\Xi - z_0)^{-1} \mathcal{K} |\varphi_0|^2 \, dt \right] + 2 m_0 m_1 = 0 \, .
\end{equation}
The first term is guaranteed not to vanish.
In particular, we have the relationship
\begin{equation}
    \label{eq:calculateangle}
    z_1 = \frac{2 m_0 m_1}{\pm i\pi \Xi'(t_0)^{-1} \mathcal{K}(t_0) |\varphi_0(t_0)|^2  - \textrm{pv} \, \int (\Xi - z_0)^{-1} \mathcal{K} |\varphi_0|^2 \, dt } \neq 0 \, ,
\end{equation}
and
\begin{equation}
    \label{eq:choiceofsign}
    \textrm{sgn} (\textrm{Im} \, {z_1}) = \mp \textrm{sgn} \frac{\mathcal{K}(t_0) m_1}{\Xi'(t_0)} \, .
\end{equation}

The perturbation expansion suggests the existence of a curve of unstable modes bifurcating from the neutral mode; see Figure~\ref{fig:figeval}. For the rigorous implementation of the ideas above to show that the existence of the unstable branch in a neighborhood of the neutral mode we refer to Lemma \ref{l:will-apply-Rouche} where ultimately Plemelj's formula will play a crucial role.  

However, it is also important to say something about the global picture of the curve of unstable modes and, in particular, to ensure that $m$ can be chosen to be an integer. This is done by continuing the curve; see Proposition~\ref{p:almost-final}. Heuristically, the main obstacle will be that the curve could ``fall" back to the plane $\textrm{Im} z = 0$. A key part of understanding the global picture will be to demonstrate that the only possible \emph{neutral limiting modes}, where the curve touches $\textrm{Im} z = 0$, are at the critical values $z_c = \Xi(t_c)$ where $A(t_c) = 0$ (otherwise, the potential $A/(\Xi - z)$ contains a singularity which will lead to a contradiction); see Proposition~\ref{p:3+4}.

\begin{figure}[ht]
    \centering
    \includegraphics[width=0.8\textwidth]{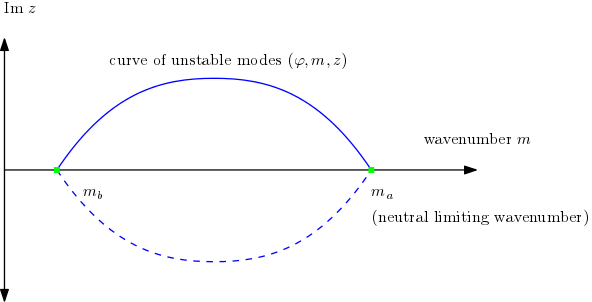}
    \caption{A schematic depiction of the emergence of stable and unstable modes from neutral modes. The dotted curve is the corresponding curve of stable modes, obtained by complex conjugation. Notably, the curve is not smooth across $\textrm{Im} z = 0$. The one-sided derivatives of $z$ at $m_a$ may be calculated as $-z_1$ in~\eqref{eq:calculateangle} with $m_1 = -1$.}
    \label{fig:figeval}
\end{figure}

The case of rotating flows in an annulus was considered by Lin in~\cite{LinSIMA2003}, and the reader to encouraged to compare with the approach therein.

\section{Overview of the proof of Theorem \ref{thm:spectral5}}\label{sec:overviewspectralthm}

The rest of the chapter is devoted to proving Theorem \ref{thm:spectral5}. The proof will be achieved through a careful study of \index{Rayleigh's stability equation@Rayleigh's stability equation} Rayleigh's stability equation~\eqref{e:eigenvalue-equation-3} and, in particular, the set $\mathscr{P}$ of pairs $(m,z)$ such that $z\in \mathscr{U}_m$ and $m>1$, i.e.,
\begin{equation}\label{e:def-pairs}
\mathscr{P}:= \left\{(m,z)\in \mathbb R \times \mathbb C: z\in \mathscr{U}_m, m>1\right\}\, .
\end{equation}\index{aalPscr@$\mathscr{P}$}
Given that $\Xi$ is strictly decreasing, we have 
\[
\lim_{t\to -\infty} \Xi (t) > \Xi (a) > \Xi (b) > \lim_{t\to \infty} \Xi (t) = 0
\]
and in order to simplify our notation we will use $\Xi (-\infty)$ for $\lim_{t\to -\infty} \Xi (t)$ and, occasionally, $\Xi (\infty)$ for $0$.

The first step in the proof of Theorem \ref{thm:spectral5} is understanding which pairs $(m,z)$ belong to the closure of $\mathscr{P}$ and have $\textrm{Im}\, z =0$. Solutions $(m,z,\varphi)$ to~\eqref{e:eigenvalue-equation-3} with $(m,z) \in \overline{\mathscr{P}}$ and $\textrm{Im}\, z = 0$ are sometimes called \emph{neutral limiting modes}~\cite{LinSIMA2003}. \index{neutral mode@neutral mode} To that end, it is convenient to introduce the following two self-adjoint operators:\index{aalLa@$L_a$}\index{aalLb@$L_b$}
\begin{align}
L_a &:= -\frac{d^2}{dt^2} + \frac{A(t)}{\Xi (t) - \Xi (a)}\label{e:def-L_a}\\
L_b &:= -\frac{d^2}{dt^2} + \frac{A(t)}{\Xi (t) - \Xi (b)}\label{e:def-L_b}\, .
\end{align}
Thanks to the definition of the class $\mathscr{C}$, it is easy to see that both functions $\frac{A(t)}{\Xi (t) - \Xi (a)}$ and $\frac{A(t)}{\Xi (t) - \Xi (b)}$ are bounded and that $\frac{A(t)}{\Xi (t) - \Xi (a)} < \frac{A(t)}{\Xi (t) - \Xi (b)}$.

Moreover, the first is negative on $]\! - \! \infty, b[$ and positive on $]b, \infty[$, while the second is negative on $]\! -\!\infty , a[$ and positive on $]a, \infty[$. Recall that the spectra of these operators are necessarily real and denote by $-\lambda_a$ and $-\lambda_b$ the smallest element in the respective spectra: observe that, by the Rayleigh quotient characterization, $-\lambda_a < -\lambda_b$.\index{aagla@$\lambda_a$}\index{aaglb@$\lambda_b$}

While \emph{a priori}, $\lambda_a = \lambda_b = 0$ is possible, for the remainder of this section, we work only in the setting $-\lambda_b < 0$; this is possible due to Proposition~\ref{p:final}.

The following proposition characterizes the possible neutral limiting modes:

\begin{proposition}\label{p:3+4}\label{P:3+4}
If $(m_0,z)\in \overline{\mathscr{P}}$ and $\textrm{Im}\, z =0$, then either $z= \Xi (a)$ or $z= \Xi (b)$. 
Moreover, in either case, if $m_0>1$ then necessarily $m_0 = \sqrt{\lambda_a}$ or $m_0 = \sqrt{\lambda_b}$. Assume in addition that $- \lambda_a < -1$. Then, for $z = \Xi (a)$, the unique $m\geq 1$ such that \eqref{e:eigenvalue-equation-3} has a nontrivial solution $\psi_a\in L^2$ is $m_a = \sqrt{\lambda_a}$. Moreover, any nontrivial solution has the property that $\psi_a (a) \neq 0$.
\end{proposition}

\begin{remark}\label{r:b-also}
We remark that the exact same argument applies with $b$ in place of $a$ when $\lambda_b >1$, even though this fact does not play any role in the rest of the notes.
\end{remark}

\index{aalma@$m_a$}\index{aalmb@$m_b$}
Observe that this does not yet show that $(m_a, \Xi (a))\in \overline{\mathscr{P}}$ corresponds to a neutral limiting mode. The latter property will be achieved in a second step, in which we seek a curve of unstable modes emanating from $(m_a, \Xi(a))$:

\begin{proposition}\label{p:5-7}\label{P:5-7}
Assume $- \lambda_a<-1$ and let $m_a=\sqrt{\lambda_a}$.
There are positive constants $\varepsilon >0$ and $\delta>0$ with the following property: 
For every $h\in ]0, \delta[$, $\mathscr{U}_{m_a-h} \cap B_\varepsilon (\Xi (a)) \neq \emptyset$. 
\end{proposition}

This is proved by making the formal argument in Section~\ref{sec:aformalexpansion} rigorous.

\begin{remark}\label{r:b-also-2} In fact, the argument given for the proposition proves the stronger conclusion that $\mathscr{U}_{m_a-h} \cap B_\varepsilon (\Xi (a))$ consists of a single point $z$, with the property that $mz$ is an eigenvalue of $\mathcal{L}_m$ with geometric multiplicity $1$. \index{geometric multiplicity@geometric multiplicity}
Moreover, the very same argument applies to $b$ in place of $a$ and $h \in ]-\delta,0[$ if $\lambda_b >1$.
\end{remark}
Combined with some further analysis, in which the curve of unstable modes is continued, the latter proposition will allow us to conclude the following:

\begin{proposition}\label{p:almost-final}\label{P:ALMOST-FINAL}
Assume $- \lambda_a<-1$, let $m_a = \sqrt{\lambda_a}$ and  $m_b:= \sqrt{\max \{1, \lambda_b\}}$. Then 
$\mathscr{U}_m\neq \emptyset$ for every $m\in ]m_b, m_a[$. \end{proposition}

Thus far, we have not selected our function $\Xi$: the above properties are valid for any element in the class $\mathscr{C}$. The choice of $\Xi$ comes in the very last step.

\begin{proposition}\label{p:final}
There is a choice of $\Xi\in \mathscr{C}$ with the property that $]m_b,m_a[$ contains an integer larger than $1$.
\end{proposition}

Clearly, the combination of Proposition \ref{p:almost-final} and Proposition \ref{p:final} gives Theorem \ref{thm:spectral5}: we first choose $\Xi$ as in Proposition \ref{p:final} and hence we select $m_0$ as the largest natural number which belongs to the interval $]m_b,m_a[$; the properties claimed in Theorem \ref{thm:spectral5} follow then from Proposition \ref{p:almost-final}. The proof of Proposition \ref{p:final} is in fact a rather straightforward application of the following.

\begin{lemma}\label{l:bottom}\label{L:BOTTOM}
Let $m_0$ be any integer. Then there exists $\Xi\in \mathscr{C}$ with $a=0$ and $b=\frac{1}{2}$ such that the smallest eigenvalue of the operator $L_a$ is smaller than $-m_0^2$.
\end{lemma}

\begin{remark}\label{rmk:veryunstablemodes}
A consequence of Lemma~\ref{l:bottom} is that the most unstable wavenumber $m_0$ can be made arbitrarily large. Only $m_0 \geq 2$ is necessary to prove non-uniqueness. We warn the reader that here $m_0$ {\em does not} refer to the neutral wavenumber of Section~\ref{sec:aformalexpansion}.
\end{remark}

The rest of the chapter will be devoted to proving the Propositions \ref{p:3+4} and \ref{p:5-7} and Lemma \ref{l:bottom}. We finish this section by giving the simple proof of Proposition \ref{p:final}

\begin{proof} For simplicity we fix $a=0$ and $b=\frac{1}{2}$ and we look at the set of functions $\Xi$ with this particular choice of zeros for $A$. We then denote by $L_{\Xi, a}$ the operator in \eqref{e:def-L_a}. We fix an arbitrary $\Xi_0\in \mathscr{C}$ and let $-\lambda (0)$ be the smallest eigenvalue of $L_{\Xi_0,a}$. We then consider the smallest integer $m_0\geq 3$ such that $m_0^2 > \lambda (0)$. By Lemma \ref{l:bottom} there is an element $\Xi_1\in \mathscr{C}$ with the property that $a=0$, $b=\frac{1}{2}$ and, if $-\lambda (1)$ is the smallest element of the spectrum of $L_{\Xi_1, a}$, then $-\lambda (1) < -m_0^2$. For $\sigma\in [0,1]$ consider $L_{\Xi_\sigma,a}$ where
\[
\Xi_\sigma = (1-\sigma) \Xi_0 + \sigma \Xi_1\, 
\]
and observe that $\Xi_\sigma \in \mathscr{C}$ for every $\sigma\in [0,1]$.

Since $\sigma \mapsto \Xi_\sigma$ is continuous in the uniform convergence, by the Rayleigh quotient characterization we see that the smallest element $-\lambda (\sigma)$ of the spectrum of $L_{\Xi_\sigma,a}$ is a continuous function of $\sigma$. There is thus one $\sigma\in [0,1[$ with $\lambda (\sigma)= m_0^2$. Let $\sigma_0$ be the largest $\sigma$ with $\lambda (\sigma)= m_0^2$. Observe now that, if we let $- \mu (\sigma_0)$ be the smallest eigenvalue of $L_{\Xi_{\sigma_0}, b}$, then $\mu (\sigma_0) < m_0^2$. In addition, $\sigma\mapsto \mu (\sigma)$ is also continuous and thus there is $h>0$ such that $\mu (\sigma) < m_0^2$ for all $\sigma\in [\sigma_0-h, \sigma_0+h]$. On the other hand $\lambda (\sigma_0+h)> m_0^2$. This shows that $m_b < m_0 < m_a$ if we choose $\Xi= \Xi_{\sigma_0+h}$, completing the proof of our claim.
\end{proof}

\section{ODE Lemmas}

An essential tool in the proofs of the Propositions \ref{p:3+4} and \ref{p:5-7} are the following two ODE lemmas.

\begin{lemma}\label{l:ODE1}
Let $m>0$. For every $f\in L^2 (\mathbb R)$ there is a unique $\psi\in L^2(\R) \cap W^{2,2}_{\text{loc}}$ s.t.
\begin{equation}\label{e:Laplacian-1d}
-\frac{d^2\psi}{dt^2} + m^2\psi = f
\end{equation}
and it is given by
\begin{equation}\label{e:potential-1d}
\psi (t) = \frac{1}{2m} \int_{\R} e^{-m|t-\tau|} f (\tau)\, d\tau\, .
\end{equation}
\end{lemma}
\begin{proof} The lemma is a classical well-known fact. At any rate the verification that
$\psi$ as in \eqref{e:potential-1d} solves \eqref{e:Laplacian-1d} is an elementary computation while, since obviosuly $e^{-m|t|}\in L^1$, $\psi\in L^2$ if $f\in L^2$. Moreover, any other solution $\hat\psi$ of \eqref{e:Laplacian-1d} must satisfy $\hat\psi (t) = \psi (t) + C_+ e^{mt} + C_- e^{-mt}$ for some constants $C_\pm$ and the requirement $\hat\psi\in L^2$ immediately implies $C_+=C_-=0$.
\end{proof}

The second ODE Lemma is the following:

\begin{lemma}\label{l:ODE2}
Let $v\in L^1 (\mathbb R, \mathbb C)$. Then for every constant $c_-$ there is a unique solution $y \in W^{2,1}_{\text{loc}} (\mathbb R, \mathbb C)$ of 
\begin{equation}\label{e:ODE2}
- \frac{d^2y}{dt^2} + (m^2 + v) y =0
\end{equation}
with the property that 
\begin{equation}\label{e:y=e^mt}
\lim_{t\to - \infty} e^{-mt} y (t) =c_-\, .
\end{equation}
Moreover we have $y(t) = e^{mt} (c_-+z(t))$ for a function $z(t)$ which satisfies the bounds
\begin{align}
|z(t)| &\leq |c_-|\left[\exp \left(\frac{1}{2m} \int_{-\infty}^t |v(s)|\, ds\right) -1\right]\label{e:est-z}\\
|z'(t)| &\leq 2m |c_-|\left[\exp \left(\frac{1}{2m} \int_{-\infty}^t |v(s)|\, ds\right) -1\right]\label{e:est-z'}
\end{align}
A symmetric statement, left to the reader, holds for solutions such that
\begin{equation}\label{e:y=e^mt-plus}
\lim_{t\to \infty} e^{mt} y (t) =c_+\, .
\end{equation}
\end{lemma}

Important consequences of the above Lemmas are the following:

\begin{corollary}\label{c:decay}
If $(m,z)\in \mathscr{P}$, then the space of solutions $\varphi\in L^2\cap W^{2,2}_{\text{loc}}$ of \eqref{e:eigenvalue-equation-3} is $1$-dimensional. Moreover for any such $\varphi$ there is a constant $C$ with the property that
\begin{align}
|\varphi (t)| &\leq C e^{-m|t|}\, 
\end{align}
and there are two constants $C_+$ and $C_-$ such that
\begin{align}
\lim_{t\to\infty} e^{mt} \varphi (t) &= C_+\\
\lim_{t\to -\infty} e^{-mt} \varphi (t) &= C_-\, .
\end{align}
The constants are either both nonzero or both zero, in which case $\varphi$ vanishes identically.

The same conclusions apply if $m>1$, $z\in \{\Xi (a), \Xi (b)\}$ and $\varphi$ solves \eqref{e:eigenvalue-equation-3}.
\end{corollary}
\begin{proof}
Observe that $|\Xi (t)-z|\geq |\textrm{Im}\, z|$, while $A (t) = 6 c_0 e^{2t}$ for $-t$ sufficiently large and $|A(t)|\leq  2 e^{-\al t}$ for $t$ sufficiently large. In particular 
\begin{equation}\label{e:estimate-A-over-Xi}
\frac{|A(t)|}{|\Xi (t)-z|} \leq C e^{-\al |t|}\, .
\end{equation}
First of all notice that, if $\varphi\in L^2\cap W^{2,2}_{\text{loc}}$ solves \eqref{e:eigenvalue-equation-3}, by Lemma \ref{l:ODE1} (applied with $f= -\frac{A\varphi}{\Xi-z}$) we have
\begin{equation}\label{e:integral-equation}
|\varphi (t)| \leq \frac{C}{2m} \int e^{-m|t-\tau|} e^{-\al |\tau|} |\varphi (\tau)|\, d\tau\, . 
\end{equation}
Using Cauchy-Schwarz and the fact that $\varphi\in L^2$ we immediately obtain that $\varphi\in L^\infty$, namely, that there is a constant $C$ such that $|\varphi|\leq C$. We now prove inductively that $|\varphi (t)|\leq C_k e^{-k\al |t|/2}$ as long as $k\al/2 \leq m$. The case $k=0$ has already been shown. Assume thus that the inequality holds for $k-1$ and that $k\al/2 \leq m$. We then observe that
\begin{align*}
e^{-m|t-\tau|} e^{-\al |\tau|} |\varphi (\tau)| &\leq C_{k-1} e^{- m|t-\tau| - k\al |\tau|/2} e^{-\al |\tau|/2}\\
&\leq C_{k-1} e^{-k\al (|t-\tau| + |\tau|)/2} e^{-\al |\tau|/2}\\
&\leq C_{k-1} e^{-k\al |t|/2} e^{-\al |\tau|/2}\, .
\end{align*}
Inserting in \eqref{e:integral-equation} and using $e^{-\al |\tau|/2}\in L^1$ we then obtain $|\varphi (t)|\leq C_k e^{-k\al |t|/2}$. Assuming now $k\al/2 \leq m < (k+1) \al/2$ we can, likewise, bound
\[
e^{-m|t-\tau|} e^{-\al |\tau|} |\varphi (\tau)| \leq C_k  e^{- m|t-\tau| - (k+1)\al |\tau|/2} e^{-\al |\tau|} \leq C_k e^{-m|t|} e^{-\al |\tau|/2}
\]
and plugging into \eqref{e:integral-equation} one last time we conclude $|\varphi (t)|\leq C e^{-m|t|}$.

In order to show that $\varphi$ is unique up to a multiplicative constants, it suffices to show that $\lim_{t\to -\infty} e^{-mt} \varphi (t)$ exists and is finite. Hence Lemma \ref{l:ODE2} would conclude that the solution is uniquely determined by $C_-$, and that the latter must be nonzero, otherwise $\varphi\equiv 0$. 
In order to show existence and finiteness of $C_-$ rewrite 
\[
\varphi (t) = \frac{e^{mt}}{2m} \int_t^\infty e^{-ms} \frac{A(s)}{\Xi (s) -z} \varphi (s)\, ds 
+ \frac{e^{-mt}}{2m} \int_{-\infty} ^t e^{m s} \frac{A(s)}{\Xi (s) -z} \varphi (s)\, ds\, .
\]
Since by our estimates both $e^{-ms} \frac{A(s)}{\Xi (s) -z} \varphi (s)$ and $e^{m s} \frac{A(s)}{\Xi (s) -z} \varphi (s)$ are integrable, we conclude that $C_{\pm}$ exist and equal
\begin{align*}
C_\pm = \frac{1}{2m}\int_{-\infty}^\infty e^{\pm ms} \frac{A(s)}{\Xi (s) -z} \varphi (s)\, ds\, .
\end{align*}

\medskip

As for the last sentence of the statement of the lemma, the same arguments can be used in the case $z\in \{\Xi (a), \Xi (b)\}$, since the crucial point is that, thanks to the assumption that $A (a) = A(b)=0$ and $\Xi' (a) \neq 0 \neq \Xi' (b)$, the estimate \eqref{e:estimate-A-over-Xi} remains valid.
\end{proof}

\begin{proof}[Proof of Lemma \ref{l:ODE2}] We distinguish between the case $c_-\neq 0$ and $c_-=0$. In the case $c_- \neq 0$ we can divide by $c_-$ and reduce the statement to $c_-=1$. For the existence it suffices to look for a solution of \eqref{e:ODE2} which satisfies \eqref{e:y=e^mt} on a half-line of type $]-\infty, T]$ for some $T$. Such solution has then a $W^{2,1}_{\text{loc}}$ continuation on $[T, \infty[$ by standard ODE theory. Likewise the uniqueness is settled once we can show the uniqueness holds on $]-\infty, T]$. Observe next that, if the solution exists, we would clearly conclude that $\frac{d^2 y}{dt^2}\in L^1 (]-\infty, T])$, hence implying that 
\[
\lim_{t\to-\infty} y' (t)
\]
exists and is finite. On the other hand \eqref{e:y=e^mt} implies that such limit must be $0$.

Let $\tilde{y} (t) = e^{-mt} y (t)$ and observe that we are looking for a solution of 
\[
(e^{2mt} \tilde{y}')' = e^{2mt} v \tilde{y}\, .
\]
Integrating between $-N$ and $t$ the latter identity and letting $t\to -\infty$ we conclude 
\begin{equation}\label{e:tildey'}
e^{2mt} \tilde{y}' (t) = \int_{-\infty}^t e^{2ms} v (s)\tilde{y} (s)\, ds\, .
\end{equation}
Divide by $e^{2mt}$ and integrate once more to reach
\begin{align*}
\tilde{y} (t) -1 &= - \int_{-\infty}^t \int_{-\infty}^r e^{2m (s-r)} v(s)\tilde{y} (s)\, ds\, dr
\\&= \frac{1}{2m} \int_{-\infty}^t \big(1-e^{-2m (t-s)}\big) v(s) \tilde{y} (s)\, ds
\end{align*}
We then define the transformation 
\begin{equation}\label{e:fixed-point}
\mathscr{F} (\tilde{y}) (t) = \frac{1}{2m} \int_{-\infty}^t \big(1-e^{-2m (t-s)}\big) v(s) \tilde{y} (s)\, ds + 1\, 
\end{equation}
which we consider as a map from $L^\infty (]-\infty, T])$ into itself.
From our discussion we conclude that $y$ solves \eqref{e:ODE2} and obeys \eqref{e:y=e^mt} if and only if $\tilde{y}$ is a fixed point of $\mathscr{F}$. Choosing $T$ large enough so that $\|v\|_{L^1 (]-\infty, T])}\leq m$ we see immediately that $\mathscr{F}$ is contraction on $L^\infty (]-\infty, T])$ and it thus has a unique fixed point. We have thus showed existence and uniqueness of the solution in question. 

Observe now that $z(t) = \tilde{y} (t) -1$ and set 
\[
Z(t) := \exp \left(\frac{1}{2m} \int_{-\infty}^t |v(s)|\, ds\right) -1\, .
\]
$Z$ solves the ODE $Z' = \frac{|v|}{2m} Z + \frac{|v|}{2m}$ and, since $\lim_{t\to-\infty} Z(t) =0$, the integral equation
\[
Z (t) = \frac{1}{2m} \int_{-\infty}^t |v(s)| Z(s)\, ds + \frac{1}{2m} \int_{-\infty}^t |v(s)|\, ds\, .
\]
We first want to show that $|z(t)|\leq Z(t)$ on $]-\infty, T]$. We set $\tilde{y}_0 := Z+1$ and define inductively $\tilde{y}_{i+1} = \mathscr{F} (\tilde{y}_i)$. From the above discussion we know that $\tilde{y}_i$ converges uniformly to $\tilde{y}$ and it suffices thus to show that $|\tilde{y}_i -1| \leq Z$ for all $i$.
By definition we have $|\tilde{y}_0-1| = Z$ and thus we need to show the inductive step. We estimate
\begin{align*}
|\tilde{y}_{i+1} (t) -1| &\leq \frac{1}{2m} \int_{-\infty}^t |v(s)| |\tilde{y}_i (s)|\, ds\\
&\leq \frac{1}{2m} \int_{-\infty}^t |v(s)| Z(s)\, ds + \frac{1}{2m} \int_{-\infty}^t |v(s)|\, ds = Z(t)\, ,
\end{align*}
We have shown \eqref{e:est-z} on $]-\infty, T]$. In order to extend the inequality to the whole real axis observe first that we can assume, without loss of generality, that $\|v\|_{L^1 (\mathbb R)}>0$, otherwise we trivially have $|\tilde{y} (t)-1| = Z(t) =0$ for all $t$. In particular we can select $T$ so that all of the above holds and at the same time $\|v\|_{L^1 (]-\infty, T])}>0$. This implies $Z(T)>0$. Moreover, by \eqref{e:fixed-point} and $\mathscr{F} (\tilde{y})= \tilde{y}$, either
\begin{align*}
|\tilde{y} (T) -1| &< \frac{1}{2m} \int_{-\infty}^{T} |v(s)| |\tilde{y} (s)|\, ds
\end{align*}
or $|v||\tilde{y}|$ vanishes identically on $]-\infty, T]$. In both cases we conclude $|\tilde{y} (T)-1|< Z(T)$. Consider now $\sup \{t\geq T: |\tilde{y} (t)-1|< Z (t)\}$. Such supremum cannot be a finite number $T_0$ because in that case we would have $|\tilde{y} (T_0)-1| = Z(t_0)$ while the same argument leading to the strict inequality $|\tilde{y} (T)-1|< Z(T)$ implies $|\tilde{y} (T_0)-1|< Z (T_0)$.

Having shown \eqref{e:est-z} we now come to \eqref{e:est-z'}. Recalling \eqref{e:tildey'} we have
\begin{align*}
z'(t) &= \int_{-\infty}^t e^{-2m (t-s)} v (s) (z(s)+1)\, ds\\
&\leq \int_{-\infty}^t e^{-2m (t-s)} |v (s)| Z(s)\, ds + \int_{-\infty}^t e^{-2m (t-s)} |v (s)|\, ds
= 2m Z(t)\, .
\end{align*}

We now come to the case $c_- =0$. In that case we need to show that the unique solution is identically~$0$. Arguing as for the case $c_- =1$ we conclude that $\varphi$ is a fixed point of the transformation
\[
\mathscr{F} (\varphi) (t) = \frac{1}{2m} \int_{-\infty}^t \big(1-e^{-2m (t-s)}\big) v(s) \tilde{y} (s)\, ds
\]
Again, for a sufficiently small $T$, $\mathscr{F}$ is a contraction on $L^\infty (]\! - \infty, T])$ and hence it has a unique fixed point. Since however $0$ is, trivially, a fixed point, we conclude that $\varphi\equiv 0$ on $]\! -\!\infty, T]$. Standard ODE theory implies then that $\varphi$ vanishes identically on the whole $\mathbb R$.
\end{proof}

\section{Proof of Proposition \ref{p:3+4}}\label{s:3+4}

We start by showing the last statement of the proposition, namely:
\begin{itemize}
    \item[(A)] For $z = \Xi (a)$ and under the assumption that $\lambda_a>1$, the unique $m\geq 1$ such that \eqref{e:eigenvalue-equation-3} has a nontrivial solution $\psi_a\in L^2$ is $m_a = \sqrt{\lambda_a}$.
\end{itemize}
Before coming to its proof we also observe that the same argument applies with $b$ in place of $a$.

First of all observe that, for $z=\Xi (a)$, the equation \eqref{e:eigenvalue-equation-3}, which becomes
\begin{equation}\label{e:eigenvalue-equation-again}
-\frac{d^2\varphi}{dt^2} + m^2 \varphi + \frac{A}{\Xi-\Xi (a)} \varphi = 0,
\end{equation}
has nontrivial solutions $\varphi\in W^{2,2}_{\text{loc}}\cap L^2 (\mathbb R; \mathbb C)$ if and only if it has nontrivial solution $\varphi \in W^{2,2}_{\text{loc}} \cap L^2 (\mathbb R;\mathbb R)$. That the equation has a nontrivial solution when $m=\sqrt{\lambda_a}$ follows from the classical theory of self-adjoint operators. We therefore only need to show that the existence of a nontrivial solution is only possible for a single $m\geq 1$. Arguing by contradiction assume there are two, $1\leq m_1< m_2$, and denote by $\psi_1$ and $\psi_2$ the respective solutions. Then there is a nontrivial linear combination
\[
\psi = C_1 \psi_1 + C_2 \psi_2
\]
which vanishes on $a$. Observe that $\psi_1$ and $\psi_2$ can be interpreted as eigenfuctions of the self-adjoint operator $-\frac{d^2}{dt^2} + \frac{A(t)}{\Xi (t)-\Xi (a)} $ relative to distinct eigenvalues and they are, therefore, $L^2$ orthogonal. Summing the equations, multiplying by $\psi$ and integrating by parts we achieve
\begin{equation}\label{e:tested}
\underbrace{\int \Big((\psi')^2 +\frac{A}{\Xi-\Xi (a)} \psi^2\Big)}_{=:I} = - C_1^2 m_1^2 \int \psi_1^2 - C_2^2 m_2^2 \int \psi_2^2\, .
\end{equation}
Recalling that $A = \Xi'' + 2\Xi' = (\Xi' + 2\Xi)'$, we wish to integrate by parts the second integrand in the left-hand side. Observe that, because $\psi$ vanishes on $a$ and $\Xi' (a) \neq 0$, the function $\frac{\psi^2}{\Xi - \Xi (a)}$ is in fact continuously differentiable. In particular we can write
\[
\int\frac{A}{\Xi-\Xi (a)} \psi^2 = \int \left(\frac{\Xi'+2\Xi}{(\Xi-\Xi(a))^2} \Xi' \psi^2 - 2\frac{\Xi'+2\Xi}{\Xi-\Xi (a)}\psi\psi'\right)\, .
\]
Substituting it into I, we achieve
\begin{align*}
I &= \int \left(\psi' - \frac{\Xi'}{\Xi-\Xi (a)}\psi\right)^2 + \int \left(\frac{2\Xi\Xi'\psi^2}{(\Xi-\Xi (a))^2} - \frac{4\Xi\psi\psi'}{\Xi-\Xi (a)} \right)\\
&= \int \left(\psi' - \frac{\Xi'}{\Xi-\Xi (a)}\psi\right)^2 + 2 \int \frac{\Xi'}{\Xi-\Xi (a)} \psi^2\, ,
\end{align*}
where to reach the second line we have written the first term in the second integral as 
\[
- 2 \, \Xi \frac{d}{dt} \left(\frac{1}{\Xi-\Xi (a)}\right) \psi^2
\]
and integrated it by parts. Again thanks to the fact that $\psi$ vanishes at $a$ we can write it as $\psi= (\Xi-\Xi (a)) \eta$ and hence conclude
\begin{align*}
I &= \int ((\Xi - \Xi (a))\eta')^2 + \int 2 (\Xi-\Xi (a))\Xi' \eta^2 
\\
&= \int ((\Xi - \Xi (a))\eta')^2 - 2 \int (\Xi-\Xi (a))^2 \eta\eta'\\
&= \int (\Xi-\Xi(a))^2 (\eta'-\eta)^2 - \int (\Xi-\Xi(a))^2 \eta^2\\
&= \int  (\Xi-\Xi(a))^2 (\eta'-\eta)^2 - \int (C_1^2\psi_2^2 + C_2^2\psi_2^2)\, .
\end{align*}
Inserting the latter in \eqref{e:tested} we conclude
\[
\int (\Xi-\Xi(a))^2 (\eta'-\eta)^2 = - C_1^2 (m_1^2-1) \int \psi_1^2 - C_2^2 (m_2^2-1) \int \psi_2^2\, .
\]
Observe that, since $m_2>1$ and $\psi_2$ is nontrivial, we conclude that $C_2=0$. This would then imply that $\psi = C_1 \psi_1$ and we can thus assume $C_1=1$ in all our computations. In particular $\eta'=\eta$, which implies $\eta (t) = C e^t$. We can now write $\psi_1 (t) = (\Xi (t)-\Xi(a)) \eta (t)$ and given the properties of $\Xi (t)$ we easily see that this would violate the decay at $+\infty$ that we know for $\psi_1$ from Corollary \ref{c:decay}. 

\begin{remark}\label{r:phi(a)-nonzero}
We record here a consequence of the above argument: a nontrivial solution $\varphi$ of \eqref{e:eigenvalue-equation-again} necessarily satisfies $\varphi (a) \neq 0$ (and thus it must be unique up to constant factors).
\end{remark}

\medskip

We next show that
\begin{itemize}
    \item[(B)] If $(m_0,z)\in \overline{\mathscr{P}}$, $m_0\geq 1$ and $z\in \mathbb R$, then $z$ is in the closure of the range of $\Xi$. 
\end{itemize}
We again argue by contradiction and assume the existence of
\begin{itemize}
    \item[(i)] A sequence $\{m_j\}\subset \, ]1, \infty[$ converging to $m_0\in [1, \infty[$;
    \item[(ii)] A sequence $\{z_j\}\subset \mathbb C$ with $\textrm{Im}\, z_j >0$ converging to $z\in \mathbb R\setminus \overline{\Xi (\mathbb R)}$;
    \item[(iii)] A sequence $\psi_j$ of nontrivial solutions of
    \begin{equation}\label{e:eigenvalue-5}
    -\frac{d^2\psi_j}{dt^2} + m^2 \psi_j + \frac{A}{\Xi-z_j} \psi_j = 0
    \end{equation}
\end{itemize}
By Corollary \ref{c:decay} we can normalize our functions $\psi_j$ so that $\psi_j (t) e^{-m_jt} \to 1$ as $t\to-\infty$ and $\psi_j (t) e^{m_jt} \to C_j\neq 0$ as $t\to\infty$. Observe also that there is a positive constant $c_0$ such that $|\Xi-z_j|\geq c_0$ for all $j$ sufficiently large, thanks to (ii). In particular, the functions $\frac{A}{\Xi-z_j}$ are uniformly bounded in $L^1$. By Lemma \ref{l:ODE2} there is a positive $T_0\geq b+1$, independent of $j$ such that
\begin{equation}\label{e:uniform-exp}
\left|\psi_j (t) - C_j e^{-m_j t}\right| \leq \frac{C_j}{2} e^{- m_j t} \qquad \forall t \geq T_0\, ,
\end{equation}
and there is a constant $C$, independent of $j$ such that
\begin{equation}\label{e:uniform-inside}
\|\psi_j\|_{L^\infty ([a,b])} \leq C\, .
\end{equation}
Next multiply \eqref{e:eigenvalue-5} by $\bar \psi_j$, integrate in $t$ and take the imaginary part of the resulting equality to conclude
\begin{equation}\label{e:imaginary-trick}
\int \frac{A}{(\Xi - \textrm{Re}\, z_j)^2 + (\textrm{Im}\, z_j)^2} |\psi_j|^2 = 0\, .
\end{equation}
We might break the integral into three integrals on the regions $]\! -\! \infty, a[$, $]a,b[$, and $]b, \infty[$, where the function $A$ is, respectively, negative, positive, and negative. This gives
\[
-\int_{T_0}^{2T_0} \frac{A}{(\Xi - \textrm{Re}\, z_j)^2 + (\textrm{Im}\, z_j)^2} |\psi_j|^2 \leq
 \int_a^b \frac{A}{(\Xi - \textrm{Re}\, z_j)^2 + (\textrm{Im}\, z_j)^2} |\psi_j|^2
\]
Now, the right-hand side of the inequality can be bounded uniformly independently of $j$ by \eqref{e:uniform-inside} and (ii). On the other hand the function $\frac{-A}{(\Xi - \textrm{Re}\, z_j)^2 + (\textrm{Im} z_j)^2}$ is larger than a positive constant $c$ independent of $j$ on $[T_0, 2 T_0]$. Using \eqref{e:uniform-exp} we can achieve a uniform bound $|C_j|\leq C$ for the constants $C_j$. The latter bound, combined with the estimates of Lemma \ref{l:ODE2} and the uniform bound on $\|\frac{A}{\Xi-z_j}\|_{L^1}$ easily imply that $\psi_j$ is precompact in $L^2$. We can thus extract a subsequence, not relabeled, converging to a nontrivial $L^2$ solution $\psi$ of 
\begin{equation}\label{e:eigenvalue-equation-6}
-\frac{d^2\psi}{dt^2} + m_0^2 \psi + \frac{A}{\Xi-z} \psi = 0\, .
\end{equation}
Without loss of generality we assume that $\psi$ is real valued, since $z$ is real. We can thus multiply \eqref{e:eigenvalue-equation-6} by $\psi$ and integrate to achieve
\[
\int ((\psi')^2 + m_0^2 \psi^2) + \int \frac{\Xi''+2\Xi'}{\Xi- z} \psi^2 = 0\, .
\]
Integrating by parts $\int \frac{\Xi''}{\Xi-z} \psi^2$ we find 
\[
\int ((\psi')^2 + m_0^2 \psi^2) + \int \left(\frac{(\Xi')^2}{(\Xi-z)^2} \psi^2 - 2 \frac{\Xi'}{\Xi-z} \psi' \psi\right) +
\int \frac{2\Xi'}{\Xi-z} \psi^2 = 0 \, ,
\]
which we can rewrite as 
\begin{equation}\label{e:energy-trick}
\int \Big(\Big(\psi' - \frac{\Xi'}{\Xi-z} \psi\Big)^2 + m_0^2 \psi^2\Big) + 2 \int \frac{\Xi'}{\Xi-z} \psi^2 = 0\, .
\end{equation}
As already done in the previous paragraphs we set $\eta = \frac{\psi}{\Xi-z}$ and write the identity as
\[
\int \left((\Xi-z)^2 (\eta')^2 + m_0^2 (\Xi-z)^2 \eta^2 + 2 \Xi' (\Xi-z) \eta^2\right) = 0 
\]
Integrating by parts the last term we find 
\[
\int (\Xi-z)^2 (\eta'-\eta)^2 + \int (m_0^2-1) (\Xi-z)^2 \eta^2 = 0\, .
\]
We thus conclude that $m_0=1$ and $\eta'=\eta$, i.e. $\eta (t) = C e^t$, but again we see that this would violate $\psi\in L^2$. 

\medskip

We next employ a suitable variation of the latter argument to show that
\begin{itemize}
    \item[(C)] $(m_0, 0)$ and $(m_0, \Xi (-\infty))$ do not belong to $\overline{\mathscr{P}}$ if $m_0\geq 1$.
\end{itemize}
We again argue by contradiction and assume the existence of
\begin{itemize}
    \item[(i)] A sequence $\{m_j\}\subset \, ]1, \infty[$ converging to $m_0\in [1, \infty[$;
    \item[(ii)] A sequence $\{z_j\}\subset \mathbb C$ with $\textrm{Im}\, z_j >0$ converging to $0$ or to $\Xi (- \infty)$;
    \item[(iii)] A sequence $\psi_j$ of nontrivial solutions of
    \begin{equation}\label{e:eigenvalue-equation-7}
    -\frac{d^2\psi_j}{dt^2} + m_j^2 \psi_j + \frac{A}{\Xi-z_j} \psi_j = 0\, .
    \end{equation}
\end{itemize}
We first focus on the case $z_j\to 0$. Normalize again the solutions so that $\psi_j (t)$ is asymptotic to $e^{m_j t}$ for $t$ negative, and to $C_j e^{-m_j t}$ for $t$ positive.

Observe that in this case we have $\frac{A}{\Xi}\in L^1 (]-\infty, N])$ for every $N$, while $\frac{A}{\Xi-z_j}$ enjoys a uniform $L^1$ bound on any $]-\infty, N]$. We can thus apply Lemma \ref{l:ODE2} and conclude the $\psi_j$ can be assumed to converge uniformly to a function $\psi$ on $]-\infty, N]$ for every $N$ and that likewise $\psi (t)$ is asymptotic to $e^{m_0 t}$ for $t$ negative. 
 
As done previously we multiply the equation  \eqref{e:eigenvalue-equation-7} by $\bar\psi_j$, integrate, and take the imaginary part. In particular we gain the inequality
\[
\int_b^\infty \frac{A}{(\Xi- \textrm{Re}\, z_j)^2 + (\textrm{Im}\, z_j)^2} |\psi_j|^2 \leq - \int_a^b \frac{A}{(\Xi- \textrm{Re}\, z_j)^2+ (\textrm{Im}\, z_j)^2} |\psi_j|^2\, .
\]
Since $z_j\to 0$ and the range of $\Xi$ on $[a,b]$ is bounded away from $0$, we conclude that the right-hand side is uniformly bounded. In particular, passing to the limit we conclude that 
\begin{equation}\label{e:info-L^1}
\Xi^{-2} A |\psi|^2 \in L^1 ([b, \infty[)\, .
\end{equation}
Observe however that 
\[
\lim_{t\to\infty} \frac{A(t)}{\Xi (t)} = \lim_{t\to \infty} \frac{-\al e^{-\al t}}{c_1 e^{-2t}+\frac{1}{2-\al} e^{-\al t}} = - \al (2-\al)\, .
\]
In particular we conclude that $\psi\in L^2$. Moreover, we can write
\[
\frac{A}{\Xi} = -\al (2-\al) + B 
\]
for a function $B$ which belongs to $L^1 ([T, \infty[)$ for every $T$. We thus have that 
\[
- \frac{d^2\psi}{dt^2} + (m_0^2 - \al (2-\al)) \psi + B \psi = 0\, .
\]
Recalling that $0<\al <1$ and $m_0\geq 1$, we have $m_0^2 - \al (2-\al)>0$ and we can therefore apply Lemma \ref{l:ODE2} to conclude that, for $\bar m := \sqrt{m_0^2 - \al (2-\al)}$ 
\[
\lim_{t\to \infty} e^{\bar m t} \psi (t)
\]
exists, it is finite, and nonzero. Observe however that \eqref{e:info-L^1} forces $e^{\al t} |\psi|^2\in L^1$, which in particular implies that $\bar m > \frac{\al}{2}$
We next argue as in the derivation of \eqref{e:energy-trick} to get
\[
\int \Big(\Big(\psi' - \frac{\Xi'}{\Xi} \psi\Big)^2 + m_0^2 \psi^2\Big) + 2 \int \frac{\Xi'}{\Xi} \psi^2 = 0\, .
\]
We again set $\psi= \Xi \eta$ and observe that, by our considerations, $\eta$ decays exponentially at $-\infty$, while it is asymptotic to $e^{(\al - \bar m) t}$ at $+\infty$. We rewrite the latter identity as 
\[
\int (\Xi^2 (\eta')^2 + m_0^2 \Xi^2 \eta^2 + 2 \Xi\Xi' \eta^2) = 0\, .
\]
We wish to integrate by parts the latter term to find
\begin{equation}\label{e:da-giustificare}
\int (\Xi^2 (\eta'-\eta)^2 + (m_0^2-1) \Xi^2 \eta^2)=0\, .
\end{equation}
Since we have exponential decay of $\eta$ at $-\infty$, while at $+\infty$ $\eta$ might grow, the latter integration by parts need some careful justification. First of all we notice that $\Xi \Xi' \eta^2$ decays exponentially at $+\infty$ and thus, since the other two integrands are positive, we can write
\[
\int (\Xi^2 (\eta')^2 + m_0^2 \Xi^2 \eta^2 + 2 \Xi\Xi' \eta^2) =
\lim_{N\to\infty} \int_{-\infty}^N (\Xi^2 (\eta')^2 + m_0^2 \Xi^2 \eta^2 + 2 \Xi\Xi' \eta^2) \, .
\]
Next, we can integrate by parts the third integrand (before passing to the limit) to write
\begin{align*}
& \int_{-\infty}^N (\Xi^2 (\eta')^2 + m_0^2 \Xi^2 \eta^2 + 2 \Xi\Xi' \eta^2)\\ 
 = & \int_{-\infty}^N (\Xi^2 (\eta'-\eta)^2 + (m_0^2-1) \Xi^2 \eta^2) + \Xi^2 (N) \eta^2 (N)\, .
\end{align*}
Since $\Xi (N) \eta (N)$ converges to $0$ exponentially, passing into the limit we conclude \eqref{e:da-giustificare}.
As before this would imply $m_0=1$ and $\eta (t) = C e^t$ at $+\infty$, while we have have already argued that $\eta$ is asymptotic to $e^{(\bar\alpha -\bar m) t}$, which is a contradiction because $\bar\alpha - \bar m \leq \bar\alpha <1$. 

We next tackle the case $z_j \to \Xi (-\infty)$. This time we observe that $\frac{A}{\Xi-z_j}$ enjoys a uniform $L^1$ bound on $[T, \infty[$ for every $T$ and we thus normalize the functions $\psi_j$ so that $\psi_j (t)$ is asymptotic to $e^{-m_j t}$ for $t\to \infty$. Arguing as above, we assume that $\psi_j$ converges uniformly on all $[T, \infty[$ to a $\psi$ which is asymptotic to $e^{-m_0 t}$ and solves
\begin{equation}\label{e:eigenvalue-equation-9}
-\frac{d^2\psi}{dt^2} + m_0^2 \psi + \frac{A}{\Xi-\Xi (-\infty)} \psi=0\, .
\end{equation}
As above we can assume that $\psi$ is real valued. Moreover, this time we infer (with the same method used to prove \eqref{e:info-L^1})
\begin{equation}\label{e:info-L1-2}
(\Xi-\Xi (-\infty))^{-2} A \psi^2 \in L^1 (\mathbb R)
\end{equation}
This time observe that, for $t$ sufficiently negative, $\frac{A(t)}{\Xi (t)- \Xi (-\infty)} = 8$. In particular we can explicitly solve the equation as 
\[
\psi (t) = C_1 e^{-t\sqrt{m_0^2+8}} + C_2 e^{t\sqrt{m_0^2+8}}
\]
when $t$ is sufficiently negative.
However, if $C_1$ were positive, \eqref{e:info-L1-2} would not hold. In particular we infer exponential decay at $-\infty$. We can now argue as for the case $z_j\to 0$: we multiply \eqref{e:eigenvalue-equation-9} by $\psi$, integrate in time and perform an integration by part to infer
\[
\int \Big( \Big(\psi' - \frac{\Xi'}{\Xi - \Xi (-\infty)} \psi\Big)^2 + m_0 \psi^2\Big) + 2 \int \frac{\Xi'}{\Xi-\Xi (-\infty)} \psi^2 = 0\, .
\]
We then introduce $\eta$ so that $\psi = (\Xi-\Xi (-\infty)) \eta$. This time we infer exponential decay for $\eta$ at both $\infty$ and $-\infty$. Arguing as above we rewrite the last identity as
\[
\int ((\Xi- \Xi (-\infty))^2 (\eta'-\eta)^2 + (m_0^2-1) (\Xi- \Xi (-\infty))^2 \eta^2)=0\, ,
\]
reaching again a contradiction.

\medskip

In order to complete the proof of the proposition we need to show
\begin{itemize}
    \item[(D)] If $(m_0, \Xi (c)) \in \overline{\mathscr{P}}$ 
    and $m_0> 1$
    , then either $c=a$ or $c=b$ 
    and moreover we have, respectively, $m_0 = \sqrt{\lambda_a}$ or $m_0 = \sqrt{\lambda_b}$
    .
\end{itemize}
As before we argue by contradiction and assume the existence of
\begin{itemize}
    \item[(i)] A sequence $\{m_j\}\subset \, ]1, \infty[$ converging to $m_0\in \, ]1, \infty[$;
    \item[(ii)] A sequence $\{z_j\}\subset \mathbb C$ with $\textrm{Im}\, z_j >0$ converging to $\Xi (c)$ for some $c\not\in \{a,b\}$;
    \item[(iii)] A sequence $\psi_j$ of nontrivial solutions of
    \begin{equation}\label{e:eigenvalue-equation-10}
    -\frac{d^2\psi_j}{dt^2} + m_j^2 \psi_j + \frac{A}{\Xi-z_j} \psi_j = 0\, .
    \end{equation}
\end{itemize}
This time we normalize the $\psi_j$'s so that
\begin{equation}\label{e:L2-normalization}
\int (|\psi_j'|^2 + m_j^2 |\psi_j|^2) =1\, .
\end{equation}
By Lemma \ref{l:ODE2} we know that $\psi_j (t)$ is asymptotic to $\rho_j^\pm e^{\mp m_j t}$ for $t\to \pm \infty$, where $\rho_j^\pm \in \mathbb C \setminus \{0\}$. Since $\Xi (c)$ has a positive distance from both $0$ and $\Xi (-\infty)$, we can apply Lemma \ref{l:ODE2} to achieve uniform times $T_\pm$ with the properties that
\begin{align}
\left|\psi_j (t) - \rho_j^+ e^{-m_j t}\right|& \leq \frac{|\rho_j^+|}{2} e^{-m_j t} \qquad\qquad \forall t\geq T_+\, ,\label{e:exp-bound-1}\\
\left|\psi_j (t) - \rho_j^- e^{m_j t}\right| &\leq \frac{|\rho_j^-|}{2} e^{m_j t} \qquad\qquad \forall t\leq T_-\, .\label{e:exp-bound-2}
\end{align}
Combining the latter inequalities with \eqref{e:L2-normalization} we conclude that $\sup_j |\rho_j^\pm| < \infty$, and in particular $\{\psi_j\}_j$ is tight in $L^2$, i.e. for every $\varepsilon >0$ there is $N = N (\varepsilon)$ such that
\[
\sup_j \int_{|t|\geq N} |\psi_j|^2 < \varepsilon\, .
\]
The latter bound combined with \eqref{e:L2-normalization} implies, up to extraction of a subsequence which we do not relabel, the strong $L^2$ convergence of $\psi_j$ to a function $\psi$. Thanks to Sobolev embedding, the convergence is uniform on any compact set and, moreover, $\psi\in C^{1/2}$.

Arguing as for \eqref{e:imaginary-trick} we infer 
\begin{equation}\label{e:imaginary-trick-2}
\int \frac{A}{(\Xi-\textrm{Re}\, z_j)^2 + (\textrm{Im}\, z_j)^2} |\psi_j|^2 =0
\end{equation}
The latter bound implies $\psi (c)=0$. In fact first we
observe that $\frac{A}{|\Xi-z_j|^2} |\psi_j|^2$ converges in $L^1$ on $\mathbb R \setminus ]c-\delta, c+\delta[$ for every $\delta$. Choosing $\delta>0$ so that $|A (t) - A(c)| \leq \frac{|A(c)|}{2}$ for $t\in [c-\delta, c+\delta]$ and recalling that $|A(c)|>0$, we easily infer that 
\[
\sup_j \int_{c-h}^{c+h} \frac{|\psi_j|^2}{(\Xi-\textrm{Re}\, z_j)^2 + (\textrm{Im}\, z_j)^2} < \infty \qquad \forall h < \delta\, .
\]
If $\psi (c)$ were different from $0$, we can select a positive $h< \delta$ and a positive $c_0$ with the property that $|\psi (t)|^2 \geq 2c_0$ for all $t\in [c-h, c+h]$. In particular, for a large enough $j$ we infer $|\psi_j (t)|^2 \geq c$ for all $t\in [c-\delta, c+\delta]$. But then we would conclude
\[
\sup_j \int_{c-h}^{c+h} \frac{1}{(\Xi-\textrm{Re}\, z_j)^2 + (\textrm{Im}\, z_j)^2} < \infty\, .
\]
Since the denominator converges to $(\Xi - \Xi (c))^2$, this is clearly not possible. 

We now wish to pass in the limit in \eqref{e:eigenvalue-equation-10} to derive that
\begin{equation}\label{e:eigenvalue-equation-11}
- \psi'' + m_0^2 \psi + \frac{A}{\Xi-\Xi (c)} \psi =0\, ,
\end{equation}
where we notice that, thanks to $\psi (c)=0$ and the H\"older regularity of $\psi$, the function $\frac{A}{\Xi-\Xi (c)} \psi$ is indeed in $L^p$ for every $p<2$. We thus understand the equation distributionally.
The equation clearly passes to the limit outside the singularity $c$ of the denominator and thus we just need to pass it to the limit distributionally in some interval $]c-h,c+h[$. We write the third term as
\begin{align*}
\frac{A}{\Xi-z_j} \psi_j &= \left(\frac{d}{dt} \ln (\Xi-z_j) \right)\frac{A}{\Xi'} \psi_j\\
 &= \frac{d}{dt} \left(\ln (\Xi-z_j) \frac{A}{\Xi'} \psi_j\right) - \ln (\Xi-z_j)\frac{A}{\Xi'} \psi'_j - \ln (\Xi-z_j) \frac{d}{dt} \left(\frac{A}{\Xi'}\right) \psi_j\, .
\end{align*}
Observe that we can define the logarithm unequivocally because $\Xi$ is real valued and $\textrm{Im}\, z_j >0$.

Next, we remark that:
\begin{itemize}
    \item[(i)] $\frac{A}{\Xi'}$ is smooth in $]c-h, c+h[$;
    \item[(ii)] $\ln (\Xi-z_j)$ converges strongly \footnote{Since $\ln (\Xi-z_j)$ converges uniformly to $\ln (\Xi-\Xi(c))$ on any compact set which does not contain $c$, in order to reach the conclusion it suffices to prove a uniform $L^q$ bound on the functions, for every $q<\infty$. This can be easily concluded as follows. Choose an interval $[c-h, c+h]$ and recall that $\Xi$ does not change sign on it. For each $j$ large enough we then find a unique $c_j \in [c-h, c+h]$ such that $\Xi (c_j) = \textrm{Re}\, z_j$. Using the mean value theorem we easily conclude that $|\Xi (t)-z_j|\geq |\Xi (t) - \Xi (c_j)|
    \geq C^{-1} |t-c_j|$ for every $t\in [c-h, c+h]$, where $C^{-1} = \min \{|\Xi'(t)|: c-h\leq t \leq c+h\}$.} to $\ln (\Xi-\Xi(c))$ in $L^q (]c-h, c+h[)$ for every $q<\infty$;
    \item[(iii)] $\psi_j' \to \psi'$ weakly in $L^2$, while $\psi_j\to \psi$ uniformly. 
\end{itemize}
We thus conclude that $\frac{A}{\Xi-z_j} \psi_j$ converges distributionally to
\[
\frac{d}{dt} \left(\ln (\Xi-\Xi (c)) \frac{A}{\Xi'} \psi\right) - \ln (\Xi-\Xi(c)) \frac{A}{\Xi'} \psi' - \ln (\Xi-\Xi(c)) \frac{d}{dt} \left(\frac{A}{\Xi'}\right) \psi\, .
\]
Using now that $\psi\in W^{1,2}$ and $\psi (c)=0$ we can rewrite the latter distribution as
\[
\frac{A}{\Xi-\Xi (c)} \psi
\]
and hence conclude the validity of \eqref{e:eigenvalue-equation-11}. 

Observe next that from \eqref{e:eigenvalue-equation-11} we infer $\psi''\in L^p$ for every $p< 2$, which in turn implies that $\psi$ is indeed $C^{1,\kappa}_{\text{loc}}$ for every $\kappa < \frac{1}{2}$. In turn this implies that $\frac{A}{\Xi-\Xi (c)} \psi$ is continuous at $c$, so that in particular $\psi$ is twice differentiable. We thus can argue as for the derivation of \eqref{e:energy-trick} and get
\begin{equation}\label{e:energy-trick-4}
\int \Big(\Big(\psi' - \frac{\Xi'}{\Xi-\Xi (c)} \psi\Big)^2 + m_0^2 \psi^2\Big) + 2 \int \frac{\Xi'}{\Xi-\Xi (c)} \psi^2 = 0\, .
\end{equation}
Once again we can set $\psi = (\Xi-\Xi (c)) \eta$ and observe that $\eta\in W^{1,2}$, to rewrite the latter identity as 
\[
\int ((\Xi- \Xi (c))^2 (\eta'-\eta)^2 + (m_0^2-1) (\Xi- \Xi (c))^2 \eta^2)=0\, ,
\]
inferring that $\eta=0$. 

We thus have concluded that $\psi$ vanishes identically, but this is not yet a contradiction since the normalization \eqref{e:L2-normalization} and the strong $L^2$ convergence does not ensure that $\psi$ is nontrivial. In order to complete our argument, note first that, by the monotonicity of $\Xi$, for each $j$ large enough there is a unique $c_j$ such that $\Xi (c_j) = \textrm{Re}\, z_j$. We then multiply the equation \eqref{e:eigenvalue-equation-10} by $\bar \psi_j - \overline{\psi_j (c_j)}$ to obtain
\[
\int \Big(|\psi_j'|^2 + m_j^2 \psi_j (\bar\psi_j - \overline{\psi_j (c_j)}) + \frac{A}{\Xi-z_j} \psi_j (\bar\psi_j - \overline{\psi_j (c_j)})\Big) = 0\, .
\]
Note that $c_j$ must converge to $c$ and that the integrals
\[
\int \psi_j (\bar\psi_j - \overline{\psi_j (c_j)})
\]
converges to $0$ because $\psi_j - \psi_j (c_j)$ converges to $0$ uniformly and, thanks to the uniform exponential decay of $\psi_j$, the latter are uniformly bounded in $L^1$. For the same reason the first integral in the sum
\begin{equation}
    \label{e:up}\int_{|t-c|\geq h} \frac{A}{\Xi-z_j} \psi_j (\bar\psi_j - \overline{\psi_j (c_j)}) +\int_{|t-c|\leq h} \frac{A}{\Xi-z_j} \psi_j (\bar\psi_j - \overline{\psi_j (c_j)})
\end{equation}
converges to $0$ for every fixed $h$. On the other hand, $|\frac{A (t)}{\Xi (t)-z_j}| |\psi_j (t) - \psi_j (c_j)|\leq C |t-c_j|^{-1/2}$ and thus the second integrand in \eqref{e:up}
converges to $0$ as well. We thus conclude that the $L^2$ norm of $\psi'_j$ converges to $0$ as well. This however contradicts the normalization \eqref{e:L2-normalization}.

\section{Proof of Proposition \ref{p:5-7}: Part I}

We set $m_0=m_a$, $z_0 = \Xi (a)$, and
we fix a $\psi_0$ solution of 
\[
-\frac{d^2\psi_0}{dt^2} + m_0^2 \psi_0 + \frac{A}{\Xi-z_0} \psi_0 = 0
\]
with $L^2$ norm equal $1$. Since the operator is self-adjoint we will indeed assume that $\psi_0$ is real. We then define the projector $P_0: L^2 (\mathbb R; \mathbb C) \to \{\kappa \psi_0:\kappa \in \mathbb C\}$ as 
\[
P_0 (\psi) = \langle \psi, \psi_0\rangle \psi_0\, .
\]
Observe that $P_0$ is self-adjoint.\index{aalP0@$P_0$}
Next,
in a neighborhood of $(m_0, z_0)$ we will look for solutions of \eqref{e:eigenvalue-equation-3} by solving
\begin{equation}\label{e:Lagrange}
\left\{
\begin{array}{l}
-\psi'' + m^2 \psi + \frac{A}{\Xi-z} \psi + P_0 (\psi) = \psi_0\\ \\

\langle \psi, \psi_0\rangle =1
\end{array}
\right.
\end{equation}
which we can rewrite as
\begin{equation}\label{e:Lagrange-2}
\left\{
\begin{array}{l}
-\psi'' + m_0^2 \psi + \frac{A}{\Xi-z_0} \psi + P_0 (\psi) = A \left(((\Xi-z_0)^{-1} - (\Xi-z)^{-1}) \psi\right)\\
\qquad\qquad\qquad\qquad\qquad\qquad\qquad\qquad + (m_0^2-m^2)\psi + \psi_0\\ \\

\langle \psi, \psi_0\rangle =1
\end{array}
\right.
\end{equation}
Next we observe that the operator $-\frac{d^2}{dt^2} + m_0^2$, considered as a closed unbounded self-adjoint operator in $L^2$ (with domain $W^{2,2}$) has an inverse $\mathcal{K}_{m_0}\colon L^2 \to L^2$ which is a bounded operator. \index{aalKcalm@$\mathcal{K}_m$} We thus rewrite \eqref{e:Lagrange-2} as 
\begin{equation}\label{e:Lagrange-3}
\left\{
\begin{array}{ll}
\underbrace{\psi + \mathcal{K}_{m_0} \left(\frac{A}{\Xi-z_0} \psi + P_0 (\psi)\right)}_{=: T (\psi)}\\
\qquad \qquad= \underbrace{\mathcal{K}_{m_0}
\left(\left(A \left((\Xi-z_0)^{-1} - (\Xi-z)^{-1}\right) + (m_0^2 -m^2)\right) \psi\right)}_{=:- \mathcal{R}_{m,z} (\psi)} +
\mathcal{K}_{m_0} (\psi_0)\\ \\
\langle \psi, \psi_0 \rangle =1\, .
\end{array}
\right.
\end{equation}
The proof of Proposition \ref{p:5-7} will then be broken into two pieces. In this section we will show the first part, which we can summarize in the following 

\begin{lemma}\label{l:solve-for-psi}
For every $\mu>0$, if $(m,z)$ sufficiently close to $(m_0, z_0)$ and $\textrm{Im}\, z\geq \mu |\textrm{Re}\, (z-z_0)|$ then there is a unique $\psi= \psi (m,z) \in L^2 (\mathbb R)$ solving 
\begin{equation}\label{e:solve-for-psi}
T (\psi) + \mathcal{R}_{m,z} (\psi) = \mathcal{K}_{m_0} (\psi_0)\, .
\end{equation}
\end{lemma}
\index{aalT@$T$}\index{aalRcalmz@$\mathcal{R}_{m,z}$}

We recognize this as, essentially, the first step in the Lyapunov-Schmidt reduction but written for the solution $\psi$ rather than the perturbation $\tilde{\psi} = \psi - \psi_0$. (The second step will be to the solve one-dimensional equation~\eqref{e:what-we-want-to-do} enforcing $\langle \psi, \psi_0 \rangle = 1$.)\index{Lyapunov-Schmidt reduction@Lyapunov Schmidt reduction}

Before coming to the proof of Lemma~\ref{l:solve-for-psi}, 
we single out two important ingredients.

\begin{lemma}\label{l:invert-T}
$T$ is a bounded operator with bounded inverse on the spaces $L^2$ and $C^\sigma$, for any $\sigma \in ]0,1[$.
\end{lemma}
\begin{proof} Recall that the operator $\mathcal{K}_m$ is given by the convolution with $\frac 1{2m} e^{-m|\cdot|}$.
{Observe that $\mathcal{K}_m$ is well-defined on $C^\sigma$ and so is the multiplication by $\frac{A}{\Xi-z_0}$, since the latter is a smooth functions with bounded derivatives, and the operator $P_0 (\psi) = \langle \psi, \psi_0\rangle \psi_0$: for the latter we just need to check that $\psi\overline{\psi_0}$ is integrable, which follows from the exponential decay of $\psi_0$, { cf. Corollary \ref{c:decay}}.}

 In this first step we prove that $T$ is a bounded operator with bounded inverse in the spaces $L^2 (\mathbb R)$ and $C^\sigma (\mathbb R)$.

Recall that $\frac{A}{\Xi-z_0}= \frac{A}{\Xi-\Xi (a)}$ is indeed a bounded smooth function (thanks to the structural assumptions on $\Xi$: in particular recall that $\Xi' (a)\neq 0$ and $A(a) =0$, which implies that $\frac{A}{\Xi-\Xi(a)}$ is in fact smooth at $a$). Moreover the function and its derivatives decay exponentially at $\pm \infty$. It follows therefore that $\psi \mapsto \mathcal{K}_{m_0} (\frac{A}{\Xi-z_0} \psi + P_0 (\psi))$ is a compact operator, both on $L^2$ and on $C^\sigma$. Thus $T$ is a Fredholm operator with index $0$. We thus just need to check that the kernel is $0$ in order to conclude that it is invertible with bounded inverse. In both cases we need to show that the equation
\begin{equation}\label{e:kernel-T}
-\frac{d^2\psi}{dt^2} + m_0^2 \psi + \frac{A}{\Xi-\Xi (a)} \psi + P_0 (\psi) = 0
\end{equation}
has only the trivial solution. Observe that the kernel $V$ of the operator $\psi \mapsto -\frac{d^2\psi}{dt^2} + m_0^2 \psi + \frac{A}{\Xi-\Xi (a)} \psi$ is $1$-dimensional by Lemma \ref{l:ODE2} and Corollary \ref{c:decay}. In particular $V$ is generated by $\psi_0$. Since the operator $P_0$ is the orthogonal projection onto $V$ and $-\frac{d^2}{dt^2} + m_0^2+  \frac{A}{\Xi-\Xi (a)}$ is self-adjoint, the kernel of $-\frac{d^2}{dt^2} + m_0^2+ \frac{A}{\Xi-\Xi (a)} + P_0$ in $L^2$ must be trivial. 

In order to argue that the kernel is $0$ on $C^\sigma$ we apply a variation of the same idea: first we observe that if $\psi$ is a $C^\sigma$ solution of \eqref{e:kernel-T}, then $\frac{A}{\Xi-\Xi (a)} \psi + P_0 (\psi)$ is also in $C^\sigma$ and hence $\psi''\in C^\sigma$. Observe also that the operator is self-adjoint and thus we can assume that $\psi$ is real-valued. We then multiply both sides of \eqref{e:kernel-T} by $\bar \psi_0$, integrate by parts and use the fact that $\psi_0$ is in the kernel of the self-adjoint operator $-\frac{d^2}{dt^2} + m_0^2 + \frac{A}{\Xi-\Xi (a)}$ to conclude that $(\langle \psi, \psi_0\rangle)^2 =0$. But then
$\psi$ is a bounded solution of $-\frac{d\psi}{dt^2} + m_0^2 \psi + \frac{A}{\Xi-\Xi (a)}\psi =0$. Given that $\frac{A}{\Xi-\Xi (a)} \psi$ is a product of an exponentially decaying function and a bounded function, we conclude that $-\frac{d^2\psi}{dt^2} + m_0^2 \psi$ is an exponentially decaying function $f$. We thus have $\psi = \mathcal{K}_m (f) + C_1 e^{-m_0t} + C_2 e^{m_0t}$ for two constants $C_1$ and $C_2$. However $\mathcal{K}_m (f)$ decays exponentially at both $\pm \infty$ and thus, given that $\psi$ is bounded, we must have $C_1=C_2=0$. In particular $\psi$ decays exponentially at both $\pm \infty$ and so it is an $L^2$ function. But we already saw that every $L^2$ solution is trivial.
\end{proof}

\begin{lemma}\label{l:Rmz-small}
For every constant $\mu>0$ we define the cone $C_\mu := \{z: \textrm{Im} z \geq \mu |\textrm{Re}\, (z-z_0)|\}$. Then
\begin{equation}
\lim_{z\in C_\mu, (m,z)\to (m_0, z_0)} \|\mathcal{R}_{m,z}\|_O = 0\, ,
\end{equation}
where $\|L\|_O$ is the operator norm of $L$ 
when considered as a bounded operator from $L^2$ to $L^2$.
\end{lemma}
\begin{proof} Clearly, it suffices to show that
\begin{equation}
\lim_{z\in C_\mu, z\to z_0} \|\mathcal{K}_{m_0} \circ (A/(\Xi-z) - A/(\Xi-z_0))\|_O = 0\, .
\end{equation}
We can rewrite the operator as 
\[
\psi \mapsto \mathcal{K}_{m_0} \Big(\frac{A (z-z_0)}{(\Xi-z) (\Xi-z_0)} \psi\Big) \, .
\]
First of all observe that the operators 
\[
\psi \mapsto L_z (\psi) = \frac{A (z-z_0)}{(\Xi-z) (\Xi-z_0)} \psi
\]
are bounded in the operator norm uniformly in $z\in C_\mu$ by a constant $M$. Moreover, we can see that the adjoint operator is given by $L_z^* (\psi)= \frac{A (\bar z-z_0)}{(\Xi-\bar z) (\Xi-z_0)} \psi$ converges strongly in $L^2$ to $0$: indeed the functions $\frac{A (\bar z-z_0)}{(\Xi-\bar z) (\Xi-z_0)}$ are uniformly bounded and they converge to $0$ on $\mathbb R \setminus \{a\}$. We now use an argument entirely similar to that used in the proof of Lemma \ref{l:three}: given any $\varepsilon >0$ we fix the orthogonal projection $P_N$ onto a finite-dimensional subspace $V$ of $L^2$ with the property that $\|\mathcal{K}_{m_0}\circ P_N - \mathcal{K}_{m_0}\|_O$ is smaller than $\frac{\varepsilon}{2M}$. We then argue that for $|z-z_0|$ sufficiently small $P_N \circ L_z$ has operator norm smaller than $\frac{\varepsilon}{2}$. Having chosen an orthonormal base $\psi_1, \ldots, \psi_N$ for $V$, we recall that
\[
P_N (\psi)= \sum_i \langle \psi_i, \psi\rangle \psi_i\, 
\]
(where $\langle \cdot, \cdot \rangle$ denotes the scalar product in $L^2$).
Therefore our claim amounts to show that
\[
|\langle \psi_i, L_z (\psi)\rangle|\leq \frac{\varepsilon}{2N}
\]
for $z$ sufficiently close to $z_0$ and every $\psi$ with $\|\psi\|_{L^2}\leq 1$. For the latter we use
\[
|\langle \psi_i, L_z (\psi)\rangle| = |\langle L_z^* (\psi_i), \psi \rangle|\leq \|L_z^* (\psi_i)\|_{L^2}\, .
\]
\end{proof}

\begin{proof}[Proof of Lemma \ref{l:solve-for-psi}]
We rewrite the equation that we want to solve as 
\[
\psi + T^{-1} \circ \mathcal{R}_{m,z} (\psi) = T^{-1} \circ \mathcal{K}_{m_0} (\psi_0)\, .
\]
Note that $P_0(\psi_0)=\psi_0$. Furthermore, since $\mathcal K_{m_0}$ is, by definition, the inverse operator of $-\frac{\mathrm d^2}{\mathrm dt^2}+m_0^2\operatorname{Id}$,
\begin{equation*}
    \mathcal K_{m_0}^{-1}\Big(\psi_0+\mathcal K_{m_0}\Big(\frac{A}{\Xi-z_0}\psi_0\Big)\Big) = -\psi_0''+m_0^2\psi_0+\frac{A}{\Xi-z_0}\psi_0 = 0.
\end{equation*}
Therefore, 
\begin{equation*}
    \psi_0+\mathcal K_{m_0}\Big(\frac{A}{\Xi-z_0}\psi_0\Big) = 0.
\end{equation*}
In combination with the definition of $T$ in \eqref{e:Lagrange-3}, we get
\begin{equation*}
    T(\psi_0) = \psi_0+\mathcal K_{m_0}\Big(\frac{A}{\Xi-z_0}\psi_0+\psi_0\Big)=\mathcal K_{m_0}(\psi_0),
\end{equation*}
in other words,
\begin{equation}\label{e:T-1K}
T^{-1} \circ \mathcal{K}_{m_0} (\psi_0) = \psi_0\, .
\end{equation}
Therefore, \eqref{e:solve-for-psi} becomes
\begin{equation}\label{e:to-Neumann}
(\operatorname{ Id} + T^{-1} \circ \mathcal{R}_{m,z}) (\psi) = \psi_0\, ,
\end{equation}
so the existence of a unique solution is guaranteed as soon as $\|T^{-1} \circ \mathcal{R}_{m,z}\|_{O} < 1$. 
\end{proof}

\begin{remark}\label{r:Neumann-series}
In the remaining part of the proof of Proposition \ref{p:5-7} we will take advantage of the representation of $\psi$ as a function of $\psi_0$ through the Neumann series coming from \eqref{e:to-Neumann}. More precisely, our proof of Lemma \ref{l:solve-for-psi} leads to the following representation:
\begin{equation}\label{e:Neumann-series}
\psi = \psi_0 - (T^{-1} \circ \mathcal{R}_{m,z}) (\psi_0) + \sum_{k=2}^\infty (-1)^k (T^{-1}\circ \mathcal{R}_{m,z})^k (\psi_0)\, .
\end{equation}
\end{remark}

\section{Proof of Proposition \ref{p:5-7}: Part II}\label{s:5-7-part-II}

We now complete the proof of Proposition \ref{p:5-7}. The positive parameter $\mu>0$ in Lemma \ref{l:solve-for-psi} will have to be chosen sufficiently small: its choice will be specified in a few paragraphs, while for the moment we assume it to be fixed. We set $m_0 = m_a$ and $z_0 = \Xi (a)$. Thus, for each $(m_0+h ,z)$ in a set 
\[
U_{\delta, \mu} := \{|h|< \delta, |z-z_0|< \delta, \textrm{Im} z > \mu |\textrm{Re}\, (z-z_0)|\}
\]
we know that that there is a solution $\psi = \psi (m_0+h,z)$ of \eqref{e:solve-for-psi} which moreover satisfies the expansion \eqref{e:Neumann-series}.
We then define the function 
\begin{equation}
H (h,z) := \langle \psi (m_0+h,z), \psi_0\rangle\, ,
\end{equation}
and obviously we are looking for those $z$ which solve
\begin{equation}\label{e:what-we-want-to-do}
H (h,z) =1  \, .
\end{equation}
The main point of our analysis is the following
\begin{lemma}\label{l:will-apply-Rouche}
The function $H$ is holomorphic in $z$ and moreover
\begin{equation}\label{e:expansion}
H (h,z) = 1 - 2m_a h + c (a) (z-z_0) + o (|z-z_0| + |h|)
\end{equation}
where $c(a)$ is a complex number with $\textrm{Im}\, c(a) > 0 $.
\end{lemma}

The coefficient $c(a)$ is computed in~\eqref{eq:limitexists} and is, in fact, $1/z_1$ from~\eqref{eq:calculateangle} with $\textrm{Im} z_1 < 0$ and $m_1 = -(2m_0)^{-1}$. The formal computation involving the Plemelj formula to obtain~\eqref{eq:calculateangle} will appear in the calculation of $c(a)$ below.\index{Plemelj's formula@Plemelj's formula}

Given Lemma \ref{l:will-apply-Rouche}, consider now $\xi (h)$ which we obtain by solving $c(a) (\xi-z_0)= 2m_a h$, namely,
\[
\xi (h) = \frac{2m_a h}{c(a)} +z_0 = \frac{2m_a h}{|c(a)|^2} \overline{c(a)} +z_0\, .
\]
The idea behind the latter definition is that, if the term $o (|z-z_0| + |h|)$ vanished identically, $z = \xi (h)$ would be the solution of $H (h,z)=1$. Even though $o (|z-z_0| + |h|)$ does not vanish, we nonetheless expect that the solution $z$ of $H (h,z)=1$ is relatively close to $\xi (h)$.

Since $\textrm{Im}\, c(a)>0$, $\xi (h)$ has positive imaginary part if $h<0$. In particular we have
\[
\textrm{Im}\, \xi (h) \geq \gamma |h| \qquad \forall h < 0\, .
\]
where $\gamma$ is a positive constant. We then rewrite
\[
H (h, z) = 1 + c (a) (z-\xi (h)) + \underbrace{o (|\xi (h)-z_0| + h)}_{=: r(h)} + o (|z-\xi (h)|)\, .
\]
Consider the disk $D_h := \{|z-\xi (h)| \leq 2 \beta |h|\}$, for a suitably chosen constant $\beta>0$. We will show below that adjusting the constants $\mu$ and $\beta$ suitably, the disk will be in the domain of the holomorphic function $H (\cdot, z)$. Leaving this aside for the moment, by Rouch\'e Theorem, if we choose $h$ sufficiently small the set $H(h, D_h)$ contains a disk of radius $|c(a)|\beta h$ centered at $1+ r (h)$. But then for $h$ sufficiently small we also have $|r(h)| \leq \frac{|c(a)|\beta h}{2}$ and so we conclude that $1\in H (h, D_h)$, namely that there is a point $z (h)$ in the disk $D_h$ which is mapped in $1$ by $H (h, \cdot)$. This would then complete the proof of Proposition \ref{p:5-7} if we were able to prove that $\textrm{Im}\, z (h) >0$. We therefore need to show that $D_h$ is in the domain of $H (h, \cdot)$, namely,
\[
\textrm{Im}\, z \geq \mu |\textrm{Re}\, (z-z_0)|\, \qquad \forall z\in D_h\, .
\]
We first estimate
\[
\textrm{Im}\, z \geq \, \textrm {Im}\, \xi (h) - 2 \beta |h| \geq (\gamma- 2 \beta) |h|\, .
\]
Then
\begin{equation}\label{e:inequality-101}
|\textrm{Re}\, z-z_0| \leq |\xi (h)-z_0| + |z-\xi (h)| \leq \left(\frac{2m_a}{|c(a)|} + 2 \beta\right) |h|\, .
\end{equation}
We thus conclude that
\begin{equation}\label{e:inequality-102}
\textrm{Im}\, z(h) \geq \frac{\gamma-2\beta}{2m_a |c(a)|^{-1} +2\beta} |\textrm{Re}\, (z (h)-z_0)|\, .
\end{equation}
Thus it suffices to choose $\beta = \frac{\gamma}{3}$ and $\mu = \frac{\gamma}{6 m_a |c(a)|^{-1} + 2 \gamma}$. This guarantees at the same time the existence of a solution and the fact that $z (h)$ has positive imaginary part when $h<0$ (which results from combining \eqref{e:inequality-101} and \eqref{e:inequality-102}.

In order to complete the proof of Proposition \ref{p:5-7} we therefore just need to show Lemma \ref{l:will-apply-Rouche}.

\begin{proof}[Proof of Lemma \ref{l:will-apply-Rouche}]
In order to show holomorphicity we just need to show that, for each fixed $z$,
\[
z\mapsto \sum_{k=0}^\infty (- T^{-1} \circ \mathcal{R}_{m,z})^k
\]
is holomorphic. Since the series converges in the operator norm, it suffices to show that each map $z\mapsto (-T^{-1} \circ \mathcal{R}_{m,z})^k$ is holomorphic for every $k$, for which indeed it suffices to show that $z\mapsto \mathcal{R}_{m,z}$ is holomorphic. This is however obvious from the explicit formula. We therefore now come to the Taylor expansion \eqref{e:expansion}. 

\medskip

\textbf{ Step 1} We will show here that 
\begin{equation}\label{e:small-in-Csigma}
\|\mathcal{R}_{m_0+h,z}\|_{\mathcal{L} (C^\sigma)} \leq C (\sigma) (|h| + |z-z_0|)\, 
\end{equation}
for every $\sigma\in ]0,1[$, where $\|L\|_{\mathcal{L} (C^\sigma)}$ is the operator norm of a bounded linear operator $L$ on $C^{\sigma}$.
The estimate will have the following consequence. First of all using \eqref{e:Neumann-series} and $\|\psi_0\|_{L^2}^2 =1$ we expand
\begin{equation}\label{e:Taylor-2}
H (h,z) = 1 - \langle T^{-1} \circ \mathcal{R}_{m_0+h,z} (\psi_0), \psi_0\rangle 
+ \underbrace{\sum_{k=2}^\infty \langle (-T^{-1} \circ \mathcal{R}_{m_0+h,z})^k (\psi_0), \psi_0\rangle}_{=: R_1 (z,h)}\, .
\end{equation}
Hence using \eqref{e:small-in-Csigma} we estimate
\begin{align}
|R_1 (z,h)| & \leq \sum_{k=2}^\infty \|(-T^{-1} \circ \mathcal{R}_{m_0+h,z})^k (\psi_0)\|_\infty \|\psi_0\|_{L^1}\nonumber\\\hspace{-1em}
&\leq C \sum_{k=2}^\infty (\|T^{-1}\|_{C^\sigma} \|\mathcal{R}_{m_0+h,z}\|_{\mathcal{L} (C^\sigma)})^k \|\psi_0\|_{C^\sigma}\|\psi_0\|_{L^1}\nonumber\\
&= o (|h|+|z-z_0|)\, ,\label{e:resto-1}
\end{align}
for some fixed $\sigma$.
In order to show \eqref{e:small-in-Csigma} we write
\[
\mathcal{R}_{m_0+h,z} (\psi) = (z-z_0) \mathcal{K}_{m_0} \Big(\frac{1}{\Xi-z} \Big(\frac{A}{\Xi-z_0} \psi\Big)\Big) + (2m_0 h +h^2) \mathcal{K}_{m_0} (\psi)\, .
\]
Since $\frac{A}{\Xi-z_0}$ is smooth, it suffices to show that the operators $B_z:= \mathcal{K}_{m_0} \circ \frac{1}{\Xi-z}$ are uniformly bounded in $\mathcal{L} (C^\sigma)$. We first fix a smooth cut-off function $\varphi \in C^\infty_c (]a-2, a+2[)$ which equals $1$ on $[a-1,a+1]$ and write 
\[
B_z= B_z^1 + B_z^2
:= \mathcal{K}_{m_0} \circ \Big(\frac{1-\varphi}{\Xi-z}\Big)+\mathcal{K}_{m_0} \circ \Big(\frac{\varphi}{\Xi-z} \Big)\, .
\]
But since $(1-\varphi)/(\Xi-z)$ enjoys a uniform bound in $C^k$, it is easy to conclude that $\|B^1_z\|_{\mathcal{L} (C^\sigma)}$ is bounded uniformly in $z$. We thus need to bound
\begin{align*}
B^2_z (\psi) (t) &= \frac 1{2m_0}\int e^{-m_0 |t-s|} \frac{\varphi (s)}{\Xi (s) -z}  \psi (s)\, ds\, .
\end{align*}
We first bound $\|B^2_z\|_{L^\infty}$. We write $z= x+iy$ and, since $x$ is close to $a$, we select the only $a'$ such that $\Xi (a')=x$ and write
\begin{align*}
B^2_z (\psi) (t) &= \frac 1{2m_0}\underbrace{\int e^{-m_0 |t-s|} \frac{\varphi (s) (\psi (s)-\psi (a'))}{(\Xi (s) -\Xi (a')) - iy} \, ds}_{=: I_1 (t)}\\
&\qquad
+ \frac {\psi(a')}{2m_0}\underbrace{\int e^{-m_0 |t-s|} \frac{\varphi (s)}{\Xi (s) -z}\, ds}_{=: I_2 (t)}
\end{align*}
Writing $\frac{1}{\Xi -z} = \frac{1}{\Xi'}\frac{d}{dt} \ln (\Xi -z)$ we can integrate by parts to get
\begin{align*}
I_2 (t) &= - \underbrace{\int m_0 \frac{t-s}{|t-s|} e^{-m_0 |t-s|} (\Xi' (s))^{-1} \ln (\Xi (s)-z) \varphi (s)\, ds}_{=:I_{2,1} (t)}\\
&\qquad  -
\underbrace{\int e^{-m_0 |t-s|} \ln (\Xi (s) -z) \frac{d}{ds} ((\Xi')^{-1} \varphi) (s)\, ds}_{=: I_{2,2} (t)}
\end{align*}
and use the uniform bound for $\ln (\Xi (s)-z)$ in $L^1 ([a-2,a+2])$ to conclude that $|I_{2,1}|$ and $|I_{2,2}|$ are both bounded uniformly. As for $I_1$, note that, on any compact interval $K$ around $a'$, we have, since $\Xi'$ is continuous and $\Xi'<0$,
\begin{equation*}
    C(K):=\inf_{x\in K} |\Xi'(x)| = -\max_{x\in K} \Xi'(x) >0.
\end{equation*}
Therefore by the mean value theorem, for all $s\in K$, there exists a $\iota = \iota(s)\in K$ such that
\begin{align*}
    \abs{\Xi(s)-\Xi(a')-iy} &=\sqrt{y^2+(\Xi(s)-\Xi(a'))^2}\\ & >\abs{\Xi(s)-\Xi(a')}= \abs{s-a'}\abs{ \Xi'(\iota)}\ge \abs{s-a'} C(K).
\end{align*}
By the definition of the Hölder semi-norm, we thus have, for all $s\in K$,
\[
\left|\frac{\psi (s)- \psi (a')}{\Xi (s) - \Xi (a') - iy}\right| \leq  \frac{\|\psi\|_{C^\sigma}}{C(K)|s-a'|^{1-\sigma}},
\]
which is integrable. Furthermore, outside of $K$ the integrand of $I_1$ is bounded and decays exponentially, therefore one can uniformly bound $I_1$.

We next wish to bound the seminorm
\[
[B^2_z (\psi)]_\sigma:= \sup_{t\neq t'} \frac{|B^2_z (\psi) (t) - B^2_z (\psi) (t')|}{|t-t'|^\sigma}\, .
\]
We write
\[
B^2_z (\psi) (t) - B^2_z (\psi) (t') = (I_1 (t) - I_1 (t')) + \psi' (a) (I_2 (t) - I_2 (t'))\, . 
\]
Using that $|e^{-m_0 |t-s|} - e^{-m_0 |t'-s|}|\leq C |t-t'|$ we can bound
\[
|I_1 (t) - I_1 (t')| \leq C |t-t'| \int |\varphi (s)| \frac{|\psi (s)-\psi (a')|}{|\Xi (s) - \Xi (a')|}\, ds
\leq C \|\psi\|_{C^\sigma} |t-t'|\, .
\]
Similarly we can write
\[
|I_{2,2} (t) - I_{2,2} (t')| \leq C |t-t'| \int \left|\ln (\Xi (s) -z) \frac{d}{ds} ((\Xi')^{-1} \varphi) (s)\right|\, ds \leq C |t-t'|\, .
\]
Next denoting the function $(\Xi' (s))^{-1} \varphi (s) \ln (\Xi (s) -z)$ by $B (s)$ we assume $t> t'$ and write further 
\begin{align*}
I_{2,1} (t) - I_{2,1} (t')&= m_0 \Big(\underbrace{\int_t^\infty e^{-m_0 (s-t)} B(s)\, ds - \int_{t'}^\infty e^{-m_0(s-t')} B(s)\, ds}_{=: J_+(t,t')}\Big)\\
&\quad - m_0 \Big(\underbrace{\int_{-\infty}^t e^{-m_0 (t-s)} B(s)\, ds - \int_{-\infty}^{t'} e^{-m_0 (t'-s)} B(s)\, ds}_{=: J_- (t,t')}\Big)\, . 
\end{align*}
Then we choose $p=\frac{1}{\sigma}$, let $p'$ be the dual exponent and estimate
\begin{align*}
|J_+ (t,t')| &\leq C |t-t'| \int_t^\infty |B(s)|\, ds + \int_{t'}^t |B (s)|\, ds\\
&\leq C |t-t'| \|B\|_{L^1} + |t-t'|^\sigma \|B\|_{L^{p'}}\, .
\end{align*}
A similar estimate for $J_- (t,t')$ finally shows the existence of a constant $C$ such that
\[
|B^2_z (\psi) (t) - B^2_z (\psi) (t')|\leq C \|\psi\|_{C^\sigma} \left(|t-t'|+|t-t'|^\sigma\right)\, .
\]
Clearly this implies
\[
|B^2_z (\psi) (t) - B^2_z (\psi) (t')|\leq C \|\psi\|_{C^\sigma} |t-t'|^\sigma \qquad \mbox{if $|t-t'|\leq 1$.}
\]
On the other hand we can trivially bound
\[
|B^2_z (\psi) (t) - B^2_z (\psi) (t')| \leq 2 \|B^2_z (\psi)\|_\infty \leq C \|\psi\|_{C^\sigma} |t-t'|^\sigma
\quad\mbox{if $|t-t'|\geq 1$.}
\]

\medskip

\textbf{ Step 2.} In this second step we compute
\begin{align*}
\langle T^{-1} \mathcal{R}_{m,z} (\psi_0), \psi_0\rangle &=
\langle T^{-1} \circ \mathcal{K}_{m_0} \left(A ((\Xi-z)^{-1} - (\Xi-z_0)^{-1})\psi_0\right), \psi_0\rangle\\
&\qquad
+ (2m_0 h + h^2) \langle T^{-1} \circ \mathcal{K}_{m_0} (\psi_0), \psi_0\rangle\, .
\end{align*}
Recalling \eqref{e:T-1K} (and using that $\mathcal{K}_{m_0}$ is self-adjoint) we rewrite the expression as
\begin{align}
\langle T^{-1} \mathcal{R}_{m,z} (\psi_0), \psi_0\rangle &=
(z-z_0) \langle T^{-1}\circ \mathcal{K}_{m_0} \big( A (\Xi-z)^{-1} (\Xi-z_0)^{-1} \psi_0\big), \psi_0\rangle\nonumber\\
&\qquad + 2m_a h + h^2\nonumber\\
&= (z-z_0) \langle A (\Xi-z)^{-1} (\Xi-z_0)^{-1} \psi_0, \mathcal{K}_{m_0}\circ (T^*)^{-1} (\psi_0)\rangle\nonumber\\
&\qquad + 2ma_ h + h^2\nonumber\\
&= (z-z_0) \underbrace{\langle A (\Xi-z)^{-1} (\Xi -z_0)^{-1} \psi_0, \psi_0 \rangle}_{=: G (z)} + 2m_a h + h^2\label{e:Taylor-3}\, .
\end{align}
We thus want to show that the following limit exists and to compute its imaginary part:
\begin{align}
\label{eq:limitexists}
- c (a) &:= \lim_{\textrm{Im}\, z >0, z\to \Xi (a)} G (z)\nonumber\\
&=
\lim_{\textrm{Im}\, z >0, z\to \Xi (a)} \int \frac{1}{\Xi (s)-z} |\psi_0 (s)|^2 \frac{A(s)}{\Xi (s) - \Xi (a)}\, ds\, . 
\end{align}
Observe indeed that inserting $G(z) = - c(a) + o (1)$ in \eqref{e:Taylor-3} and taking into account \eqref{e:Taylor-2} and \eqref{e:resto-1} we conclude that \eqref{e:expansion} holds.

In order to compute $c(a)$ we observe first an 
\[
\phi (s) := |\psi_0 (s)|^2 \frac{A(s)}{\Xi (s) - \Xi (a)}
\]
is smooth and decays exponentially. We thus rewrite
\[
G (z) = \int \frac{1}{\Xi (s)-z} \phi (s)\, ds\, .
\]
Next we decompose $z$ into its real and imaginary part as $z = x + iy$ and observe that 
\begin{align*}
\lim_{\textrm{Im}\, z >0, z\to \Xi (a)} \textrm{Re}\, G (z) &= \lim_{x\to \Xi(a), y \downarrow 0} \int \frac{\Xi (s)-x}{(\Xi (s)-x)^2 + y^2} \phi (s)\, ds 
\end{align*}
Here we are only interested in showing that the limit exists and we thus fix a cut-off function $\varphi\in C^\infty_c (]a-2, a+2[)$, identically $1$ on $[a-1, a+1]$ and split the integral into
\begin{align*}
\textrm{Re}\, G (z) &=  \int \frac{\Xi (s)-x}{(\Xi (s)-x)^2 + y^2} \phi (s) \varphi (s)\, ds\\
&\qquad +
\int \frac{\Xi (s)-x}{(\Xi (s)-x)^2 + y^2} \phi (s) (1-\varphi (s))\, ds\, .
\end{align*}
The second integral has a limit, while in order to show that the first has a limit we write 
\[
\frac{\Xi (s)-x}{(\Xi (s)-x)^2 + y^2} = \frac{1}{2\Xi' (s)} \frac{d}{ds} \ln ((\Xi(s)-x)^2 + y^2)\, .
\]
We then integrate by parts and use the fact that $\ln ((\Xi (s)-x)^2 +y^2)$ converges to $2 \ln |(\Xi(s)-\Xi (a)|$ strongly in $L^q ([a-2,a+2])$ for every $q$ to infer the existence of the limit of the first integral. 

As for the imaginary part we write instead
\begin{align}\label{e:arctan-integral}
\lim_{\textrm{Im}\, z >0, z\to \Xi (a)} \textrm{Im}\, G (z) &= \lim_{x\to \Xi(a), y \downarrow 0} \int \frac{y}{(\Xi (s)-x)^2 + y^2} \phi (s)\, ds\, .
\end{align}
We wish to show that the latter integral converges to 
\begin{equation}\label{e:arctan-integral-2}
I = \phi (a) \int \frac{ds}{(\Xi' (a))^2 s^2 +1} = \frac{\pi \phi (a)}{|\Xi' (a)|}\, .
\end{equation}
On the other hand $\phi (a) = |\psi_0 (a)|^2 A' (a) (\Xi' (a))^{-1}$. 
Since $A' (a) > 0$ and $\Xi' (a)<0$, we conclude that $c (a)$ exists and it is a complex number with positive imaginary part, which completes the proof of the lemma. 

It remains to show the convergence of \eqref{e:arctan-integral} to \eqref{e:arctan-integral-2}. First observe that for each $x$ sufficiently close to $\Xi (a)$ there is a unique $a' = \Xi^{-1} (x)$ such that $\Xi (a')=x$. Changing variables ($s$ becomes $a'+s$), the integral in \eqref{e:arctan-integral} becomes
\begin{equation}
\int \frac{y}{(\Xi (a'+s)-x)^2 + y^2} \phi (a'+s)\, ds\, 
\end{equation}
and we wish to show that its limit is $I$ as $(a',y)\to (a,0)$.
Next, fix any $\delta>0$ and observe that 
\[
\lim_{y\to 0} \int_{|s|\geq \delta} \frac{y}{(\Xi (a'+s)-x)^2 + y^2} \phi (a'+s)\, ds=0
\]
uniformly in $a' \in [a-1, a+1]$. We therefore define
\[
I (\delta, a', y) := \int_{-\delta}^\delta \frac{y}{(\Xi (a'+s)-x)^2 + y^2} \phi (a'+s)\, ds
\]
and we wish to show that, for every $\varepsilon >0$ there is a $\delta>0$ such that
\begin{equation}\label{e:arctan-integral-3}
\limsup_{(a',y) \downarrow (a,0)} \left| I (\delta, a', y) - I\right| \leq C \varepsilon\, ,
\end{equation}
where $C$ is a geometric constant.
We rewrite
\[
I (\delta, a', y) = \int_{-\delta y^{-1}}^{\delta y^{-1}} \frac{\phi (a'+ys)}{y^{-2} (\Xi (a' + ys) - \Xi (a'))^2 +1}\, ds\, .
\]
Fix now $\varepsilon$ and observe that, since $\Xi'$ and $\phi$ are continuous, if $\delta$ is chosen sufficiently small, then
\begin{align}
&((\Xi' (a))^2 - \varepsilon^2) s^2 \leq y^{-2} (\Xi (a' + ys) - \Xi (a'))^2 \leq ((\Xi' (a))^2 + \varepsilon^2) s^2\\
& |\phi (a' + ys) - \phi (a)| \leq \varepsilon\, .
\end{align}
for all $|a'-a|<\delta$ and $y |s| \leq \delta$. Choosing $\varepsilon>0$ so that $\varepsilon \leq \frac{|\Xi' (a)|}{2}$ we easily see that, when $|a'-a| < \delta$, we have 
\[
\left|I (\delta, a', y) - \phi (a) \int_{-\delta y^{-1}}^{\delta y^{-1}} \frac{ds}{(\Xi' (a))^2 s^2 +1}\right| \leq C \varepsilon\, .
\]
In particular, as $y\downarrow 0$, we conclude \eqref{e:arctan-integral-3}.
\end{proof}

\section{Proof of Proposition \ref{p:almost-final}}

We reduce the proof of Proposition \ref{p:almost-final} to the following lemma.

\begin{lemma}\label{l:almost-final-2}
Consider $G:= \{m> 1, m \neq m_a, m_b : \mathscr{U}_m \neq \emptyset\}$. Then $G$ is relatively open and relatively closed in $]1, \infty[\setminus \{m_a, m_b\}$.
\end{lemma}

Proposition \ref{p:almost-final} is an obvious consequence of the latter lemma and of Proposition \ref{p:5-7}: Lemma \ref{l:almost-final-2} implies that $G$ is the union of connected components of $[1, \infty[\setminus \{m_a, m_b\}$. On the other hand the connected component $]m_b, m_a[$ intersects $G$ because of Proposition \ref{p:5-7} and thus it is contained in $G$.
We thus complete the proof of Proposition \ref{p:almost-final} showing Lemma \ref{l:almost-final-2}
 
 \begin{proof}[Proof of Lemma \ref{l:almost-final-2}] We start with some preliminary considerations. Fix an interval $[c,d]\subset ]1, \infty[\setminus \{m_a, m_b\}$.
 Recalling Proposition \ref{p:all-m} we know that, since the operator norm of $\mathcal{L}_m$ is bounded uniformly in $m\in [c,d]$,
 \begin{itemize}
     \item[(a)] There is $R>0$ such that $\mathscr{U}_m\subset B_R (0)$ for all $m\in [c,d]$.
\end{itemize}
However it also follows from Proposition \ref{p:3+4} that 
\begin{itemize}
\item[(b)] There is a $\delta >0$ such that $\mathscr{U}_m\subset \{\textrm{Im}\, z > \delta\}$. 
\end{itemize}

\medskip

\textbf{ Step 1.} We first prove that $G$ is relatively closed. To that end we fix a sequence $m_j \to m\in ]1, \infty[\setminus \{m_a, m_b\}$ such that $m_j$ belongs to $G$. Without loss of generality we can assume $\{m_j\}\subset [c,d]\subset ]1, \infty[\setminus \{m_a, m_b\}$. For each $m_j$ we can then consider $z_j\in \mathscr{U}_{m_j}$, which by (a) and (b) we can assume to converge to some $z\in \mathbb C$ with positive imaginary part. We then let $\psi_j$ be a sequence of nontrivial elements in $L^2$ such that 
\begin{equation}\label{e:eigenvalue-equation-21}
-\psi_j'' + m_j^2 \psi_j + \frac{A}{\Xi -z_j} \psi_j = 0\, ,
\end{equation}
and normalize them to $\|\psi_j\|_{L^2}=1$
Since $\textrm{Im}\, z_j \geq \delta >0$, the sequence of functions $\frac{A}{\Xi-z_j}$ enjoy uniform bounds in the spaces $L^1$ and $C^k$. We can then argue as in Section \ref{s:3+4} to find that
\begin{itemize}
    \item[(i)] $\|\psi_j'\|_{L^2}$ enjoy a uniform bound;
    \item[(ii)] There are uniformly bounded nonzero constants $\{C^\pm_j\}$ with the property that $\psi_j$ is asymptotic to $C^\pm_j e^{\mp m_j t}$ and $\pm \infty$;
    \item[(iii)] There is a $T_0>0$ independent of $j$ with the property that
    \[
    |\psi_j (t) - C^\pm_j e^{\mp m_j t}| \leq \frac{|C^\pm_j|}{2} e^{\mp m_j t} \qquad \forall \pm t > T_0\, .
    \]
\end{itemize}
These three properties together imply that a subsequence, not relabeled, converges strongly in $L^2$ to some $\psi$. Passing into the limit in \eqref{e:eigenvalue-equation-21} we conclude that 
\[
-\psi'' + m^2 \psi + \frac{A}{\Xi-z} \psi = 0\, .
\]
This shows that $z\in \mathscr{U}_m$, i.e. that $m \in G$.

\medskip

\textbf{ Step 2.} Here we show that $G$ is relatively open. To that end we consider some sequence $m_j \to m \in ]1, \infty[\setminus \{m_a, m_b\}$ with the property that $m_j \not\in G$ and we show that $m\not \in G$. By (a) and (b) above, it suffices to show that the domain
\[
\Delta := \{|z|< R : \textrm{Im}\, z > \delta\}
\]
does not contain any element of $\textrm{spec}\, (\mathcal{L}_m)$. Observe first that, since we know that it does not intersect $\gamma = \partial \Delta$, the distance between $\gamma$ and any element in $\textrm{spec}\, (\mathcal{L}_m)$ is larger than a positive constant $\varepsilon$. Recalling that the spectrum on the upper half complex space is discrete, we have that
\[
P_m := \int_\gamma (\mathcal{L}_m -z)^{-1}\, dz 
\]
is a projection on a finite-dimensional space which contains all eigenspaces of the elements $z\in \textrm{spec}\, (\mathcal{L}_m)\cap \Delta = m \mathscr{U}_m = \{mz: z\in \mathscr{U}_m\}$. And since all such elements belong to the discrete spectrum, $\mathscr{U}_m = \emptyset$ if and only if $P_m = 0$. On the other hand 
\[
P_{m_j} := \int_\gamma (\mathcal{L}_{m_j} -z)^{-1}\, dz
\]
equals $0$ precisely because $m_j \not \in G$. We thus just need to show that $P_{m_j}$ converges to $P_m$ to infer that $m\in G$. The latter follows from the following observations:
\begin{itemize}
    \item[(i)] Since $\gamma$ is a compact set and does not intersect the spectrum of $\mathcal{L}_m$, there is a constant $M$ such that $\|(\mathcal{L}_m -z)^{-1}\|_O \leq M$ for all $z\in \gamma$;
    \item[(ii)] $\mathcal{L}_{m_j}$ converges to $\mathcal{L}_m$ in the operator norm;
    \item[(iii)] Writing 
    \begin{align*}
    &\qquad (\mathcal{L}_{m_j} - z)^{-1}\\ 
    &= (\textrm{Id} + (\mathcal{L}_m -z)^{-1} 
    (\mathcal{L}_{m_j} - \mathcal{L}_m))^{-1}( \mathcal{L}_m - z)^{-1}\, ,
    \end{align*}
    when $\|\mathcal{L}_{m_j} - \mathcal{L}_m\|_{\mathcal{L}} \leq \frac{1}{2M}$ we can use the Neumann series for the inverse to infer
    \[
    \sup_{z\in \gamma} \|(\mathcal{L}_{m_j} -z)^{-1} - (\mathcal{L}_m -z)^{-1}\|_O \leq C \|\mathcal{L}_m - \mathcal{L}_{m_j}\|_O\, ,
    \]
    for some constant $C$ independent of $j$.
\end{itemize}
We then conclude that $P_{m_j}$ converges to $P_m$ in the operator norm.
\end{proof}

\begin{remark}\label{rmk:algebraic dim const}
An immediate outcome of the argument above is that the sum of the \index{algebraic multiplicity@algebraic multiplicity} algebraic multiplicities of $z\in \mathscr{U}_m$, as eigenvalues of $m^{-1} \mathcal{L}_m$, is constant on any connected component of $]\! -\! \infty, \infty[\setminus \{m_a, m_b\}$. Indeed, it coincides with the rank of the operator $P_m$ defined in Step 2.
\end{remark}

\section{Proof of Lemma \ref{l:bottom}}\label{s:choice-of-A}

\index{aagO@$\Xi$}\index{aalA@$A$}
We start by observing that 
\[
\Xi' (t) := \int_{-\infty}^t e^{-2(t-\tau)} A(\tau)\, d\tau\, . 
\]
In fact, since $A = \Xi''+2\Xi$, the identity follows from the classical solution formula for first order ODEs with constant coefficients once we show that the right hand side and the left hand side coincide for all sufficiently negative $t$. The latter can be verified with a direct computation using that $\Xi' (t) = -2 c_0 2^{2t}$ for all $t$ sufficiently negative.

We next read the conditions for $\Xi\in \mathscr{C}$ in terms of $A$ to find that it suffices to impose
\begin{itemize}
    \item[(i)] $A (t) = - 8 c_0 e^{2t}$ for all $t$ sufficiently negative;
    \item[(ii)] $A(t) = -\al e^{-\al t}$ for all $t\geq \ln 2$;
    \item[(iii)] There are exactly two zeros $a<b$ of $A$ and $A' (a) >0$, $A' (b)<0$;
    \item[(iv)] $\int_{-\infty}^t e^{-2(t-\tau)} A(\tau) d\tau <0$ for every $t$.
\end{itemize}
Note indeed that, by imposing $\Xi (\infty) =0$, with a simple computation we can determine $\Xi (-\infty)$ as 
\[
\Xi (-\infty) = - \int_{-\infty}^\infty \Xi' (\tau)\, d\tau =
-\frac{1}{2} \int_{-\infty}^\infty A (\tau)\, d\tau\, 
\]
and moreover (using $\Xi'<0$, which follows from (iv)) we conclude that the conditions (i)-(iv) ensure 
\[
\int_{-\infty}^\infty A (\tau)\, d\tau < 0\, .
\]
We next fix $a=0$ and $b=\frac{1}{2}$ and rather than imposing (iv) we impose the two conditions
\begin{itemize}
    \item[(v)] $\int_{-\infty}^0 e^{2\tau} A (\tau) d\tau = -1$;
    \item[(vi)] $\max A \leq \frac{1}{e}$.
\end{itemize}
Observe, indeed, that (iv) is equivalent to 
\[
\int_{-\infty}^t e^{2\tau} A (\tau) \, d\tau < 0
\]
and that, since $A$ is negative on $]-\infty, 0[$ and $]\frac{1}{2}, \infty[$, the integral on the left-hand side is maximal for $t=\frac{1}{2}$. We then can use (v) and (vi) to estimate
\[
\int_{-\infty}^{\frac{1}{2}} e^{2\tau} A(\tau)\,d \tau \leq -1 + \frac{e}{2} \max A \leq -\frac{1}{2}\, .
\]
We next recall that, by the Rayleigh criterion, 
\[
- \lambda_a = \min_{\|\psi\|_{L^2} = 1} \langle \psi, L_a \psi\rangle = \min_{\|\psi\|_{L^2} = 1} \int \Big(|\psi'|^2 + \frac{A}{\Xi-\Xi (0)} |\psi|^2 \Big)\, . 
\]
We test the right-hand side with 
\[
\psi (t) :=
\left\{
\begin{array}{ll}
0 \qquad & \mbox{for $|t|\geq \frac{1}{2}$}\\ \\
\sqrt{2} \cos (\pi t) \qquad &\mbox{for $|t|\leq \frac{1}{2}$.}
\end{array}\right.
\]
We therefore get 
\begin{equation}\label{e:bottom-est-1}
- \lambda_a \leq 2 \pi^2 + 2 \int_{-1/2}^{1/2} \frac{A (t)}{\Xi (t) - \Xi (0)} \cos^2 \pi t\, dt\, . 
\end{equation}
Next, for any fixed positive constant $B>0$, we impose that $A (t) = B t$ on the interval $]- \sqrt{B}^{-1}, 0]$ and we then continue it smoothly on $[0, \infty[$ so to satisfy (ii), (iii), and (v) on $[0, \infty[$ (the verification that this is possible is rather simple). We also can continue it smoothly on $]\! -\!\infty, - \sqrt{B}^{-1}]$ so to ensure (i). In order to ensure (v) as well we just need to show that
\[
\int_{-\sqrt{B}^{-1}}^0 e^{2\tau} A(\tau)\, d\tau \geq -\frac{1}{2}\, .
\]
The latter is certainly ensured by
\[
\int_{- \sqrt{B}^{-1}}^0 e^{2\tau} A(\tau)\, d\tau \geq \int_{-\sqrt{B}^{-1}}^0 A(\tau)\, d\tau = -\frac{1}{2}\, .
\]
Now, observe that $\Xi' (0) = \int_{-\infty}^0 e^{2\tau} A (\tau) \, d\tau = -1$. For $t\in \, ]-\sqrt{B}^{-1}, 0[$ we wish to estimate $\Xi' (t)$ and to do it we compute
\begin{align*}
|\Xi' (t)- \Xi' (0)| & = \left|e^{-2t}\int_{-\infty}^t e^{2\tau} A (\tau)\, d\tau - \int_{-\infty}^0 e^{2\tau} A(\tau)\, d\tau\right|\\
&\leq e^{-2t} \left|\int_t^0 e^{2\tau} A(\tau)\, d\tau\right| + (e^{-2t}-1) \left|\int_{-\infty}^0 e^{2\tau} A(\tau)\, d\tau\right|\\
&\leq \frac{e^{2\sqrt{B}^{-1}}}{2} + (e^{2\sqrt{B}^{-1}}-1) \leq \frac{3}{4}\, ,
\end{align*}
which can be ensured by taking $B$ sufficiently large. In particular $-\frac{1}{4} \geq \Xi' (t) \geq -2 $ for $t\in \, ]-\sqrt{B}^{-1}, 0[$. We thus conclude that
\[
-2t \leq \Xi (t) - \Xi (0) \leq -\frac{t}{4} \qquad \forall t\in \, ]-\sqrt{B}^{-1}, 0[\, .
\]
In turn the latter can be used to show 
\[
\frac{A (t)}{\Xi (t) - \Xi(0)} \leq - \frac{B}{2} \qquad \forall t\in \, ]-\sqrt{B}^{-1}, 0[\, .
\]
Since $\frac{A}{\Xi - \Xi (0)}$ is otherwise negative on $]-\frac{1}{2}, \frac{1}{2}[$, we conclude
\begin{equation}\label{e:bottom-est-2}
- \lambda_a \leq 2\pi^2 - 2 \int_{-\sqrt{B}^{-1}}^0 B \cos^2 \pi t\, dt\, .
\end{equation}
By taking $B$ large enough we can ensure that $\cos^2 \pi t\geq \frac{1}{2}$ on the interval $]-\sqrt{B}^{-1}, 0[$. In particular we achieve
\[
- \lambda_a \leq 2\pi^2 - \sqrt{B}\, .
\]
Since we can choose $\sqrt{B}$ as large as we wish, the latter inequality completes the proof of the lemma.

\chapter{Nonlinear theory}\label{sect:Proof-main4}

This final chapter will prove Theorem~\ref{thm:main4} and hence complete the argument leading to Theorem~\ref{thm:main}. To that end we fix a choice of $\bar \Omega$, $\bar V$, $m$ and $\eta$ as given by Theorem~\ref{thm:spectral}, where $\bar a>0$ is a large parameter whose choice will be specified only later. We moreover fix a suitable subspace $X$ of $L^2_m$ and we will prove an estimate corresponding to \eqref{e:H2-estimate} in this smaller space.  

\begin{definition}\label{d:X}
We denote by $X$ the subspace of elements $\Omega\in L^2_m$ for which the following norm is finite:
\begin{equation}\label{e:X-norm}
\|\Omega\|_X:= \|\Omega\|_{L^2} + \||x| \nabla \Omega\|_{L^2} + \|\nabla \Omega\|_{L^4}\, .
\end{equation}
\end{definition}

The above norm has two features which will play a crucial role in our estimates. The first feature is that it ensures an appropriate decay of the $L^2$ norm of $D\Omega$ on the complements of large disks $\mathbb R^2\setminus B_R$. The second feature is that it allows to bound the $L^\infty$ norm of $\Omega$ and $\nabla (K_2 *\Omega)$ and to give a bound on the growth of $K_2*\Omega$ at infinity. More precisely, we have the following:

\begin{proposition}\label{p:X-bounds}\label{P:X-BOUNDS}
For all $\kappa \in ]0,1[$, there is a constant $C(\kappa) >0$ such that the following estimates hold for every $m$-fold symmetric $\Omega\in X$:
\begin{align}
|\nabla (K_2*\Omega) (x)|+|\Omega (x)| &\leq \frac{C(\kappa)}{1+|x|^{1-\kappa}}\|\Omega\|_X\qquad\forall x\in\R^2 \label{e:decay-Omega}\\
|K_2* \Omega (x)| &\leq C  \|\Omega\|_X \min\{ |x| , 1\} \qquad \forall x\in \mathbb R^2\, \label{e:Hoelder}.
\end{align}
\end{proposition}

The aim of this chapter is therefore to give the bound 
\begin{equation}\label{e:final-bound}
\|\Omega_{\text{per}, k} (\cdot, \tau)\|_{X} \leq e^{\tau (a_0+\delta_0)} \qquad\qquad\qquad \forall \tau\leq \tau_0\,  
\end{equation}
for some appropriately chosen constants $\delta_0>0$ and $\tau_0<0$, independent of $k$. Of course the main difficulty will be to give the explicit estimate \eqref{e:final-bound}. However a first point will be to show that the norm is indeed finite for every $\tau\geq -k$. This will be a consequence of the following:

\begin{lemma}\label{l:initial-bound}\label{L:INITIAL-BOUND}\label{l:pointwise}\label{L:POINTWISE}
Provided $a_0$ is large enough, the eigenfunction $\eta$ of Theorem \ref{thm:spectral} belongs to $C^2 (\mathbb R^2\setminus \{0\})$ and satisfies the pointwise estimates
\begin{equation}\label{e:eta-pointwise-decay}
|D^j \eta (x)|\leq C (1+ |x|)^{-m-2-j-\al} \qquad \forall x, \forall j\in \{0,1,2\}\, .
\end{equation}
In particular $\eta\in C^1 (\mathbb R^2)$ and its first derivatives are Lipschitz (namely, $\eta\in W^{2,\infty} (\mathbb R^2)$). Moreover $K_2*\eta \in L^2$.
\end{lemma}

One relevant outcome of Lemma 
\ref{l:initial-bound} is that the bound \eqref{e:final-bound} holds at least for $\tau$ sufficiently close to $-k$, given that $\Omega_{\text{per}, k} (\cdot, -k)\equiv 0$. The main point of \eqref{e:final-bound} is then that we will be able to deduce the following estimates.

\begin{lemma}\label{l:final-estimates}
Under the assumptions of Theorem \ref{thm:main4} there is a constant $C_0$ (independent of $k$) such that the following holds. Assume that $\bar\tau \leq 0$ is such that for all $\tau\in [-k, \bar\tau]$ we have the estimate
\begin{equation}\label{e:a-priori}
\|\Omega_{\text{per}, k} (\cdot, \tau)\|_X \leq e^{(a_0+\delta_0) \tau}.
\end{equation}
Then
\begin{align}
\|\Omega_{\text{per},k} (\cdot, \bar \tau)\|_{L^2} &\leq C_0 e^{(a_0+\delta_0+1/2)\bar\tau}\, ,
\label{e:stima-L2}\\
\||x| D\Omega_{\text{per}, k} (\cdot, \bar\tau)\|_{L^2} &\leq C_0 e^{(a_0+2\delta_0)\bar\tau}\, , \label{e:stima-H1-pesata}\\
\|D \Omega_{\text{per},k} (\cdot, \bar \tau)\|_{L^4} &\leq C_0 e^{(a_0+2\delta_0)\bar\tau}\label{e:stima-L4}\, .
\end{align}
\end{lemma}
With the above lemma we easily conclude \eqref{e:final-bound} (and hence Theorem \ref{thm:main4}). Indeed denote by $\tau_k$ the largest non-positive time such that \begin{equation}\label{e:assumed-for-the-moment}
\|\Omega_{\text{per}, k} (\cdot, \tau)\|_X \leq e^{(a_0+\delta_0) \tau} \qquad \forall \tau\in [-k, \tau_k]\, .
\end{equation}
Then we must have
\begin{equation}\label{e:forza}
\|\Omega_{\text{per}, k} (\cdot, \tau_k)\|_X = e^{(a_0+\delta_0) \tau_k}\, .
\end{equation}
On the other hand, summing the three estimates \eqref{e:stima-L2}, \eqref{e:stima-H1-pesata}, and \eqref{e:stima-L4} we conclude 
\begin{equation}\label{e:forza2}
\|\Omega_{\text{per}, k} (\cdot, \tau_k)\|_X \leq \bar C e^{(a_0+2 \delta_0) \tau_k}
\end{equation}
for some constant $\bar C$ independent of $k$. However \eqref{e:forza} and \eqref{e:forza2} give
$e^{\delta_0 \tau_k}\geq \bar C^{-1}$, i.e. $\tau_k \geq - \frac{1}{\delta_0} \ln \bar C$, implying that \eqref{e:final-bound} holds with $\tau_0:= - \frac{1}{\delta_0} \ln \bar C$.

After proving Proposition \ref{p:X-bounds} and Lemma \ref{l:initial-bound}, we will dedicate two separate sections to the three estimates \eqref{e:stima-L2}, \eqref{e:stima-H1-pesata}, and \eqref{e:stima-L4}. The first estimate, which we will call {\em baseline estimate}, will differ substantially from the other two, and to it we will dedicate a section. In order to accomplish the gain in the exponent in \eqref{e:stima-L2} we will use crucially the information on the semigroup which comes from Theorem \ref{thm:spectral} and Theorem \ref{t:group}, namely, that the growth bound $\omega (L_{ss})$ is precisely $a_0$ (i.e., the growth achieved by $\Omega_{\text{lin}}$). Since, however, the terms in the Duhamel's formula depend on derivatives, we need to invoke an a priori control of them, which is present in the norm $\|\cdot\|_X$. Indeed, one such term experiencing the derivative loss arise from the nonlinearity is the following:
\begin{equation}
    \int_{-k}^\tau e^{(\tau-s) L_\textrm{ss}} [(K_2 * \Omega_{\text{per}, k}) \cdot \nabla \Omega_{\text{per}, k}](\cdot,s) \, ds \, .
\end{equation}      
Note that $\|\cdot\|_X$ also includes the weighted $L^2$ norm $\||x| D\Omega\|_{L^2}$ because we encounter a term where $D\Omega_{\text{per}, k}$ is multiplied by the function $V^r$, which grows at $\infty$ like $|x|^{1-\al}$ when $\al \in ]0,2[$. In order to close the argument we then need to control the $L^4$ norm and the weighted $L^2$ norm of $D\Omega_{\text{per}, k}$. The latter estimates will not be accomplished through a growth bound on the semigroup $e^{\tau L_{ss}}$ (which would invoke controls on yet higher derivatives), but rather through some careful energy estimates. The structure of the problem will then enter crucially, since the term $(K_2 * \Omega_{\text{per}, k}) \cdot \nabla \bar{\Omega}$ which we need to bound in the energy estimates will take advantage of the improved baseline estimate on the $L^2$ norm. The above term,
which is responsible for the creation of unstable eigenvalues, actually \emph{gains a derivative}. Finally, there is one remaining difficulty when estimating $D \Omega_{\text{per}, k}$ due to the transport term. Namely, differentiating the equation in cartesian coordinates contributes a term $(\partial_i \bar{V}) \cdot \nabla \Omega_{\text{per} ,k}$, which could destabilize the estimates. We exploit the structure of the problem again in a crucial way by estimating angular and radial derivatives, rather than derivatives in cartesian coordinates, and by estimating the angular derivatives {\em first} and the radial derivatives {\em after}.


\section{Proof of Proposition \ref{p:X-bounds}}

We start by bounding $|\Omega (x)|$. Since $W^{1,4} (B_2)$ embeds in $C^{1/2}$, the bound $|\Omega (x)|\leq C \|\Omega\|_X$ is true for every $x\in B_2$. Consider further $R:= \frac{|x|}{2} \geq 1$ and define $u_R (y) := \Omega (Ry)$. In particular let $B:= B_1 (\frac{x}{R})$ and notice that
\begin{align}
\|u_R\|_{L^2 (B)} &= R^{-1} \|\Omega\|_{L^2 (B_R (x))} \leq R^{-1}\|\Omega\|_X\, ,\\
\|Du_R\|_{L^2 (B)} &= \|D\Omega\|_{L^2 (B_R (x))} \le  R^{-1} \||\cdot |D\Omega\|_{L^2 (B_R (x))} \leq R^{-1} \|\Omega\|_X\, ,\\
\|D u_R\|_{L^4 (B)} &= R^{1/2} \|D\Omega\|_{L^4 (B_R (x))} \leq R^{1/2} \|\Omega\|_X\, .
\end{align}
By interpolation, for $\frac{1}{p} = \frac{\lambda}{2} + \frac{1-\lambda}{4}$ we have
\[
\|Du_R\|_{L^p (B)} \leq C \|\Omega\|_X R^{-\lambda + (1-\lambda)/2}\, .
\]
Choosing $p$ very close to $2$, but larger, we achieve $\|Du_R\|_{L^p (B)} \leq C R^{-1+\kappa} \|\Omega\|_X$. Since the average of $u_R$ over $B$ is smaller than $\|u_R\|_{L^2 (B)} \leq C R^{-1} \|\Omega\|_X$, from Poincar\'e we conclude $\|u_R\|_{W^{1,p} (B)} \leq C R^{-1+\kappa} \|\Omega\|_X$. In particular using the embedding of $W^{1,p}$ in $L^\infty$ we conclude
\begin{equation}\label{e:Morrey}
\|u_R\|_{L^\infty (B)} \leq C (p) R^{-1+\kappa} \|\Omega\|_X\, .
\end{equation}
Since however $|\Omega (x)|\leq \|u_R\|_{L^\infty (B)}$, we reach the estimate
\begin{equation}\label{e:decay-Omega-2}
|\Omega (x)|\leq \frac{C}{|x|^{1-\kappa}} \|\Omega\|_X\, .
\end{equation}
Note that the constant $C$ depends on $\kappa$: $\kappa$ is positive, but a very small choice of it forces to choose a $p$ very close to $2$, which in turn gives a dependence of the constant $C(p)$ in \eqref{e:Morrey} on $\kappa$. 

We next come to the estimates for $\nabla K_2 * \Omega$. First of all, observe that 
\begin{align*}
\|\nabla K_2 * \Omega\|_{L^2} \leq & C \|\Omega\|_{L^2} \leq C \|\Omega\|_X\\
\|D^2 K_2*\Omega\|_{L^2} + \|D^2 K_2* \Omega\|_{L^4} \leq & C \|D\Omega\|_{L^2} + C \|D\Omega\|_{L^4}
\leq C \|\Omega\|_X
\end{align*}
and thus $|\nabla K_2*\Omega (x)|\leq C \|\Omega\|_X$ follows for every $x\in B_4$. Consider now $|x|\geq 4$, set $R:= \frac{|x|}{4}$ and let $\varphi\in C^\infty_c (B_{2R} (x))$ be a cut-off function identically equal to $1$ on $B_R (x)$ and $\psi\in C^\infty_c (B_{2R})$ equal to $1$ on $B_R$. We choose them so that $\|D^k \psi\|_{C^0} + \|D^k \varphi\|_{C^0} \leq C (k) R^{-k}$. 

We split 
\begin{equation*}
\nabla K_2 * \Omega = \nabla K_2 * (\varphi \Omega) + \nabla K_2 * (\psi \Omega) + \nabla K_2 * ((1-\varphi -\psi) \Omega)\, =: F_1 + F_2 + F_3\, .
\end{equation*}
We have 
\begin{align*}
\|F_1\|_{L^2} &\leq C\|\varphi \Omega\|_{L^2} \leq C \|\Omega\|_X\\
\|DF_1\|_{L^2} &\leq C\|D (\varphi \Omega)\|_{L^2} \leq C R^{-1} \|\Omega\|_{L^2} + C \|\varphi D\Omega\|_{L^2}
\leq C R^{-1} \|\Omega\|_X\\
\|D F_1\|_{L^4} & \leq C \|\Omega\|_X\, .
\end{align*}
The argument used above implies then $|F_1 (x)| \leq C (\kappa) |x|^{\kappa-1} \|\Omega\|_X$. As for estimating $F_2$ we observe that $F_2$ is harmonic outside $B_{2R}$. On the other hand $\|F_2\|_{L^2} \leq C \|\Omega\|_X$. Using the mean-value inequality for harmonic functions we then get 
\[
|F_2 (x)| \leq \frac{1}{\pi (2R)^2} \int_{B_{2R} (x)} |F_2| \leq \frac{C}{R} \|\Omega\|_X\, .
\]
As for $F_3$ we write, using the bound on $|\Omega|$ and $|\nabla K (x-y)|\leq C |x-y|^{-2}$,
\begin{align*}
|F_3 (x)| &\leq \int_{\mathbb R^2\setminus (B_{2R} (x) \cup B_{2R})} \frac{C(\kappa) \|\Omega\|_X}{|x-y|^2 |y|^{1-\kappa}}\\
&\leq \int_{(\mathbb R^2\setminus B_{2R})\cap \{|x-y|\geq |y|\}} \frac{C(\kappa) \|\Omega\|_X}{|y|^{3-\kappa}}
\\& \qquad 
+ \int_{(\mathbb R^2\setminus B_{2R} (x)) \cap \{|y|\geq |x-y|\}} \frac{C(\kappa) \|\Omega\|_X }{|x-y|^{3-\kappa}}
\\&
\leq \frac{C(\kappa) \|\Omega\|_X}{R^{1-\kappa}}\, .
\end{align*}

Recalling that $K_2*\Omega (0)=0$, integrating \eqref{e:decay-Omega} on the segment with endpoints $0$ and $x$ we conclude \eqref{e:Hoelder} for $|x|\leq 2$. In order to prove the bound when $|x|\geq 1$, fix a point $y$ with $3 R:= |y|\geq 1$. Let $\varphi$ be a radial cut-off function which is identically equal to $1$ on $B_{R} (0)$, is supported in $B_{2R} (0)$ and whose gradient is bounded by $C R^{-1}$. We then write 
\begin{equation}\label{e:decompose}
K_2 * \Omega = K_2 * (\varphi \Omega) + K_2 * ((1-\varphi) \Omega)\, .
\end{equation}
Since the distance between $y$ and the support of $\varphi \Omega$ is larger than $R$, we can estimate
\begin{equation}\label{e:cut-inside}
|K_2*(\varphi \Omega) (y)|\leq \frac{C}{R} \int |\varphi \Omega|
\leq C\|\Omega\|_{L^2}\leq C\|\Omega\|_X\, .
\end{equation}
Next observe that, by Calderon-Zygmund,
\begin{align*}
   \|D^2 K_2 * (\Omega (1-\varphi))\|_{L^2} 
   & = \|D ((1-\varphi) \Omega)\|_{L^2} \\&\leq \frac{C}{R} \|\Omega\|_{L^2} + C \|D\Omega\|_{L^2 (\mathbb R^2\setminus B_R)}
   \\& \leq \frac{C}{R} \|\Omega\|_X\, .
\end{align*}
 
Since $(1-\varphi) \Omega$ belongs to $L^2_m$, the average over $B$ of $K_2* ((1-\varphi) \Omega)$ equals $0$. Hence we conclude from Poincar\'e inequality and Calderon-Zygmund that
 \begin{align*}
 \|K_2 * ((1-\varphi) \Omega) \|_{L^{2}(B)} &\leq  CR\|DK_2 * ((1-\varphi) \Omega) \|_{L^{2}(B)} 
 \\&\leq CR\|((1-\varphi) \Omega) \|_{L^{2}(B)} \, .
 \\& 
 \leq CR \|\Omega\|_X\, .
 \end{align*}
 From Gagliardo-Nirenberg interpolation inequality applied on $B$ we have
 \begin{align*}
    \|K_2* ((1-\varphi) \Omega)\|_{L^\infty(B)}&\leq C\|D^2 K_2 * (\Omega (1-\varphi))\|_{L^2}^{1/2}\|K_2* ((1-\varphi) \Omega)\|^{1/2}_{L^{2}(B)} 
    \\
    &\qquad + \frac C R  \|K_2 * ((1-\varphi) \Omega) \|_{L^{2}(B)}
    \leq C \|\Omega\|_X.
 \end{align*}

\section{Proof of Lemma \ref{l:initial-bound}}

The proof of the estimates for $\eta$ and its derivatives can be found below.

As for the second conclusion of the Lemma, observe that, by going back to the solutions $\omega_{\varepsilon, k}$, it follows from the regularity and decay of the initial data in \eqref{e:Euler-later-times} (just proved in the above lemma) and the regularity and decay of the forcing term $f$ for positive times, that $\omega_{\varepsilon, k}\in C ([t_k, T], X)$ for every $T> t_k$. Given the explicit transformation from $\omega_{\varepsilon, k}$ to $\Omega_{\varepsilon, k}$ we conclude that $\Omega_{\varepsilon, k} \in C ([-k, T], X)$ for every $T> -k$. Since the same regularity is enjoyed by $\tilde{\Omega}$ on $[-k, T]$ (the latter is in fact smooth and compactly supported on $\mathbb R^2\times [-k,T]$) and by $\Omega_{\text{lin}}$, we infer that $\Omega_{\text{per}, k} = \Omega_{\varepsilon, k} - \tilde{\Omega}-\Omega_{\text{lin}}$ belongs to $C ([-k, T], X)$. 

\begin{proof}[Proof of the first part of Lemma \ref{l:initial-bound}]
Consider $\eta\in L^2_m$, for $m\geq 2$ as in Theorem \ref{thm:spectral} and write it as $\eta (\theta, r) = \vartheta (r) e^{ik m\theta}$ when $b_0\neq 0$ or $\eta (\theta, r) = \vartheta (r) e^{ik m \theta} +\bar\vartheta (r) e^{-ikm\theta}$ if $b_0=0$ (where $\bar\vartheta$ denotes the complex conjugate of $\vartheta$). 
In both cases $\vartheta (r) e^{ikm\theta}$ is an eigenfunction and through this property we will show that it satisfies the estimates of the Lemma. We can therefore assume without loss of generality that $\eta (x) = \vartheta (r) e^{ikm\theta}$. We will also see that the argument leads in fact to estimates \eqref{e:eta-pointwise-decay} with $km$ replacing $m$ and hence without loss of generality we assume
$k=1$. Furthermore an outcome of the argument below is that $\eta$ is smooth except possibly at the origin. 
Note moreover that after having shown the pointwise bounds \eqref{e:eta-pointwise-decay} for $\eta$ outside of the origin, the $W^{2,\infty}$ regularity of $\eta$ follows immediately: indeed $\eta$ and $D\eta$ can be continuously extended to the origin by setting them equal to $0$, hence showing that $\eta\in C^1 (B_1)$, while the uniform bound of $|D^2 \eta|$ in $B_1\setminus \{0\}$ easily implies that $D\eta$ is Lipschitz.

\medskip

\textbf{ Step 1. Exponential coordinates.} We recall that the distribution $K_2* \eta$ is well defined, according to Lemma \ref{l:extension}, and its action on a test function in $\mathscr{S}$ is given by integrating the scalar product of the test with a function $v\in W^{1,2}_{\text{loc}} (\mathbb R^2, \mathbb C^2)$, cf. the proof of Lemma \ref{l:extension} in the appendix. It follows from the argument given there that $R_{2\pi/m} v (R_{2\pi/m} x) = v (x)$. Given that $\textrm{div}\, V =0$, $-V^\perp$ is the gradient of a continuous function, which is determined by its value $\psi$ at the origin. $\psi$ inherits the symmetry and can thus be written as $\psi (\theta, r) = f (r) e^{im \theta}$. We thus conclude that
\begin{align}
\vartheta =&f'' +\frac{1}{r} f' -\frac{m^2}{r^2} f\label{e:Poisson}\\
v = &\frac{f'}{r} \frac{\partial}{\partial \theta} - \frac{im f}{r} \frac{\partial}{\partial r}\, .
\end{align}
Observe moreover that $v\in W^{1,2}_{\text{loc}}$ implies $f'', \frac{f'}{r}$, $\frac{f}{r^2}\in L^2_{\text{loc}}$. Therefore $f$ is determined by \eqref{e:Poisson} and the boundary conditions
\begin{itemize}
    \item[(a)] $f (0) =0$;
    \item[(b)] $\int_0^1 \frac{|f'(r)|^2}{r} \, dr < \infty$.
\end{itemize}
(observe that condition (b) implies $f'(0)=0$ when $f\in C^1$ and we can therefore interpret it as a surrogate for a boundary condition). We recall also that we have the estimate $\|D^2\psi\|_{L^2} \leq C \|\vartheta\|_{L^2} < \infty$ and owing to $\int_{B_R} D\psi = 0$, by Poincar\'e we achieve
\[
\|D\psi\|_{L^p (B_R)} \leq C (p) R^{2/p}\, .
\]
Using Morrey's embedding and the fact that $\psi (0)=0$ we conclude, in turn
\[
|\psi (x)|\leq C \|D\psi\|_{L^p (B_{|x|})} |x|^{1-2/p} \leq C |x|\, .
\]
In particular we conclude 
\begin{equation}\label{e:linear-bound-f}
|f(r)|\leq C r\, .
\end{equation}

The equation satisfied by $\eta$ can thus be written in terms of the function $f$ as
\begin{equation}\label{e:third-order}
\left(1+\frac{r}{\alpha}\frac{d}{dr} - z_0 - im\beta \zeta\right) \left(f''+\frac{f'}{r} - \frac{m^2}{r^2} f\right) - \frac{imf \beta g'}{r} = 0\, ,
\end{equation}
where $g$ is the smooth function such that $\bar\Omega (x) = g (|x|)$ (in particular $g$ is constant in a neighborhood of the origin, and equals $r^{-\alpha}$ for $r\geq 2$) and $\zeta$ is given by the formula \eqref{e:def-zeta}. We next set $s = \ln r$ and in particular
\begin{align}
\tilde{g} (s) &= g (e^s)\\
h (s) & = f (e^s)\\
\tilde\zeta (s) &= \zeta (e^s)\, .
\end{align}
Note that 
\begin{equation}\label{e:Gamma}
\vartheta (e^s) = e^{-2s} (h'' (s) - m^2 h (s)) =: e^{-2s} \Gamma (s)\, .   
\end{equation}
In these new coordinates we observe that the claim of the lemma corresponds then to showing that $\vartheta\in C^2_{\text{loc}} (\mathbb R)$ and 
\begin{align}
|\Gamma (s)| + |\Gamma' (s)| + |\Gamma'' (s)|&\leq C e^{4s} \qquad &\forall s\leq 0\label{e:decad-negativo}\, ,\\
|\Gamma (s)| + |\Gamma' (s)| + |\Gamma'' (s)|&\leq C e^{- (m+\al)s} \qquad &\forall s\geq 0\label{e:decad-positivo}\, .
\end{align}
In order to achieve the latter estimates we will need the following bounds on $\tilde{g}'$, $\tilde{g}''$, $\tilde\zeta$, and $\tilde{\zeta}'$, which we record here and can be easily checked:
\begin{align}
|\tilde{g} (s)| + |\tilde{g}' (s)| + |\tilde\zeta (s)| + |\tilde\zeta' (s)| &\leq C e^{-\al s}\qquad &\forall s \geq 0\, ,\label{e:stima-tilde-g-positivo}\\
|\tilde{g} (s) - g (0)| + |\tilde{g}' (s)| + |\tilde\zeta (s) - \zeta (0)| +
|\tilde\zeta' (s)| &\leq C e^{2s}\qquad &\forall s\leq 0\, .\label{e:stima-tilde-g-negativo}
\end{align}
We observe next that, by \eqref{e:linear-bound-f} 
\begin{equation}\label{e:the-very-first-crappy-bound}
|h(s)|\leq C e^s\, .
\end{equation}

\medskip

\textbf{ Step 2. The equation for $h$.} In these new coordinates the equation \eqref{e:third-order} becomes then
\begin{equation}
\left(1+\frac{1}{\alpha} \frac{d}{ds} -z_0 - im \beta \tilde{\zeta}\right) \left(e^{-2s} (h''-m^2 h)\right) - im h e^{-2s} \tilde{g}' = 0\, ,
\end{equation}
which we simplify as 
\begin{equation}
\left[ \frac{d}{ds} - \alpha \left(im\beta \tilde{\zeta} + z_0 - 1 +\frac{2}{\alpha}\right)\right] (h''-m^2 h) - i\alpha mh \tilde{g}' = 0\, .
\end{equation}
We then define the integrating factor 
\[
I (s) = \exp \left[- \alpha \int_0^s \left(im \beta \tilde{\zeta} + z_0 -1 +\frac{2}{\alpha}\right)\, d\sigma\right]
\]
We can thus write
\[
\frac{d}{ds} \left[ I (h'' -m^2 h)\right] = i \alpha m I h \tilde{g}'\, .
\]
Given that $z_0= a_0 + i b_0$, 
\begin{equation}\label{e:exact-identity-I}
|I (s)|\le C e^{-(2+\alpha (a_0-1)) s}
\end{equation}
and in particular, by \eqref{e:stima-tilde-g-positivo}
\begin{equation}\label{e:first-crappy-bound}
|I h \tilde{g}'| (s) \leq C e^{-(1+\alpha a_0 + {(\al-\alpha)}) s}\, .
\end{equation}
This implies that the latter is an integrable function on every halfline $[s, \infty[$ so that we can write
\begin{equation}\label{e:ODE-again}
\Gamma (s) = h'' (s) - m^2 h (s) = - \alpha I(s)^{-1} \int_s^\infty i m I h \tilde{g}' \, . 
\end{equation}
Since $\Gamma (s) = e^{2s} \vartheta (e^s)$ and $\vartheta \in L^2 (rdr)$, $e^{-s} \Gamma \in L^2 (ds)$. We claim in particular that $e^{-m|s|} \Gamma (s)$ is integrable. Indeed: 
\[
\int_{\mathbb R}|\Gamma (s)| e^{-m|s|}\, ds \leq \|e^{-s} \Gamma\|_{L^2 (\mathbb R)} \left(\int_{\mathbb R} e^{-2 (m|s|-s)}\, ds\right)^{1/2}< \infty\, .
\]
We claim then that for the function $h$ we have the formula
\begin{equation}\label{e:formulozza}
h (s) = -\frac{1}{2m} e^{-ms} \int_{-\infty}^s e^{ms'} \Gamma (s') ds' - \frac{1}{2m} e^{ms} \int_s^\infty e^{-ms'} \Gamma (s')\, ds'\, .
\end{equation}
In order to prove the identity denote the right hand side by $H$ and observe that it is a solution of the
same ODE satisfied by $H$, namely, $H''-m^2 H = \Gamma$. Hence $(\frac{d^2}{ds^2} -m^2) (H-h) =0$, which implies that $H (s)-h(s) = C_1 e^{ms} + C_2 e^{-ms}$. On the other hand, using the information that $e^{-s}\Gamma\in L^2$, it can be readily checked that $H (s) = o (e^{m|s|})$ at both $\pm \infty$. Since this property is shared by $h$, thanks to the bound \eqref{e:the-very-first-crappy-bound}, we conclude that $C_1 e^{ms} + C_2 e^{-ms}$ must be $o (e^{m|s|})$ at $\pm \infty$, implying $C_1=C_2=0$.

\medskip

\textbf{ Step 3. Estimates at $+\infty$.} In this step we give bounds for the asymptotic behavior of $h$ at $+\infty$.
We recall \eqref{e:stima-tilde-g-positivo} and hence observe that
\begin{equation}\label{e:stima-Gamma}
|\Gamma (s)| \leq C e^{(2 + \alpha a_0 - \alpha) s} \int_s^\infty |h(\sigma)| e^{-(2+\alpha a_0+{(\al-\alpha)}) \sigma}\, d\sigma\, ,
\end{equation}
for $s$ positive.
On the other hand, for $s\geq 0$ we can also write from \eqref{e:formulozza}
\begin{equation}\label{e:stima-h-positiva}
|h(s)| \leq C e^{-ms} + C e^{-ms} \int_0^s e^{m\sigma} |\Gamma (\sigma)|\, d\sigma + C e^{ms} \int_s^\infty e^{-m\sigma} |\Gamma (\sigma)|\, d\sigma\, ,
\end{equation}
Starting with the information $|h(s)|\leq C e^s$ for $s>0$, we then infer from \eqref{e:stima-Gamma} that $|\Gamma (s)|\leq C e^{(1-\al) s}$ for $s>0$. In turn plugging the latter into \eqref{e:stima-h-positiva} we infer $|h (s)|\leq C e^{(1-\al) s}$ for $s>0$. The latter, plugged into \eqref{e:stima-Gamma} turns into $|\Gamma (s)|\leq C e^{(1-2\al) s}$ for $s>0$. We then can keep iterating this procedure. The bootstrap argument can be repeated until we reach the largest integer $k$ such that $(1-k\al) > -m$: one last iteration of the argument gives then 
\begin{equation}\label{e:bound-finale-h-positivo}
|h(s)|\leq C e^{-ms}
\end{equation}
and hence, inserting one last time in \eqref{e:stima-Gamma} 
\begin{equation}\label{e:decad-positivo-1}
|\Gamma (s)|\leq C e^{-(m+\al) s}\, . 
\end{equation}
In order to estimate the first and second derivatives of $\Gamma$ we observe that
\[
\frac{I'}{I} = -\alpha (im \beta \tilde{\zeta} + z_0-1+2\alpha)
\]
and we compute explicitly
\begin{align}
\Gamma' &= \alpha (im \beta \tilde\zeta + z_0-1+2\alpha) \Gamma + i m h \tilde{g}'\label{e:Gamma'}\\
\Gamma'' &=\alpha im \beta \tilde\zeta' \Gamma + \alpha (im \beta \tilde\zeta + z_0-1+2\alpha) \Gamma'
+ i m h \tilde{g}'' + i m h' \tilde{g}'\, .\label{e:Gamma''}
\end{align}
From \eqref{e:Gamma'} and the bounds \eqref{e:decad-positivo-1},\eqref{e:bound-finale-h-positivo}, and \eqref{e:stima-tilde-g-positivo}, we immediately conclude
\begin{equation}\label{e:decad-positivo-2}
|\Gamma' (s)|\leq C e^{-(m+\al)s}\, .
\end{equation}
As for the second derivative, using \eqref{e:decad-positivo-2}, \eqref{e:decad-positivo-1},\eqref{e:bound-finale-h-positivo}, and \eqref{e:stima-tilde-g-positivo}, we conclude
\[
|\alpha im \beta \tilde\zeta' \Gamma + \alpha (im \beta \tilde\zeta + z_0-1+2\alpha) \Gamma'
+ i m h \tilde{g}''|\leq C e^{-(m+\al)s}\, .
\]
In order to estimate the term $i m h' \tilde{g}'$ we differentiate \eqref{e:formulozza} to infer
\begin{equation}\label{e:h'}
h' (s) = \frac{1}{2} e^{-ms} \int_{-\infty}^s e^{ms'} \Gamma (s') ds' - \frac{1}{2} e^{ms} \int_s^\infty e^{-ms'} \Gamma (s')\, ds'\, . 
\end{equation}
We then derive $|h'(s)|\leq C e^{-ms}$ using the same argument for bounding $h$. In turn, combined again with \eqref{e:stima-tilde-g-positivo} we conclude $|i m h' \tilde{g}' (s)|\leq C e^{-(m+\al)s}$, hence completing the proof of \eqref{e:decad-positivo}.

\medskip

\textbf{ Step 4. Estimates at $-\infty$.}
For the bound at $-\infty$ we use instead \eqref{e:stima-tilde-g-negativo} (which we observe holds for positive $s$ as well). This leads to the inequality
\begin{equation}\label{e:bootstrap-negative-1}
|\Gamma (s)|\leq C e^{(2+ \alpha (a_0-1))s} \int_s^{\infty} e^{-\alpha (a_0-1) \sigma} |h(\sigma)|\, d\sigma
\end{equation}
In this argument we assume that $a_0$ is selected very large, depending on $m$.
In turn we estimate $h$ for negative $s$ by
\begin{equation}\label{e:bootstrap-negative-2}
|h (s)|\leq C e^{ms} + C e^{ms} \int_s^0 e^{-m\sigma} |\Gamma (\sigma)|\, d\sigma
+ C e^{-ms} \int_{-\infty}^s e^{m\sigma} |\Gamma (\sigma)|\, d\sigma\, .
\end{equation}
Observe now that we have $|h(s)|\leq C e^s$ for every $s$. Inserting this bound in \eqref{e:bootstrap-negative-1} and assuming that $a_0$ is large enough we conclude $|\Gamma (s)|\leq C e^{3s}$. In turn we can insert the latter bound in \eqref{e:bootstrap-negative-2} to conclude
$|h (s)|\leq C (e^{ms} + e^{3s})$. Since $m\geq 2$ we can then conclude $|h(s)|\leq C e^{2s}$ and inserting it in \eqref{e:bootstrap-negative-1} we conclude $|\Gamma (s)|\leq C e^{4s}$. 

For the first and second derivatives we use the formulae \eqref{e:Gamma'}, \eqref{e:Gamma'}, and \eqref{e:h'} and argue as above to conclude $|\Gamma' (s)| + |\Gamma'' (s)|\leq C e^{4s}$.

\medskip

\textbf{Step 5. Estimate on $v = K_2* \eta$.} Recall that we can explicitly compute $K_2* \eta = \nabla^\perp \psi$ and that we already argued that $v\in W^{1,2}_{loc}$. Hence to complete the proof we need to show
\begin{equation}\label{e:K2*eta-L2}
\int_{\{|x|\geq 1\}} |\nabla \psi|^2\, dx = \int_1^\infty \left((f' (r))^2 + \frac{(f(r))^2}{r^2}\right)\, r\, dr < \infty\, .
\end{equation}
Recall next that $f(r) = h (\ln r)$ and $f'(r) = \frac{1}{r} h' (\ln r)$. Thus the exponential decays for $h$ proved in the previous steps imply $f (r) \leq C (1+ r)^{-m}$ and $f' (r) \leq C (1+ r)^{-m-1}$. Since $m\geq 2$, this clearly implies \eqref{e:K2*eta-L2} and completes the proof. 
\end{proof}

\section{Proof of the baseline \texorpdfstring{$L^2$}{L2} estimate}

In this section we prove \eqref{e:stima-L2}. In order to simplify the notation, from now on we will use $\Omega$ in place $\Omega_{\text{per}, k}$. We recall next equation \eqref{e:master}
\begin{align}
 (\partial_{\tau} - L_{\text{ss}}) \Omega
= &- \underbrace{(V_{\text{lin}}\cdot \nabla) \Omega}_{=:\mathscr{F}_1} - \underbrace{(V_r\cdot \nabla) \Omega}_{=:\mathscr{F}_2}- \underbrace{(V \cdot \nabla) \Omega_{\text{lin}}}_{=:\mathscr{F}_3} + \underbrace{(V\cdot \nabla) \Omega_r}_{=:\mathscr{F}_4} + \underbrace{(V \cdot \nabla) \Omega}_{=:\mathscr{F}_5}\nonumber\\
 &-\underbrace{(V_{\text{lin}}\cdot \nabla) \Omega_{\text{lin}}}_{=:\mathscr{F}_6} - \underbrace{(V_r\cdot \nabla) \Omega_{\text{lin}}}_{=:\mathscr{F}_7} - \underbrace{(V_{\text{lin}}\cdot \nabla) \Omega_r}_{=:\mathscr{F}_8}\,  .\label{e:master-2}
\end{align}
We then define $\mathscr{F} := - \sum_{i=1}^8 \mathscr{F}_i$. 
Recalling Theorem \ref{t:group} and the fact that $\Omega (\cdot, -k) =0$, we estimate via Duhamel's formula
\begin{equation}\label{e:Duhamel}
\|\Omega (\cdot, \bar\tau)\|_{L^2} \leq C (\varepsilon) \int_{-k}^{\bar\tau} e^{(a_0+\varepsilon) (\bar\tau - s)} \|\mathscr{F} (\cdot, s)\|_{L^2}\, ds\, .
\end{equation}
We next estimate the $L^2$ norms of the various $\mathscr{F}_i$. In order to keep our notation simpler we use $\|\cdot\|_2$ for $\|\cdot\|_{L^2}$ and $\|\cdot\|_\infty$ for $\|\cdot\|_{L^\infty}$. $\mathscr{F}_1$ is simple:
\begin{equation}\label{e:F-1}
\|\mathscr{F}_1 (\cdot, s)\|_2 \leq \|V_{\text{lin}} (\cdot, s)\|_\infty \|D\Omega (\cdot, s)\|_{L^2} \leq C e^{a_0 s} e^{(a_0+\delta_0) s} \leq C e^{(2a_0+\delta_0) s}\, .
\end{equation}
As for $\mathscr{F}_2$ we use the fact that 
\begin{align*}
\int |\mathscr{F}_2 (\xi, \tau)|^2\, d\xi & \leq C \int_{|\xi|\geq e^{-\tau/\alpha}} |\xi|^{2-{2\al}} |D \Omega (\xi, \tau)|^2\, d\xi\\
& \leq C e^{{\frac{2\al}{\alpha}}\tau} \int |\xi|^2 |D \Omega (\xi, \tau)|^2\, d\xi \leq C e^{{\frac{2\al}{\alpha}}\tau} \|\Omega (\cdot, \tau)\|^2_X\, . 
\end{align*}
We hence conclude
\begin{equation}\label{e:F2}
\|\mathscr{F}_2 (\cdot, \tau)\|_{L^2} 
\leq C e^{(a_0+{\frac{\al}{\alpha}}+\delta_0)\tau}
\leq C e^{(a_0+1+\delta_0)\tau}\, .
\end{equation}
As for $\mathscr{F}_3$, for every fixed $\tau$ with $\kappa=\frac{1}{2}$ we can use Proposition \ref{p:X-bounds} to conclude
\begin{align}\label{e:F3}
\|\mathscr{F}_3 (\cdot, \tau)\|_{L^2} &\leq \|V (\cdot, \tau)\|_{L^\infty} \|\nabla \Omega_{\text{lin}} (\cdot, \tau)\|_{L^2} 
\\&\leq C \|\Omega (\cdot, \tau)\|_X \|\nabla \Omega_{\text{lin}} (\cdot, \tau)\|_{L^2}
\\& \leq C e^{(2a_0 +{\delta_0}) \tau}\, .
\end{align}
To estimate $\mathscr{F}_4$ we recall that
\begin{equation}
\Omega_r (\xi, \tau) = \beta (1-\chi (e^{\tau/\alpha} (\xi)) \bar\Omega (\xi) + e^{\tau/\alpha} (\beta \zeta (|\xi|) |\xi|) \chi' (e^{\tau/\alpha} \xi)\, .   
\end{equation}
Differentiating the latter identity we get:
\begin{align}
|\nabla \Omega_r (\xi, \tau)| \leq & C \mathbf{1}_{|\xi|\geq e^{-\tau/\alpha}} |D \bar \Omega| (\xi)
\nonumber \\&
+  C e^{\tau/\alpha} (|\bar\Omega| (\xi) + | D (\zeta (|\xi|) |\xi|)) \mathbf{1}_{e^{-\tau/\alpha} R\geq |\xi|\geq e^{-\tau/\alpha}}
\nonumber \\&
+ C e^{2\tau/\alpha} (|\zeta (\xi)||\xi|) \mathbf{1}_{e^{-\tau/\alpha} R \geq |\xi|\geq e^{-\tau/\alpha}}\nonumber\\
\leq & C \mathbf{1}_{|\xi|\geq e^{-\tau/\alpha}} |\xi|^{-1-\al} 
\nonumber \\&+ C (e^{\tau/\alpha} |\xi|^{-\al} + e^{2\tau/\alpha} |\xi|^{1-\al}) \mathbf{1}_{e^{-\tau/\alpha} R \geq |\xi|\geq e^{-\tau/\alpha}}\label{e:sfava}
\end{align}
where we are assuming that $\textrm{spt}\, (\chi) \subset B_R$. We next use Proposition \ref{p:X-bounds} with $\kappa = \alpha/2$ to get $\|V (\cdot, \tau)\|_{L^\infty} \leq C \|\Omega (\cdot, \tau)\|_X\leq C e^{(a_0+\delta_0)\tau}$.
In particular we can estimate
\begin{align*}
\int |\mathscr{F}_4 (\xi, \tau)|^2 d\xi  \leq & C e^{2(a_0+\delta_0) \tau} \int_{e^{-\tau/\alpha}}^\infty r^{-1-{ 2\al}} d r
\\& + C e^{2(a_0+\delta_0 + 1/\alpha)\tau} \int_{e^{-\tau/\alpha}}^{e^{-\tau/\alpha} R} r^{1 - { 2\al} }\, dr\\
& + C e^{2(a_0+\delta_0
+ 2/\alpha)\tau} \int_{e^{-\tau/\alpha}}^{e^{-\tau/\alpha} R} r^{3 -{ 2\al}}\, dr \leq C e^{(2 a_0+ 2\delta_0 + {2})\tau}\, .
\end{align*}
We thus conclude
\begin{equation}\label{e:F-4}
\|\mathscr{F}_4 (\cdot, \tau)\|_2 \leq C e^{(a_0+\delta_0+1) \tau}\, .
\end{equation}
For $\mathscr{F}_5$ we use again $\|V (\cdot, \tau)\|_{L^\infty} \leq C e^{(a_0+\delta_0) \tau}$ to get
\begin{equation}\label{e:F-5}
\|\mathscr{F}_5 (\cdot, \tau)\|_2 \leq C e^{(a_0+\delta_0)\tau} \| D\Omega (\cdot, \tau)\|_2 \leq C e^{2(a_0+\delta_0)\tau}\, .    
\end{equation}
$\mathscr{F}_6$ follows easily from Lemma \ref{l:initial-bound}
\begin{equation}\label{e:F-6}
\|\mathscr{F}_6 (\cdot, \tau)\|_2 \leq \|\Omega_{\text{lin}} (\cdot, \tau)\|_\infty \|\Omega_{\text{lin}} (\cdot, \tau)\|_2 
\leq C e^{2(a_0+\delta_0) \tau}\, .
\end{equation}
$\mathscr{F}_7$ and $\mathscr{F}_8$ can be easily estimated using the explicit formula for $\Omega_r$ and the decay estimates given by Lemma \ref{l:pointwise} for $\Omega_{\text{lin}}$, in particular they enjoy an estimate which is better than \eqref{e:F-4}, i.e.
\begin{equation}\label{e:F-7+8}
\|\mathscr{F}_7 (\cdot, \tau)\|_2 + \|\mathscr{F}_8 (\cdot, \tau)\|_2 \leq C e^{(a_0+\delta_0+1/2) \tau}\, .
\end{equation}
Assuming that $a_0$ is sufficiently large we then achieve the estimate
\begin{equation}\label{e:F-all-together}
\|\mathscr{F} (\cdot, \tau)\|_2 \leq C e^{(a_0+\delta_0+1/2) \tau}\, .
\end{equation}
Inserting in \eqref{e:Duhamel} we choose $\varepsilon < 1/2 +\delta_0$ and we then achieve 
\begin{equation}\label{e:baseline-bar-tau}
\|\Omega (\cdot, \bar\tau)\|_2 \leq C e^{(a_0+\varepsilon) \bar \tau} \int_{-k}^{\bar \tau} e^{(\delta_0+1/2 -\varepsilon) s}\, ds
\leq C e^{(a_0+\delta_0 + 1/2)\bar \tau}\, .
\end{equation}
In fact, observe that the argument just given implies the stronger conclusion
\begin{equation}\label{e:stronger-baseline}
\|\Omega (\cdot, \tau)\|_2 \leq C e^{(a_0+\delta_0 + 1/2) \tau}\, , \qquad \forall \tau \in [-k, \bar \tau]\, .
\end{equation}

\section{Estimates on the first derivative}

In this section we prove \eqref{e:stima-H1-pesata} and \eqref{e:stima-L4}. 
The proof will be achieved via $L^2$ and $L^4$ energy estimates, where we will differentiate \eqref{e:master} first with respect to the angular variable and then with respect to the radial variable. We start rewriting \eqref{e:master} as 
\begin{align}
& \partial_\tau \Omega - \Omega + \left(\left(-\frac{\xi}{\alpha} + \beta \bar V + V_r + V_{\text{lin}} + V \right)\cdot \nabla\right) \Omega\nonumber\\ 
& = - \beta (V \cdot \nabla) \bar \Omega - (V\cdot \nabla)\Omega_{\text{lin}} - (V\cdot \nabla) \Omega_r - (V_{\text{lin}}\cdot \nabla) \Omega_{\text{lin}} 
\nonumber\\
& \qquad
- (V_r \cdot \nabla) \Omega_{\text{lin}} - (V_{\text{lin}}\cdot \nabla) \Omega_r
\nonumber \\&
= : \mathscr{G}\, .\label{e:master-10}
\end{align}
We next differentiate in $\theta$. In order to simplify our notation we will write $\theta$ in the subscript (or eventually $,\theta$ if there is already another subscript). We also recall that $\Omega_r$, $\bar\Omega$ are radial functions, while $(V_r \cdot \nabla)$ and $(\bar V \cdot \nabla)$ are angular derivatives times a function which is radial, and $\xi\cdot \nabla$ is a radial derivative times a radial function. So we can write
\begin{align}
& \partial_\tau \Omega_\theta - \Omega_\theta + \left(\left(-\frac{\xi}{\alpha} + \beta \bar V + V_r + V_{\text{lin}} + V \right)\cdot \nabla\right) \Omega_\theta\nonumber\\ 
=  &\mathscr{G}_\theta - (V_{\text{lin}, \theta}\cdot \nabla) \Omega - (V_\theta \cdot \nabla) \Omega =: \mathscr{H}_1\, \label{e:master-angular-1}
\\
& \partial_\tau \frac{\Omega_\theta}{r} - \left(\frac{1}{\alpha} + 1\right) \frac{\Omega_\theta}{r} + \left(\left(-\frac{\xi}{\alpha} + \beta \bar V + V_r + V_{\text{lin}} + V \right)\cdot \nabla\right) \frac{\Omega_\theta}{r}\nonumber\\
= &\frac{1}{r} \mathscr{G}_\theta - \frac{1}{r} (V_{\text{lin}, \theta}\cdot \nabla) \Omega - \frac{1}{r} (V_\theta \cdot \nabla) \Omega
+ \Omega_\theta ((V_{\text{lin}} + V) \cdot \nabla)\frac{1}{r}
= :\mathscr{H}_2\, . \label{e:master-angular-2}
\end{align}
We then multiply by $\Omega_\theta$ the first equation and integrate by parts the terms on the left-hand side to conclude
\begin{align}
\frac{d}{d\tau} \frac{1}{2} \|\Omega_\theta (\cdot, \tau)\|_2^2 &= \left(1 -\frac{1}{\alpha}\right) \|\Omega_\theta (\cdot, \tau)\|_2^2 
+ \int \mathscr{H}_1 (\xi, \tau) \Omega_\theta (\xi, \tau)\nonumber\\
& \leq \|\mathscr{H}_1 (\cdot, \tau)\|_2 \|\Omega_\theta (\cdot, \tau)\|_2\label{e:first-energy-est}
\end{align}
Likewise we multiply the second identity by $(\frac{1}{r} \Omega_\theta)^3$ and integrate by parts to achieve
\begin{align}
\frac{d}{d\tau} \frac{1}{4} \|r^{-1} \Omega_\theta (\cdot, \tau)\|_4^4 &= \left(1 + \frac{1}{2\alpha}\right) \|r^{-1} \Omega_\theta (\cdot, \tau)\|_4^4 
+ \int \mathscr{H}_2 (\xi, \tau) (r^{-1} \Omega_\theta (\xi, \tau))^3\nonumber
\end{align}
which implies
\begin{equation}\label{e:second-energy-est}
    \|r^{-1} \Omega_\theta (\cdot, \tau)\|_4
    \le
    \int_{-k}^\tau e^{\left(1+\frac{1}{2\alpha}\right)(\tau - \hat \tau)} \|\mathscr{H}_2 (\cdot, \hat \tau)\|_4\, d\hat \tau\, \quad
    \forall \tau \in [-k, \bar \tau]\, .
\end{equation}
We next wish to estimate the two integrals in the right-hand sides of both equations. 

We summarize the relevant estimates in the Lemma \ref{l:ugly-lemma} below. Note that they imply
\begin{align}
\|\Omega_\theta (\cdot, \hat\tau)\|_2 &\leq C e^{(a_0+2 \delta_0)\hat \tau}\,  \qquad &\forall \hat \tau\in [-k, \bar \tau ]\, ,\\
\|r^{-1} \Omega_\theta (\cdot, \hat \tau)\|_4 &\leq C e^{(a_0+2 \delta_0)\hat \tau}\qquad &\forall \hat \tau\in [-k, \bar\tau]\, ,
\end{align}
provided $a_0$ is chosen big enough.

\begin{lemma}\label{l:ugly-lemma}
Under the assumptions of Lemma \ref{l:final-estimates} we have
\begin{align}
\|D \mathscr{G} (\cdot, \tau)\|_{4} &\leq C e^{(a_0+2\delta_0) \tau}\label{e:DG}\\
\|r D\mathscr{G} (\cdot, \tau)\|_{2} &\leq C e^{(a_0+2\delta_0) \tau}\label{e:rDG}\\
\||D V_{\text{lin}}| |\nabla \Omega| (\cdot, \tau)\|_4 + \||D V| |\nabla \Omega| (\cdot, \tau) \|_4 &\leq C e^{(a_0+2 \delta_0) \tau}\label{e:DVDOmega}\\
\|r |D V_{\text{lin}}| |\nabla \Omega| (\cdot, \tau)\|_2 + \|r |D V| |\nabla \Omega| (\cdot, \tau) \|_2 &\leq C e^{(a_0+2 \delta_0) \tau}\label{e:rDVDOmega}\\
\|r^{-1}V_{\text{lin}} D \Omega (\cdot, \tau)\|_4 + \|r^{-1} V D\Omega (\cdot, \tau)\|_4 &\leq C e^{(a_0+2 \delta_0)\tau}\, .\label{e:comm-term}
\end{align}
\end{lemma}

\begin{proof}
\textbf{ Proof of \eqref{e:DG} and of \eqref{e:rDG}.} We break the terms as
\begin{align}
\|D\mathscr{G}\|_4 \leq & C \|DV D\bar\Omega\|_4 + C \|V D^2 \bar \Omega\|_4 
+ \|DVD\Omega_{\text{lin}}\|_4 
\nonumber\\&
+ \|VD^2\Omega_{\text{lin}}\|_4 
+ C \|DVD\Omega_r\|_4 
+ C \|VD^2\Omega_r\|_4 
\nonumber\\&
+ C \|DV_{\text{lin}}D\Omega_{\text{lin}}\|_4 
+ C \|V_{\text{lin}}D^2 \Omega_{\text{lin}}\|_4
+ \|D V_r D\Omega_{\text{lin}}\|_4 
\nonumber\\&
+ \|V_r D^2 \Omega_{\text{lin}}\|_4 + \|DV_{\text{lin}}D\Omega_r\|_4 + \|V_{\text{lin}} D^2 \Omega_r\|_4\,  ,
\end{align}
and
\begin{align}
\|rD\mathscr{G}\|_2 \leq & 
C \|rDV D\bar\Omega\|_2 
+ C \|rV D^2 \bar \Omega\|_2 
+ \|rDVD\Omega_{\text{lin}}\|_2 
\nonumber\\&
+ \|rVD^2\Omega_{\text{lin}}\|_2
+ C \|rDVD\Omega_r\|_2 
+ C \|VrD^2\Omega_r\|_2 
\nonumber\\&
+ C \|rDV_{\text{lin}}D\Omega_{\text{lin}}\|_2 
+ C \|rV_{\text{lin}}D^2 \Omega_{\text{lin}}\|_2
+ \|rD V_r D\Omega_{\text{lin}}\|_2 
\nonumber\\&
+ \|rV_r D^2 \Omega_{\text{lin}}\|_2 + \|rDV_{\text{lin}}D\Omega_r\|_2 + \|rV_{\text{lin}} D^2 \Omega_r\|_2\, . 
\end{align}
The terms involving $\Omega$ and $\bar\Omega$ is where we use the baseline $L^2$ estimate. Observe that
\begin{equation}\label{e:interpolating-baseline}
\|\Omega (\cdot, \tau)\|_{4} \leq \|\Omega (\cdot, \tau)\|_{\infty}^{1/2} \|D\Omega (\cdot, \tau)\|^{1/2}_{2}
\leq C e^{(a_0 + \delta_0+1/4)\tau}\, 
\end{equation}
and, by Calder{\'o}n-Zygmund,
\begin{equation}\label{e:CZ-baseline}
\|D K_2* \Omega (\cdot, \tau)\|_{4} \leq C \|\Omega (\cdot, \tau)\|_{4} \leq C e^{(a_0+\delta_0+1/4) \tau}\, .
\end{equation}
Next we estimate
\begin{align}
\|DV D\bar\Omega (\cdot, \tau)\|_4 & \leq \|D \bar \Omega (\cdot, \tau)\|_\infty \|DV (\cdot, \tau)\|_4 
\nonumber\\&
\leq C \|\Omega (\cdot, \tau)\|_4
\nonumber\\&
\leq C e^{(a_0+\delta_0+1/4) \tau}\, , \label{e:per-bar-1}\\
\|r DV D\bar\Omega (\cdot, \tau)\|_{L^2} &\leq 
\|r D \bar \Omega (\cdot, \tau)\|_\infty \|DV (\cdot, \tau)\|_2 
\nonumber\\&
\leq C \|\Omega (\cdot, \tau)\|_{L^2}
\nonumber \\&
\leq C e^{(a_0+\delta_0+1/2)\,  \tau}\, .\label{e:per-bar-1-weight}
\end{align}
Next, recalling Lemma \ref{l:extension} we get 
\begin{equation}
\|V (\cdot , \tau)\|_{L^2 (B_R)} \leq C R \|\Omega (\cdot, \tau)\|_2 \leq C R e^{(a_0+\delta_0+1/2) \tau}\, .    
\end{equation}
However, using $\int_{B_R} V (\cdot, \tau) =0$, we can in fact estimate also
\begin{equation}
\|V (\cdot, \tau)\|_{L^4 (B_R)} \leq C R^{1/2} \|\Omega (\cdot, \tau)\|_2 \leq C R^{1/2} e^{(a_0+\delta_0+1/2)\tau}  \, . 
\end{equation}
In particular we can infer
\begin{equation}
\| (1+|\cdot|)^{-(1+\varepsilon)} V (\cdot, \tau)\|_2 \leq C(\varepsilon) e^{(a_0+\delta_0+1/2) \tau} 
\end{equation}
and 
\begin{equation}
\| (1+|\cdot|)^{-(1/2+\varepsilon)} V (\cdot, \tau)\|_4 \leq C (\varepsilon) e^{(a_0+\delta_0+1/2) \tau}    
\end{equation}
for every positive $\varepsilon$. On the other hand, given that $|D^2 \bar \Omega (\xi)|\leq C (1+|\xi|)^{-2-\al}$, we easily infer
\begin{align}
\|V D^2 \bar \Omega (\cdot, \tau)\|_2 &\leq C \|(1+|\cdot|)^{-{1}} V (\cdot, \tau)\|_4 \leq C e^{(a_0+\delta_0+1/2) \tau}\label{e:per-bar-2}\\
\|r V D^2 \bar \Omega (\cdot, \tau)\|_2 &\leq C \|(1+|\cdot|)^{-1-\al} V (\cdot, \tau)\|_2 \leq C e^{(a_0+\delta_0+1/2) \tau}\, .\label{e:per-bar-2-r}
\end{align}
From now on we will not handle the terms with the weight $r$ as the proof is entirely analogous: we will just focus on the $L^4$ estimates and leave to the reader the computations with the weight.

For the two quadratic terms in $V_{\text{lin}}$ we can use Lemma \ref{l:pointwise} to achieve
\begin{equation}\label{e:lin-lin}
\|DV_{\text{lin}}D\Omega_{\text{lin}} (\cdot, \tau)\|_4 + \|V_{\text{lin}}D^2 \Omega_{\text{lin}} (\cdot, \tau)\|_4 \leq C e^{2a_0 \tau}\, .
\end{equation}
Likewise we can estimate
\begin{equation}\label{e:lin-per}
\|DV D \Omega_{\text{lin}} (\cdot, \tau)\|_4 + \|V D^2\Omega_{\text{lin}}\|_4 \leq C e^{a_0\tau}\|\Omega (\cdot, \tau)\|_X \leq C e^{(2a_0+\delta_0) \tau}\, ,
\end{equation}
(where for the second term we use the decay at infinity of $D^2 \Omega_{\text{lin}}$ to compensate the moderate growth of $V$, the argument is the same as for \eqref{e:per-bar-2} and we do not repeat it here). 
Observe next that, by \eqref{e:sfava}, $\|D \Omega_r (\cdot, \tau)\|_\infty \leq C$ for $\tau\leq 0$. Hence the term $DV D\Omega_r$ can be estimated as in \eqref{e:per-bar-1}:
\begin{equation}\label{e:per-err-1}
\|DV D\Omega_r (\cdot, \tau)\|_4 \leq C \|DV (\cdot, \tau)\|_4 \leq C e^{(a_0+\delta_0+1/4) \tau}\, .
\end{equation}
As for the other term, differentiating once more and arguing as for \eqref{e:sfava} we get:
\begin{align}
|D^2 \Omega_r (\xi, \tau)| 
\leq & C |\xi|^{-2-\al}\mathbf{1}_{|\xi|\geq e^{-\tau/\alpha}}
\nonumber\\&
 + (e^{\tau/\alpha} |\xi|^{-1-\al} + e^{2\tau/\alpha} |\xi|^{-\al} + e^{3\tau/\alpha} |\xi|^{1-\al}) \mathbf{1}_{e^{-\tau/\alpha} R \geq |\xi|\geq e^{-\tau/\alpha}}\nonumber\\
\leq &  C |\xi|^{-2-\al} \mathbf{1}_{|\xi|\geq e^{-\tau/\alpha}} + C e^{3\tau/\alpha} |\xi|^{1-\al} \mathbf{1}_{e^{-\tau/\alpha} R \geq |\xi|\geq e^{-\tau/\alpha}}\, . \label{e:sfava-2}
\end{align}
We can thus argue similarly as for \eqref{e:per-bar-2} to conclude
\begin{align}\label{e:per-err-2}
\|V D^2 \bar \Omega_r (\cdot, \tau)\|_4 &\leq C \|(1+|\cdot|)^{-3/2} V (\cdot, \tau)\|_4 + C e^{3\tau/\alpha} \|V\|_{L^4 (B_{R e^{-\xi/\tau}})}\nonumber\\ &\leq C e^{(a_0+\delta_0+1/4) \tau}\, .
\end{align}
In order to handle the remaining three terms, we recall that, by Lemma \ref{l:pointwise}, 
\begin{align}
\|V_{\text{lin}} (\cdot, \tau) \|_\infty &\leq C e^{a_0 \tau}\\
|DV_{\text{lin}} (\xi, \tau)|
&\leq C e^{a_0\tau} |\xi|^{-2-\al}\\
|D^k \Omega_{\text{lin}} (\xi, \tau)| &\leq C e^{a_0\tau} |\xi|^{-2-k-\al}\, .
\end{align}
On the other hand, owing to the computations in this and the previous section we can also write
\begin{align*}
|V_r (\xi, \tau)| &\leq C |\xi|^{1- \al} \mathbf{1}_{|\xi|\geq e^{-\tau/\alpha}}\\
|DV_r (\xi, \tau)| + |\Omega_r (\xi, \tau)| &\leq C |\xi|^{-\al}\mathbf{1}_{|\xi|\geq e^{-\tau/\alpha}} + C e^{\tau/\alpha} |\xi|^{1-\al}
\mathbf{1}_{e^{-\tau/\alpha} R \geq |\xi|\geq e^{-\tau/\alpha}}\\
|D^k \Omega_r (\xi, \tau)| &\leq C |\xi|^{-k-\al}\mathbf{1}_{|\xi|\geq e^{-\tau/\alpha}}\\
&\qquad+ C e^{k \tau/\alpha} |\xi|^{1-\al}
\mathbf{1}_{e^{-\tau/\alpha} R \geq |\xi|\geq e^{-\tau/\alpha}}\, .
\end{align*}
Integrating the estimates in the respective domain we easily get 
\begin{equation}\label{e:lin-err}
\|D V_r D\Omega_{\text{lin}}\|_4 + \|V_r D^2 \Omega_{\text{lin}}\|_4 + \|DV_{\text{lin}}D\Omega_r\|_4 + \|V_{\text{lin}} D^2 \Omega_r\|_4\leq C e^{(a_0+1)\tau}\, .
\end{equation}

\medskip

\textbf{ Remaining estimates.} The two terms \eqref{e:DVDOmega} and \eqref{e:rDVDOmega} have already been covered in the argument above. It remains to handle \eqref{e:comm-term}. Notice that, by Lemma \ref{l:pointwise} and Proposition \ref{p:X-bounds} we have 
\begin{align}
\|r^{-1} V (\cdot, \tau)\|_\infty &\leq C \|\Omega (\cdot, \tau)\|_X \leq C e^{(a_0+\delta_0) \tau}\\
\|r^{-1} V_{\text{lin}} (\cdot, \tau)\|_\infty & \leq C e^{a_0\tau}\, .
\end{align}
We thus conclude easily
\begin{align}
\|r^{-1} V D\Omega (\cdot, \tau)\|_4 & \leq C e^{(a_0+\delta_0) \tau}\|D\Omega (\cdot, \tau)\|_4 \leq C e^{2(a_0+\delta_0) \tau}\, \\
\|r^{-1} V_{\text{lin}}  D\Omega(\cdot, \tau)\|_4 &\leq C e^{a_0\tau}\|D\Omega (\cdot, \tau)\|_4 \leq C e^{(2a_0+\delta_0) \tau}\, .
\end{align}
\end{proof}

We next differentiate in $r$ \eqref{e:master-10} in order to achieve similar identities to \eqref{e:master-angular-1} and \eqref{e:master-angular-2}. This time, given the ambiguity with $V_r$ and $\Omega_r$, we write $,r$ in the subscript to denote the radial derivative of {\em any} function.
\begin{align}
& \partial_\tau \Omega_{,r} + \left(1-\frac{1}{\alpha}\right) \Omega_{,r} + \left(\left(-\frac{\xi}{\alpha} + \beta \bar V + V_r + V_{\text{lin}} + V \right)\cdot \nabla\right) \Omega_{,r}\nonumber\\ 
=  &\mathscr{G}_{,r} - (V_{\text{lin}, r}\cdot \nabla) \Omega - (V_{,r} \cdot \nabla) \Omega - \beta (\bar V_{,r} \cdot \nabla) \Omega - (V_{r,r}\cdot \nabla ) \Omega \label{e:master-radial-1}\\
& \partial_\tau r\Omega_{,r} 
+ r \Omega_{,r} + \left(\left(-\frac{\xi}{\alpha} + \beta \bar V+ V_r + V_{\text{lin}} + V \right)\cdot \nabla\right) (r \Omega_{,r})\nonumber\\
= & r \mathscr{G}_{,r} - r (V_{\text{lin}, r}\cdot \nabla) \Omega - r (V_{,r} \cdot \nabla) \Omega - r \beta (\bar V_{,r} \cdot \nabla) \Omega - r (V_{r,r}\cdot \nabla ) \Omega\nonumber\\
& + \Omega_{,r} (V_{\text{lin}} + V)\cdot \nabla r \, . \label{e:master-radial-2}
\end{align}
Multiplying by $(\Omega_{,r})^3$ and $r \Omega_{,r}$ respectively, and using the estimates \eqref{e:DG} and \eqref{e:rDG}, we achieve, in the respective cases:
\begin{align}
\frac{d}{d\tau} \|\Omega_{,r} (\cdot, \tau)\|_4 &\leq C e^{(a_0+2 \delta_0)\tau} 
+ C \|D V_{\text{lin}} D\Omega (\cdot, \tau)\|_4 + C\|D V D \Omega (\cdot, \tau)\|_4\nonumber\\
& \qquad\qquad + C \|\bar{V}_{,r}\cdot \nabla \Omega (\cdot, \tau)\|_4 + \|DV_r D \Omega (\cdot, \tau)\|_4\, ,\label{e:master-radial}\\
\frac{d}{d\tau} \|r\Omega_{,r} (\cdot, \tau)\|_2 &\leq 
C e^{(a_0+2 \delta_0)\tau} 
+ C \|r V_{\text{lin}, r} D\Omega (\cdot, \tau)\|_2 
\nonumber\\
&
+ C\|r (V_{,r}\cdot \nabla) \Omega (\cdot, \tau)\|_2
+ C \|r \bar{V}_{,r}\cdot \nabla \Omega (\cdot, \tau)\|_2 
\nonumber\\
&
+ \|r (V_{r,r} \cdot\nabla) \Omega (\cdot, \tau)\|_2
+ \|V_{\text{lin}} D\Omega \|_2 + \|V D\Omega\|_2\label{e:master-radial-r}\, .
\end{align}
Note next that 
\begin{align}
&\|D V_{\text{lin}} D\Omega (\cdot, \tau)\|_4 + \|D V D \Omega (\cdot, \tau)\|_4 \\&\leq (\|DV_{\text{lin}} (\cdot, \tau)\|_\infty +
\|DV (\cdot, \tau)\|_\infty) \|D\Omega (\cdot, \tau)\|_4\nonumber
 \leq C e^{(2a_0+\delta_0)\tau}\, ,
\end{align}
and likewise
\begin{align}
\|r D V_{\text{lin}} D\Omega (\cdot, \tau)\|_2 & + \|r D V D \Omega (\cdot, \tau)\|_2 
\nonumber\\&\leq (\|DV_{\text{lin}} (\cdot, \tau)\|_\infty +
\|DV (\cdot, \tau)\|_\infty) \|r D\Omega (\cdot, \tau)\|_2\nonumber\\
& \leq C e^{(2a_0+\delta_0)\tau}\, ,
\end{align}
The terms $(\bar V_{, r} \cdot \nabla) \Omega$ and $r(\bar V_{,r} \cdot \nabla)\Omega$ can be bounded observing that they involve only the angular derivative and that $\bar V_{,r}$ is bounded. Since the angular derivative has already been estimated (this is the reason for estimating it {\em before} estimating the radial derivative), we get
\begin{align}
\|\bar{V}_{,r}\cdot \nabla \Omega (\cdot, \tau)\|_4 &\leq C \|r^{-1} \Omega_\theta (\cdot, \tau)\|_4 \leq C e^{(a_0+2\delta_0)\tau}\, .\\
\|r \bar{V}_{,r}\cdot \nabla \Omega (\cdot, \tau)\|_2 &\leq C \|\Omega_\theta (\cdot, \tau)\|_2 \leq C e^{(a_0+2 \delta_0)\tau}\, .
\end{align}
As for $DV_r D\Omega$ and $r DV_r D\Omega$\, we observe that $\|DV_r\|_\infty \leq e^\tau$ and thus
we easily get
\begin{align}
\|DV_r D\Omega (\cdot, \tau)\|_4 &\leq Ce^{\tau} \| D\Omega (\cdot, \tau)\|_4 \leq Ce^{(a_0+\delta_0+1)\tau}\\
\|r DV_r D\Omega (\cdot, \tau)\|_2 &\leq Ce^{\tau} \|r D\Omega (\cdot, \tau)\|_2 \leq Ce^{(a_0+\delta_0+1)\tau}\, .
\end{align}
We finally need to estimate $\|V_{\text{lin}} D\Omega\|_2$ and $\|V D\Omega\|_2$, but we observe that this has already been done in the previous section, since they correspond to the terms $\mathscr{F}_1$ and $\mathscr{F}_4$ in \eqref{e:master-2}, cf. \eqref{e:F-1} and \eqref{e:F-4}.
Summarizing, we conclude 
\begin{align}
\frac{d}{d\tau} \|\Omega_{,r} (\cdot, \tau)\|_2 &\leq C e^{(a_0+\delta_0+1/2)\tau}\\
\frac{d}{d\tau} \|r\Omega_{,r} (\cdot, \tau)\|_2 &\leq C e^{(a_0+\delta_0+1/2)\tau}\, ,
\end{align}
which we then integrate between $0$ and $\bar\tau$ to achieve the desired conclusion. 

\appendix

\chapter{A more detailed spectral analysis}\label{a:better}

\addtocontents{toc}{\vspace*{\baselineskip}}
\addtocontents{toc}{\hspace*{-1.5em}\textbf{Appendix~\thechapter}}

\vspace*{-4.5pt}

\section{From Remark \ref{r:better2}(i) to Remark \ref{r:better}(c)}

Let us assume the validity of \ref{r:better2}(i) and prove Remark \ref{r:better}(c). Let $m_0\ge 2$ be the integer such that 
\begin{align}\label{z1}
    &\textrm{spec}\, (L_{\text{st}}, U_{m_0}) \cap \{\textrm{Re}\, z > 0\} \neq \emptyset  \, ,
    \\
    &\textrm{spec}\, (L_{\text{st}}, U_{m}) \cap \{\textrm{Re}\, z > 0\} = \emptyset \, 
    \quad 
    \text{for any $m>m_0$} \, .
\end{align}
We show that Remark \ref{r:better}(c) holds with $m = m_0$.

For any $z\in \textrm{spec}_{m_0}\, (L_{\text{st}}) \cap \{\textrm{Re}\, z > 0\}$ we denote by $V_z:= P_z(L^2_{m_0})$ the image of the Riesz projector 
\begin{equation*}
    P_z = \frac{1}{2\pi i} \int_\gamma (w-L_{\text{st}})^{-1} dw \, ,
\end{equation*}
where $\gamma$ parameterizes the boundary of a ball containing $z$ and no other eigenvalues of $L_{\text{st}}$.

It is enough to show that $P_z(U_{km_0}) = \{0\}$ for any $k\in \mathbb{Z}\setminus\{-1, 1\}$, $z\in \textrm{spec}_{m_0}\, (L_{\text{st}}) \cap \{\textrm{Re}\, z > 0\}$ since it gives
$$V_z = P_z(U_{m_0}\cup U_{-m_0}) \subset U_{m_0} \cup U_{-m_0}\, ,$$
where the second inclusion follows from the fact that $U_m$ is always an invariant space of $L_{\text st}$.

If $k>1$, from \eqref{z1} we know that $z\notin \textrm{spec}\, (L_{\text{st}}, U_{km_0})$, hence $P_z(U_{km_0})$ is trivial. If $k<-1$,
we reduce to the previous situation by observing that $P_z(U_{km_0}) = \overline{ P_{\bar z}(U_{-km_0})}$.

\section{Proof of Remark \ref{r:better2}(i)} In order to show this point, given Lemma \ref{l:almost-final-2}, we just need to prove the following statement.

\begin{lemma}\label{l:almost-final-1}
For every fixed $\Xi\in \mathscr{C}$ there is $M_0>0$ such that $\mathscr{U}_m$ is empty for every $m\geq M_0$.
\end{lemma}

Indeed, given the conclusion above we infer that $\mathscr{U}_m$ is empty for every $m\geq m_a$ and it thus suffices to select $m_0$ as the largest integer strictly smaller than $m_a$.

Before coming to the proof of Lemma \ref{l:almost-final-1} we state an auxiliary fact which will be used in the argument and which can be readily inferred from the computations in Step 1 of the proof of Lemma \ref{l:will-apply-Rouche}.

\begin{lemma}\label{l:operator-B_z}
For every $0 < \sigma < \tau < 1$ there is a constant $C$ (depending only upon $\sigma$ and $\tau$) such that $B_z := \mathcal{K}_{m_0} \circ \frac{1}{\Xi-z}$ is a bounded operator from $C^\sigma$ to $C^\tau$ for every $z$ with $\textrm{Im}\, z>0$ and 
\[
\|B_z\|_{\mathcal{L} (C^\sigma, C^\tau)} \leq C\, .
\]
\end{lemma}

\begin{proof}[Proof of Lemma \ref{l:almost-final-1}] The proof will be by contradiction and thus, assuming that the statement is false, we can select:
\begin{itemize}
    \item[(i)] a sequence $\{m_j\}\subset [1, \infty[$ with $m_j\to \infty$;
    \item[(ii)] a sequence $\{z_j\}\subset \mathbb C$ with $\textrm{Im}\, z_j >0$;
    \item[(iii)] and a sequence $\{\psi_j\}\subset L^2 (\mathbb R)$ solving the equation
    \begin{equation}\label{e:eigenvalue-equation-20}
    -\frac{d^2 \psi_j}{dt^2} + m_j^2 \psi_j + \frac{A}{\Xi -z_j} \psi_j = 0\, .
    \end{equation}
\end{itemize}

\medskip

\textbf{ Step 1.} We first prove that $\{z_j\}$ is bounded and every cluster point must be an element of $[0, \Xi (-\infty)]$. Otherwise for a subsequence, not relabeled, we get the estimate
\[
\sup \left\|\frac{A}{\Xi-z_j}\right\|_{L^\infty} =: C_0 < \infty\, .
\]
By scalar multiplying \eqref{e:eigenvalue-equation-20} by $\psi_j$ and taking the real part of the resulting equation we then conclude
\[
\int (|\psi_j'|^2 + m_j^2 |\psi_j|^2) \leq C_0 \int |\psi_j|^2\, ,
\]
which clearly it is not feasible because $C_0 < m_j^2$ for a sufficiently large $j$ (and $\psi_j$ is nontrivial).

Up to subsequences we can thus assume that $z_j$ converges to some $z_0 \in [0, \Xi (-\infty)]$.

\medskip

\textbf{ Step 2.} We next analyze the cases $z_0 =0$ and $z_0 = \Xi (-\infty)$. The argument is similar to that used in Section \ref{s:3+4} in case (C). Let us argue first for $z_0=0$. We observe that $\Xi^{-1} |A|$ belongs to $L^1 (]-\infty, N])$ for any fixed $N$ and that, likewise, $|\Xi-z_j|^{-1} |A|$ have a uniform $L^1$ bound on any $]-\infty, N]$. We can then 
use the Lemma \ref{l:ODE2} to normalize $\psi_j$ so that it is asymptotic to $e^{m_j t}$ and also to write
\[
\psi_j (t) = e^{m_j t} (1+z_j (t)) 
\]
with
\[
|z_j (t)| \leq \exp \left(\frac{1}{m_j} \int_{-\infty}^N \frac{|A|}{|\Xi-z_j|}\right) -1 \qquad 
\mbox{for all $t\leq N$.}
\]
In particular, we have $|z_j (t)|\leq \frac{C(N)}{m_j}$ on $]-\infty, N]$. We next scalar multiply \eqref{e:eigenvalue-equation-20} by $\psi_j$ and take the imaginary part to conclude
\[
- \left(\int_{-\infty}^a + \int_b^\infty\right) \frac{A}{|\Xi-z_j|^2} |\psi_j|^2\leq 
\int_a^b \frac{A}{|\Xi-z_j|^2} |\psi_j|^2\, .
\]
In particular, since $\frac{A}{|\Xi-z_j|^2}$ is bounded from above by a constant $C$ independent of $j$ on $[a,b]$ and  $-\frac{A}{|\Xi-z_j|^2}$ is bounded from below by a constant $c>0$ independent of $j$ on $[b+1, b+2]$, we conclude
\[
\int_{b+1}^{b+2} |\psi_j|^2 \leq \frac{C}{c} \int_a^b |\psi_j|^2\, .
\]
We next choose $N$ larger than $b+2$ and use the estimate $|z_j (t)|\leq \frac{C(N)}{m_j}$ to argue that, for $j$ large enough, we have $\frac{1}{2} e^{m_j t} \leq |\psi_j (t)| \leq 2 e^{m_j t}$ on $]-\infty, N]$. In particular, we infer
\[
\int_{b+1}^{b+2} e^{2m_j t} \leq C \int_a^b e^{2m_j t}
\]
provided the constant $C$ is chosen large enough (but independent of $j$) and $j$ is large enough. The latter inequality is certainly impossible for $m_j$ large enough, leading to a contradiction. 

The argument to exclude $z_0 = \Xi (-\infty)$ is entirely analogous, this time normalizing for $t\to \infty$ and reaching an inequality of type
\[
\int_{a-2}^{a-1} e^{-2m_j t} \leq C \int_a^b e^{-2m_j t}
\]
for a constant $C$ independent of $j$ and any $j$ large enough.

\medskip

\textbf{ Step 3.} We next examine the last case, that is $z_0 = \Xi (c)$. This time we fix a $\sigma \in \, ]0,1[$ and normalize $\psi_j$ so that 
\begin{equation}\label{e:normalization_psi_j}
\|\psi_j\|_{C^\sigma}=1\, .
\end{equation}
We observe that
\[
\psi_j = - \mathcal{K}_{m_j} \left(\frac{A}{\Xi-z_j} \psi_j\right)\, ,
\]
and also recall that $\mathcal{K}_{m_j} (\varphi) = \frac{1}{2m_j} e^{-m_j |\cdot|} * \varphi$.
We set $m_0=m_a$ write further
\[
\psi_j = - \mathcal{K}_{m_j} \circ \left(-\frac{d^2}{dt^2} +m_0^2\right) \left(\mathcal{K}_{m_0} \left(\frac{A}{\Xi-z_j}\right) \psi_j\right)\, .
\]
Recalling Lemma \ref{l:operator-B_z}, we can fix a $\tau \in ]\sigma,1[$ to achieve 
\[
\left\|\left(\mathcal{K}_{m_0} \left(\frac{A}{\Xi-z_j}\right) \psi_j\right)\right\|_{C^\tau}\leq C
\]
for some constant $C$ independent of $j$. 

We will show in the final step that 
\begin{itemize}
    \item[(Cl)] $\|\mathcal{K}_{m_j} \circ (-\frac{d^2}{dt^2} + m_0^2)\|_{\mathcal{L} (C^\tau, C^\tau)} \leq C$ for some constant $C$ independent of $k$.
\end{itemize}
In particular, we achieve
\begin{equation}\label{e:estimate-C-tau}
\|\psi_j\|_{C^\tau} \leq C\, .
\end{equation}
We now wish to show that indeed $\|\psi_j\|_{C^\sigma} \leq \frac{1}{2}$ for $j$ large enough, which obviously would be a contradiction. In order to achieve the latter estimate we use a Littlewood-Paley decomposition. We fix a cut-off function $\chi$ which is supported in $]\frac{1}{2}, 2[$, define $\chi_\ell (t) :=\chi (2^{-\ell} t)$ for $\ell\in \mathbb N\cup \{0\}$ and assume that $\chi$ has been chosen so that
\[
\sum_{\ell \in \mathbb N\cup \{0\}} \chi_\ell \equiv 1 \qquad \mbox{on $[1, \infty[$}.
\]
We then define 
\[
\chi_{-1} := 1 - \sum_{\ell \in \mathbb N} \chi_\ell\, 
\]
and introduce the Littlewood-Paley operator $\Delta_\ell$ as $\Delta_\ell (\varphi) = \mathscr{F}^{-1} (\chi_\ell \mathscr{F} (\varphi))$, where $\mathscr{F}$ is the Fourier transform. 
We finally recall that (see \cite[Section 1.4.2]{Grafakos}), if we define
\[
\|\varphi\|_{X^\sigma} := \sum_{\ell \geq -1} 2^{\sigma \ell} \|\Delta_\ell \varphi\|_{L^\infty}\, ,
\]
then
\[
C (\sigma)^{-1} \|\varphi\|_{X^\sigma} \leq \|\varphi\|_{C^\sigma} \leq C (\sigma) \|\varphi\|_{X^\sigma}\, .
\]
We are now ready to perform our final estimate. We fix a large $N$, which will be chosen later, and for $\ell\geq N$ we write 
\begin{align*}
\sum_{\ell \geq N} 2^{\sigma \ell} \|\Delta_\ell \psi_j\|_\infty 
&\leq 2^{-N (\tau-\sigma)} \sum_{\ell\geq N} 2^{\tau \ell} \| \Delta_\ell \psi_j\|_\infty\\
&\leq 2^{-N (\tau-\sigma)} C (\tau) \|\psi_j\|_{C^\tau} \leq C 2^{-N (\tau-\sigma)}\, ,
\end{align*}
where the constant $C$ is independent of both $N$ and $j$. Next, for any $\ell$ we observe that
\[
\Delta_\ell \psi_j = \mathcal{K}_{m_j} \circ \left(-\frac{d^2}{dt^2} + m_0^2\right) \underbrace{\left(\Delta_\ell \left(\mathcal{K}_{m_0} \left(\frac{A}{\Xi-z} \psi_j \right)\right)\right)}_{=: \Gamma_{\ell,j}}\, . 
\]
Now 
\[
\|\Gamma_{\ell,j}\|_{L^\infty} \leq C 2^{-\ell\sigma} \left\|\mathcal{K}_{m_0} \left(\frac{A}{\Xi-z} \psi_j \right)\right\|_{C^\sigma} \leq C 2^{-\ell \sigma}\, .
\]
On the other hand, because of the frequency localization, we have
\[
\Delta_\ell \psi_j = \mathcal{K}_{m_j} \circ \left(-\frac{d^2}{dt^2} + m_0^2\right) \circ (\Delta_{\ell-1}+\Delta_\ell+\Delta_{\ell+1}) (\Gamma_{\ell,j})
\]
and the estimate
\[
\left\|\mathcal{K}_{m_j} \circ \left(-\frac{d^2}{dt^2} + m_0^2\right) \circ \Delta_\ell\right\|_{\mathcal{L} (L^\infty, L^\infty)} \leq \frac{C}{m_j^2} \left(2^{2\ell} + m_0^2\right)\, .
\]
We can therefore write the estimate
\begin{align*}
    \|\psi_j\|_{C^\sigma} 
    &\leq \frac{C}{m_j^2} \sum_{\ell=-1}^{N} (2^{2 \ell} + m_0^2) + C 2^{-N (\tau-\sigma)} \\
    &\leq \frac{CN}{m_j^2} \left(2^{2 N} + m_0^2\right) + C 2^{-N (\tau-\sigma)}\, ,
\end{align*}
where the constants $C$ are independent of $N$ and $j$. In particular, we fix first $N$ large enough to get $C 2^{-N (\tau-\sigma)} \leq \frac{1}{4}$ and we then choose $m_j$ large enough so that 
\[
\frac{CN}{m_j^2} \left(2^{2 N} + m_0^2\right) \leq \frac{1}{4}\, .
\]
These two estimates imply $\|\psi_j\|_{C^\sigma} \leq \frac{1}{2}$, contradicting the normalization \eqref{e:normalization_psi_j}. 

\medskip

\textbf{ Step 4.} To complete the proof of the Lemma we need to show (Cl). We first write 
\[
T_{m, \ell} := \Delta_\ell \circ \mathcal{K}_m \circ \left(-\frac{d^2}{dt^2} + m_0^2\right)\, .
\]
The operator $T_{m, \ell}$ is the convolution with a kernel $K_{m, \ell}$ whose Fourier symbol is given by 
$\chi_\ell (\xi) \frac{|\xi|^2 + m_0^2}{|\xi|^2 +m^2}$. Hence, for $\ell \geq 0$ we have
\[
K_{m, \ell} (t) = \frac{1}{2\pi} \int \chi \left(\frac{\xi}{2^\ell}\right) \frac{|\xi|^2 + m_0^2}{|\xi|^2 + m^2} e^{i \xi t}\, d\xi\, .
\]
and 
\[
(-it)^k K_{m, \ell} (t) = \frac{1}{2\pi} \int \frac{d^k}{d\xi^k} \left(  \chi \left(\frac{\xi}{2^\ell}\right) \frac{|\xi|^2 + m_0^2}{|\xi|^2 + m^2}\right) e^{it\xi}\, d\xi\, .
\]
In particular, we easily conclude
\[
\| |t|^k K_{m, \ell}\|_{L^\infty} \leq C (k) 2^{\ell (1-k)}\, ,
\]
for a constant $C (k)$ independent of both $m\geq 1$ and $\ell$, but which depends on $k$. From the latter we can estimate
\begin{align*}
\|K_{m, \ell}\|_{L^1} &\leq \int_{|t|\leq 2^{-\ell}} |K_{m, \ell} (s)|\, ds + 
\int_{|t|\geq 2^{-\ell}} \frac{|s^2 K_{m, \ell} (s)|}{|s|^2}\, ds\\
&\leq C + C 2^{-\ell} \int_{2^{-\ell}}^\infty \frac{1}{s^2}\, ds \leq C\, .
\end{align*}
For $\ell = -1$ we just conclude, likewise
\[
\||t|^k K_{m, -1}\|_{L^\infty} \leq C (k)
\]
for a constant $C(k)$ independent of $m$, but depending of $k$. Once again using the cases $k=0$ and $2$ of the latter inequality we achieve 
\[
\|K_{m, -1}\|_{L^1} \leq C \, .
\]
We have thus bounded all $\|K_{m, \ell}\|_{L^1}$ with a universal constant $C$ independent of both $m\geq 1$ and $\ell\in \mathbb N \cup \{-1\}$. In particular, since $\|T_{m, \ell}\|_{\mathcal{L} (L^\infty, L^\infty)} = \|K_{m, \ell}\|_{L^1}$ and 
\[
\mathcal{K}_m \circ \left(-\frac{d^2}{dt^2} + m_0^2\right) = \sum_{\ell \geq -1} T_{m, \ell} 
= \sum_{\ell \ge -1} T_{m, \ell} \circ (\Delta_{\ell-1}+\Delta_\ell+\Delta_{\ell+1})\, ,
\]
we can estimate
\begin{align*}
\left\|  \mathcal{K}_m \circ \left(-\frac{d^2}{dt^2} + m_0^2\right) (\varphi)\right\|_{C^\sigma} 
&\leq C (\sigma) \sum_{\ell \geq -1} 2^{\sigma \ell} \|T_{m, \ell} (\varphi)\|_{L^\infty}\\
&= C (\sigma) \sum_{\ell \geq -1} 2^{\sigma \ell} \|T_{m, \ell} (\Delta_\ell \varphi)\|_{L^\infty}\\
&\leq C (\sigma) \sum_{\ell \geq -1} 2^{\sigma \ell} \|\Delta_\ell \varphi\|_{L^\infty}\\ 
&\leq C (\sigma) \|\varphi\|_{C^\sigma}\, .
\end{align*}
This completes the proof of (Cl) and hence of the entire Lemma.
\end{proof}

\section{Proof of Theorem \ref{thm:spectral-stronger-2}}

In \cite{Vishik2}, Vishik claims the following improved version of Theorem \ref{thm:spectral5}, which would immediately imply Theorem \ref{thm:spectral-stronger-2}.

\begin{theorem}\label{thm:Vishikversion}
There are a function $\Xi\in \mathscr{C}$ and an integer $m_0\geq 2$ such that $\mathscr{U}_{m} = \emptyset$ for any integer $m>m_0$ and  $\mathscr{U}_{m_0}$ consists of a single $z$. Moreover, the \index{algebraic multiplicity@algebraic multiplicity} algebraic multiplicity of $m_0 z$ as an eigenvalue of $\mathcal{L}_{m_0}$ is $1$. 
\end{theorem}

Vishik's suggested proof of Theorem \ref{thm:Vishikversion} builds upon Proposition \ref{p:3+4} and the following improved versions of Proposition \ref{p:5-7} and Proposition \ref{p:almost-final}.

\begin{proposition}\label{prop:Vishicimproved1}\label{PROP:VISHICIMPROVED1}
Assume $- \lambda_a < -1$ and let $m_a=\sqrt{\lambda_a}$. Then there
exist $\varepsilon >0$ and $\delta>0$ with the following property. 
For every $h\in ]0, \delta[$, $\mathscr{U}_{m_a-h} \cap B_\varepsilon (\Xi (a)) = \{z_{m_a-h}\}$, where $(m_a-h) z_{m_a-h}$ is an eigenvalue of $\mathcal{L}_{m_a-h}$ with algebraic multiplicity $1$. 
\end{proposition}

In \cite{Vishik2} Vishik only gives the argument that $\mathscr{U}_{m_a-h} \cap B_\varepsilon (\Xi (a))$ contains a single element $z_{m_a-h}$ and the corresponding eigenspace of $(m_a-h)^{-1} \mathcal{L}_{m_a-h}$ has dimension $1$ \index{geometric multiplicity@geometric multiplicity} (i.e. its {\em geometric} multiplicity is $1$, cf. Remark \ref{r:b-also-2}). However it is essential to have the {\em algebraic} multiplicity equal to $1$ in order to complete his suggested argument. After we pointed out to him the gap in his paper, he suggested in \cite{Vishik3} the proof of Proposition \ref{prop:Vishicimproved1} reported below. Before coming to it, we point out that 
a spectral perturbation argument as in the proof of Lemma \ref{l:almost-final-2} (which we outline anyway below) 
easily implies the following.

\begin{proposition}\label{prop:Vishicimproved2}
Assume $- \lambda_a<-1$ and let 
\begin{align*}
m_a &:= \sqrt{\lambda_a}\, ,\\
m_b &:= \max \{1, \sqrt{\lambda_b}\}\, .
\end{align*}
Then $\mathscr{U}_m$ consists of a single element $z_m$ for every $m\in ]m_b, m_a[$ and moreover the  algebraic multiplicity of $z_m$ as an eigenvalue of $m^{-1} \mathcal{L}_m$ is $1$.
\end{proposition}

Taking the previous proposition for granted, we just need a choice of $\Xi$ for which $\lambda_a >1$ and $]m_b, m_a[$ contains an integer, which is guaranteed by Lemma \ref{L:BOTTOM}, and we conclude Theorem \ref{thm:Vishikversion}.

\medskip

We now explain how to prove Proposition \ref{prop:Vishicimproved2}. 
From Proposition \ref{p:almost-final} and Lemma \ref{l:almost-final-1} we know that $\mathscr{U}_m \neq  \emptyset$, for every $m\in ]m_b, m_a[$, and $\mathscr{U}_m = \emptyset$ for $m \ge m_a$. Moreover, Remark \ref{rmk:algebraic dim const} implies that the sum of the algebraic multiplicities of $z\in \mathscr{U}_m$, as eigenvalues of $m^{-1} \mathcal{L}_m$, is constant for $m\in ]m_b, m_a[$. Hence, to conclude we just need to prove that the latter is $1$ for some $m\in ]m_b, m_a[$. 
To that aim we show that for any $\varepsilon>0$ there exists $\delta>0$ such that  $\mathscr{U}_{m_a-h} = \mathscr{U}_{m_a-h}\cap B_{\varepsilon}(\Xi(a))$ for any $h\in ]0,\delta[$. This is enough for our purposes since, together with Proposition \ref{prop:Vishicimproved1}, it gives $\mathscr{U}_{m_a-h}= \mathscr{U}_{m_a-h}\cap B_{\varepsilon}(\Xi(a))=\{ z_{m_a-h}\}$ where $z_{m_a-h}$ is an eigenvalue of $(m_a-h)^{-1} \mathcal{L}_{m_a - h}$ with algebraic multiplicity $1$.

Assume for contradiction the existence of a sequence $(m_j)_{j\in\mathbb N}$ in $]m_b,m_a[$ converging to $m_a$ such that there are $z_j\in \mathscr{U}_{m_j}$ with $|z_j - \Xi(a)|>\varepsilon$ for some $\varepsilon>0$. Up to extracting a subsequence, we may assume $z_j \to z$ for some $z\in \mathbb C$ with $|z-\Xi(a)|\ge \varepsilon$. Proposition \ref{p:3+4} implies that the imaginary part of $z$ is positive. Arguing as in the first step of proof of Proposition \ref{p:3+4} we can prove that $z\in \mathscr{U}_{m_a}$ and reach a contradiction.

\section{Proof of Proposition \ref{prop:Vishicimproved1}}

The proof of Proposition \ref{prop:Vishicimproved1} can be reduced to the following weaker version using Remark \ref{rmk:algebraic dim const} and the argument outlined in the previous paragraph.

\begin{proposition}\label{prop:Vishikimproved-weaker}
Assume $- \lambda_a < -1$ and let $m_a=\sqrt{\lambda_a}$. Let $h$ and $\varepsilon$ be sufficiently small so that Proposition \ref{p:3+4} and Remark \ref{r:b-also-2} apply, namely
$\mathscr{U}_{m_a-h} \cap B_\varepsilon (\Xi (a)) = \{z_{m_a-h}\}$, where $z_{m_a-h}$ is an eigenvalue of $(m_a-h)^{-1} \mathcal{L}_{m_a-h}$ with {\em geometric} multiplicity $1$. Then, if $h$ is chosen possibly smaller, the algebraic multiplicity of $z_{m_a-h}$ is also $1$.
\end{proposition}

We now come to the proof of the latter, which is the heart of the matter. First of all, we introduce a suitable transformation of the space $\mathcal{H}$ (which we recall is the domain of the operator $\mathcal{L}_m$, defined in \eqref{e:def-H}). We introduce the Hilbert space
\[
\mathcal{H}^e :=\left\{ f: \mathbb R \to \mathbb C\, :\,  \int |f (t)|^2 e^{-2t}\, dt < \infty\right\} 
\]
and the isometry $T: \mathcal{H} \to \mathcal{H}^e$ given by
\[
\gamma (r) \mapsto e^{2t} \gamma (e^t)\, .
\]
Rather than considering the operator $\mathcal{L}_m$ on $\mathcal{H}$, it turns out to be more convenient to consider the operator $T \circ \mathcal{L}_m \circ T^{-1}$ on $\mathcal{H}^e$. Since the spectra of the two operators coincide, with a slight abuse of notation we will keep writing $\mathcal{L}_m$ in place of $T \circ \mathcal{L}_m \circ T^{-1}$, and we will keep $\mathscr{U}_m$ to denote the point spectrum of ${m^{-1}}T \circ \mathcal{L}_m \circ T^{-1}$ in the upper half plane. 

Simple computations show that the operator $\mathcal{L}_m$ is given, on $\mathcal{H}^e$ by 
\[
\mathcal{L}_m (\alpha) = m \Xi \alpha -  m A \varphi
\]
where $\varphi$ is the unique $\mathcal{H}^e$ solution of 
\[
\varphi'' - m^2 \varphi = \alpha\, .
\]

We can now come to the main idea behind the simplicity of $z_{m_a-h}$, which is borrowed from \cite{Vishik3}. A prominent role is played by the adjoint of $\mathcal{L}_m$ (considered as a bounded linear operator from $\mathcal{H}^e$ into itself): for the latter we will use the notation $\mathcal{L}_m^\star$.

\begin{lemma}\label{l:aggiunto}
Suppose that $h$ and $\varepsilon$ are sufficiently small such that \[
\{z_{m_a-h}\} =\mathscr{U}_{m_a-h}\cap B_\varepsilon (\Xi (a))
\]
and $z_{m_a-h}$ has geometric multiplicity $1$ in $\textrm{spec}\, ((m_a-h)^{-1} \mathcal{L}_{m_a-h}, \mathcal{H}^e)$. Let $\alpha_h \in \mathcal{H}^e\setminus \{0\}$ be such that $(m_a-h)^{-1}\mathcal{L}_{m_a-h}(\alpha_h) - z_{m_a - h}\alpha_h=0$. If $h$ is small enough, then there is $\beta_h\in \mathcal{H}^e$ such that $(m_a-h)^{-1}\mathcal{L}_{m_a-h}^\star(\beta_h) - \bar z_{m_a-h}\beta_h =0$ and
\begin{equation}\label{e:dual-pairing}
\langle \alpha_h, \beta_h \rangle_{\mathcal{H}^e} = \int \alpha_h (t) \bar \beta_h (t)\, e^{-2t}dt \neq 0\, .
\end{equation}
\end{lemma}


Let us show how the latter implies Proposition \ref{prop:Vishikimproved-weaker}. Assume $z_{m_a-h}$ were an element of $\textrm{spec}\, ((m_a-h)^{-1}\mathcal{L}_{m_a-h}, \mathcal{H}^e)\cap B_\varepsilon (\Xi (a))$ with geometric multiplicity $1$ and algebraic multiplicity larger than $1$: our goal is to show that $h$ cannot be too small. 
The properties just listed mean that the following bounded operator on $\mathcal{H}^e$,
\[
L_h := (m_a-h)^{-1}\mathcal{L}_{m_a-h} - z_{m_a-h} \, ,
\]
has a 1-dimensional kernel, $0$ is in its point spectrum, and $0$ has algebraic multiplicity strictly larger than $1$. These properties imply that any element $\alpha_h$ in the kernel of $L_h$ (i.e. any eigenfunction of $(m_a-h)^{-1}\mathcal{L}_{m_a-h}$ with eigenvalue $z_{m_a-h}$) is in the image of $L_h$. Fix one such element $\alpha_h$ and let $\eta_h$ be such that $L_h (\eta_h) = \alpha_h$. If $h$ is small enough, we can fix $\beta_h$ as in Lemma \ref{l:aggiunto}, and observe that it is in the kernel of the adjoint operator $L^\star_h$. We then must have
\[
0 \neq \int \alpha_h \bar\beta_h \, e^{-2t}dt= \int L_h (\eta_h) \bar\beta_h \, e^{-2t}dt
= \int \eta_h \overline{L_h^\star (\beta_h)} \, e^{-2t}dt = 0\, ,
\]
which is not possible.

\begin{proof}[Proof of Lemma \ref{l:aggiunto}]

We begin by proving the following claim:
\begin{itemize}
\item[(Cl)] For any $z\in \mathcal{U}_m$, with $m>1$, such that 
$m^{-1}\mathcal{L}_m(\alpha_z) - z \alpha_z = 0$,
there exists $\beta_z\in \mathcal{H}^e$ such that
\begin{equation}
     m^{-1}\mathcal{L}_m^\star(\beta_z) - \bar z \beta_z = 0 \, ,
\end{equation}
and
\begin{equation}\label{eq:keyfunction}
    \left\langle \alpha_z, \beta_z \right\rangle_{\mathcal{H}^e}
    =
    \int_{\mathbb{R}} \frac{A(t)}{(\Xi(t) - z)^2}\varphi_z(t)^2\, d t\, ,
\end{equation}
where $\varphi_z$ is the unique solution in $\mathcal{H}^e$ of $\varphi_z'' - m^2\varphi_z = \alpha_z$.
\end{itemize}
To that aim we first observe that the adjoint of $\mathcal{L}_m$ in $\mathcal{H}^e$ is given by
\begin{equation}
    \mathcal{L}_m^\star (\alpha) =  m( \Xi \alpha + e^{2t} \mathcal{K}_m(A\alpha e^{-2t})) \, ,
\end{equation}
where $\mathcal{K}_m$ is the inverse of $-\frac{d^2}{dt^2} + m^2$ as a closed unbounded self-adjoint operator in $L^2(\mathbb{R})$. 
Notice that $\mathcal{L}^\star$ is well defined because $e^{-t}\alpha \in L^2(\mathbb{R})$ and $A\sim e^{2t}$ as $t\to -\infty$.

We now observe that, if $z\in \mathcal{U}_m$, $m^{-1}\mathcal{L}_m (\alpha_z) = z \alpha_z$, and $\beta_z$ is defined by
\[
\beta_z:= e^{2t}\frac{\bar \varphi_z}{\Xi - \bar z}
    = -e^{2t}\frac{\mathcal{K}_m(\bar \alpha_z)}{\Xi - \bar z}\, ,
\]
then 
\begin{equation}\label{eq:adj}
     m^{-1}\mathcal{L}_m^\star (\beta_z) = \bar z \beta_z \, .
\end{equation}
Notice that $\beta_z\in W^{2,2}_\textrm{loc}\cap \mathcal{H}^e$ decays exponentially fast at $\infty$ thanks to the bound $| \varphi_z(t)|\le C e^{-m|t|}$, for every $t\in \mathbb{R}$, proven in Lemma \ref{c:decay}.
Let us now verify \eqref{eq:adj}: 
We first observe that
\begin{equation}
    \alpha_z = \frac{A\varphi_z}{\Xi - z}  \, ,
\end{equation}
hence
\begin{align*}
    m^{-1}\mathcal{L}_m^\star (\beta_z) & = \Xi \beta_z + e^{2t}\mathcal{K}_m \left(\frac{A \bar \varphi_z}{\Xi - \bar z}\right) =
    \Xi \beta_z + e^{2t}\mathcal{K}_m(\bar \alpha_z)
    \\& = \Xi \beta_z -(\Xi - \bar z) \beta_z = \bar z \beta_z \, .
\end{align*}
It is now immediate to conclude \eqref{eq:keyfunction}.

\medskip

In order to simplify our notation we use $\mathcal{L}_h$, $z_h$ and $m_h$ in place of $\mathcal{L}_{m_a-h}$, $z_{a-h}$ and $m_a-h$.
Given $\alpha_h\in \mathcal{H}^e$ as in the statement of the Lemma we denote by $\varphi_h$ the unique $L^2$ solution of $\varphi_h'' - m_h^2 \varphi_h = \alpha_h$. We now can apply (Cl) above to find $\beta_h\in \mathcal{H}^e$ which solves $m_h^{-1}\mathcal{L}_h^\star (\beta_h) = \bar z_h \beta_h$ and such that
\begin{equation}\label{eq:keyfunction1}
    \left\langle \alpha_h, \beta_h \right\rangle_{\mathcal{H}^e}
    =
    \int_{\mathbb{R}} \frac{A(t)}{(\Xi(t) - z_h)^2}\varphi_h(t)^2\, d t\, .
\end{equation}
To conclude the proof it suffices to show that, after appropriately normalizing the functions $\alpha_h$ (i.e. after multiplying them by an appropriate constant factor, which might depend on $h$) we have 
\begin{equation}\label{e:limit-nonzero}
    \lim_{h\to 0} \left\langle \alpha_h, \beta_h \right\rangle_{\mathcal{H}^e}
    = c \neq 0 \, .
\end{equation}
Note that for the latter conclusion, which we will prove in the next two steps, we will use the assumption that $\alpha_h\neq 0$.

\bigskip

\textbf{ Step 1:} We show that, up to multiplication of $\alpha_h$ by a suitable constant factor (which might vary with $h$), $\varphi_h \to \varphi$ in $W^{1,2}$ and in $C^{1,\alpha}$, as $h\to 0$, where $\varphi \in W^{2,\infty}$ is a nontrivial solution to
\begin{equation}
- \frac{d^2 \varphi}{dt^2} + m_a^2 \varphi + \frac{A}{\Xi - \Xi (a)} \varphi = 0\, .
\end{equation}
By Remark \ref{r:phi(a)-nonzero}, the nontriviality of $\varphi$ implies $\varphi (a) \neq 0$, and hence, up to multiplication by another constant factor, we will assume, without loss of generality, that $\varphi (a)=1$.

Recall that $\varphi_h$ solves the equation
\begin{equation}\label{e:ODE-again-100}
- \frac{d^2 \varphi_h}{dt^2} + m_h^2 \varphi_h + \frac{A}{\Xi - z_h} \varphi_h
= 0 \, .
\end{equation}
For the moment, let us normalize the functions so that 
\begin{equation}\label{e:normalization-2}
\int (|\varphi_h'|^2 + m_h^2 |\varphi_h|^2) = 1\, ,
\end{equation}
as in \eqref{e:L2-normalization}. We then can argue as for the bounds \eqref{e:exp-bound-1} and \eqref{e:exp-bound-2} to derive the existence of constants $C$ and $\beta>2$ (independent of $h$) such that
\begin{equation}\label{e:exp-bound-3}
\left|\varphi_h (t)\right| \leq C e^{- \beta |t|} \quad \forall t\, .
\end{equation}
Recalling Section \ref{s:5-7-part-II}, we know that $z_h = \Xi (a) + c(a) h + o (h)$, where $c(a)$ is a complex number with positive imaginary part, which we denote by $d(a)$. Using the monotonicity of $\Xi$ we can write $z_h = \Xi (t (h)) + i (d(a) h + o (h))$ for some $t(h)$ which satisfies the bound $|t (h) -a|\leq C h$ for some positive constant $C$. In particular, using the mean value theorem and the fact that the derivative of $\Xi$ does not vanish on $[a-1, a+1]$, we get
\[
|\Xi (t) - z_h|^2\geq C^{-1} (|t-t(h)|^2 + |h|^2)\qquad \forall t\in [a-1,a+1]\, ,
\]
where $C$ is some positive constant. Next, using that $|t(h)-a|\leq C h$, we conclude the estimate
\[
|\Xi (t) - z_h|\geq C^{-1} |t-a| \qquad \forall t\in [a-1, a+1]\, ,
\]
with a constant $C$ independent of $h$. Since $a$ is a zero of $A$, we finally conclude that the functions
\[
\frac{A(t)}{\Xi (t) - z_h}
\]
are in fact uniformly bounded, independently of $h$. Using the latter estimate and \eqref{e:exp-bound-3} we thus infer that 
\begin{equation}\label{e:exp-bound-4}
|\varphi_h'' (t)|\leq C e^{-\beta |t|}\, .
\end{equation}
In particular, upon extraction of a subsequence, we can assume that the $\varphi_h$ converge to a function $\varphi$ strongly in $W^{1,2}$, weakly in $W^{2,\infty}$, and hence strongly in $C^{1, \alpha}$ for every $\alpha<1$. In particular, because of the normalization \eqref{e:normalization-2}, $\varphi$ is a nontrivial $W^{2, \infty}$ function, satisfying the same exponential decay as in \eqref{e:exp-bound-3} and \eqref{e:exp-bound-4}. Moreover, given the bound on the functions $\frac{A(t)}{\Xi (t) -z_h}$, $\varphi$ is in fact a solution of 
\begin{equation}\label{e:ODE-again-101}
- \frac{d^2 \varphi}{dt^2} + m_a^2 \varphi + \frac{A}{\Xi - \Xi (a)} \varphi = 0\, .
\end{equation}
Recalling Remark \ref{r:phi(a)-nonzero}, $\varphi (a)\neq 0$ and $\varphi$ is unique up to a constant factor. In particular, we must have 
\begin{equation}\label{e:comoda}
\liminf_{h\downarrow 0} |\varphi_h (a)|>0, 
\end{equation}
otherwise for a suitable subsequence we would have convergence to a nontrivial solution $\varphi$ for which $\varphi (a)=0$. Because of \eqref{e:comoda} we can use the different normalization $\varphi_h (a) = 1$, which in turn implies that $\varphi_h$ converges (without extracting subsequences) to the unique $W^{2,2}$ solution $\varphi$ of \eqref{e:ODE-again-101} which satisfies $\varphi (a)=1$.

\medskip

\textbf{ Step 2:} We prove that
\begin{equation}
    \lim_{h\to 0} \, \textrm{Im} \int_{\mathbb{R}} \frac{A(t)}{(\Xi(t) - z_h)^2}\varphi_h(t)^2\, d t
    =
    \frac{2A'(a)\varphi(a)^2}{d(a)\Xi'(a)^2}\int_{\mathbb{R}} \frac{s^2}{(1+s^2)^2} ds \, .
\end{equation}

Recalling that $z_h = \Xi(t(h)) + i (d(a)h + o(h))$, we write
 \begin{align*}
 	&\quad\; \textrm{Im}\, \left[\frac{A}{(\Xi-z_h)^2} \varphi_h^2\right]\\
 	 & = \textrm{Im}\,  \left[\frac{A((\Xi-\Xi(t(h)))+ i(d(a)h + o(h)))^2}{((\Xi-\Xi(t(h)))^2 + (d(a)h + o(h))^2)^2}(\textrm{Re}\, \varphi_h + i \textrm{Im}\, \varphi_h)^2\right]
	 \\& 
	 = \frac{2(d(a)h + o(h)) A(\Xi - \Xi(t(h)))}{((\Xi-\Xi(t(h)))^2 + (d(a)h + o(h))^2)^2}(\mathrm{Re}\, \varphi_h^2 - \textrm{Im}\, \varphi_h^2)
	\\ & \qquad + 
	 \frac{2A }{(\Xi-\Xi(t(h)))^2 + (d(a)h + o(h))^2}\textrm{Re}\, \varphi_h \textrm{Im}\, \varphi_h 
 	 \\  & \qquad -\frac{ 4(d(a)h + o(h))^2 A}{((\Xi-\Xi(t(h)))^2 + (d(a)h + o(h))^2)^2}\textrm{Re}\, \varphi_h \textrm{Im}\, \varphi_h
	 \\ & =: I_h + II_h + III_h \, .
\end{align*}
To ease notation we set
 \begin{equation}
 	f_h := \textrm{Re}\, \varphi_h^2 - \textrm{Im}\, \varphi_h^2 \, ,
 	\quad
 	g_h := \textrm{Re}\, \varphi_h\, \textrm{Im}\varphi_h \, ,
 \end{equation}
 and observe that $f_h \to \varphi^2$, $g_h \to 0$ as $h\to 0$, where the convergence is, in both cases, in the strong topologies of $L^2$ and of $C^\alpha$, for every $\alpha <1$.
We will show below that:
\begin{align}
\lim_{h\to 0} \int I_h &=  \frac{2A'(a)\varphi(a)^2}{d(a) \Xi'(a)^2}\int_{\mathbb{R}} \frac{s^2}{(1+s^2)^2} ds=: L (a)\label{e:limit-I}\, ,\\   
\lim_{h\to 0} \int II_h &= 0\label{e:limit-II}\, ,\\
\lim_{h\to 0} \int III_h &=0\label{e:limit-III}\, .
\end{align}
Considering that none of the numbers $\Xi' (a)$, $A'(a)$, $\varphi (a)$, and $d(a)$ vanish, $L(a)\neq 0$. This implies \eqref{e:limit-nonzero} and concludes the proof. We next study separately the three limits above. 

\medskip

\textbf{ Proof of \eqref{e:limit-I}.}
There exists $\delta>0$ and $r>0$ such that for any $h$ sufficiently small one has $|\Xi(t) - \Xi(t(h))|>\delta$ for all $t\in \mathbb{R}\setminus (a-r/2,a+r/2)$. This implies that
\begin{equation}
	\lim_{h \to 0} \int_{\mathbb{R}\setminus (t(h) - r, t(h) + r)} I_h = 0 \, ,
\end{equation}
hence, we are left with
\begin{equation}
	\lim_{h \to 0} 2h d(a)\int_{t(h) - r}^{t(h) + r} \frac{A(t)(\Xi(t)- \Xi(t(h))}{((\Xi(t) - \Xi(t(h)))^2 + (d(a)h + o(h))^2)^2} f_h(t)
	\, d t \, .
\end{equation}
We change variables according to $t= t(h) + sh$:
\begin{align*}
	 C(h)&:=2\int_{-\frac{r}{h}}^{\frac{r}{h}}  s \left(\frac{A(t(h) + sh)}{h} \frac{\Xi(t(h) + sh )- \Xi(t(h))}{sh}\right) \times 
	\\&
	\left( s^2\left(\frac{\Xi(t(h) + sh)- \Xi(t(h))}{sh}\right)^2  + (d(a) + o(h)/h)^2 \right)^{-2}
	 f_h(t(h) + sh)\, d s
	 \, .
\end{align*}
Notice that, for any $s\in \mathbb{R}$, we have
\begin{equation}
	\lim_{h \to 0}\frac{\Xi(t(h) + sh)- \Xi(t(h))}{sh} = \Xi'(a) \, .
\end{equation}
Moreover the monotonicity of $\Xi$ implies that
\begin{equation}
	1/C \le \left| \frac{\Xi(t(h) + sh)- \Xi(t(h))}{sh} \right| \le C
	\quad \text{for any $s\in (-r/h, r/h)$} \, .
\end{equation}
Notice that
\begin{equation}
	\frac{A(t(h) + sh)}{h}
	=
	\frac{A(t(h) + sh)-A(t(h))}{h} + \frac{A(t(h)) - A(a)}{h} \, ,
\end{equation}
hence, up to extracting a subsequence $h_i \to 0$, we have
\begin{equation}
	\frac{A(t(h_i) +  sh_i)}{h_i} \to A'(a)s + x
\end{equation}
for some $x\in \mathbb{R}$ (recall that $|t(h)-a|\leq C h$ and note that $x$ might depend on the subsequence).

Collecting all the estimates above, and using the dominated convergence theorem we deduce that, along the subsequence $h_i$, 
\begin{equation}
	\lim_{i\to\infty} C(h_i)= 2\int_{\mathbb{R}} \frac{s(A'(a)s + x)\Xi'(a)}{(s^2 \Xi'(a)^2 + d(a)^2)^2} \varphi(a)^2\,  ds
	=
	\frac{2A'(a)\varphi(a)^2}{d(a) \Xi'(a)^2}\int_{\mathbb{R}} \frac{s^2}{(1+s^2)^2} ds \, .
\end{equation}
Observe that the limit does not depend on $x$ and hence does not depend on the chosen subsequence.

\medskip

\textbf{ Proof of \eqref{e:limit-III}.} Arguing as we did for $I_h$ and using that $g_h\to 0$ in $C^{1/2}$, as $h\to 0$, we easily deduce that
\begin{equation}
\lim_{h\to 0} \int III_h = 0 \, .
\end{equation}

\medskip

\textbf{ Proof of \eqref{e:limit-II}.}
We need to show that
\begin{equation}
	\lim_{h \to 0} \int_{\mathbb{R}} \frac{A(t)}{(\Xi(t) - \Xi(t(h))^2 + (d(a)h + o(h))^2} g_h(t)\,  dt = 0 \, .
\end{equation}
Observe that $G_h := g_h/|\varphi_h|^2 = |\varphi_h|^{-2} \textrm{Re}\, \varphi_h \textrm{Im}\, \varphi_h $, and in particular, $|G_h|\leq \frac{1}{2}$. Moreover there exists $r>0$ such that $G_h \to 0$ in $(a-r, a+r)$ in the $C^\alpha$ topology for every $\alpha<1$ (here, we are using that $|\varphi_h(a)|^2 \to |\varphi(a)|^2\neq 0$ as $h\to 0$).
We write
\begin{align}
	\int_\mathbb{R} & \frac{A(t)|\varphi_h(t)|^2 }{(\Xi(t) - \Xi(t(h)))^2 + (d(a)h + o(h))^2} G_h(t)\, dt 
   \\&	=
	\int_\mathbb{R} \frac{A(t)|\varphi_h(t)|^2 }{(\Xi(t) - \Xi(t(h))^2 + (d(a)h + o(h))^2} (G_h(t) - G_h(t(h))) dt \, ,
\end{align}
where we took advantage of the identity
\begin{equation}
		\int_\mathbb{R} \frac{A(t)|\varphi_h(t)|^2 }{(\Xi(t) - \Xi(t(h))^2 + (d(a)h + o(h))^2}\, d t = 0 \, ,
\end{equation}
proven in \eqref{e:imaginary-trick}.

Arguing as we did for $I_h$ we can reduce the problem to show
\begin{align}
	\lim_{h \to 0} & \int_{t(h) - r}^{t(h) + r} \frac{A(t)|\varphi_h(t)|^2 }{(\Xi(t) - \Xi(t(h)))^2 + (d(a)h + o(h))^2} (G_h(t) - G_h(t(h)))\, dt = 0\, .
\end{align}
We split the integral to the sum of 
\begin{align*}
J_1 (h) &:= \int_{t(h)-r}^{t(h)+r} \frac{(A(t)- A(t (h))|\varphi_h(t)|^2 }{(\Xi(t) - \Xi(t(h)))^2 + (d(a)h + o(h))^2} (G_h(t) - G_h(t(h)))\, dt\\
J_2 (h) &:= A(t (h))  \int_{t(h)-r}^{t(h)+r} \frac{|\varphi_h(t)|^2}{(\Xi(t) - \Xi(t(h)))^2 + (d(a)h + o(h))^2} (G_h(t) - G_h(t(h)))\, dt\, .
\end{align*}
Next observe that, in the interval that interests us, the following inequalities hold provided $r$ and $h$ are sufficiently small:
\begin{align*}
|A (t) - A(t(h))|&\leq C |t- t(h)|\\
|A(t(h)| &= |A (t (h))-A(a)|\leq C|t(h)-a|\leq C h\\
|G_h (t) - G_h (t (h))|&\leq \|G_h\|_{C^{1/2} (a-r, a+r)} |t-t(h)|^{1/2}\\
|\Xi (t) - \Xi (t(h))| &\geq C^{-1} |t-t(h)|\\
(d(a)h + o (h))^2 &\geq C^{-1} h^2\, .
\end{align*}
Since $\|\varphi_h\|_{L^\infty} \leq C$, we can change variable in the integrals to $\sigma =t-t(h)$ and estimate them as follows:
\begin{align*}
|J_1 (h)|&\leq C \|G_h\|_{C^{1/2} (a-r,a+r)} \int_{-r/2}^{r/2} \sigma^{-1/2}\, d\sigma \leq C\|G_h\|_{C^{1/2}} r^{1/2}\, ,\\
|J_2 (h)| &\leq C h \int_{-\infty}^\infty \frac{\sigma^{1/2}}{\sigma^2 + C^{-1} h^2}\, d\sigma 
= C h^{1/2} \int \frac{\tau^{1/2}}{\tau^2 + C^{-1}} d\tau \leq C h^{1/2}\, .
\end{align*}
Clearly $J_2 (h)\to 0$, while $J_1 (h) \to 0$ because $\|G_h\|_{C^{1/2} (a-r,a+r)} \to 0$.
\end{proof}

\chapter{Proofs of technical statements}
\section{Proof of Remark \ref{r:bounded}}
More generally, we will show here that, for any $q_0\in[1,2[$ and $q_1\in]2,\infty]$, it is true that \begin{equation*}\lVert K_2*\omega\rVert_{L^\infty(\R^2;\R^2)}\le C(q_0, q_1) (\lVert \omega\rVert_{L^{q_0}(\R^2)}+\lVert \omega\rVert_{L^{q_1}(\R^2)})\end{equation*} for all $\omega\in L^{q_0}\cap L^{q_1}$.
Indeed, passing to polar coordinates, one sees that $K_2\vert_{B_1}\in L^{q_1^*}(B_1; \R^2)$ and $K_2\vert_{\R^2\setminus B_1}\in L^{q_0^*}(\R^2\setminus B_1;\R^2)$, where $q_0^*, q_1^*$ are given by $\frac1{q_i}+\frac1{q_i^*}=1$ for $i\in\{0,1\}$. Hölder's inequality implies that for any $x\in\R^2$, 
	\begin{equation*}
	\begin{split}
		\abs{(K_2*\omega)(x)} &= \abs{((K_2 \mathbf 1_{B_1})*\omega)(x)+((K_2 (1-\mathbf 1_{B_1}))*\omega)(x)} \\
		&\le\norm{K_2}_{L^{q_1^*}(B_1)}\norm{\omega}_{L^{q_1}(\R^2)} + \norm{K_2}_{L^{q_0^*}(\R^2\setminus B_1)}\norm{\omega}_{L^{q_0}(\R^2)} \\
		&\le C(q_0, q_1) (\lVert \omega\rVert_{L^{q_0}(\R^2)}+\lVert \omega\rVert_{L^{q_1}(\R^2)}).
	\end{split}
	\end{equation*}
Since $x$ is arbitrary, this achieves a proof of the claim above.

\section{Proof of Theorem \ref{thm:Yudo}}\label{a:Yudo}
\textbf{Existence.} The existence argument is a classical density argument. Take any sequence $(\omega_0^{(n)})_{n\in\mathbb N}$ of functions in $L^1\cap C^\infty_c$ that converges strongly in $L^1$ to $\omega_0$. Analogously, pick a sequence of smooth functions $(f_n)_{n\in\mathbb N}$ in $C^\infty_c (\R^2\times[0, T])$ converging in $L^1(\R^2\times[0,T])$ to $f$ and satisfying the bound $\|f_n (\cdot, t)\|_{L^\infty} \leq \|f (\cdot, t)\|_{L^\infty}$ for a.e. $t$. Then, let $\omega^{(n)}$ denote the solution of the corresponding Cauchy problem of the Euler equations in vorticity form. The existence of such solutions is a classical well-known fact, see for instance \cite[Theorem A]{McGrath}. Following Remark \ref{r:A-priori-estimates}, these solutions satisfy all the \`a priori estimates needed in Proposition \ref{p:convergence}. Therefore, following the proof of Proposition \ref{p:convergence}, one obtains, in the limit $n\to\infty$, a solution $\omega\in L^\infty([0,T]; L^1\cap L^\infty)$ of the given Cauchy problem. Furthermore, since the à priori estimates of Remark \ref{r:A-priori-estimates} are uniform in $n$, one gets $K_2*\omega\in L^\infty([0,T]; L^2)$.

\begin{remark}
    The proof of Proposition \ref{p:convergence} has a fixed force $f$ but a straightforward adaptation of the arguments handles the case above, namely with a sequence of forces $(f_n)_{n\in\mathbb N}$ that converges in $L^1(\R^2\times[0,T])$ to a given $f$. More precisely the only difference occurs in the term $I_4 (k)$ of \eqref{e:term-I4}, which anyway enjoys convergence to the same limit.  
\end{remark}

\textbf{Uniqueness.} The uniqueness proof needs two important facts. The first is a well-known ODE inequality, whose short proof is given, for the reader's convenience, at the end of the section. 

\begin{lemma}\label{l:ODE lemma}
    Let $T>0$ and let $E:[0,T]\to[0,\infty[$ be a differentiable function satisfying
    \begin{equation}\label{e:ODE inequality for E}
        \dot E(t)\le p M E(t)^{1-1/p} \quad \text{ and }\quad E(0)=0
    \end{equation}
    for some fixed $M>0$. Then $E(t)\le (Mt)^p$ for all $t\in[0,T]$.
\end{lemma}

The second is the classical Calderón-Zygmund $L^p$ estimate, where we need the sharp $p$-dependence of the corresponding constant. This fact is also well known, cf. for instance \cite[Formula (8.45), page 322]{MajdaBertozzi}).

\begin{lemma}\label{l:Estimate on Lp norm of gradient of velocity}
For every $p_0>1$ there is a constant $c (p_0)$ with the following property. 
    If $v=K_2*\omega$ for some $\omega\in L^1 \cap L^p(\R^2)$ with $p\in [p_0, \infty[$, then $\norm{D v}_{L^p}\le p c \norm{\omega}_{L^p}$.
\end{lemma}

Now, let $v_1=K_2*\omega_1, v_2=K_2*\omega_2$ be two solutions of \eqref{e:Euler} satisfying the assumptions of Theorem \ref{thm:Yudo} and note that $w:=v_1-v_2$ solves
\begin{equation}\label{e:Gleichung fuer die Differenz}
    \partial_t w +(v_1\cdot\nabla)w +(w\cdot\nabla)v_2=-\nabla(p_1-p_2)
\end{equation}
(where $p_1, p_2$ are the pressures corresponding to $v_1$ and $v_2$). Clearly
\[
E(t):=\int_{\R^2} |w(x,t)|^2\,\mathrm dx \leq 2 \int_{\R^2} |v_1 (x,t)|^2\,\mathrm dx + 2 \int_{\R^2} |v_2 (x,t)|^2\,\mathrm dx <\infty 
\]
is a bounded function on $[0,T]$.

We scalar multiply \eqref{e:Gleichung fuer die Differenz} with $w$, integrate by parts and use the divergence free conditions of $v_1, v_2$, and $w$ to conclude and 
\begin{align*}
\dot E(t)  =  & - 2\int_{\R^2}((w\cdot\nabla)v_2)w\,\mathrm dx 
 \le 2\int_{\R^2}|w(x,t)|^2 \abs{D v_2(x,t)}\,\mathrm dx\\
 \le & 2\norm{\nabla v_2(\cdot, t)}_{L^p}\norm{w(\cdot, t)}_{L^\infty}^{2/p}\norm{w(\cdot,t)}_{L^2}^{2-2/p}\,.
\end{align*}
Using Remark \ref{r:bounded}, we also have
\begin{equation*}
\begin{split}
    &\sup_{t\in[0,T]} \norm{w(\cdot, t)}_{L^\infty} \le \sup_{t\in[0,T]}(\norm{v_1(\cdot, t)}_{L^\infty}+\norm{v_2(\cdot, t)}_{L^\infty}) \\
    &\hspace{0.7cm}\le C \sup_{t\in[0, T]}(\norm{\omega_1(\cdot, t)}_{L^1}+\norm{\omega_1(\cdot, t)}_{L^\infty}+\norm{\omega_2(\cdot, t)}_{L^1}+\norm{\omega_2(\cdot, t)}_{L^\infty}) <\infty\, .
\end{split}
\end{equation*}
Next fix any $p\geq 2$. From Lemma \ref{l:Estimate on Lp norm of gradient of velocity} and the classical $L^p$ interpolation we conclude
\begin{equation*}
    \norm{D v_2(\cdot, t)}_{L^p}\le p c \norm{\omega_2 (\cdot, t)}_{L^p}\le p c \norm{\omega_2 }_{L^\infty([0,T]; L^1)}^{1/p}\norm{\omega_2 (\cdot, t)}_{L^\infty([0,T]; L^\infty)}^{1-1/p}.
\end{equation*}
Therefore, $\dot E(t)\le p M_p E(t)^{1-1/p}$
with
\begin{align*}
    M_p &= 2\norm{w}_{L^\infty([0,T];L^\infty)}^{2/p} c \norm{\omega_2}_{L^\infty([0,T]; L^1)}^{1/p}\norm{\omega_2 }_{L^\infty([0,T]; L^\infty)}^{1-1/p}\\
    &\le 2c \left({\textstyle{\frac{1}{p}}}\norm{w}_{L^\infty([0,T];L^\infty)}^2\norm{\omega_2}_{L^\infty([0,T];L^1)}+\left(1-{\textstyle{\frac{1}{p}}}\right)\norm{\omega_2}_{L^\infty([0,T];L^\infty)}\right) \\
    &\le 2c (\norm{w}_{L^\infty([0,T];L^\infty)}^2\norm{\omega_2}_{L^\infty([0,T];L^1)}+\norm{\omega_2}_{L^\infty([0,T];L^\infty)})=: M<\infty\, .
\end{align*}
We can thus apply
Lemma \ref{l:ODE lemma} to obtain that $E(t)\le (M_p t)^p \leq (Mt)^p$.
In particular, for any $t\leq \frac{1}{2M}$ we have $E(t)\le \frac 1{2^p}$ and we can let $p$ tend to $\infty$ to infer $E(t)=0$. Since the same estimates apply to any translation $\tilde{E} (t):= E (t+t_0)$ of the function $E$, we immediately conclude that $E$ vanishes identically, namely that $v_1=v_2$ on $\mathbb R^2\times [0,T]$.

\begin{proof}[Proof of Lemma \ref{l:ODE lemma}] Fix an arbitrary $t_0\leq T$ and note that if $E(t_0)=0$ there is nothing to show. Hence assume $E(t_0)> 0$ and set $a:=\sup\{t:E(t)=0\text{ and }t\le t_0\}$ (note that the set is nonempty because $E (0)=0$). $E(a)=0$ by continuity of $E$ and clearly $E(t)>0$ for all $t\in]a,t_0]$. Therefore, we can divide \eqref{e:ODE inequality for E} by $E(t)^{1-1/p}$ to obtain that $\dot E(t) E^{1/p-1}(t)\le p M$
    for all $t\in]a, t_0]$. Integrating both sides gives
    \begin{equation}\label{e:Integral bound on E}
        \int_{a}^{t_0} \dot E(t) E^{1/p-1}(t)\le p M (t_0-a)\, .
    \end{equation} 
     But the left hand side equals $p E^{1/p}(t_0)-pE^{1/p}(a)=p E^{1/p}(t_0)$, from which we infer
    $E^{1/p}(t_0)\le M (t_0-a) \le M t_0$.
\end{proof}

\section{Proof of Proposition \ref{p:convergence}}

Recall first the following classical metrizability result of weak${}^*$ topologies of separable Banach spaces.
\begin{lemma}[Metrizability Lemma]\label{l:Metrizability}
    Let $X$ be a separable Banach space and let $K\subset X^*$ be weakly${}^*$-compact. Then $K$ is metrizable in the weak${}^*$ topology inherited from $X^*$ and a metric that induces this topology is given by 
    \begin{equation}\label{e:Metrization-of-weak-star-topology}
        d(l, \tilde l)=\sum_{n=1}^\infty 2^{-n}\min\{1, \vert l(x_n)-\tilde l(x_n)\vert\},
    \end{equation}
    where $(x_n)_{n\in \mathbb N}$ is any sequence in $X$ such that $\{x_n:n\in\mathbb N\}$ is dense in $X$.
\end{lemma}

Now on to the proof of Proposition \ref{p:convergence}. We will prove convergence of the $\omega_{\varepsilon, k}$ to a $\omega_\varepsilon$ for fixed $\varepsilon$ and $k\to\infty$ in the space $C([0, T]; K)$, where $K:=\{u\in L^q(\R^2):\lVert u\rVert_{L^q}\le R\}$ is equipped with the weak${}^*$ topology inherited from $L^q_{\text w}$. (We will talk about the choice of $q$ later.) Here, $R$ is the uniform bound obtained in \eqref{e:uniform_bound} of Corollary \ref{c:omega_k_epsilon}. Note that since every $L^q$ space is reflexive, one can work just as well with the weak topology on $K$. Let $(\phi_n)_{n\in\mathbb N}$ be a sequence of smooth functions such that $\{\phi_n:n\in\mathbb N\}$ is dense in every $L^q$. The metric given by \eqref{e:Metrization-of-weak-star-topology} now induces the topology of $K$, and it does not depend on $q$. Therefore, using the uniform bound \eqref{e:uniform_bound}, we conclude that \emph{the choice of $q$ does not matter}. It is sufficient to prove the statement of Proposition \ref{p:convergence} for one fixed $q\in]1, p]$ in order to prove it for all $q\in]1, p]$. 
\begin{claim}
    $\omega_{\varepsilon, k}$, seen as functions from $[0, T]$ to $K$, are equicontinuous (for simplicity we define each $\omega_{\varepsilon, k} (\cdot, t)$ on the interval $[0,t_k]$ as constantly equal to $\omega_{\varepsilon, k} (\cdot, t_k)$).
\end{claim}
\begin{proof}
    For $\tilde\omega, \hat\omega \in L^q(\R^2)$, let
    \begin{equation*}
        d_i(\tilde\omega, \hat \omega) \overset{\text{Def.}}=\left\lvert\int_{\R^2} (\tilde\omega-\hat\omega) \phi_i \mathrm dx\right\rvert.
    \end{equation*}
   Since each $\omega_{\varepsilon, k}$ solves the Euler equations in vorticity form, we can estimate
    \begin{align}
        &\;\;d_i(\omega_{\varepsilon, k}(t, \cdot), \omega_{\varepsilon, k}(s, \cdot))\nonumber\\
        &= \left\lvert\int_{\R^2}\int_s^t\partial_\tau\omega_{\varepsilon, k}(\sigma, x)\phi_i(x)\,d\sigma\, dx\right\rvert\nonumber\\
        &=\left\lvert\int_{\R^2}\int_s^t ((K_2*\omega_{\varepsilon, k})\cdot\nabla)\omega_{\varepsilon, k}(x, \sigma) + f (x, \sigma)\phi_i(x) \, d\sigma\,dx\right\rvert\nonumber\\
        &\le\lVert\nabla\phi_i\rVert_{L^\infty(\R^2)}\int_{\R^2}\int_s^t\lvert K_2*\omega_{\varepsilon, k}\rvert\lvert\omega_{\varepsilon, k}\rvert\, d\sigma\, dx\nonumber\\
        & \qquad + \lVert\phi_i\rVert_{L^\infty(\R^2)} \int_{\R^2}\int_s^t |f(x, \sigma)|\, d\sigma\, dx\nonumber\\
        &\le C(\lVert\nabla\phi_i\rVert_\infty + \|\phi_i\|_\infty) \lvert s-t\rvert^{\gamma}  \label{e:bound-on-ith-distance}
    \end{align}
    whenever $t\geq s \geq t_k$, where $\gamma = \gamma(\alpha,  \al)$. Indeed, the last inequality follows from $f\in L^r_tL^1_x$ for some $r=r(\alpha,\al)>1$ (see the proof of Lemma \ref{lem:Curl of tilde v}). Let $\tilde\varepsilon>0$. We can find a $N\in\mathbb N$ (depending on $\tilde\varepsilon$) such that $\sum_{n=N^+1}^\infty 2^{-n}\le\frac{\tilde\varepsilon}2$. If \begin{equation*}\lvert t-s\rvert\le \left(\frac{\tilde\varepsilon}{2NC\max_{i\in\{1,\dots,N\}}(\lVert\nabla\phi_i\rVert_{\infty} + \|\phi_i\|_\infty)}\right)^{\frac 1{\gamma}},\end{equation*} where $C$ is the constant from \eqref{e:bound-on-ith-distance}, then, by the bound in \eqref{e:bound-on-ith-distance}, we get
    \begin{equation*}
        d(\omega_{\varepsilon, k}(t, \cdot),\omega_{\varepsilon, k}(s, \cdot))\le\frac{\tilde\varepsilon}2+\sum_{i=1}^N \frac{\tilde\varepsilon}{2N}=\tilde\varepsilon.\qedhere
    \end{equation*}
\end{proof}

By the Banach-Alaoglu theorem, bounded subsets are relatively compact in the weak${}^*$ topology. Therefore, using reflexivity, for every $t\in[0,T]$, the bounded set $\{\omega_{\varepsilon, k}(\cdot, t): k\in\mathbb N\}$ is also relatively compact in $L^q_{\text w}$.

Therefore, using Arzelà-Ascoli, we can conclude that there exists a subsequence of $(\omega_{\varepsilon, k})_{k\in\mathbb N}$, not relabeled, that converges in $C([0, T]; L_{\text w}^q)$, for every $q$, to the same $\omega_\varepsilon\in C([0, T]; L_{\text w}^q)$.

\begin{claim}
    The $\omega_\varepsilon$ is a solution of the Euler equations in vorticity formulation.
\end{claim}

\begin{proof}
We have, for every $k\in\mathbb N$ and $\phi\in C_{\text c}^\infty(\R^2\times[0, T])$ with $\phi(\cdot, T)=0$, (cf. \eqref{e:distrib})
\begin{align}
0 = &\underbrace{\int_{\R^2} \omega_{\varepsilon, k}(x, t_k)\phi(x, t_k)\,\mathrm dx}_{=:I_1 (k)} + \underbrace{\int_{t_k}^T \int_{\R^2}\ \omega_{\varepsilon, k}(x,t)\partial_t\phi(x,t)\,\mathrm dx\,\mathrm dt}_{=:I_2 (k)}\nonumber\\
& +\underbrace{\int_{t_k}^T \int_{\R^2} \omega_{\varepsilon, k}(x,t)((K_2*_x\omega_{\varepsilon, k})(x,t)\cdot\nabla)\phi(x,t)\,\mathrm dx\,\mathrm dt}_{=:I_3 (k)}\nonumber\\
&+\underbrace{\int_{t_k}^T \int_{\R^2} f(x, t)\phi(x, t)\,\mathrm dx\,\mathrm dt}_{=:I_4 (k)} = 0\label{e:term-I4}\, .
\end{align}
The term $I_4(k)$ converges to 
\begin{equation*}
\int_{\R^2\times[0, T]} f(x, t)\phi(x,t)\,\mathrm dx\,\mathrm dt\, .
\end{equation*} 
By the convergence of the $\omega_{\varepsilon, k}$, \begin{equation*}\lim_{k\to\infty} I_2(k)=\int_{\R^2\times[0, T]} \omega_\varepsilon(x, t)\partial_t\phi(x,t)\,\mathrm dx\,\mathrm dt.\end{equation*} By the definition of the initial condition of $\omega_{\varepsilon, k}$ (cf. \eqref{e:Euler-later-times}), $\omega_{\varepsilon, k}(\cdot, t_k)$ converges strongly in $L^1(\R^2)$ to $\tilde\omega(\cdot, 0)=\omega_0=\omega_\varepsilon(\cdot, 0)$. Therefore, \begin{equation*}\lim_{k\to\infty} I_1(k)=\int_{\R^2}\omega_\varepsilon(x, 0)\phi(x, 0)\,\mathrm dx.\end{equation*}

It therefore only remains to prove the convergence of $I_3$, for which we will require yet another claim.

\begin{claim}
    For every $r\in[2, \infty[$ and every $t\in[0,T]$, the set $\{v_{\varepsilon, k}(\cdot, t): k\in\mathbb N\}$ is compact in $L^r (B_R)$ for every $R>0$.
\end{claim}
\begin{proof}
    From \eqref{e:uniform_bound}, we know that $\|v_{\varepsilon, k}(\cdot, t)\|_{L^2(\R^2)}\le C$ for some constant $C$ that is independent of $t$. Recall that $v_{\varepsilon, k}=\nabla^\bot\psi_{\varepsilon, k}$, where $\psi_{\varepsilon, k}$ solves $\Delta\psi_{\varepsilon, k} = \omega_{\varepsilon, k}$. Therefore, using the Calder\'{o}n-Zygmund inequality, one gets
    \begin{equation*}
        \norm{\nabla v_{\varepsilon, k}(\cdot, t)}_{L^2}\le C\norm{\omega_{\varepsilon, k}(\cdot, t)}_{L^2}.
    \end{equation*}
    Since the $L^2$ norms of the $\omega_{\varepsilon, k}(\cdot, t)$ are uniformly bounded, we can conclude that 
    \begin{equation*}
        \sup_{k\in\infty} \norm{\nabla v_{\varepsilon, k}(\cdot, t)}_{L^2}<\infty.
    \end{equation*}
Hence we conclude the compactness in $L^r (B_R)$ from Rellich's Theorem.    
\end{proof}

Therefore, the $v_{\varepsilon, k}(\cdot, t)$ converge to $v_\varepsilon(\cdot, t)$ strongly in every $L^r (B_R)$ with $r\in[2,\infty[$. Moreover, thanks to \eqref{e:uniform_bound}, we can apply the dominated convergence theorem by Lebesgue to conclude that $v_{\varepsilon, k}\to v_\varepsilon$ as $k\to\infty$ in the space $L^1([0, T]; L^r (B_R))$ for every $r\in[2,\infty[$.

By definition,
\begin{equation*}
    \omega_{\varepsilon, k} (v_{\varepsilon, k}\cdot\nabla)\phi-\omega_{\varepsilon} (v_{\varepsilon}\cdot\nabla)\phi = \omega_{\varepsilon, k} (v_{\varepsilon, k}-v_{\varepsilon})\cdot \nabla \phi +
    (\omega_{\varepsilon,k} - \omega_\varepsilon) v_\varepsilon \cdot \nabla \phi\, .
\end{equation*}
We thus rewrite 
\begin{align}
I_3 (k) &= \int_0^T \int_{B_R} \omega_{\varepsilon, k} (v_{\varepsilon,k}-v_\varepsilon)\cdot \nabla \phi\, dx\, dt
+ \int_0^T \underbrace{\int_{B_R} (\omega_{\varepsilon,k} - \omega_\varepsilon) v_\varepsilon \cdot \nabla \phi\, dx}_{=:J_k(t)}\, dt\, .\label{e:I3-converges-to-0}
\end{align}
Observe first that, for each fixed $t$,
\[
\lim_{k\to\infty} J_k (t) = 0\, ,
\]
since $\omega_{\varepsilon, k} (\cdot, t) - \omega_\varepsilon (\cdot, t)$ converges weakly to $0$ in $L^2$, while $v_\varepsilon (\cdot, t)\cdot \nabla \phi (\cdot, t)$ is a fixed $L^2$ function. On the other hand
\[
|J_k (t)|\leq \|\nabla \phi (\cdot, t)\|_{L^\infty} (\|\omega_{\varepsilon, k} (\cdot, t)\|_{L^2} + \|\omega_\varepsilon (\cdot, t)\|_{L^2}) \|v_\varepsilon (\cdot, t)\|_{L^2}\, .
\]
Therefore the second integral in \eqref{e:I3-converges-to-0} converges to $0$. The first integral can be bounded by
\[
\|\nabla \phi\|_{L^\infty} \|v_{\varepsilon, k} - v_\varepsilon\|_{L^1 ([0,T], L^2 (B_R))} \|\omega_{\varepsilon, k}\|_{L^\infty ([0,T], L^2 (B_R))}
\]
and converges to $0$ as well.
\end{proof}

\section{Proof of Lemma \ref{l:extension}}

Consider $\vartheta \in L^2_m\cap \mathscr{S}$ for $m\geq 2$ and let $v:= K_2*\vartheta$. We first claim that
\begin{equation}\label{e:average}
\int_{B_R} v = 0 \qquad \qquad \mbox{for every $R>0$.}
\end{equation}
With \eqref{e:average} at our disposal, since $\|Dv\|_{L^2 (\mathbb R^2)} = \|\vartheta\|_{L^2 (\mathbb R^2)}$, we use the Poincar\'e inequality to conclude
\begin{equation}
R^{-1} \|v\|_{L^2 (B_R)} + \|Dv\|_{L^2 (B_R)} \leq C \|\vartheta\|_{L^2 (\mathbb R^2)}   
\end{equation}
for a geometric constant $C$. This is then enough to infer the remaining conclusions of the lemma.

In order to achieve \eqref{e:average} observe first that $v = \nabla^\perp h$, where $h$ is the unique potential-theoretic solution of $\Delta h = \vartheta$, given by $h = K * \vartheta$ with $K (x) = \frac{1}{2\pi} \log |x|$. Since $K(R_\theta x) = K (x)$ and $\vartheta (x) = \vartheta (R_{2\pi/m} x)$, it follows that $h (R_{2\pi/m} x) = h (x)$, i.e. $h$ is $m$-fold symmetric. Therefore $R_{-2\pi/m} \nabla h (R_{2\pi/m} x) = \nabla h (x)$. In particular, integrating in $x$ and using that the rotation is a measure-preserving transformation of the disk, we conclude
\[
\int_{B_R} \nabla h = R_{2\pi/m} \int_{B_R} \nabla h \, ,
\]
and thus,
\[
\int_{B_R} \nabla h = \frac{1}{m} \sum_{k=0}^{m-1} R_{2k\pi/m} \int_{B_R} \nabla h \, .
\]
However, since $m\ge 2$, $\sum_{k=0}^{m-1} R_{2k\pi/m} = 0$, showing that $\int_{B_R} \nabla h = 0$.

\begin{remark}\label{r:Camillo_dumb}
We next show that it is not possible to find a continuous extension of the operator $L^2\cap \mathscr{S} \ni \vartheta \mapsto K_2* \vartheta \in \mathscr{S}'$ to the whole $L^2$. First of all we observe that, if such an extension exists, it then needs to coincide with $K_2* \vartheta$ when $\vartheta \in L^1 \cap L^2$. We next exhibit a sequence of divergence free vector fields $\{v_k\}\subset W^{1,1}\cap W^{1,2}$ with the property that $\omega_k = \curl v_k$ converge to $0$ strongly in $L^2$ but $v_k$ converge locally to a constant vector field $v_0\neq 0$. In order to do this, we first define the following functions $\phi_k$ on the positive real axis:
\[
\phi_k (r) :=
\left\{
\begin{array}{ll}
1 + \frac{3}{4\ln k} \qquad &\mbox{for $r\leq \frac{k}{2}$}\\ \\
1 + \frac{3}{4\ln k} - \frac{1}{k^2 \ln k} \left(r-\frac{k}{2}\right)^2\qquad &\mbox{for $\frac{k}{2} \leq r \leq k$}\\ \\
1 + \frac{1}{2 \ln k}- \frac{1}{\ln k} \ln \frac{r}{k} \qquad  & \mbox{for $k\leq r \leq k^2$}\\ \\
\frac{1}{2 k^4 \ln k} (r-2k^2)^2\qquad & \mbox{for $k^2 \leq r\leq 2k^2$}\\ \\
0 \qquad &\mbox{for $r\geq 2 k^2$}\, .
\end{array}
\right.
\]
Observe that $\phi_k$ is $C^1$ and its derivative is Lipschitz. Next we define the stream functions 
\[
\psi_k (x) = - \phi_k (|x|) v_0^\perp \cdot x 
\]
and the vector field $v_k (x) = \nabla^\perp \psi_k (x)$. By construction $v_k$ is divergence free, compactly supported, and Lipschitz. In particular, it belongs to $W^{1,p}$ for every $p$. Moreover, $v_k$ equals $(1+\frac{3}{4\ln k}) v_0$ on $B_{k/2}$ and it thus follows that, as $k\to \infty$, $v_k$ converges locally to the constant vector field $v_0$. It remains to check that $\curl v_k = \Delta \psi_k$ converges to $0$ strongly in $L^2$. We compute 
\[
\Delta \psi_k = - \underbrace{v_0^\perp \cdot x\, \Delta (\phi_k (|x|))}_{=:f_k} - \underbrace{\nabla (\phi_k (|x|))\cdot v_0^\perp}_{=: g_k}\, 
\]
and we seek to bound $f_k$ and $g_k$ pointwise. For what concerns $f_k$ observe that $\Delta\phi_k$ vanishes on $|x|\leq \frac{k}{2}$, $k\leq |x| \leq k^2$, and $2k^2 \leq |x|$. On the remaining regions, using the formula for the Laplacian in polar coordinates, we can estimate
\[
|f_k (x)|\leq |v_0| |x| (|\phi'' (|x|)| + |x|^{-1} |\phi' (|x|)|)\, .
\]
In particular, we conclude
\[
|f_k (|x|)| \leq \frac{C}{|x| \ln k}\, , 
\]
for a constant $C$ independent of $k$.
As for $g_k$, it vanishes for $|x|\leq \frac{k}{2}$ and $|x|\geq k^2$, and where it does not vanish we have the estimate
\[
|g_k (x)|\leq |v_0| |\psi' (|x|)| \leq \frac{C}{|x|\ln k}\, ,
\]
again for a constant $C$ independent of $k$. Passing to polar coordinates, we can thus estimate
\begin{align*}
\|\Delta \psi_k\|^2_{L^2 (\mathbb R^2)} \leq & \frac{C}{(\ln k)^2} \int_{k/2}^{2k^2} \frac{1}{r} \, dr = 
\frac{C}{(\ln k)^2} \left(\ln (2k^2) - \ln {\textstyle{\frac{k}{2}}}\right)
= \frac{C \ln (4k)}{(\ln k)^2}\, .
\end{align*}
\end{remark}

\printindex
\printbibliography

\end{document}